\documentclass[10pt,a4paper]{article}
\usepackage{amssymb}
\usepackage{amsmath}
\usepackage{theorem}
\usepackage{epsfig}
\usepackage{epstopdf}
\usepackage{verbatim}
\usepackage{graphicx}
\usepackage{subfigure}
\usepackage{enumerate}
\usepackage{enumitem}
\usepackage{color}
\usepackage[a4paper,left=3cm,right=2cm,top=2.5cm,bottom=2.5cm]{geometry}

\usepackage{amssymb}

\newcommand{\R}{{\mathbb R}}

\newcommand{\C}{\mathbb{C}}
\newcommand{\Z}{\mathbb{Z}}
\newcommand{\I}{\mathbb{I}}
\newcommand{\T}{\mathbb{T}}

\newcommand{\grad}{\nabla}

\newcommand{\la}{\lambda}

\newcommand{\al}{\alpha}
\newcommand{\ep}{\varepsilon}

\newcommand{\pa}{\partial}
\newcommand{\va}{\varphi}

\newcommand{\ut}{\tilde{u}}
\newcommand{\tX}{\tilde{X}}

\newcommand{\Om}{\Omega}

\newcommand{\tv}{\tilde{v}}
\newcommand{\p}[1]{\left(#1\right)}

\newcommand{\gradj}{\grad_{J}}
\newcommand{\htn}[2]{\left|\left|#1\right|\right|_{H_{(0)}^{ht,#2}}}
\newcommand{\hhn}[3]{\left|\left|#1\right|\right|_{H^{#2}_{(0)}H^{#3}}}
\newcommand{\li}[2]{\left|\left|#1\right|\right|_{L^\infty H^{#2}}}
\newcommand{\lit}[2]{\left|\left|#1\right|\right|_{L^\infty_{1/4}H^{#2}}}
\newcommand{\htb}[2]{\left|#1\right|_{H_{(0)}^{ht,#2}}}
\newcommand{\hhb}[3]{\left|#1\right|_{H_{(0)}^{#2}H^{#3}}}
\newcommand{\fn}[1]{\left|\left|#1\right|\right|_{F^{s+1}}}
\newcommand{\lhn}[2]{\left|\left|#1\right|\right|_{L^2H^{#2}}}
\newcommand{\hln}[2]{\left|\left|#1\right|\right|_{H_{(0)}^{#2}L^2}}
\newcommand{\kt}[1]{\left|\left|#1\right|\right|_{\overline{H}_{(0)}^{ht,s}}}
\newcommand{\lhb}[2]{\left|#1\right|_{L^2H^{#2}}}
\newcommand{\hlb}[2]{\left|#1\right|_{H_{(0)}^{#2}L^2}}
\newcommand{\lib}[2]{\left|#1\right|_{L^\infty H^{#2}}}

\newcommand{\xn}{X^{(n)}}
\newcommand{\xm}{X^{(n-1)}}
\newcommand{\vn}{v^{(n)}}
\newcommand{\vm}{v^{(n-1)}}
\newcommand{\zn}{\zeta^{(n)}}
\newcommand{\zm}{\zeta^{(n-1)}}
\newcommand{\wn}{w^{(n)}}
\newcommand{\wm}{w^{(n-1)}}
\newcommand{\qn}{q^{(n)}}
\newcommand{\qm}{q^{(n-1)}}

\newtheorem{thm}{Theorem}[section]
\newtheorem{cor}[thm]{Corollary}
\newtheorem{defi}[thm]{Definition}
\newtheorem{rem}[thm]{Remark}

\newtheorem{proposition}[thm]{Proposition}
\newtheorem{lemma}[thm]{Lemma}

\newenvironment{proof}{\begin{trivlist} \item[] {\em Proof:}}{\hfill $\Box$
                      \end{trivlist}}

\newenvironment{prooflemma}[1]{\begin{trivlist} \item[] {\em Proof of Lemma \ref{#1}:}}{\hfill $\Box$
                       \end{trivlist}}

\newenvironment{proofthm}[1]{\begin{trivlist} \item[] {\em Proof of Theorem \ref{#1}:}}{\hfill $\Box$
                       \end{trivlist}}

\title{Splash singularities for the free boundary Navier-Stokes equations}
\date{\today}
\author{ Angel Castro, Diego C\'ordoba, Charles Fefferman, \\ Francisco Gancedo and Javier G\'omez-Serrano}

\begin{document}

\maketitle

\begin{abstract}

In this paper, we prove the existence of smooth initial data for the 2D free boundary incompressible Navier-Stokes equations, for which the smoothness of the interface breaks down in finite time into a splash singularity.


\end{abstract}

\section{Introduction}

In this paper, we prove that an initially smooth solution of the 2D water wave equation with non-zero viscosity may break down in finite time by forming a splash singularity, see Figure \ref{figcharlie1}.

The analogous result for inviscid water waves was proven in our previous paper \cite{Castro-Cordoba-Fefferman-Gancedo-GomezSerrano:finite-time-singularities-free-boundary-euler}. The strategy there was to start with a splash configuration and solve backwards in time. To do so, we first made a conformal map (essentially a branch of the square root) $P(z)$ from physical space to the ``tilde domain'' and then adapted the proof of Ambrose-Masmoudi \cite{Ambrose-Masmoudi:zero-surface-tension-2d-waterwaves} (see also \cite{Cordoba-Cordoba-Gancedo:interface-water-waves-2d}) of short time existence of solutions of the inviscid water wave equation.

The strategy of \cite{Castro-Cordoba-Fefferman-Gancedo-GomezSerrano:finite-time-singularities-free-boundary-euler} cannot work for the present case of nonzero viscosity, since the equations cannot be solved backwards in time. We will instead make use of the transformation to the tilde domain in a new way, which we explain below.

We refer the reader to the further historical discussion at the end of the introduction, including references to alternate proofs by Coutand-Shkoller of several of our results.

\begin{figure}
\centering
\includegraphics[scale=0.7]{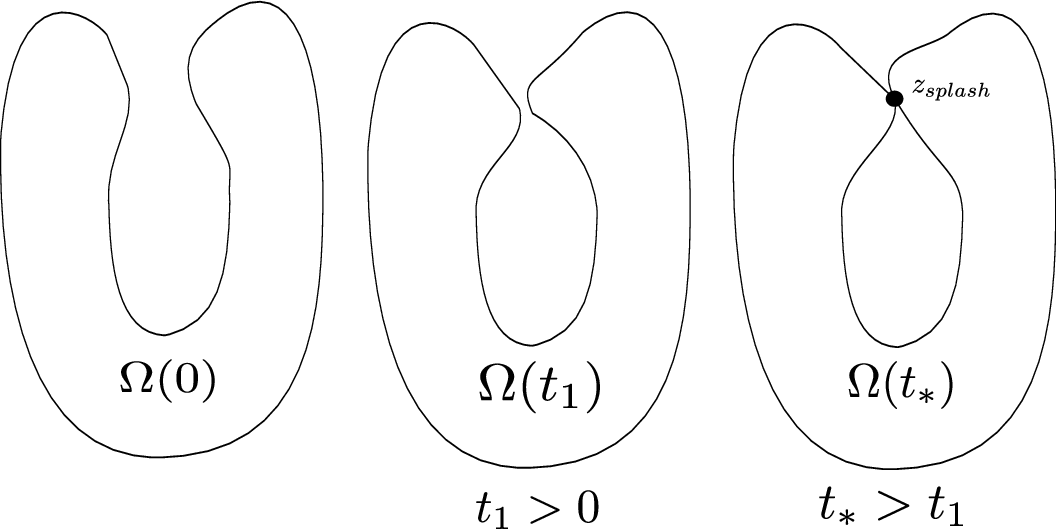}
\caption{How a splash forms.}
\label{figcharlie1}
\end{figure}

Let us first present the inviscid and viscous water wave equations, next define a splash singularity, and then state our main result.

The 2D water wave equations govern a system consisting of water, vacuum and the interface between them. At time $t \in \R$, the water occupies a region $\Omega(t) \subset \R^{2}$, and the vacuum occupies the complementary region $\R^{2} \setminus \Omega(t)$. For points $x$ in the water region $\Omega(t)$, the velocity of the water at position $x$ and time $t$ is $u(x,t) \in \R^{2}$, and the pressure is $p(x,t) \in \R$. Thus, $u(x,t), p(x,t)$ are defined only for $x \in \Omega(t)$; finding $\Omega(t)$ is part of the problem.

We assume here that the interface $\pa \Omega(t) \subset \R$ is a smooth simple closed curve, which we write in parametric form:

\begin{align*}
\pa \Omega(t) = \{z(\al,t): \al \in \R/\Z\},
\end{align*}

where $z: \R/\Z \to \R^{2}$ is smooth and satisfies the chord-arc condition

\begin{align*}
|z(\al,t) - z(\al',t)| \geq CA(t)\|\al-\al'\| \text{ for } \al,\al' \in \R/\Z.
\end{align*}

Here, $CA(t) > 0$ is the ``chord-arc constant'', and $\|\al-\al'\|$ denotes the distance from $\al$ to $\al'$ in $\R/\Z$.

The parametrization of the interface has no physical meaning, and can be picked to simplify our analysis.

The inviscid water wave equations are as follows:

\begin{align}
(\pa_{t} + u \cdot \nabla_{x}) u(x,t) & = -\nabla_{x} p(x,t) \text{ for } x \in \Omega(t) \nonumber \\
\text{div} u(x,t) & = 0 \text{ for } x \in \Omega(t) \nonumber \\
\text{curl} u(x,t) & = 0 \text{ for } x \in \Omega(t) \nonumber \\
p(x,t) & = 0 \text{ for } x \in \pa \Omega(t) \nonumber \\
\pa_{t} z(\al,t) & = u(z(\al,t),t) + c(\al,t) \pa_{\al}z(\al,t) \text{ for } \al \in \R/\Z. \label{inviscidww}
\end{align}

(The last equation asserts that the interface moves with the fluid. The function $c(\al,t)$ affects only the parametrization of the interface, and may be chosen arbitrarily).

The initial conditions for inviscid water waves are as follows:

\begin{itemize}
\item $\Omega(0) = \Omega_{0}$ (a given domain bounded by a smooth simple closed curve).
\item $u(x,0) = u_0$ (a given smooth divergence-free irrotational vector field) for $x \in \Omega_{0}$.
\end{itemize}

For water waves with nonzero viscosity, the relevant equations take the following form in suitable units:

\begin{align}
(\pa_{t} + u \cdot \nabla_{x}) u(x,t) & = \Delta_{x} u(x,t) -\nabla_{x} p(x,t) \text{ for } x \in \Omega(t) \nonumber \\
\text{div} u(x,t) & = 0 \text{ for } x \in \Omega(t) \nonumber \\
\left(p\, \mathbb{I}-\left(\nabla u + \left(\nabla u\right)^* \right)\right)n & = 0, \text{ for } x \in \pa \Omega(t) \nonumber \\
\pa_{t} z(\al,t) & = u(z(\al,t),t) + c(\al,t) \pa_{\al}z(\al,t) \text{ for } \al \in \R/\Z. \label{viscousww}
\end{align}

Again, $c(\al,t)$ may be chosen arbitrarily.

The initial conditions are:

\begin{itemize}
\item $\Omega(0) = \Omega_{0}$ (as before).
\item $u(x,0) = u_0$ for $x \in \Omega_{0}$, where $u_0$ is a given smooth divergence-free vector field on $\Omega_{0}$, satisfying the constraint
\item $n_0^{\bot}\left(\left(\nabla u_{0} + \left(\nabla u_{0}\right)^* \right)\right)n_0  = 0$ on $\pa \Omega_{0}$.
\end{itemize}

Next, we adapt from \cite{Castro-Cordoba-Fefferman-Gancedo-GomezSerrano:finite-time-singularities-free-boundary-euler} the definition of a splash singularity for the compact case:

Note that the inviscid water wave equations \eqref{inviscidww} have a symmetry under time reversal, but the viscous equations \eqref{viscousww} have no such symmetry. This reflects the presence of the Euler equation in \eqref{inviscidww} and the Navier-Stokes equation in \eqref{viscousww}.

\begin{defi}
\label{defsplash}
We say that $z(\al) = (z_1(\al),z_2(\al))$ is a \emph{splash curve} if
\begin{enumerate}
\item $z_{1}(\al), z_2(\al)$ are smooth functions and $2\pi$-periodic.
\item $z(\al)$ satisfies the arc-chord condition at every point except at $\alpha_1$ and $\alpha_2$, with $\alpha_1 < \alpha_2$ where $z(\al_1) = z(\al_2)$ and $|z_{\al}(\al_1)|, |z_{\al}(\al_2)| > 0$. This means $z(\al_1) = z(\al_2)$, but if we remove either a neighborhood of $\al_1$ or a neighborhood of $\al_2$ in parameter space, then the arc-chord condition holds.
\item The curve $z(\alpha)$ separates the complex plane into two regions; a  connected water region and a vacuum region (not necessarily connected). We choose the parametrization such that the normal vector $n=\frac{(-\pa_\alpha z_2(\alpha), \pa_\alpha z_1(\alpha))}{|\pa_\alpha z(\alpha)|}$ points to the vacuum region. We regard the interface to be part of the water region.

\item We can choose a point $c$ outside the water region and a single-valued branch of the function $P(z) = \sqrt{z-c}$ on the water region with the following properties:

The image of the water region under $P$ is bounded by a curve $\tilde{z}(\al) = (\tilde{z}_1(\al),\tilde{z}_2(\al)) = P(z(\al))$ satisfying:
\begin{enumerate}
\item $\tilde{z}_1(\al)$ and $\tilde{z}_2(\al)$ are smooth and $2\pi$-periodic.
\item $\tilde{z}$ is a closed contour.
\item $\tilde{z}$ satisfies the arc-chord condition.
\end{enumerate}
\end{enumerate}
\end{defi}

See Figure \ref{figcharlie2} and Figure \ref{figcharlie5} for examples of splash and non-splash curves. Although we have taken the slit $\Gamma$ in Figure \ref{figcharlie2} to be a half-line, we could just as well have taken any smooth arc joining the origin to infinity, passing through the splash point but otherwise avoiding the water region.

The referee points out another type of splash scenario indicated in Figure \ref{boceto}. Our proof can be easily adapted to this scenario by replacing $P(z) = \sqrt{z-c}$ by a branch of $\sqrt{\frac{z-a}{z-b}}$ with suitable $a$ and $b$.

\begin{figure}[h!]\centering
\includegraphics[scale=0.5]{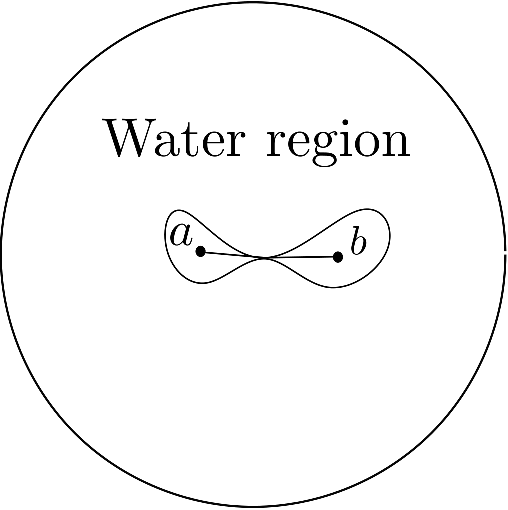}
\caption{An additional splash scenario.}
\label{boceto}
\end{figure}


\begin{figure}[ht]
\centering
\subfigure
{
\includegraphics[width=0.45\textwidth, height=0.3\textwidth]{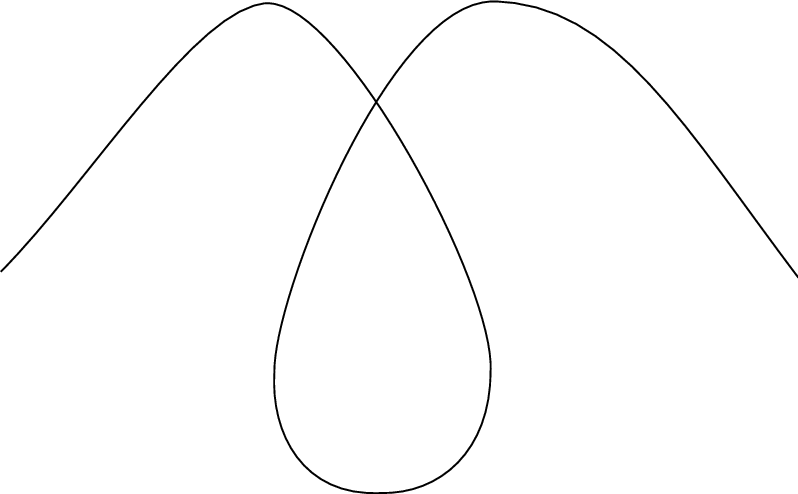}
\label{figcharlie5a}
}
\subfigure
{
\includegraphics[width=0.45\textwidth, height=0.3\textwidth]{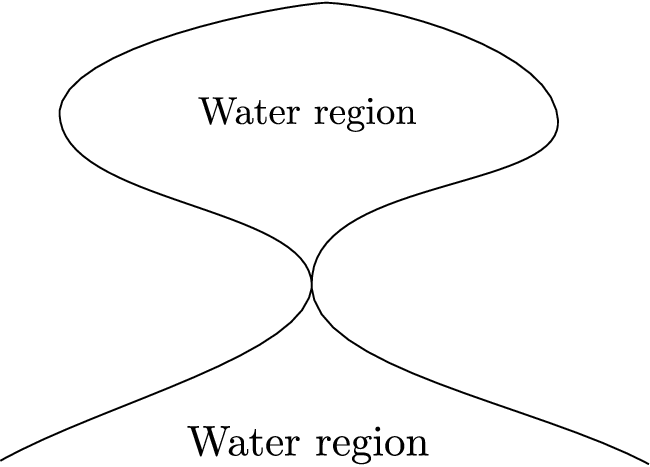}
\label{figcharlie5b}
}
\caption{Two examples of non splash curves.}
\label{figcharlie5}
\end{figure}


We can now state the main result of this paper:

\begin{thm}
There exists a solution of the viscous water wave equation that remains smooth for times $t \in [0,t_{*})$ but forms a splash singularity at time $t_{*}$.
\end{thm}

Next, we recall from \cite{Castro-Cordoba-Fefferman-Gancedo-GomezSerrano:finite-time-singularities-free-boundary-euler} how to produce the inviscid water waves \footnote{We treat here the case in which the water region $\Omega(t) \subset \R^{2}$ is a bounded region. In \cite{Castro-Cordoba-Fefferman-Gancedo-GomezSerrano:finite-time-singularities-free-boundary-euler}, we studied the case in which $\Omega(t)$ is periodic with respect to horizontal translation as in Figure \ref{figcharliex}. In this introduction we ignore the distinction between the compact and periodic cases.} that end in a splash at time $t_{*} > 0$. We start with the splash $\Omega(t_{*}), u(\cdot, t_{*})$ and solve the inviscid equations \eqref{inviscidww} backwards in time. It is well known that the inviscid equations \eqref{inviscidww} can be solved (forward or backwards in time) starting from smooth initial data (See S. Wu \cite{Wu:well-posedness-water-waves-2d} and \cite{Lannes:water-waves-book} for a comprehensive list of references) The difficulty here is that the initial $\Omega(t_{*})$ is singular. To overcome this difficulty, we make a slit $\Gamma$ in the complex plane as in Figure \ref{figcharlie2}, and then make the conformal mapping $\tilde{z} = P(z)$ for $z \in \C \setminus \Gamma$; here $P(z)$ is a branch of $\sqrt{z}$. The inverse map is simply $P^{-1}(\tilde{z}) \equiv \tilde{z}^{2}$, which of course is well defined and smooth on the whole complex plane. We remark that we can apply this procedure to any splash curve (see Definition \ref{defsplash}) simply picking the conformal map $P(z) = \sqrt{z-c}$ with a suitable $c \in \mathbb{C}$ and choosing a branch of the square root that separates the two splash points.

\begin{figure}
\centering
\includegraphics[scale=0.7]{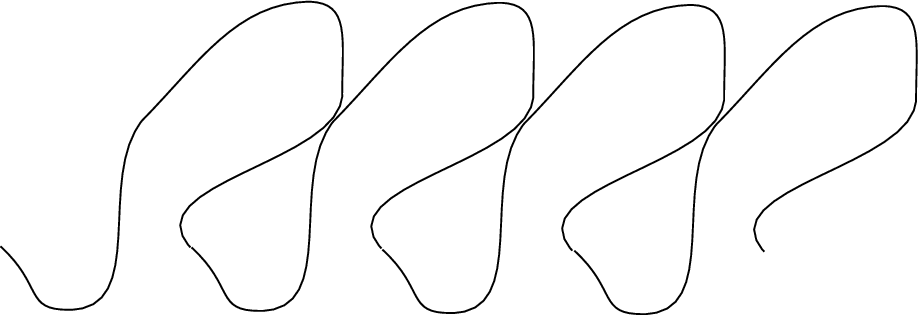}
\caption{Horizontally periodic setting.}
\label{figcharliex}
\end{figure}

\begin{figure}[ht]
\centering
\subfigure[The $z$-plane.]
{
\includegraphics[width=0.25\textwidth]{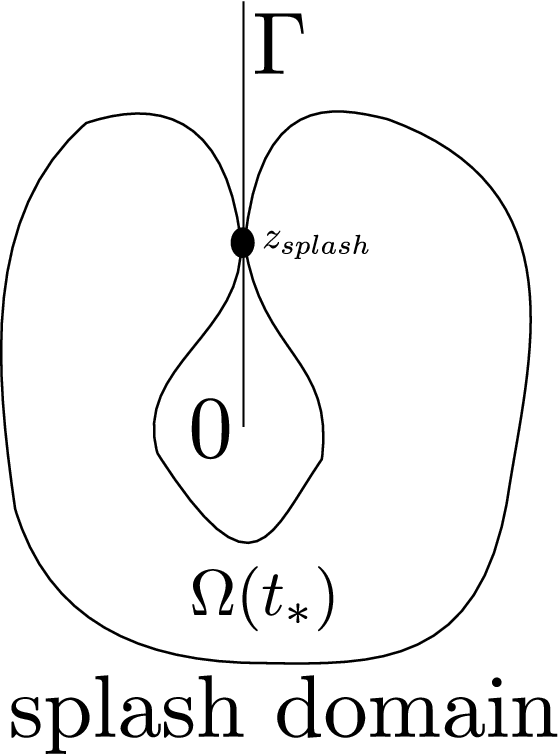}
\label{figcharlie2a}
}
\subfigure[The $\tilde{z}$-plane.]
{
\includegraphics[width=0.45\textwidth]{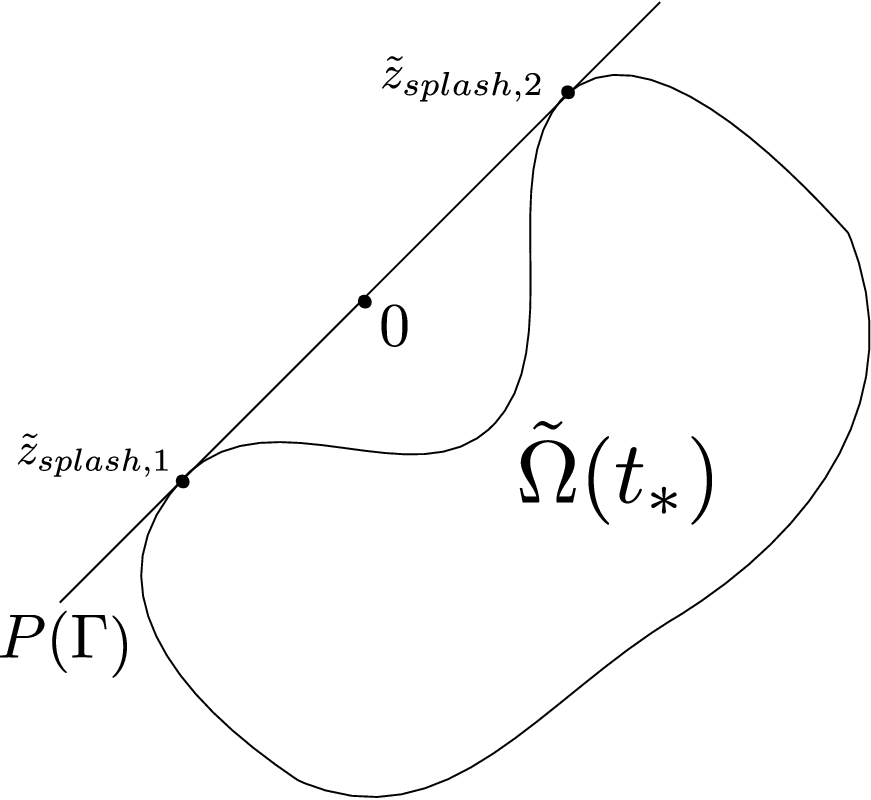}
\label{figcharlie2b}
}
\caption{Desingularization of the splash domain.}
\label{figcharlie2}
\end{figure}

\begin{figure}[ht]
\centering
\subfigure[ ]
{
\includegraphics[width=0.45\textwidth]{pictures/Figure2b_tilted.eps}
\label{figcharlie3a}
}
\subfigure[ ]
{
\includegraphics[width=0.45\textwidth]{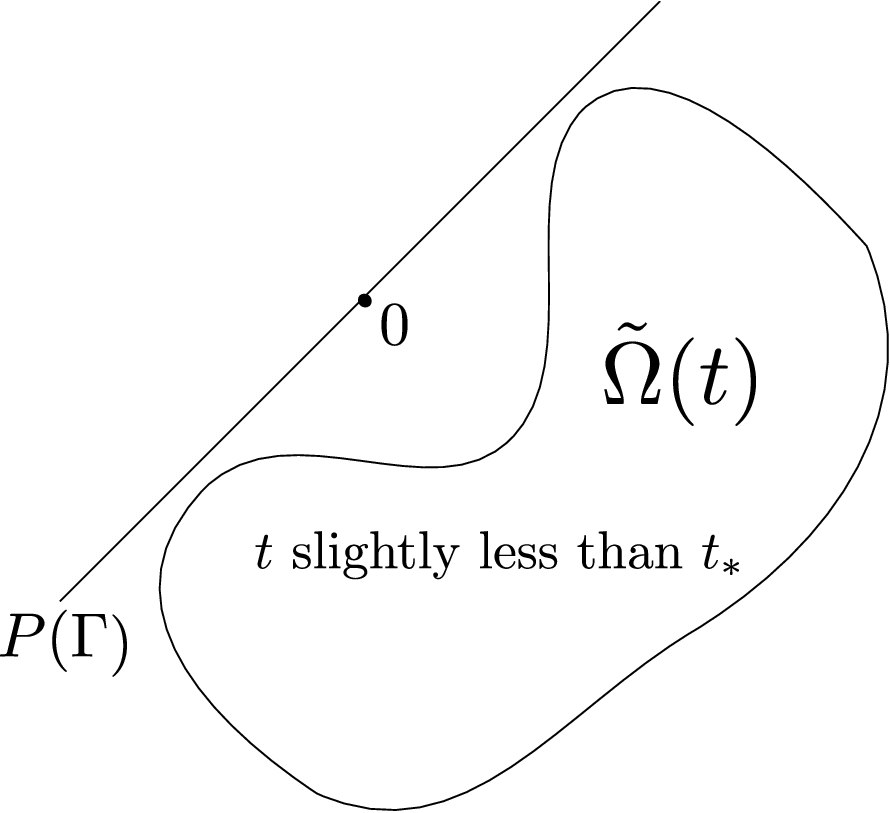}
\label{figcharlie3b}
}
\label{figcharlie3}
\caption{Evolution of a splash in the tilde world for $t = t_{*}$ and $t < t_{*}$.}
\end{figure}

\begin{figure}[ht!]
\centering
\subfigure[ ]
{
\includegraphics[width=0.65\textwidth]{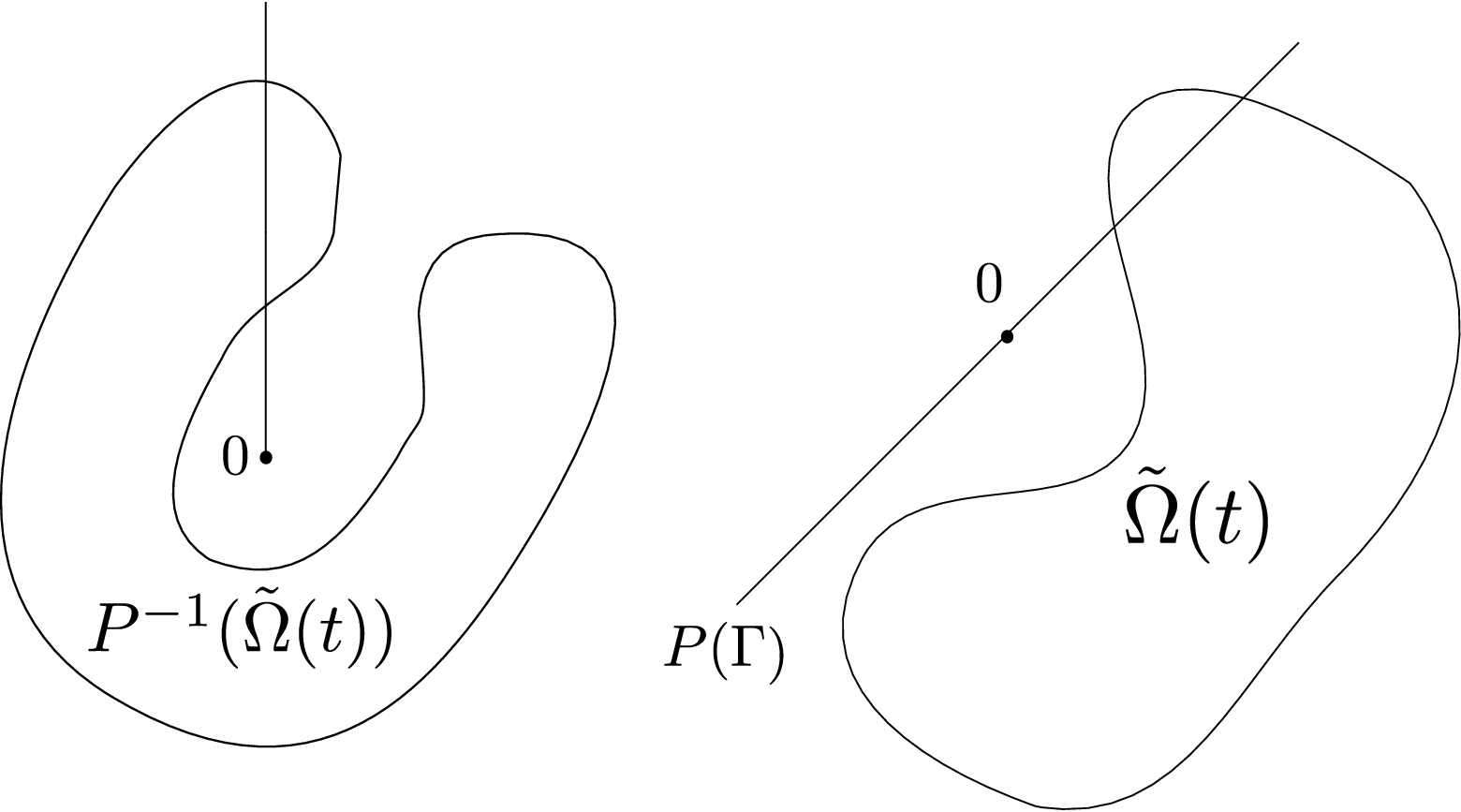}
\label{figcharlie4a}
}
\subfigure[ ]
{
\includegraphics[width=0.65\textwidth]{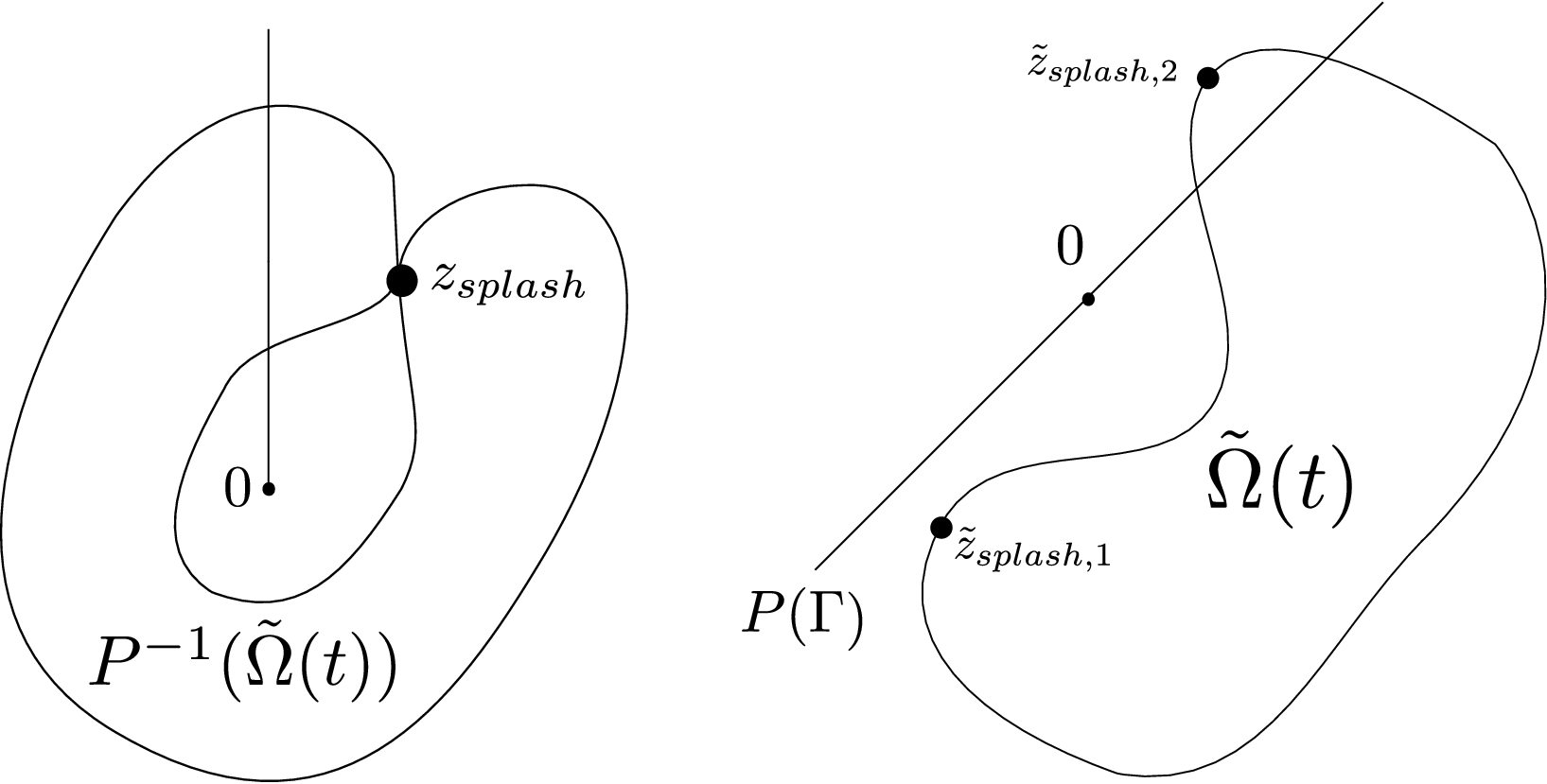}
\label{figcharlie4b}
}
\subfigure[ ]
{
\includegraphics[width=0.65\textwidth]{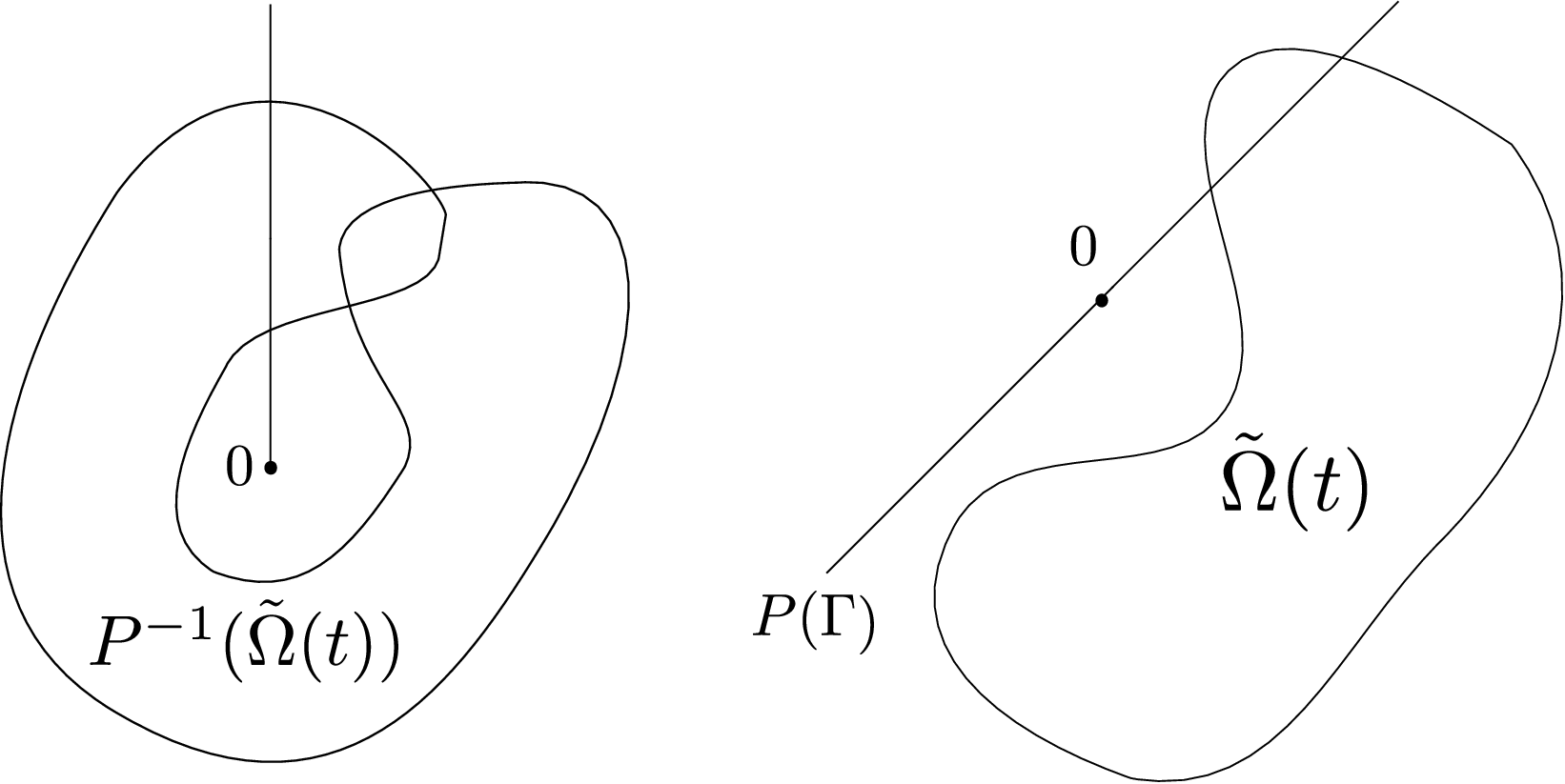}
\label{figcharlie4c}
}
\label{figcharlie4}
\caption{Possibilities for $P^{-1}(\tilde{\Omega}(t))$.}
\end{figure}

We want a solution of \eqref{inviscidww} for which $\Omega(t_{*})$ is as in Figure \ref{figcharlie2}, but for times $t < t_{*}$ ($t$ close to $t_{*}$), $\Omega(t)$ avoids the slit $\Gamma$. The corresponding domains $\tilde{\Omega}(t)$ in the tilde domain (i.e. $\tilde{\Omega}(t) = P(\Omega(t))$ ) behave as in Figure \ref{figcharlie3}.

In \cite{Castro-Cordoba-Fefferman-Gancedo-GomezSerrano:finite-time-singularities-free-boundary-euler} we give simple transformation laws that allow us to pass from the velocity $u(x,t)$ and pressure $p(x,t)$ (defined for $x\in \Omega(t)$) to a transformed velocity and pressure $\tilde{u}(\tilde{x},t), \tilde{p}(\tilde{x},t)$, defined for $\tilde{x} \in \tilde{\Omega}(t)$.

We can then write down the analogue of the inviscid water wave equations \eqref{inviscidww} in the tilde domain. We call these equations $\tilde{\eqref{inviscidww}}$; they govern the evolution of $\tilde{\Omega}(t), \tilde{u}(\tilde{x},t)$ backwards in time, with initial conditions at $t = t_{*}$.

Note that, whereas $\Omega(t_{*})$ is singular (its chord-arc constant is zero), $\tilde{\Omega}(t_{*})$ is bounded by a smooth simple closed curve. Moreover, the transformed water wave equation $\tilde{\eqref{inviscidww}}$ behave much like the original equations \eqref{inviscidww}. Adapting the energy estimates of Ambrose and Masmoudi \cite{Ambrose-Masmoudi:zero-surface-tension-2d-waterwaves}, we prove existence of a smooth solution of $\tilde{\eqref{inviscidww}}$ for times $t \in [t_{*}-\varepsilon, t_{*}+\varepsilon]$, with $\tilde{\Omega}(t_{*})$ as in Figure \ref{figcharlie2}. As long as $\tilde{\Omega}(t)$ avoids $P(\Gamma)$ for all $t \in [t_{*}-\varepsilon, t_{*})$, we obtain a corresponding solution of \eqref{inviscidww} with $\Omega(t) = P^{-1}(\tilde{\Omega}(t))$ a smooth simple closed curve for $t \in [t_{*}-\varepsilon, t_{*})$, We can guarantee that $\tilde{\Omega}(t)$ avoids $P(\Gamma)$ for $t \in [t_{*}-\varepsilon, t_{*})$ by taking $\varepsilon$ smaller and carefully picking the initial velocity $\tilde{u}$ at the two points $\tilde{z}_{splash,1}$ and $\tilde{z}_{splash,2}$ in Figure \ref{figcharlie3}. Thus, we have produced an inviscid water wave that starts out smooth at time $t_{*}-\varepsilon$ and ends in a splash at time $t_{*}$.

This concludes our discussion of the inviscid water wave equations \eqref{inviscidww}.

We pass to the viscous case, where we no longer have the luxury of solving backwards in time. Just as in the inviscid case, there is an analogue of the water wave equations \eqref{viscousww} in the tilde domain, which we call equations $\tilde{\eqref{viscousww}}$ (we won't write them out in the introduction, see section \ref{sectionequations} for a precise definition).

The unknowns in the tilde domain are a time-varying domain $\tilde{\Omega}(t)$, a velocity field $\tilde{u}(\tilde{x},t)$ and a pressure $\tilde{p}(\tilde{x},t)$ with $\tilde{u}, \tilde{p}$ defined for $\tilde{x} \in \tilde{\Omega}(t)$.

The relationship between equations \eqref{viscousww} and equations $\tilde{\eqref{viscousww}}$ is as follows:

Every solution $\Omega(t), u(\cdot,t), p(\cdot,t)$ of \eqref{viscousww} such that $\Omega(t)$ avoids the slit $\Gamma$ gives rise to a solution $\tilde{\Omega}(t), \tilde{u}(\cdot,t), \tilde{p}(\cdot,t)$ of $\tilde{\eqref{viscousww}}$, with $\tilde{\Omega}(t) = P(\Omega(t))$. On the other hand, let $\tilde{\Omega}(t), \tilde{u}(\cdot,t), \tilde{p}(\cdot,t)$ be a solution of $\tilde{\eqref{viscousww}}$. We would like to define a solution $\Omega(t), u(\cdot,t), p(\cdot,t)$ of \eqref{viscousww} such that $\Omega(t) = P^{-1}(\tilde{\Omega}(t))$. However, this may not be possible, because $P^{-1}(\pa \tilde{\Omega}(t))$ may self-intersect, as in Figures \ref{figcharlie4b} and \ref{figcharlie4c}. In particular, $P^{-1}(\pa \tilde{\Omega}(t))$ in Figure \ref{figcharlie4c} is clearly not the boundary of any physically meaningful water region.

The good news is that Figures  \ref{figcharlie4b} and \ref{figcharlie4c} are the only obstructions; as long as $P^{-1}(\pa\tilde{\Omega}(t))$ is as in Figure \ref{figcharlie4a}, we can easily pass from our solution $\tilde{\Omega}(t), \tilde{u}(\cdot,t), \tilde{p}(\cdot,t)$ of $\tilde{\eqref{viscousww}}$ to a solution $\Omega(t), u(\cdot,t), p(\cdot,t)$ of \eqref{viscousww} with $\Omega(t) = P^{-1}(\tilde{\Omega}(t))$.

Let us now solve equations $\tilde{\eqref{viscousww}}$ for times $t>0$, starting with smooth initial $\tilde{\Omega}(0)$ and $\tilde{u}(\cdot,0)$. Adapting the analysis of Beale \cite{Beale:initial-value-problem-navier-stokes} from \eqref{viscousww} to the tilde domain, we prove that smooth solutions $\tilde{\Omega}(t), \tilde{u}(\cdot,t), \tilde{p}(\cdot,t)$ of $\tilde{\eqref{viscousww}}$ with the given initial conditions exist for short time, i.e. for $t \in [0,T]$ with small positive $T$ depending on $\tilde{\Omega}(0), \tilde{u}(\cdot,0)$. Moreover, the solutions of $\tilde{\eqref{viscousww}}$ depend stably on the initial conditions.

For suitable one-parameter families of initial conditions $\tilde{\Omega}_{\ep}(0), \tilde{u}_{\ep}(\cdot,0)$ depending on a small parameter $\ep$, there is a family of smooth solutions $\tilde{\Omega}_{\ep}(t), \tilde{u}_{\ep}(\cdot,t), \tilde{p}_{\ep}(\cdot,t)$ solving $\tilde{\eqref{inviscidww}}$ up to time $T$, with $\|\pa \tilde{\Omega}_{\ep}(t) - \pa \tilde{\Omega}(t)\| = O(\ep)$ in a suitable norm $\|\cdot\|$.

We are ready to combine the ingredients above. We start with smooth initial conditions $\tilde{\Omega}(0), \tilde{u}(\cdot,0)$ with $P^{-1}(\pa \tilde{\Omega}(0))$ as in Figure \ref{figcharlie4b}. Solving $\tilde{\eqref{viscousww}}$ for a short time, we obtain smooth solutions $\tilde{\Omega}(t), \tilde{u}(\cdot,t), \tilde{p}(\cdot,t)$ for times $t \in [0,T]$ (some $T > 0$). By making $T$ smaller and picking the initial velocity $\tilde{u}(\cdot,0)$ so that $\tilde{u}(z_{splash,1},0)$ and $\tilde{u}(z_{splash,2},0)$ point in the right direction, we can guarantee that $P^{-1}(\pa \tilde{\Omega}(T))$ behaves as in Figure \ref{figcharlie4c}.

Next, we pick a one-parameter family of initial conditions $\tilde{\Omega}_{\ep}(0)$ and $\tilde{u}_{\ep}(\cdot,0)$ perturbing our $\tilde{\Omega}(0), \tilde{u}(\cdot,0)$. We can easily arrange that for small positive $\ep$, $P^{-1}(\pa \tilde{\Omega}_{\ep}(0))$ looks like Figure \ref{figcharlie4a}, even though $P^{-1}(\pa \tilde{\Omega}(0))$ is as in Figure \ref{figcharlie4b}. The perturbed solution $\tilde{\Omega}_{\ep}(t), \tilde{u}_{\ep}(\cdot,t), \tilde{p}_{\ep}(\cdot,t)$ will satisfy $\|\pa \tilde{\Omega}_{\ep}(T) - \pa \tilde{\Omega}(T)\| = O(\ep)$. Hence, for $\ep > 0$ small enough, $P^{-1}(\pa \tilde{\Omega}_{\ep}(T))$ will be as in Figure \ref{figcharlie4c}, since the same holds for $P^{-1}(\pa \tilde{\Omega}(T))$. For such $\ep$, $P^{-1}(\pa \tilde{\Omega}_{\ep}(t))$ starts out as in Figure \ref{figcharlie4a} for $t = 0$  and ends as in Figure \ref{figcharlie4c} for $t = T$. Fix such an $\ep$ and let

\begin{align*}
t_{*} = \inf\left\{t \in [0,T]: P^{-1}(\pa \tilde{\Omega}_{\ep}(t)) \text{ is as in Figure \ref{figcharlie4b} or \ref{figcharlie4c}}\right\}
\end{align*}

Then, $0 < t_{*} < T$, $P^{-1}(\pa \tilde{\Omega}_{\ep}(t_{*}))$ is as in Figure \ref{figcharlie4b}, and $P^{-1}(\pa \tilde{\Omega}_{\ep}(t))$ is as in Figure \ref{figcharlie4a} for $0 \leq t < t_{*}$. Consequently, $\tilde{\Omega}_{\ep}(t), \tilde{u}_{\ep}(\cdot,t), \tilde{p}_{\ep}(\cdot,t)$ gives rise to a solution of \eqref{viscousww}, the viscous water wave equation, for $t \in [0,t_{*})$, ending in a splash at time $t_{*}$.

The paper is organized as follows: in Section \ref{sectionequations} we derive the equations in the tilde domain, in Section \ref{sectionspaceslemmas} we setup the different spaces and we prove the auxiliary technical lemmas that we will use throughout the estimates. Section \ref{sectionlineal} is devoted to the study of the linear part of the Navier-Stokes equation, whereas Section \ref{fixedpoint} incorporates the effects of the nonlinear part. Section \ref{sectionstability} closes the argument by showing the structural stability of the equation. Finally, in Section \ref{sectioninitialvel} we show that we can pick an initial velocity in such a way that the splash is formed.

We discuss briefly the types of singular interfaces that our methods produce.

Given any splash curve $\Gamma$, our main result produces initially smooth solutions of the viscous water wave equations that end in splash curves $\tilde{\Gamma}$ arbitrarily close to $\Gamma$ in, say, $C^2$. Other scenarios are possible. For instance, we believe it will be easy to produce initially smooth solutions that end with an interface as in Figure \ref{figcharliexx}. Thus, at the moment of breakdown, the interface self-intersects at two points $A$ and $B$.

\begin{figure}
\centering
\includegraphics[scale=0.4]{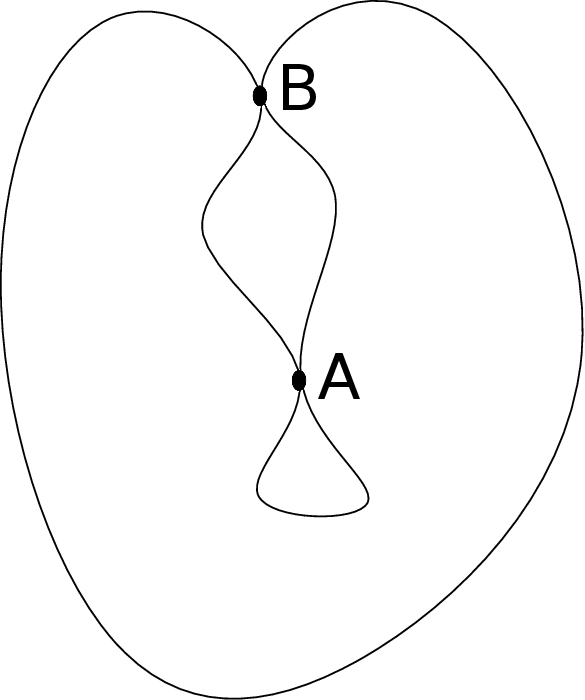}
\caption{A splash singularity forms at two points simultaneously.}
\label{figcharliexx}
\end{figure}

To do so, we introduce a suitable two-parameter family of initial conditions. Let these initial conditions be parametrized by $(\lambda,\mu)$ in a small neighborhood of the origin in $\mathbb{R}^2$. The initial condition corresponding to any $(\lambda,\mu)$ will be smooth and will depend smoothly on $\lambda,\mu$ when viewed in the tilde domain. We write $\Gamma(\lambda,\mu)$ to denote the time-zero interface for the initial conditions arising from $(\lambda,\mu)$. We can arrange that $\Gamma(0,0)$ is as in Figure \ref{figcharliexx}, and that for small positive $\lambda,\mu$, the curve $\Gamma(\lambda,\mu)$ is as in Figure \ref{figcharlieyy}, where the distance from $A_1$ to $A_2$ is comparable to $\lambda$, and the distance from $B_1$ to $B_2$ is comparable to $\mu$.

\begin{figure}
\centering
\includegraphics[scale=0.4]{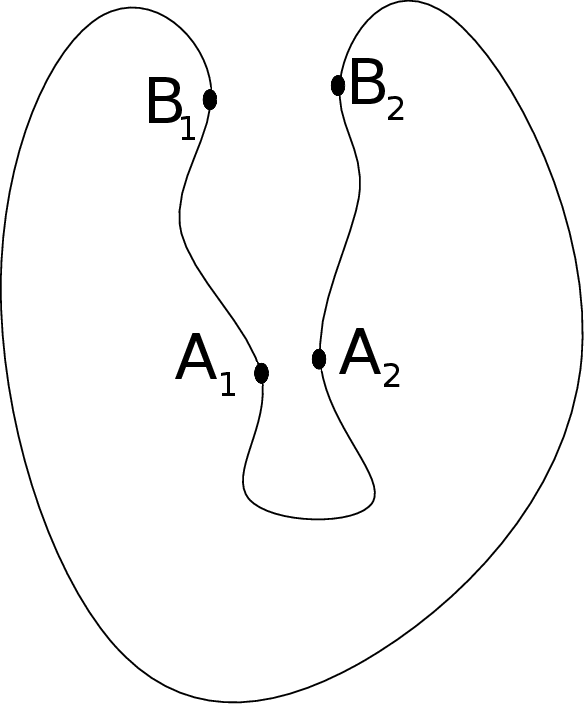}
\caption{Sketch of the situation before a splash singularity.}
\label{figcharlieyy}
\end{figure}

We arrange the initial velocities so that at time zero, points $A_1$ and $A_2$ are moving towards each other with velocity $\sim$ 1, and similarly for points $B_1$ and $B_2$.

The time evolution of a viscous water wave with initial conditions arising from $(\lambda,\mu)$ can be controlled by passing to the tilde domain and applying the analysis in the forthcoming sections below. We find that the interface begins to self-intersect near point $A$ at time $t_{A}(\lambda,\mu) \sim \lambda$. Similarly, the interface begins to self-intersect near point $B$ at time $t_B(\lambda,\mu) \sim \mu$. If $0 < \lambda \ll \mu \ll 1$, then $t_B(\lambda,\mu) < t_A(\lambda,\mu)$; but if $0 < \mu \ll \lambda \ll 1$, then $t_A(\lambda,\mu) < t_B(\lambda,\mu)$. Hence, for some small nonzero $(\lambda,\mu)$, we have that $t_A(\lambda,\mu) = t_B(\lambda,\mu)$, at which time the interface looks like Figure \ref{figcharliexx}. Thus, we can produce initially smooth viscous water waves that end with an interface that self-intersects at two distinct points.

We close the introduction by citing some of the relevant literature on viscous water waves.

Some of the earliest papers in this topic were written by Solonnikov. He studied the problem of a viscous fluid bounded by a free boundary in the vacuum (the fluid domain is bounded). He showed local existence of solutions with \cite{Solonnikov:solvability-motion-viscous-incompressible} using H\"older spaces in the frame of parabolic systems for bounded domains and without surface tension \cite{Solonnikov:an-initial-bvp-stokes-free-boundary,Solonnikov:solvability-problem-evolution-bounded-free-surface}.  Local existence and uniqueness of solutions was shown by Beale in \cite{Beale:initial-value-problem-navier-stokes} in Sobolev spaces with non-slip boundary condition at a regular bottom for the fluid and extended to $L^q$ spaces by Abels \cite{Abels:initial-value-problem-navier-stokes-free-surface-lq}. A second theorem in the former paper showed that for any $T>0$ there exist solutions of sufficiently small initial data on $[0,T]$. For the case with surface tension, see also Tani \cite{Tani:small-time-existence-3d-navier-stokes} and Coutand-Shkoller \cite{Coutand-Shkoller:unique-solvability-free-boundary-navier-stokes-surface-tension}. Masmoudi and Rousset \cite{Masmoudi-Rousset:uniform-regularity-navier-stokes} studied the case where the viscosity tends to zero for the free boundary problem (see \cite{Elgindi-Lee:uniform-regularity-free-boundary-navier-stokes-surface-tension} for the case with surface tension). The case of a viscous fluid lying above a bottom has been extensively studied. For the case without surface tension Sylvester, in \cite{Sylvester:large-time-small-viscous-waves-no-surface-tension}, showed global well-posedness for small initial data. Hataya, in \cite{Hataya:decaying-solution-navier-stokes-no-surface-tension}, showed the existence of solutions which decay algebraically in time for a periodic in the horizontal variable surface. Guo and Tice, in \cite{Guo-Tice:local-well-posedness-surface-wave-no-surface-tension,Guo-Tice:decay-viscous-surface-waves-no-surface-tension,Guo-Tice:almost-exponential-decay-surface-waves-no-surface-tension}, have proved algebraic decay rate in time for asymptotically flat surfaces and almost exponential decay rate in time for periodic in the horizontal variable surfaces.

 Global in time regularity was first given for small initial data in \cite{Beale:large-time-regularity-viscous-waves} by adding to the system surface tension effects (see also \cite{Bae:solvability-free-boundary-navier-stokes} for an alternative proof).
In \cite{Sylvester:large-time-small-viscous-waves-no-surface-tension, Tani-Tanaka:large-time-surface-waves-viscous-fluids} global existence is obtained without the help of surface tension. Decay rates have been also considered to understand the long time behavior of the solution. For the case with surface tension see \cite{Beale-Nishida:large-time-viscous-surface-waves, Nishida-Teramoto-Yoshihara:global-existence-viscous-surface-waves-horizontally-periodic}.

In the case of a two fluid problem there are some recent results where local well-posedness and global existence for small data is shown. The situations consider low regular initial velocities in critical spaces, in some cases within the chain of Besov spaces (see \cite{Danchin-Mucha:lagrangian-approach-incompressible-navier-stokes-variable-density,Danchin-Mucha:incompressible-flows-piecewise-constant-density, Huang-Paicu-Zhang:global-well-posedness-bounded-density-non-lipschitz-velocity, Paicu-Zhang-Zhang:global-existence-inhomogeneous-navier-stokes,Liao-Zhang:global-regularity-2d-patch-navier-stokes} and the references therein for more details).

Existence of splash singularities for inviscid water waves was proven in our paper \cite{Castro-Cordoba-Fefferman-Gancedo-GomezSerrano:finite-time-singularities-free-boundary-euler}, see also Coutand-Shkoller \cite{Coutand-Shkoller:finite-time-splash} for an alternate proof with applications to a 3D setting.

The inviscid splash is not prevented by taking into account gravity or surface tension (see \cite{Castro-Cordoba-Fefferman-Gancedo-GomezSerrano:finite-time-singularities-free-boundary-euler,Castro-Cordoba-Fefferman-Gancedo-GomezSerrano:finite-time-singularities-water-waves-surface-tension}), but it is prevented by replacing the vacuum in $\R^{2} \setminus \Omega(t)$ by an incompressible opposing fluid; see \cite{Fefferman-Ionescu-Lie:absence-splash-singularities} as well as an alternate proof by Coutand-Shkoller \cite{Coutand-Shkoller:no-splash-vortex-sheet}. It has been shown in \cite{Cordoba-Enciso-Grubic:splash-almost-splash-stationary-Euler} that there exist ``almost splash'' stationary solutions in the two fluid case. We caution that the nonexistence of a splash does not rule out a breakdown in which the chord-arc constant degenerates and the solution loses smoothness; again, see \cite{Fefferman-Ionescu-Lie:absence-splash-singularities}.

Our strategy for the viscous splash, as outlined above, was announced at the \emph{40$\grave{e}$mes journ\'ees EDP (2013)} in Biarritz, at the \emph{2013 Clay Research Conference} (Oxford) and more recently at the \emph{Minerva Distinguished Visitor Lectures} in Princeton in 2014. Finally, in this paper we provide details.

We remark that the estimates we make in the tilde domain are given in detail in the Appendices for the reader's convenience, but they are simply adapted from well known estimates for the nontilde domain.

We refer the reader to Coutand-Shkoller \cite{Coutand-Shkoller:splash-navier-stokes} for a different proof of the formation of splash singularities for viscous water waves.

\section{Equations: transformation to a nonsplash domain}
\label{sectionequations}

We have to solve the 2D-Navier-Stokes equations in the fluid domain $\Omega(t)$,
\begin{align}\label{NSO}
\pa_t u + (u\cdot\nabla) u-\Delta u & = -\nabla p,\quad \text{in $\Omega(t)$},\\
\nabla\cdot u & = 0, \quad \text{in $\Omega(t)$},\label{div0}\\
\left(p\, \mathbb{I}-\left(\nabla u + \left(\nabla u\right)^* \right)\right)n\quad & = 0,\quad \text{on $\pa \Omega(t)$},\label{tensor}\\
\left.u(t)\right|_{t=0}& = u_0,\quad \text{in $\Omega(0)=\Omega_0$,} \label{di}
\end{align}
where
\begin{align*}
\Omega(t)=X(\Omega(0),t)
\end{align*}
with $X(\alpha,t)$ solving
\begin{align*}
\frac{d X(\alpha,t)}{dt}=& u(X(\alpha,t),t)\\
X(\alpha,0)=&\alpha, \quad \al\in \Omega(0),
\end{align*}
and $n(t)$ is unit normal vector to $\pa\Omega(t)$ (pointing out).
Here, $\Omega_0$ and $u_0$ are given and satisfy the compatibility conditions
\begin{align*}
\nabla \cdot u_0 & = 0,\quad \text{in $\Omega_0$}\\
n_0^\perp\left(\nabla u_0 + \left(\nabla u_0\right)^*\right)n_0 & = 0 \quad \text{on $\pa\Omega_0$,}
\end{align*}
with $n_0=n(0)$. The condition
$$\left.\left(p\, \mathbb{I}-\left(\nabla u + \left(\nabla u\right)^* \right)\right)\right|_{\pa\Omega(t)} n = 0$$
states the continuity of the normal stress through the interface. We will use the symbol $^*$ to denote the transpose of a matrix.

We notice that the pressure, $p$, will be given by the following elliptic problem
\begin{align*}
-\Delta p = & \nabla\cdot ((u\cdot\nabla)u)\quad \text{in $\Omega(t)$}\\
p = &  n\left(\nabla u + \left(\nabla u\right)^*\right)n\quad\text{on $\pa\Omega(t)$}.
\end{align*}


Let $P(z)$ be defined as in the introduction and let $\tilde{\Om}(t) = P(\Om(t))$. Next we will write the system (\ref{NSO}-\ref{di}) in the domain $\tilde{\Omega}(t)$ and after that, by using Lagrangian coordinates, we will fix the domain in order to work in the domain $\tilde{\Omega}(0)$ rather than in $\tilde{\Omega}(t)$.
\begin{rem}
At this point one should notice that we are assuming that $\tilde{\Omega}(t)$ is the projection by $P$ of $\Omega(t)$ and that $\tilde{\Omega}(t)$ is a simply connected bounded domain. Therefore $P^{-1}$ is a well defined analytic function. Once we have written the N-S system in $\tilde{\Omega}(t)$ this assumption will not be needed anymore.
\end{rem}
We define $\ut= u \circ P^{-1}$ and $\tilde{p} = p \circ P^{-1}$ and
\begin{align}\label{definiciondeQ}Q^2=\left|\frac{dP}{dz}\circ P^{-1}\right|^2.\end{align} Then
$$ \pa_j u^i=(\pa_k \tilde{u}^i\circ P )\pa_jP^k$$ and therefore
$$\pa_j u^i\circ P^{-1}= A_{kj}\pa_k \ut^i,$$
where
\begin{align}\label{definiciondeA}A_{kj}= \pa_j P^k\circ P^{-1}.\end{align}
In addition, let's assume that we traverse (clockwise) the boundary of $\Omega(t)$ using the parametrization $z(\gamma,t)$, i.e.,
$$\pa\Omega(t)=\{ z(\gamma,t)\in \R^2 \,:\, \gamma \in [-\pi, \pi)\times [0,T]\}.$$
and
$$n= (-z^2_\gamma(\gamma,t), z^1_\gamma(\gamma,t)).$$

Since the boundary of $\tilde{\Om}(t)$ can be parameterized by $\tilde{z}(\gamma,t)=P(z(\gamma,t))$ we have that
$$\tilde{z}^i_\gamma=\pa_k P^i\circ P^{-1} (\tilde{z}(\gamma,t))z^k_\gamma(\gamma,t)$$ thus

$$\tilde{n}=-JA|_{\pa\tilde{\Omega}(t)}J  n,$$

where $$J=\left(\begin{array}{cc} 0 & -1\\ 1 & 0\end{array}\right).$$

Using the previous expression and the fact that $P$ is a conformal map we can write the system (\ref{NSO}-\ref{di}) in the domain $\tilde{\Omega}(t)$ as follows:

\begin{align}\label{NSOtilde}
\pa_t \ut + (A\ut\cdot\nabla) \ut-Q^2\Delta \ut = & -A^*\nabla \tilde{p},\quad \text{in $\tilde{\Omega}(t)$},\\
Tr\left(\nabla \tilde u A\right)  =& 0, \quad \text{in $\tilde{\Omega}(t)$},\label{div0tilde}\\
\left(\tilde{p}\, \mathbb{I}-\left(\nabla \tilde{u} A + \left(\nabla \tilde{u} A\right)^* \right)\right)A^{-1}\tilde{n}\quad = & 0,\quad \text{on $\pa \tilde{\Omega}(t)$},\label{tensortilde}\\
\left.\tilde{u}(t)\right|_{t=0}=&\tilde{u}_0,\quad \text{in $\tilde{\Omega}(0)=\tilde{\Omega}_0$} \label{ditilde}.
\end{align}
where $Tr(A)$ is the trace of the matrix $A$ and

$$\tilde{\Omega}(t)=\tilde{X}(\tilde{\Omega}(0),t)$$ with $\tilde{X}(\alpha,t)$ solving

\begin{align*}
\frac{d\tilde{X}(\alpha,t)}{dt}& =\left(A\circ \tilde{X}(\alpha,t)\right)\left(\ut\circ \tilde{X}(\alpha,t)\right)\\
\tilde{X}(\alpha,0)& = \alpha, \quad \alpha\in \tilde{\Omega}(0).
\end{align*}
Here we have used that
$$Q^2A^{-1}= -JAJ.$$

Now we will fix the domain by working with the variables
\begin{align*}
\tv(\alpha,t)= & \ut\circ \tilde{X}(\alpha,t)\quad \alpha\in \tilde{\Omega},\\
\tilde{q}(\alpha,t)= & \tilde{p}\circ \tX(\alpha,t)\quad \alpha\in \tilde{\Omega}.
\end{align*}

The system (\ref{NSOtilde}-\ref{ditilde}) in terms of $(\tilde{v},\tilde{q})$ reads
\begin{align}
\pa_t \tv_i -Q^2\circ \tX \tilde{\zeta}_{kj}\pa_k\left(\tilde{\zeta}_{lj}\pa_l\tv_i\right)&=-A_{ki}\circ \tX\tilde{\zeta}_{jk}\pa_j \tilde{q}\quad \text{in $\tilde{\Om}_0$},\label{NSOf}\\
Tr\left(\nabla \tv \tilde{\zeta}A\circ \tX\right) & = 0 \quad \text{in $\tilde{\Om}_0$}\label{divf}\\
\left(\tilde{q}\mathbb{I}-\left(\left(\nabla \tv \tilde{\zeta} A\circ \tX\right)+\left(\nabla \tv \tilde{\zeta} A\circ \tX\right)^*\right)\right)A^{-1}\circ \tX \gradj \tX\tilde{n}_0 & = 0 \quad \text{ on $\partial \tilde{\Om} _0$},\label{tensorf}\\
\tv|_{t=0}\equiv &  \tv_0 = \ut_0, \label{dif}
\end{align}
where
$$
\tilde{\zeta}=\left(\nabla \tX\right)^{-1},\quad \left(\tilde{\zeta}^{\,-1}\right)_{ij}=\pa_j \tX^i,\quad \gradj \tX= -J \nabla \tX J$$
and
\begin{align}
\frac{d \tX(\alpha,t)}{dt}=& A\circ \tX(\alpha,t) \tilde{v}(\alpha,t)\label{dtx} \\
\tX(\alpha,0)=& \alpha,\quad \alpha \in \tilde{\Om}_0\label{dtx0}.
\end{align}

We will solve the system (\ref{NSOf}-\ref{dif}) by iteration towards a fixed point. We study first the linear system that will be used for that purpose.

\begin{align}
\pa_t \tv^{(n+1)} -Q^2\Delta \tv^{(n+1)}=-A^*\nabla \tilde{q}^{(n+1)}+\tilde{f}^{(n)}&\quad \text{in $\tilde{\Om}_0$},\label{NSOfl}\\
Tr\left(\nabla \tv^{(n+1)} A\right)=   \tilde{g}^{(n)}&\quad \text{in $\tilde{\Om}_0$}\label{divfl}\\
\left(\tilde{q}^{(n+1)}\mathbb{I}-\left(\left(\nabla \tilde{v}^{(n+1)}  A\right)+\left(\nabla \tilde{v}^{(n+1)}  A\right)^*\right)\right)A^{-1}\tilde{n}_0=\tilde{h}^{(n)}& \quad \text{on $\partial \tilde{\Om} _0$},\label{tensorfl}\\
\tilde{v}^{(n+1})|_{t=0}\equiv   \tilde{v}^{(n+1)}_0 = \ut_0,& \label{difl}
\end{align}
where
\begin{align}
\tilde{f}^{(n)}_i=&Q^2\circ \tX^{(n)} \tilde{\zeta}^{(n)}_{kj}\pa_k\left(\tilde{\zeta}^{(n)}_{lj}\pa_l\tilde{v}^{(n)}_i\right)-Q^2\Delta \tv_i^{(n)} -A_{ki}\circ \tX^{(n)} \tilde{\zeta}^{(n)}_{jk}\pa_j \tilde{q}^{(n)}+ A_{ki}\pa_k \tilde{q}^{(n)} \label{efe} \\
\tilde{g}^{(n)}= & -Tr\left(\nabla \tv^{(n)} \tilde{\zeta}^{(n)}A\circ \tX^{(n)}\right)+Tr\left(\nabla \tv^{(n)}A\right)\label{ge}\\
\tilde{h}^{(n)}=& \left(\left(\nabla \tv^{(n)} \tilde{\zeta}^{(n)} A\circ \tX^{(n)}\right)+\left(\nabla \tv^{(n)} \tilde{\zeta}^{(n)} A\circ \tX^{(n)}\right)^*\right)A^{-1}\circ \tX^{(n)} \gradj\tX^{(n)}\tilde{n}_0\nonumber\\
&-\left(\left(\nabla \tv^{(n)}  A\right)+\left(\nabla \tv^{(n)}  A\right)^*\right)A^{-1}\tilde{n}_0+\tilde{q}^{(n)}A^{-1}\tilde{n}_0-\tilde{q}^{(n)}A^{-1}\circ\tX^{(n)}\gradj \tX^{(n)}\tilde{n}_0 ,\label{ache}
\end{align}
We define $\tX^{(n+1)}$ as
\begin{align}\label{Xtilda}
\tX^{(n+1)}(\alpha,t)=\alpha+\int_0^t A\circ \tX^{(n)}(\alpha,\tau)\tv^{(n)}(\alpha,\tau) d\tau,
\end{align}
and \begin{align*}\tilde{\zeta}^{(n)}=\left(\nabla \tX^{(n)}\right)^{-1} && \gradj\tX^{(n)}=-J\nabla \tX^{(n)}J.\end{align*}

Formally, assuming convergence as $n\to\infty$, it is easy to check that in the limit we find a solution of the nonlinear system. In what follows, we will either remove the tilde from the notation or it will become clear from the context.

\section{Definitions of the spaces and auxiliary lemmas}
\label{sectionspaceslemmas}

This section is devoted to present the main tools used for both the linear and the nonlinear case. We will also define all the spaces used for the construction of the solutions, and their norms.

\subsection{The spaces $H^s$}

For a positive integer $m$, we will denote the standard Sobolev space by $H^{m}([0,T])$ with norm $\|v\|_{H^m}^2=\sum_{j=0}^m\|\partial_t^j v\|_{L^2}^2$. We will indistinctly refer to $L^{2}([0,T])$ as either $L^{2}([0,T])$ or $H^{0}([0,T])$.

Here we will give a precise definition of the Sobolev spaces with fractional derivatives in time that we are going to use.

As in the classical paper \cite{Beale:initial-value-problem-navier-stokes} we define $H^s_{(0)}([0,T])$, for $0<s<1$, as the interpolation between $L^2([0,T])$ and $H^1_{(0)}([0,T])$, where  to  interpolate we use  the operator $S=1-\pa_t^2$, with domain $D(S)=\{ u \in H^2[0,T] : v(0)=\pa_t v(T)=0\}$. The reader can consult \cite[p. 9]{Lions-Magenes:non-homogeneous-bvp-I}, for further information about this interpolation (notice that in this book the operator $S$ is called $\Lambda$). After that, one can define the norm in $H^s_{(0)}([0,T])$ as the graph norm of the operator $\Lambda^s$, with $\Lambda=S^\frac{1}{2}$. An explicit computation shows that $\left\{\sin\p{\frac{(2n+1)\pi}{2T}t}\sqrt{\frac{2}{T}}\right\}_{n=1}^\infty$ is an orthogonal basis of $L^2([0,T])$ of eigenfunctions of $S$, with eigenvalues $\left\{1+\frac{(2n+1)^2\pi^2}{4T^2}\right\}_{n=0}^\infty$ and that  $H^s_{(0)}([0,T])$ consists of functions $v\in L^2([0,T])$ such that $$||v||_{H^s_{(0)}}^2=\sum_{n=0}^\infty\p{v_n^s}^2\p{\frac{(2n+1)\pi}{2T}}^{2s}<\infty \quad  \text{where $v^s_n=\int_{0}^T v(t)\sin\p{\frac{(2n+1)\pi}{2T}t}\sqrt{\frac{2}{T}}dt$}.$$
 For $s>\frac{1}{2}$ we have that $v\in H^s_{(0)}([0,T])$ implies $v(0)=0$. An important fact, remarked in \cite{Beale:initial-value-problem-navier-stokes} is, that if $A$ is an operator bounded from $L^2([0,T])$ to $Y$ and from $H^1_0([0,T])$ to $Z$ with constant independent of $T$, where $Y$ and $Z$ are Hilbert spaces with $Z\subseteq Y$, then $A$ maps $H^s_{(0)}$ onto the interpolated space $[Z,Y]_{1-s}$ with norm bounded independent of $T$.

 For larger exponents,  the space $H^{m+s}_{(0)}([0,T])$, $m=1,2,3,...$ and $0 \leq s<1$, is regarded as the subspace of $\{ v\in H^{m}([0,T])\,:\, (\pa^k_tv)(0)=0,\, k=0,...,m-1\}$ with $\pa^m_t v\in H^s_{(0)}([0,T])$. We equip this space with the norm
 \begin{align}\label{norm0}
 ||v||^2_{H^{m+s}_{(0)}([0,T])}=||v||^2_{L^2([0,T])}+||\pa_t v||^2_{L^2([0,T])}+...+||\pa^m_tv||^2_{H^s_{(0)}([0,T])}.
\end{align}

This is the norm for fractional derivatives in time that we will use in this paper.

Again as in \cite{Beale:initial-value-problem-navier-stokes}, we also  introduce the space $H^s([0,T])$. This space is defined, for $0 < s <1$, as the interpolation of $H^1([0,T])$ and $L^2([0,T])$ with $S=1-\pa^2_t$ and domain $D(S)=\{v\in H^2([0,T])\,:\, (\pa_tv)(0)=(\pa_t v)(T)=0\}$. In this case $\left\{\frac{1}{\sqrt{T}},\, \left\{\cos\p{\frac{n\pi}{T}t}\sqrt{\frac{2}{T}}\right\}_{n=1}^\infty\right\} $ is a basis of $L^2([0,T])$ of eigenfunctions of $S$ with eigenvalues $\left\{1+\frac{n^2\pi^2}{T^2}\right\}_{n=0}^\infty$. Thus we can define \begin{align}\label{norma1}||v||_{H^s([0,T])}^2=\sum_{n=0}^\infty \p{1+\frac{n^2\pi^2}{T^2}}^s\p{v^c_n}^2,\end{align}
where
\begin{align*} v_0^c=\int_{0}^T\frac{v(t)}{\sqrt{T}}dt,\quad v^c_n=\int_{0}^Tv(t)\cos\p{\frac{n\pi}{T}t}\sqrt{\frac{2}{T}}dt\quad n\geq 1.
\end{align*} A similar interpolation statement holds in this space.  For larger derivatives we regard $H^{m+s}([0,T])$, $m=1,2,3...$, $0<s<1$ as the subspace of $H^m([0,T])$ with $\pa^m_t v\in H^s([0,T])$. It happens that $H^{m+s}_{(0)}([0,T])=\{ v\in H^{m+s}([0,T])\,:\, (\pa_t^kv)(0)=0,\ k=0,1,...,m\}$, for $s>\frac{1}{2}$ and $H^{m+s}_{(0)}([0,T])=\{ v\in H^{m+s}([0,T])\,:\, (\pa_t^kv)(0)=0,\ k=0,1,...,m-1\}$, for $s<\frac{1}{2}$ .  Here we remark that this space will always be equipped with the norm  \eqref{norm0}.

The space $H^s(\T)$ will be the classical Sobolev space of $2\pi$-periodic functions such that
 $$
 ||f||_{H^s(\T)}^2=\sum_{n\in \Z}(1+|n|^2)^{s}|f_n^p|^2\quad\mbox{ is finite with }\quad f_n^p=\frac1{2\pi}\int_{0}^{2\pi}f(\theta)e^{- i n\theta}d\theta.
 $$
The space $H^s(\R^n)$, $n\geq 1$, will be the classical Sobolev space equipped with the norm $$||f||_{H^s(\R^n)}^2=\int_{\R^n}\p{1+|\xi|^2}^s|\widehat{f}(\xi)|^2 d\xi,$$ where $\widehat{f}$ is the Fourier transform of $f$ in $\R^n$. For a domain $\Omega_0\subset\R^2$ with regular boundary (see discussion in Theorem \ref{nolineal} below for the regularity of $\partial\Omega_0$) the space $H^s(\Omega_0)$ is defined classically as the space of functions with $s$ derivatives in $L^2(\Omega_0)$ if $s$ is an integer, or the usual generalization otherwise. It can be related with $H^s(\R^2)$ through the classical extension map. The space $H^{r}(\partial\Omega_0)$ is given by functions $f$ defined on $\partial\Omega_0=\{z_0(\theta):\,\theta\in[0,2\pi] \}$ such that $f(z_0(\theta))\in H^r(\T)$. In this paper $r>1/2$ so that the classical restriction (or trace) map properties on Sobolev spaces can be applied. Finally, $H^{-1}(\Omega_0)$ is defined as the dual space of $H^{1}(\Omega_0)$.

\subsection{Space-time definitions}\label{spacedefi}

Once we have defined the spaces $H^s$ we introduce the spaces we will use to solve the free boundary Navier-Stokes equations in the tilde domain, where $\nabla$ denotes the space gradient:

$$H^{ht, s+1}_{(0)}\left([0,T]; \Omega_0\right)= L^2\left([0,T];\,H^{s+1}(\Om_0)\right)\cap H_{(0)}^{\frac{s+1}{2}}\left([0,T];\, L^2(\Omega_0)\right), \quad s > 0$$
$$
H^{ht, s}_{pr\, (0)}\left([0,T]; \Omega_0\right)=\left\{q\in L^{\infty}([0,T];\dot{H}^1(\Omega_0))\,:\, \nabla q \in H_{(0)}^{ht,s-1}\left([0,T]; \Omega_0\right),\,q\in H_{(0)}^{ht,s-\frac{1}{2}}\left([0,T]; \partial\Omega_0\right)\right\}, \quad 2 < s < \frac{5}{2}
$$
$$
\overline{H}_{(0)}^{ht,s}([0,T];\Omega_0)=L^{2}([0,T]; H^{s}(\Omega_0))\cap H_{(0)}^{\frac{s+1}{2}}([0,T]; H^{-1}(\Omega_0)), \quad s > 0
$$
We now fix $s$ with $2 < s < \frac{5}{2}$ and pick a small enough $\ep > 0$ depending only on $s$. We also set
\begin{equation*}F^{s+1}_{\gamma}\left([0,T];\,\Omega_0\right)= L^{\infty}_{1/4}([0,T];H^{s+1}(\Omega_0))\cap  H_{(0)}^{2}\left([0,T];H^\gamma(\Om_0)\right),\quad s-1-\ep<\gamma<s-1
\end{equation*}
with
$$||f||_{L^{\infty}_{1/4}}=\sup_{t\in [0,T]} t^{-1/4} |f(t)|.$$

For the spaces $H^{ht, s+1}\left([0,T]; \Omega_0\right)$, $H^{ht, s}_{pr}\left([0,T]; \Omega_0\right)$ and $\overline{H}^{ht,s}([0,T];\Omega_0)$ we give analogous definitions than above 
but removing the subscript $(0)$ in the time Sobolev spaces, i.e., removing the vanishing conditions at $t=0$.



 Also we will use the following notation:
\begin{align*}
&||\cdot||_{H_{(0)}^{ht, s}([0,T],\Om_0)} = ||\cdot||_{H_{(0)}^{ht,s}}\qquad\qquad &||\cdot||_{\overline{H}^{ht, s}_{(0)}([0,T],\Om_0)} =||\cdot||_{\overline{H}_{(0)}^{ht,s}}\\
&||\cdot||_{H_{(0)}^{ht, s}([0,T],\pa\Om_0)} =  |\cdot|_{H_{(0)}^{ht,s}}\qquad\qquad &||\cdot||_{F^{s+1}_{\gamma}([0,T],\Om_0)} =  ||\cdot||_{F^{s+1}}\\
&||\cdot||_{H_{(0)}^r\left([0,T];\,H^{s}(\Om_0)\right)} = ||\cdot||_{H_{(0)}^rH^s} \qquad\qquad &||\cdot||_{H_{(0)}^r\left([0,T];\,H^{s}(\pa\Om_0)\right)} = |\cdot|_{H_{(0)}^rH^s}\\
&||\cdot||_{L^{\infty}([0,T];H^{s}(\Omega_0))} =  ||\cdot||_{L^\infty H^s} \qquad\qquad
&||\cdot||_{L^{\infty}_{1/4}([0,T];H^{s}(\Omega_0))} =  ||\cdot||_{L^\infty_{1/4}H^s}\\
&||\cdot||_{H^{ht,\,s}_{pr\,(0)}([0,T],\Om_0)} =||\cdot||_{H^{ht,\, s}_{pr\, (0)}} \qquad\qquad &.
\end{align*}

\subsection{Auxiliary lemmas}

\begin{lemma}\label{lema22} Let $B$ be a Hilbert space.
\begin{enumerate}
\item For $s\geq 0$, there is a bounded extension operator $H^s((0,T); B)\to H^s((-\infty,\infty);B)$.
\item For $0\leq s\leq 2$, $s-\frac{1}{2}$ not an integer, there is an extension operator from
$$\left\{ v\in H^s((0,T);B)\,;\, \pa_t^kv|_{t=0}=0, 0\leq k< s-\frac{1}{2}\right\}\to H^s((-\infty,\infty);B)$$ with norm bounded uniformly if $T$ is bounded above. Furthermore, if $v^\sharp$ is the extension of $v$, $$||v^\sharp||_{H^s((-\infty,\infty); B)}\leq C ||v||_{H_{(0)}^s((0,T);B)}.$$
\item Similar 
statements apply to the extension of $H^{ht, 2s}$ and $H^{ht, 2s}_{(0)}$.
\end{enumerate}
\end{lemma}
\begin{proof}
The proof can be found in \cite[p.365, Lemma 2.2]{Beale:initial-value-problem-navier-stokes}. In this paper the  statements asserts that the operator extends to $H^s((0,\infty);B)$ rather than $H^s((-\infty,\infty);B)$. But one can easily adapt the proof to the case  $H^s((-\infty, \infty); B)$.
\end{proof}

\begin{lemma}[Parabolic Trace]\label{lema21}

1. Suppose $\frac{1}{2}<r\leq 5$. The mapping $v\to \pa_n^j v$ extends to a bounded operator $H^{ht, r}([0,T];\Omega_0)\to H^{ht, r-j-\frac{1}{2}}([0,T];\partial\Omega_0)$, where $j$ is an integer with $0\leq j<r-\frac{1}{2}$. The mapping $v\to \pa_t^kv(x,0)$ also extends to a bounded operator $H^{ht,r}\to H^{r-2k-1}(\Omega_0)$ if $k$ is an integer with $0\leq k < \frac{1}{2}(r-1)$.

2. Suppose $\frac{3}{2}<r<5$, $r\neq 3$, and $r-\frac{1}{2}$ is not an integer. Let
$$W^r = \prod_{0\leq j <r-\frac{1}{2}} H^{ht,r-j-\frac{1}{2}}([0,T];\partial\Omega_0)\times \prod_{0\leq k < \frac{r-1}{2}}H^{ht, r-2k-1}([0,T];\Omega_0),$$
and let $W_0^r$ be the subspace consisting of $\{a_j,w_j\}$ so that
$$\pa^k_t a_j(x,0)=\pa_n^j w_k(x)\quad x\in\pa\Omega,$$
for $j+2k<r-\frac{3}{2}$. Then the restrictions at $\pa\Omega_0$ together with the restrictions at time $t=0$ in point 1. above form a bounded operator $H^{ht,r}\to W^r_0$, and this operator has a bounded right inverse.
\end{lemma}

\begin{proof}
See \cite[Lemma 2.1]{Beale:initial-value-problem-navier-stokes}.
\end{proof}

\begin{lemma}\label{2.3}
Suppose $0\leq r\leq 4$.
\begin{enumerate}
\item The identity extends to a bounded operator $$H^{ht,r}\to H^pH^{r-2p},$$
$p\leq \frac{r}{2}$.
\item If $r$ is not an odd integer, the restriction of this operator to the subspace with $\pa_t^kv|_{t=0}$, $0\leq k < \frac{r-1}{2}$ is bounded uniformly if $T$ is bounded above. Indeed
    $$\hhn{v}{p}{r-2p}\leq C||v||_{H^{ht,\,r}_{(0)}},$$
    where $C$ does not depend on $T$ if   $T$ is bounded above.
\end{enumerate}
\end{lemma}
\begin{proof}
The proof can be found in \cite[p.365, Lemma 2.3]{Beale:initial-value-problem-navier-stokes}.
\end{proof}

\begin{lemma}\label{2.4}Let $T_0>0$ be arbitrary and $B$ a Hilbert space, and choose $T\leq T_0$.
 For $v\in H^0((0,T); B)$, we define $V\in H^1((0,T);B)$ by
$$V(t)=\int_0^t v(\tau)d\tau.$$
Suppose $0<s<1$, $s\neq \frac{1}{2}$, for $s>\frac{1}{2}$ we impose  $v|_{t=0}=0$,    and $0\leq \ep<s$. Then $v\to V$ is a bounded operator from $H^s_{(0)}((0,T);B)$ to $H^{s+1-\ep}_{(0)}((0,T);B)$, and
$$||V||_{H_{(0)}^{s+1-\ep}((0,T);B)}\leq C_0 T^\ep||v||_{H_{(0)}^s((0,T);B)}.$$
where $C_0$ is independent of $T$ for $0<T\leq T_0$.
\end{lemma}

\begin{proof}

By definition
\begin{align*}
||V||^2_{H_{(0)}^{1+s-\ep}([0,T];B)}=||V||_{L^2}^2+||\pa_tV||^2_{H_{(0)}^{s-\ep}([0,T];B)}=||V||_{L^2}^2+||v||^2_{H_{(0)}^{s-\ep}([0,T];B)}
\end{align*}
On one hand, since $V(t)=\int_{0}^t v(\tau)d\tau$ we have that $||V||_{L^2([0,T];B)}\leq C T||v||_{L^2([0,T];B)}$.
On the other hand
\begin{align*}
&||v||^2_{H_{(0)}^{s-\ep}([0,T];B)}=\sum_{n=0}^\infty\left(\frac{(2n+1)^2\pi^2}{(2T)^2}\right)^{s-\ep}||v_n||^2_B=T^{-2s+2\ep}
\sum_{n=0}^\infty\left(\left(\frac{2n+1}{2}\right)^2\pi^2\right)^{s-\ep}||v_n||^2_B\\
&\leq T^{-2s+2\ep}
\sum_{n=0}^\infty\left(\left(\frac{2n+1}{2}\right)^2\pi^{2}\right)^{s}||v_n||^2_B=T^{2\ep}||v||^2_{H_{(0)}^{s}([0,T];B)}.
\end{align*}
\end{proof}

\begin{lemma}\label{2.5}
\begin{enumerate}
 \item Suppose $r>1$ and $r\geq s\geq 0$. If $v\in H^r(\Omega)$ and $w\in H^s(\Om)$, then $v w\in H^s(\Om)$, and
$$ || v w||_{H^s}\leq C ||v||_{H^r}||w||_{H^s}.$$
\item If $v, w \in H^1(\Om)$, then $vw\in H^0(\Om)$, and $$||vw||_{H^0}\leq C||v||_{H^1}||w||_{H^1}.$$
\item If $v\in H^r(\Om)$, $r>1$, and $w\in H^{-1}(\Om)$, the dual space of $H^1(\Om)$, then $vw$ is defined in $H^{-1}(\Om)$ and
$$||vw||_{H^{-1}}\leq ||v||_{H^r}||w||_{H^{-1}}. $$
\item If $v\in H^1(\Om)$ and $w\in H^0(\Om)$, then $vw$ is defined in $H^{-1}(\Om)$; and
$$||vw||_{H^{-1}}\leq C||v||_{H^1}||w||_{H^0}.$$
\end{enumerate}
\end{lemma}
\begin{proof}
The proof can be found in \cite[p. 366, Lemma 2.5]{Beale:initial-value-problem-navier-stokes}. Here we notice that we work in dimension 2 rather than in dimension 3 as in \cite{Beale:initial-value-problem-navier-stokes} and the second statement of lemma \ref{2.5} can be improved in dimension 2.
\end{proof}
\begin{lemma}\label{2.5m} If $v\in H^{\frac{1}{q}}$ and $w\in H^{\frac{1}{p}}$ with $\frac{1}{p}+\frac{1}{q}=1$ and $1<p<\infty$  then
$$||vw||_{H^0}\leq C ||v||_{H^{\frac{1}{q}}}|| w ||_{H^{\frac{1}{p}}}.$$
\end{lemma}
\begin{proof}
Applying H{\"o}lder's inequality yields
$$||v w||_{H^0}\leq || v||_{L^{2p}}||w||_{L^{2q}},$$
Now the Gagliardo-Nirenberg inequality provides
$$|| v ||_{L^{2p}}\leq C || v||_{H^s}$$
for $\frac{1}{2p}=\frac{1}{2}-\frac{s}{2}$, that  implies $s=\frac{1}{q}$. Proceeding similarly for $w$ we have that
$$||v w||_{H^0}\leq C||v||_{H^{\frac{1}{q}}}||w||_{H^{\frac{1}{p}}}.$$
\end{proof}

\begin{lemma}\label{2.6}
Suppose $B, Y, Z$ are Hilbert spaces, and $M: B\times Y \to Z$ is a bounded, bilinear operator. Suppose $v\in H^s((0,T);B)$ and $w\in H^s((0,T);Y)$, where $s>\frac{1}{2}$. If $vw$ is defined by $M(v,w)$, then $vw\in H^s((0,T);Z)$ and
\begin{enumerate}
\item $$||vw||_{H^s((0,T);Z)}\leq C ||v||_{H^s((0,T);B)}||w||_{H^s((0,T);Y)}.$$
\item Also, if $s\leq 2$ and $v$, $w$ satisfy the additional condition $\pa^k_t v|_{t=0}=\pa^k_t w|_{t=0}=0$, $0\leq k< s-\frac{1}{2}$, and $s-\frac{1}{2}$ is not an integer, then the constant $C$ in 1 can be chosen independent of $T$. Indeed
    $$||vw||_{H_{(0)}^s((0,T);Z)}\leq C ||v||_{H_{(0)}^s((0,T);B)}||w||_{H_{(0)}^s((0,T);Y)}$$ where $C$ does not depend on $T$.
\end{enumerate}
\end{lemma}
\begin{proof}
The proof of 1 and 2 can be found in \cite[p. 366, Lemma 2.6]{Beale:initial-value-problem-navier-stokes}.
\end{proof}

\begin{lemma}\label{bestiario}
For $2<s<2.5$, $\delta$, $\ep>0$ small enough and $v\in F^{s+1}$ the following estimates hold:

1. $\hhn{v}{\frac{s+1}{2}}{1-\ep}\leq C\fn{v}.$

2. $\hhn{v}{\frac{s+1}{2}+\ep}{1+\delta}\leq C \fn{v}.$

3. $\hhn{v}{\frac{s-1}{2}+\ep}{2+\delta}\leq C\fn{v}.$

4. $\hhn{v}{\frac{s}{2}-\frac{1}{4}+\ep}{2+\delta}\leq C\fn{v}.$

5. $\hhn{v}{1}{s-1}\leq C \fn{v}.$

6. $\hhn{v}{\frac{1}{2}+2\ep }{s}\leq C \fn{v}.$
\end{lemma}

\begin{proof}

We will show the most singular cases, which are 2. and 4. The rest is proved in an analogous way.


We will use the extension given by the lemma \ref{lema22} and thus we can consider $t \in \R$. Since $\Omega_{0}$ is a regular domain, we also consider the extension to the whole plane $\R^{2}$. This way, we can think of $v: \R \times \R^{2} \to \R$.

\underline{Case 2:}
We first consider $\ep = \delta = 0$. Using the Fourier transform in $\R \times \R^{2}$, we have

\begin{align}
\label{pacoestrella41}
\|v\|_{H^{\frac{s+1}{2}}H^{1}}^{2} \sim \int_{\R \times \R^{2}} (1 + |\tau|^{s+1}) (1 + |\xi|^{2}) |\hat{v}(\tau,\xi)|^{2} d\tau d\xi
= \int_{\R \times \R^{2}}(1 + |\xi|^{2} + |\tau|^{s+1} + |\tau|^{s+1}|\xi|^{2})|\hat{v}(\tau,\xi)|^{2} d\tau d\xi.
\end{align}

We only need to bound the previous integral by

\begin{align*}
 \|v\|_{L^{2}H^{s+1}}^{2} + \|v\|_{H^{2}H^{\gamma}}^{2} \sim \int_{\R \times \R^{2}}(2 + |\xi|^{2(s+1)} + |\tau|^{4} + |\xi|^{2\gamma} + |\tau|^{4} |\xi|^{2\gamma})|\hat{v}(\tau,\xi)|^{2} d\tau d\xi.
\end{align*}

We are only left to bound the integral with $|\tau|^{s+1}|\xi|^{2}$. Using Young's inequality:

\begin{align*}
 |\tau|^{s+1}|\xi|^{2\lambda}|\xi|^{2(1-\lambda)} \leq C(|\xi|^{2\lambda p } + |\tau|^{(s+1)q}  |\xi|^{2(1-\lambda)q}),
\end{align*}

where $p^{-1} + q^{-1} = 1$. Taking $q = \frac{4}{s+1}, p = \frac{4}{3-s}, \lambda = \frac{(s+1)(3-s)}{4}$, we are only left to check that

\begin{align*}
 (1-\lambda)q < \gamma < s-1,
\end{align*}

but elementary formulas show it as long as $s > 1$. It is easy to see that in the previous calculation there is room for the case $\ep, \delta > 0$ as long as $\ep$ and $\delta$ are sufficiently small.

\underline{Case 4:}

Proceeding similarly to Case 2, it is enough to show that

\begin{align*}
 |\tau|^{s-\frac12}|\xi|^{4} \leq C(1 + |\tau|^{4} + |\xi|^{2(s+1)} + |\tau|^{4} |\xi|^{2\gamma}).
\end{align*}

Using Young

\begin{align*}
 |\tau|^{s-\frac12}|\xi|^{4\lambda}|\xi|^{4(1-\lambda)} \leq C(|\xi|^{4\lambda p} + |\tau|^{(s-\frac12)q}|\xi|^{4(1-\lambda)q}).
\end{align*}

Taking $q = \frac{4}{s-1/2}, p = \frac{4}{9/2-s}, \lambda = \frac{(s+1)(9-2s)}{16}$, we are only left to check that

\begin{align*}
 2(1-\lambda)q < \gamma < s-1,
\end{align*}

which is true for $s > \frac{3}{2}$. Similarly, there is room for the case $\ep, \delta > 0$ and sufficiently small. Finally, the boundedness with respect to $T$ is shown if one considers the extension given by Lemma \ref{lema22}, part 2.

\end{proof}

\begin{lemma}\label{littlelemma} Let $f\in H^s([0,T])$ with $0\leq s \leq 2$ and $\pa_t^kf|_{t=0}=0$, $0\leq k<s-\frac{1}{2}$, with $s-\frac{1}{2}$ not an integer, then:
\begin{align*}||tf||_{H_{(0)}^s}&\leq CT||f||_{H_{(0)}^s}\\ ||t^2f||_{H_{(0)}^s}&\leq CT^2||f||_{H_{(0)}^s}\end{align*}
\end{lemma}
\begin{proof} We take $f\in H^2([0,T])$ such that, $f|_{t=0}=\pa_t f|_{t=0}=0$. We have that $||tf||_{L^2(0,T)}\leq T||f||_{L^2}$ and that $||tf||_{H_{(0)}^2[0,T]}\leq C T||f||_{H_{(0)}^2[0,T]}$. This is because $f(t)=\int_{0}^t\pa_\tau f(\tau)d\tau \Rightarrow ||f||_{L^2([0,T])}\leq T ||f_t||_{L^2([0,T])}$ and because of an analogous inequality for $\pa_t f(t)$. We get the inequality for the rest of the exponents by interpolation. A similar proof holds for $||t^2f||_{H_{(0)}^s}$.
\end{proof}

In the following lemmas we deal with the estimation of composition of functions. The functions $A$ and $Q$ given by the expressions \eqref{definiciondeA} and \eqref{definiciondeQ} respectively will appear together with the initial velocity $v_0 \,:\, \tilde{\Omega}\to \R^2$ which we assume regular enough (see discussion in Theorem \ref{nolineal} below). To indicate the dependence of a constant, $C$, with   respect to a quantity, $K$, we will use the notation $C[K]$.

\begin{lemma} \label{composicion1} Let $X-\alpha-Av_0t\in F^{s+1}$ with $2<s<2.5$. Then, for $T>0$ less than a small enough constant determined by $v_0$, $\inf_{\alpha\in \tilde{\Omega}} |\alpha|$ and $\fn{X-\alpha-Av_0t}$,
\begin{align*}
&||A\circ X||_{L^\infty H^{s+1}}\leq C[M,||X-\alpha-Av_0t||_{F^{s+1}}, v_0,||\alpha||_{L^2(\tilde{\Omega})}]\\
&||A\circ X-A||_{L^\infty H^{s+1}}\leq C[M,||X-\alpha-Av_0t||_{F^{s+1}}, v_0,||\alpha||_{L^2(\tilde{\Omega})}]\p{|| X-\alpha-Av_0t||_{L^\infty H^{s+1}}+||Av_0t||_{L^\infty H^{s+1}}}.
\end{align*}
with
\begin{align*}
M=  \frac{1}{\inf_{\alpha\in \tilde{\Omega}} |\alpha|-C[v_0]T-T^\frac{1}{4}||X-\alpha-Av_0t||_{F^{s+1}}}.
\end{align*}
\end{lemma}
\begin{proof}
By definition the matrix $A_{kj}=\pa_jP^k\circ P^{-1}$ contains terms of the form $\frac{x^i}{|x|^2}$ and therefore, three derivatives of $A\circ X$ contains terms of the form
\begin{align*}
&\frac{\pa^3_{jki}X^l}{|X|^2},\quad \frac{X^l\pa_jX^r\pa^2_{ik}X^u}{|X|^4},\quad \frac{\pa_jX^i\pa_kX^l\pa_uX^r}{|X|^4},\quad \frac{X^pX^q\pa_jX^i\pa_kX^l\pa_uX^r}{|X|^6},\quad \frac{X^iX^j\pa^3_{lmn}X^r}{|X|^4},\\
&\frac{X_iX_jX_k\pa^2_{m}X^l\pa^2_{rn}X^n}{|X|^6}, \frac{X^iX^jX^kX^l\pa_rX^u\pa_mX^p\pa_nX^q}{|X|^8}, \quad \text{for $i,j,k,l,m,n,r,u,p,q=1,2$.}
\end{align*}
Now we notice that $||f||_{H^{s+1}}=||f||_{H^{3+\delta}}=||f||_{L^2}+\sum_{i,j,k=1}^2||\pa^3_{ijk}f||_{H^\delta}$ for some $0<\delta<\frac{1}{2}$. And then, using that, $||fg||_{H^\delta}\leq ||f||_{H^2}||g||_{H^\delta}$ and that $H^2$ is an algebra we can check that
\begin{align*}
||A\circ X||_{H^{s+1}}\leq C \p{\left|\left|\frac{1}{|X|}\right|\right|^2_{L^\infty}+\left|\left|\frac{1}{|X|}\right|\right|_{L^\infty}^{12}}\p{||X||_{H^{s+1}}+||X||^{12}_{H^{s+1}}}
\end{align*}
In addition,
\begin{align*}
|X|\geq |\alpha|-|X-\alpha|\geq \inf_{\alpha\in\tilde{\Omega}} |\alpha|-||X-\alpha||_{L^\infty(\Omega)},
\end{align*}
and
\begin{align} &||X-\alpha||_{L^\infty(\Omega)}\leq \li{X-\alpha}{s+1}\leq \li{X-\alpha-Av_0t}{s+1}+T||Av_0||_{H^{s+1}}\nonumber\\ &\leq T^\frac{1}{4}\lit{X-\alpha-Av_0t}{s+1}+T||Av_0||_{H^{s+1}}. \label{parabajo}\end{align}

Thus
\begin{align*}
\left|\left|\frac{1}{|X|}\right|\right|_{L^\infty}\leq \frac{1}{\inf_{\alpha\in \tilde{\Omega}} |\alpha|-T||Av_0||_{H^{s+1}}-T^\frac{1}{4}\lit{X-\alpha-Av_0t}{s+1}}
\end{align*}

In addition we have that \begin{align}\label{pararriba}\li{X}{s+1}\leq \li{X-\alpha-Av_0t}{s+1}+\li{\alpha+Av_0t}{s+1}.\end{align}

The proof of the second inequality follows similar steps.
\end{proof}

\begin{lemma}\label{composicion2}Let $X-\alpha-Av_0t\in F^{s+1}$ with $2<s<2.5$. Then, for $T>0$ less than a small enough constant determined by $v_0$, $\inf_{\alpha\in \tilde{\Omega}} |\alpha|,$ and $\fn{X-\alpha-Av_0t}$,
\begin{align*}
\hhn{A\circ X -A}{1}{\gamma}\leq C[M,\fn{X-\alpha-Av_0t},v_0]\hhn{X-\alpha}{1}{\gamma}.
\end{align*}
with
\begin{align*}
M=  \frac{1}{\inf_{\alpha\in \tilde{\Omega}} |\alpha|-C[v_0]T-T^\frac{1}{4}||X-\alpha-Av_0t||_{F^{s+1}}}.
\end{align*}
\end{lemma}
\begin{proof}
By definition $\hhn{A\circ X-A}{1}{\gamma}=\lhn{A\circ X-A}{\gamma}+\lhn{\pa_t\p{A\circ X-A}}{\gamma}$. To control $\lhn{A\circ X-A}{\gamma}$ we first notice that we can write $\gamma=1+\delta$ with $0<\delta<1/2$, thus we need to control $\lhn{A\circ X-A}{0}+\lhn{\pa_\alpha \p{A\circ X-A}}{\delta}$. We focus on the term $\lhn{\pa_\alpha \p{A\circ X-A}}{\delta}$. Since $A\circ X$ contains terms of the form $\frac{X_i}{|X|^2}$, one spatial derivative of $A\circ X-A$ contains terms of the form $\frac{\pa_jX^i}{|X|^2}-\frac{\pa_j\al^i}{|\al|^2}$ and $\frac{X^iX^l\pa_j X^k}{|X|^4}-\frac{\alpha^i\alpha^l\pa_j \al^k}{|\al|^4}$. The bound for these different kinds of terms follows similar steps. For example, $$\left|\left|\frac{\pa_jX^i-\pa_j \al^i}{|X|^2}\right|\right|_{H^\delta}\leq ||\frac{1}{|X|^2}||_{H^2}||\pa_j(X^i-\al^i)||_{H^\delta}\leq C[||X||_{H^2}]C\left[\frac{1}{\inf |\alpha|-||X-\alpha||_{H^{s+1}}}\right]||X-\alpha||_{H^\gamma}.$$
This type of estimate allows to prove that
$$\lhn{A\circ X-A}{\gamma}\leq C[||X||_{L^\infty H^2}]C\left[\frac{1}{\inf |\alpha|-||X-\alpha||_{L^\infty H^{s+1}}}\right]||X-\alpha||_{H^0H^\gamma}.$$

The time derivative of $A\circ X-A$ contains terms of the form $\frac{\pa_t X^i}{|X|^2}-\frac{\pa_t \alpha^i}{|\alpha|^2}$ and $\frac{X^iX^j\pa_tX^l}{|X|^4}-\frac{\alpha^i\alpha^j\pa_t\alpha^l}{|\alpha|^4}$. In addition the most singular term we need to control in $\lhn{\pa_t\p{A\circ X-A}}{\gamma}$ is $\lhn{\pa_t\pa_\alpha\p{A\circ X-A}}{\delta}.$ Therefore we can check that
\begin{align*}
\lhn{\pa_t\p{A\circ X}}{\gamma}\leq C[||X||_{L^\infty H^2}]C\left[\frac{1}{\inf |\alpha|-||X-\alpha||_{L^\infty H^{s+1}}}\right]||\pa_t (X-\alpha)||_{H^0H^\gamma}.
\end{align*}
Finally we use \eqref{parabajo} and \eqref{pararriba} and that $\hhn{X-\alpha}{1}{\gamma}\leq \hhn{X-\alpha-Av_0t}{1}{\gamma}+\hhn{Av_0t}{1}{\gamma}$.

\end{proof}

\begin{lemma}\label{composicion3} Let $X-\alpha-Av_0t$, $Y-\alpha-Av_0t\in F^{s+1}$ with $2<s<2.5$. Then, for $T>0$ less than a small enough constant determined by $v_0$, $\inf_{\alpha\in \tilde{\Omega}} |\alpha|,$  $\fn{X-\alpha-Av_0t}$ and $\fn{Y-\alpha-Av_0t}$,
\begin{align*}
&\li{A\circ X-A\circ Y}{s+1}\leq C\left[M,\fn{X-\alpha-Av_0t},\fn{Y-\alpha-Av_0t}\right]\li{X-Y}{s+1},\\
&\hhn{A\circ X-A\circ Y}{1}{\gamma}\leq C\left[M,\fn{X-\alpha-Av_0t},\fn{Y-\alpha-Av_0t}\right]\hhn{X-Y}{1}{\gamma}
\end{align*}
where
\begin{align*}
M=\max\left\{\frac{1}{\inf |\alpha|-C[v_0]T-CT^\frac{1}{4}\lit{X-\alpha-Av_0t}{s+1}}, \frac{1}{\inf |\alpha|-C[v_0]T-CT^\frac{1}{4}\lit{Y-\alpha-Av_0t}{s+1} }\right\}
\end{align*}
\end{lemma}
\begin{proof}  The proof follows similar steps to those in the proofs of lemmas \ref{composicion1} and \ref{composicion2}.
\end{proof}

\begin{lemma}\label{composicionz}Let $X-\alpha-Av_0t\in F^{s+1}$ with $2<s<2.5$ and $\zeta=\p{\nabla X}^{-1}$. Then, for $T>0$ less than a small enough constant determined by $v_0$ and  $\fn{X-\alpha-Av_0t}$,
$$\li{\zeta}{s-1}+\sum_{j=1}^2||\pa_j \zeta||_{L^\infty H^{s-1}}\leq C[M,\fn{X-\alpha-Av_0t}]$$

$$\hhn{\zeta-\I}{\frac{s-1}{2}+\ep}{1+\delta}\leq C[M,\fn{X-\alpha-Av_0t}]\hhn{X-\alpha}{\frac{s-1}{2}+\ep}{2+\delta}$$



where
\begin{align*}
M=  \frac{1}{1-C[v_0]T-CT^\frac{1}{4}\fn{X-\alpha-Av_0t}-CT^\frac{1}{2}\fn{X-\alpha-Av_0t}^2}.
\end{align*}
\end{lemma}
\begin{proof} For the first estimate we proceed as follows. We estimate $\det\p{\nabla X}$ from below. We have that $\det \p{\nabla X}=\pa_1X^1\pa_2X^2-\pa_1X^2\pa_2X^1=1+\nabla\cdot (X-\alpha)+\det\nabla\p{X-\alpha}.$ Therefore \begin{align}\label{detbelow}|\det\p{\nabla X}|\geq 1-||\nabla \p{X-\alpha}||_{L^\infty(\Omega)}-||\nabla \p{X-\alpha}||^2_{L^\infty(\Omega)}\geq 1-C\li{X-\alpha}{1+s}-C\li{X-\alpha}{1+s}^2.\end{align} The rest of the proof follows similar steps to those in the proof of lemma \ref{composicion1}. Indeed, since $s=\delta+2$ with $0<\delta<\frac{1}{2}$ we need to look at two derivatives of $\zeta$. In addition $\zeta$ contains terms of the form $\frac{\pa_j X^i}{\det\p{\nabla X}}$. Thus, for example we need to control $\frac{\pa^3_{ijk}X^l}{\det \p{\nabla X}}$ in $H^\delta$. To do it we can proceed as follows, $\left|\left|\frac{\pa^3_{ijk}X^l}{\det \p{\nabla X}}\right|\right|_{H^\delta}\leq \left|\left|\frac{1}{\det \p{\nabla X}}\right|\right|_{H^2}\left|\left|\pa^3_{ijk}X^l\right|\right|_{H^\delta}\leq C[M]C[\li{X}{s+1}]$ thanks to \eqref{detbelow}. Finally we proceed as in \eqref{parabajo} and \eqref{pararriba}.

For the second estimate we first write
\begin{align*}
\zeta=\p{\nabla X}^{-1}=\frac{1}{\det\nabla X}\p{\nabla X}^\dag=\p{\frac{1}{\det{\nabla X}}-1}\p{\nabla X-\I}^\dag+\p{\frac{1}{\det\nabla X}-1}\I+\p{\nabla X-\I}^\dag+\I,
\end{align*}
where for a matrix
$
\left(\begin{array}{cc} a & b\\ c & d \end{array}\right)\quad \text{we define} \quad \left(\begin{array}{cc} a & b\\ c & d \end{array}\right)^\dag= \left(\begin{array}{cc} d & -b\\ -c & a \end{array}\right).
$

Since $\det\nabla X=1+\det \nabla\p{X-\alpha} + \nabla \cdot (X - \al)$ we have that
\begin{align*}
\zeta-\I = -\frac{\det\nabla\p{X-\alpha}+\nabla\cdot\p{X-\alpha}}{1+\nabla\cdot\p{X-\alpha}+\det\nabla\p{X-\alpha}}\p{\nabla\p{X-\alpha}}^\dag
-\frac{\det\nabla\p{X-\alpha}+\nabla\cdot\p{X-\alpha}}{1+\nabla\cdot\p{X-\alpha}+\det\nabla\p{X-\alpha}}\I+\p{\nabla\p{X-\alpha}}^\dag
\end{align*}
Then, we can estimate using lemma \ref{2.6}
\begin{align*}
&\hhn{\zeta-\I}{\frac{s-1}{2}+\ep}{1+\delta}\leq \hhn{\frac{\nabla\cdot\p{X-\alpha}+\det\nabla\p{X-\alpha}}{1+\nabla\cdot\p{X-\alpha}+\det\nabla\p{X-\alpha}}}{\frac{s-1}{2}+\ep}{1+\delta}\hhn{\nabla\p{X-\alpha}}{\frac{s-1}{2}+\ep}{1+\delta}\\
&+\hhn{\frac{\nabla\cdot\p{X-\alpha}+\det\nabla\p{X-\alpha}}{1+\nabla\cdot\p{X-\alpha}+\det\nabla\p{X-\alpha}}}{\frac{s-1}{2}+\ep}{1+\delta}+\hhn{\nabla\p{X-\alpha}}{\frac{s-1}{2}+\ep}{1+\delta},
\end{align*}
where we see that it is enough to estimate $\hhn{\frac{\det\nabla\p{X-\alpha} + \nabla \cdot(X-\al)}{1+\nabla \cdot\p{X-\alpha}\det\nabla\p{X-\alpha}}}{\frac{s-1}{2}+\ep}{1+\delta}$. In order to do it we notice that since $X-\alpha-Av_0t\in F^{s+1}$ and $||\nabla{X-\alpha}||_{L^\infty L^\infty}\leq ||X-\alpha-Av_0t||_{L^\infty H^{s+1}}+C T\leq C[v_0,\fn{X-\alpha-Av_0 t}]T^\frac{1}{4}$. Therefore, for $T$ small enough, we have that
\begin{align*}
\frac{\nabla\cdot\p{X-\alpha}+\det\nabla\p{X-\alpha}}{1+\nabla\cdot\p{X-\alpha}+\det\nabla\p{X-\alpha}}=\sum_{n=1}^\infty (-1)^{n+1} \p{\nabla\cdot\p{X-\alpha}+\det\nabla\p{X-\alpha}}^n.
\end{align*}
Thus using the inequality $\hhn{fg}{\frac{s-1}{2}+\ep}{1+\delta}\leq C\li{f}{1+\delta}\hhn{g}{\frac{s-1}{2}+\ep}{1+\delta}+C\hhn{f}{\frac{s-1}{2}+\ep}{1+\delta}\li{g}{1+\delta}$ yields
\begin{align*}
\hhn{\frac{\nabla\cdot\p{X-\alpha}+\det\nabla\p{X-\alpha}}{1+\nabla\cdot\p{X-\alpha}+\det\nabla\p{X-\alpha}}}{\frac{s-1}{2}+\ep}{1+\delta}\leq C\frac{\hhn{\nabla\p{X-\alpha}}{\frac{s-1}{2}+\ep}{1+\delta}+\hhn{\nabla\p{X-\alpha}}{\frac{s-1}{2}+\ep}{1+\delta}^2}{\p{1-C\li{\nabla\p{X-\alpha}}{1+\delta}-C\li{\nabla\p{X-\alpha}}{1+\delta}^2}^2}.
\end{align*}

This concludes the proof of the lemma.
\end{proof}

\begin{lemma}\label{composicionzfi} Let $\xn-\alpha-Av_0t \in F^{s+1}$ and $\zeta_\phi=\I+t\pa_t\zn|_{t=0}=\I-t\nabla\p{Av_0}$. Then, for $T>0$ less than a small enough constant determined by $v_0$ and  $\fn{\xn-\alpha-Av_0t}$,
$$\hhn{\zn-\zeta_\phi}{\frac{s-1}{2}+\ep}{1+\delta}\leq C[v_0, M]\p{\hhn{\xn-\alpha-Av_0t}{\frac{s-1}{2}+\ep}{2+\delta}+T^\frac{1}{2}}.$$
$$ \hhn{\zn-\zeta_\phi}{\frac{s+1}{2}+\ep}{0}\leq C[v_0, M]\p{\hhn{\xn-\alpha-Av_0t}{\frac{s+1}{2}+\ep}{1}+T^\frac{1}{2}}$$
where
\begin{align*}
M=  \frac{1}{1-C[v_0]T-CT^\frac{1}{4}\fn{\xn-\alpha-Av_0t}-CT^\frac{1}{2}\fn{\xn-\alpha-Av_0t}^2}.
\end{align*}
\end{lemma}
\begin{proof}
We write
\begin{align*}
&\zn=\frac{1}{\det\nabla \xn}\p{\nabla \xn}^\dag=\p{\frac{1}{\det\nabla \xn}-1+t\nabla\cdot\p{Av_0}}\p{\nabla(\xn-\alpha-Av_0t)}^\dag\\
&+\p{\frac{1}{\det\nabla \xn}-1+t\nabla\cdot\p{Av_0}}\p{\nabla\p{\alpha+tAv_0}}^\dag+ (1-t\nabla\cdot\p{Av_0})\p{\nabla\p{\xn-\alpha-Av_0t}}^\dag\\
&+\p{1-t\nabla\cdot\p{Av_0}}\p{\nabla\p{\alpha+Av_0t}}^\dag.
\end{align*}
Since $-\nabla\cdot\p{Av_0}\I+\p{\nabla \p{Av_0}}^\dag=-\nabla\p{Av_0}$ we find that
\begin{align*}
&\zn-\zeta_\phi=\frac{1}{\det\nabla \xn}\p{\nabla \xn}^\dag=\p{\frac{1}{\det\nabla \xn}-1+t\nabla\cdot\p{Av_0}}\p{\nabla(\xn-\alpha-Av_0t)}^\dag\\
&+\p{\frac{1}{\det\nabla \xn}-1+t\nabla\cdot\p{Av_0}}\p{\nabla\p{\alpha+tAv_0}}^\dag+ (1-t\nabla\cdot\p{Av_0})\p{\nabla\p{\xn-\alpha-Av_0t}}^\dag+M[v_0]t^2,
\end{align*}
where $M[v_0]$ is a matrix whose coefficients just depend on $v_0,\alpha$.  Since $||t^2||_{H^2[0,T]}\leq C T^\frac{1}{2}$ we only need to care about the terms different from $M[v_0]t^2$ in the previous expression. Using lemmas \ref{2.5m}, \ref{2.6} and \ref{littlelemma} we also see that actually it is enough to care about the term, $\frac{1}{\det\nabla \xn}-1+t\nabla\cdot\p{Av_0}$. For this term, since $\det\nabla \xn= 1+\nabla\cdot (\xn-\alpha)+\det\nabla (\xn-\alpha)$, we have that \begin{align*}&\frac{1}{\det\nabla X}-1+t\nabla\cdot\p{Av_0}\\&=\frac{-\nabla\cdot(\xn-\alpha-Av_0t)-\det(\nabla\p{\xn-\alpha})
+t\nabla\cdot(Av_0)\nabla\cdot(\xn-\alpha)+t\nabla\cdot(Av_0)\det(\nabla(\xn-\alpha))}
{1+\nabla\cdot (\xn-\alpha)+\det\nabla(\xn-\alpha)}\end{align*}

In the previous expression we can use that $\nabla\cdot(Av_0)\nabla\cdot(\xn-\alpha)=\nabla\cdot(Av_0)\nabla\cdot(\xn-\alpha-Av_0t)+t(\nabla\cdot(Av_0))^2$. In addition, since $\det\nabla (\xn-\alpha)=O(t^2)$ when $t$ goes to zero, we can check that, $\det\nabla(\xn-\alpha)=\det\nabla (\xn-\alpha-Av_0t)+F[v_0](\nabla(\xn-\alpha-Av_0t)+G[v_0]t^2$, where $F[v_0](\alpha)$ is a function linear in $\alpha$ whose coefficients just depend on $v_0$ and $G[v_0]$ is a coefficient that just depends on $v_0$. The previous splitting allows us to prove the lemma by using a similar strategy to the one in the proof of lemma \ref{composicionz}.
\end{proof}

\begin{lemma}\label{composiciondizeta}
Let $\xn-v_0-Av_0t,\, \xm-v_0-Av_0t \in F^{s+1}$ with  $2<s<2.5$. Then, for $T>0$ less than a small enough constant determined by $v_0$, $\fn{\xn-\alpha-Av_0t}$ and $\fn{\xm-\alpha-Av_0t}$,
\begin{align*}
&\hhn{\zeta^{(n)}-\zeta^{(n-1)}}{\frac{s-1}{2}+\ep}{1+\delta}\leq C[v_0, M]\hhn{\xn-\xm}{\frac{s-1}{2}+\ep}{2+\delta}\\
&\hhn{\zeta^{(n)}-\zeta^{(n-1)}}{\frac{s+1}{2}+\ep}{0}\leq C[v_0,M]\hhn{\xn-\xm}{\frac{s+1}{2}+\ep}{1}
\end{align*}
where

\begin{align*}
M= \max_{m=n,n-1}\left\{ \frac{1}{1-C[v_0]T-CT^\frac{1}{4}\fn{X^{(m)}-\alpha-Av_0t}-CT^\frac{1}{2}\fn{X^{(m)}-\alpha-Av_0t}^2}\right\}.
\end{align*}

\end{lemma}
\begin{proof}
The strategy to prove this lemma is similar to the one of lemma \ref{composicionzfi}. Here we need to make the splitting
\begin{align*}
&\zeta^{(n)}-\zeta^{(n-1)}=\frac{1}{\det \nabla \xn}\p{\nabla\xn}^\dag-\frac{1}{\det \nabla \xm}\p{\nabla\xm}^\dag\\
&=\p{\frac{1}{\det \nabla \xn}-\frac{1}{\det \nabla X^{(n-1)}}}\p{\nabla\xn}^\dag+\frac{1}{\det \nabla \xm}\p{\nabla(\xn-\xm)}^\dag\\
&=\p{\frac{1}{\det \nabla \xn}-\frac{1}{\det \nabla X^{(n-1)}}}\p{\nabla(\xn-\alpha-Av_0t)}^\dag+\p{\frac{1}{\det \nabla \xn}-\frac{1}{\det \nabla X^{(n-1)}}}\p{\nabla(\alpha-Av_0t)}^\dag\\
&+\p{\frac{1}{\det \nabla \xm}+1-t\nabla\cdot(Av_0)}\p{\nabla(\xn-\xm)}^\dag-(1-t\nabla\cdot(Av_0))\p{\nabla\p{\xn-\xm}}^\dag.
\end{align*}
\end{proof}

\begin{lemma} \label{composicionQ1} Let $X-v_0-Av_0t\in F^{s+1}$ with  $2<s<2.5$. Then, for $T>0$ less than a small enough constant determined by $v_0$, $\inf_{\alpha\in\tilde{\Omega}}|\alpha|$ and  $\fn{X-\alpha-Av_0t}$,
\begin{align*}
&||A\circ X-A||_{L^\infty H^{s-1}}\leq C[M,\fn{X-\alpha-Av_0t},v_0]\p{\li{X-\alpha-Av_0t}{s+1}+T||Av_0||_{H^{s+1}}}\\
&\hhn{A\circ X-A}{\frac{s-1}{2}}{1+\delta}\leq C[M,\fn{X-\alpha-Av_0t},v_0]\p{\hhn{X-\alpha-Av_0t}{\frac{s-1}{2}+\ep}{1+\delta}+T}\\
&||Q^2\circ X-Q^2||_{L^\infty H^{s-1}}\leq C[M,\fn{X-\alpha-Av_0t},v_0]\p{\li{X-\alpha-Av_0t}{s+1}+T||Av_0||_{H^{s+1}}}\\
&\hhn{Q^2\circ X-Q^2}{\frac{s-1}{2}}{1+\delta}\leq C[M,\fn{X-\alpha-Av_0t},v_0]\p{\hhn{X-\alpha-Av_0t}{\frac{s-1}{2}+\ep}{1+\delta}+T}
\end{align*}
with
\begin{align*}
M=  \frac{1}{\inf_{\alpha\in \tilde{\Omega}} |\alpha|-C[v_0]T-CT^\frac{1}{4}||X-\alpha-Av_0t||_{H^{s+1}}}.
\end{align*}
\end{lemma}

\begin{proof} The proof is similar to the ones of lemmas \ref{composicion1}, \ref{composicion2} and \ref{composicionz}. Notice that $Q^2(\alpha)=\frac{1}{|\alpha|^2}.$
\end{proof}

\begin{lemma}\label{composicionAfi} Let $\xn-v_0-Av_0t\in F^{s+1}$, with  $2<s<2.5$, and $\left(A_\phi\right)_{ij}=A_{ij}+t\left.\left(\frac{d}{dt}\left(A^{ij}\circ X^{(n)}\right)\right)\right|_{t=0}=A_{ij}+t\pa_kA_{ij}A_{kl}v_{0l}.$ Then, for $T>0$ less than a small enough constant determined by $v_0$, $\inf_{\alpha\in\tilde{\Omega}}|\alpha|$ and  $\fn{\xn-\alpha-Av_0t}$,
$$\hhn{A\circ\xn-A_\phi}{\frac{s-1}{2}+\ep}{1+\delta}\leq C[v_0,M]\p{\hhn{\xn-\alpha-Av_0t}{\frac{s-1}{2}}{1+\delta}+T^\frac{1}{2}}$$
$$\hhn{A\circ\xn-A_\phi}{\frac{s+1}{2}+\ep}{1+\delta}\leq C[v_0,M]\p{\hhn{\xn-\alpha-Av_0t}{\frac{s+1}{2}}{1+\delta}+T^\frac{1}{2})}$$
with
\begin{align*}
M=  \frac{1}{\inf_{\alpha\in \tilde{\Omega}} |\alpha|-C[v_0]T-CT^\frac{1}{4}||X-\alpha-Av_0t||_{H^{s+1}}}.
\end{align*}
\end{lemma}
\begin{proof}
The proof is similar to that one for lemma \ref{composicionzfi}.
\end{proof}

%
%
%

\section{Solving the linear equation}

In this section we want to solve the system given by:
\begin{align}
v_t - \nu Q^{2} \Delta v + A^{*} \nabla q & = f & \text{ in } \Omega_0 \times [0,T] \nonumber \\
Tr(\nabla v A) & = g & \text{ in } \Omega_0 \times [0,T] \nonumber \\
(q + (\nabla v A) + (\nabla v A)^{*})A^{-1} n & = h & \text{ on } \pa \Omega_0 \times [0,T] \nonumber \\
 v(x,0) & = 0 & \text{ in } \Omega_0 \label{plg}
\end{align}

Defining the following spaces:
\begin{align*}
 X_0 & := \{(v,q) : v \in H_{(0)}^{ht,s+1}, q \in H^{ht,s}_{pr,\,(0)}\} \\
Y_0 & := \{(f,g,h): f \in H_{(0)}^{ht,s-1}, g \in \overline{H}_{(0)}^{ht,s}, h \in H_{(0)}^{ht,s-\frac12}(\pa \Om \times [0,T]), \}
\end{align*}
(here we remark that in $X_0$ and $Y_0$: $v(0)=\pa_tv(0)=q(0)=f(0)=g(0)=\pa_tg(0)=h(0)=0$), then, we can write \eqref{plg} as:
\begin{align*}
 L(v,q) = (f,g,h,0); \quad L: X_0 \rightarrow Y_0, \quad 2 < s < \frac{5}{2}.
\end{align*}

\begin{thm}\label{Lmenos1}
 $L: X_0 \to Y_0$ is invertible for $2<s< \frac{5}{2}$. Moreover, $\|L^{-1}\|$ is bounded uniformly if $T$ is bounded above.
\end{thm}

\begin{proofthm}{Lmenos1}
The proof can be found in Appendix \ref{sectionlineal}.
\end{proofthm}

\section{Fixed point argument}\label{fixedpoint}

In section \ref{sectionlineal} the system \eqref{NSOfl}, \eqref{divfl} and \eqref{tensorfl} was solved uniformly in $T$ (for small $T$) but with the initial conditions $\tv|_{t=0}=0$ and $\pa_t \tv|_{t=0}=0$. Then we still need to carry out a small modification of (\ref{NSOfl}-\ref{difl}) to be able to apply the result of that section.

Let's define the approximated solution $\phi$ in $\tilde{\Om}_0 \times \mathbb{R}$ as follows
\begin{align}\label{fi}\phi= \tv_0 + t\left(Q^2\Delta \tv_0 -A^*\nabla \tilde{q}_\phi \right)\equiv \tv_0+t\psi,\end{align}
We shall show we can choose $\tilde{q}_\phi$ in such a way that $\pa_t\phi|_{t=0}=\pa_t v^{(n)}|_{t=0}$ for all $n$.

We specify who $\tilde{q}_\phi$ is. Given the system (\ref{NSOfl}-\ref{difl}),
for \begin{align*} v^{(n)}=&\tilde{v}^{(n)}\circ P\\ q^{(n)}=&\tilde{q}^{(n)}\circ P\end{align*} we have that
\begin{align*}
\pa_t v^{(n+1)} -\Delta v^{(n+1)} = & -\nabla q^{(n+1)}+f^{(n)} \quad \text{in $\Omega_0$}\\
\nabla \cdot v^{(n+1)} =&  g^{(n)}\quad \text{in $\Omega_0$}\\
\left(q^{(n+1)}\mathbb{I}-(\nabla v^{(n+1)}+(\nabla v^{(n+1)})^*)\right)n_0 = & h^{(n)} \quad \text{on $\pa\Omega_0$}\\
v^{(n+1)}|_{t=0}=& v_0,\end{align*}
where $v_0=\tilde{v}_0\circ P$, $f^{(n)}=\tilde{f}^{(n)}\circ P$, $g^{(n)}=\tilde{g}^{(n)}\circ P$ and $h^{(n)}=\tilde{h}^{(n)}\circ P$.
Taking the divergence on the first equation yields
$$\pa_t \nabla\cdot v^{(n+1)} -\Delta \nabla \cdot v^{(n+1)} =-\Delta q^{(n+1)} +\nabla\cdot f^{(n)}.$$
Thus, taking into account the third equation, we find $q^{(n+1)}$ by solving
\begin{align*}
-\Delta q^{(n+1)} = & \pa_t g^{(n)} -\Delta g^{(n)} -\nabla\cdot f^{(n)}\quad \text{in $\Omega_0$}\\
q^{(n+1)}|_{\pa\Omega_0}= & n_0 \left(\nabla v^{(n)} +(\nabla v^{(n)})^*\right)|_{\pa\Omega_0} \cdot n_0 +h^{(n)}|_{\pa\Omega_0}\cdot n_0.
\end{align*}
Next notice that $\tilde{f}^{(n)}$, $\tilde{g}^{(n)}$ and $\tilde{h}^{(n)}$ in \eqref{efe}, \eqref{ge} and \eqref{ache} are equal to zero at $t=0$.  Thus, calling $q^{(n)}_0=q^{(n)}|_{t=0}$, we have that
\begin{align*}
-\Delta q^{(n+1)}_0 = &(\pa_t g^{(n)})|_{t=0}\\
q^{(n)}_0|_{\pa\Omega_0}= & n_0 \left(\nabla v_0 +(\nabla v_0)^*\right)n_0.
\end{align*}
Now we focus our attention in the structure of $\tilde{g}^{(n)}$ in \eqref{ge}. It easy to check that $\tilde{g}^{(n)}\circ P$ can be written as
$$\tilde{g}^{(n)}\circ P= -Tr\left(\nabla v'^{(n)} \left(\nabla X^{(n)}\right)^{-1}\right)+Tr\left( \nabla v'^{(n)}\right),$$
where
$$v'^{(n)}=\tv^{(n)}\circ \tilde{X}^{-1}\circ P\circ X$$
with
\begin{align*}
\frac{dX^{(n)}}{dt}=&v'^{(n)}(\alpha,t)\\
X^{(n)}(\alpha,0)=&\alpha\quad \alpha\in \Omega_0
\end{align*}

Therefore
$$(\pa_t g^{(n)})|_{t=0}=Tr\left(\nabla u_0 \nabla u_0\right)=\nabla \cdot\left((u_0\cdot \nabla) u_0\right).$$

Because of the previous discussion we will choose $q_\phi$ solving
\begin{align*}
-\Delta q_\phi= & \nabla\cdot \left((u_0\cdot\nabla) u_0\right)\quad \text{in $\Omega_0$}\\
q_\phi|_{\pa\Omega_0}= & n_0\left(\nabla u_0 +(\nabla u_0)^*\right)|_{\pa\Omega_0}n_0.
\end{align*}
which is independent on the superscript $n$.

Finally one finds $\tilde{q}_\phi$ changing from $\Omega_0$ to $\tilde{\Omega}_0$,
\begin{align}
-Q^2\Delta \tilde{q}_\phi = & Tr\left(\nabla \tilde{u}_0 A \nabla \tilde{u}_0 A\right)\quad \text{in $\tilde{\Omega}_0$}\label{q0}\\
\tilde{q}_\phi|_{\pa\tilde{\Omega}_0} \left(A^{-1}\tilde{n}_0\cdot A^{-1}\tilde{n}_0\right) = & A^{-1}\tilde{n}_0\left(\nabla \ut_0A+(\nabla \ut_0 A)^*\right)A^{-1}\tilde{n}_0\quad \text{on $\pa\tilde{\Omega}_0$}\label{q0b}.
\end{align}

Once we have defined $\phi$ we define the velocity $\tilde{w}^{(n)}$ and the pressure $q^{(n)}_w$ by the expression
$$\tilde{w}^{(n)}=\tilde{v}^{(n)}-\phi \quad q^{(n)}_w=q^{(n)}-q_\phi$$
Now it is easy to check that
\begin{align*}\tilde{w}^{(n)}|_{t=0}=0, \quad  \left(\pa_t\tilde{w}^{(n)}\right)|_{t=0}=0.\end{align*}

Then it is better for our purpose to write the system (\ref{NSOfl}-\ref{difl}) in terms of $\tilde{w}^{(n)}$ rather than in terms of $\tv^{(n)}$. We obtain that

\begin{align}
&\pa_t \tilde{w}^{(n+1)} -Q^2\Delta \tilde{w}^{(n+1)}=-A^*\nabla \tilde{q}_w^{(n+1)}+\tilde{f}^{(n)}-\pa_t\phi+Q^2\Delta\phi- A^*\nabla \tilde{q}_\phi\quad \text{in $\tilde{\Om}_0$},\label{NSOflw}\\
&Tr\left(\nabla \tilde{w}^{(n+1)} A\right)=   \tilde{g}^{(n)}-Tr\left(\nabla \phi A\right)\quad \text{in $\tilde{\Om}_0$}\label{divflw}\\
&\left(\tilde{q}_w^{(n+1)}\mathbb{I}-\left(\left(\nabla \tilde{w}^{(n+1)}  A\right)+\left(\nabla \tilde{w}^{(n+1)}  A\right)^*\right)\right)A^{-1}\tilde{n}_0\nonumber\\&=\tilde{h}^{(n)}
-\tilde{q}_\phi A^{-1}\tilde{n}_0+\left(\left(\nabla \phi  A\right)+\left(\nabla \phi  A\right)^*\right)A^{-1}\tilde{n}_0 \quad \text{on $\partial \tilde{\Om} _0$},\label{tensorflw}\\
&\tilde{w}^{(n+1)}|_{t=0}=  0  \label{diflw}
\end{align}
where $\tilde{f}^{(n)}$, $\tilde{g}^{(n)}$ and $\tilde{h}^{(n)}$ are given by \eqref{efe}, \eqref{ge} and \eqref{ache} with $\tilde{v}^{(n)}=\tilde{w}^{(n)}+\phi$.
\begin{rem}
With this choice of the function $\phi$ we lose regularity of the solution with respect to the initial data. One could look for more sophisticated choices of $\phi$, like Beale in \cite{Beale:initial-value-problem-navier-stokes}, in order to avoid this loss. However, this exceeds the scope of this paper.
\end{rem}

We now prove the following theorem:

\begin{thm}\label{nolineal} Let $\left(\{\tilde{w}^{(n)}\}_{n=0}^\infty,\,\, \{\tilde{q}_
w^{(n)}\}_{n=0}^\infty,\,\,\{\tilde{X}^{(n)}\}_{n=0}^\infty\right)$ be the sequence given by the system \eqref{NSOflw}, \eqref{divflw}, \eqref{tensorflw} and \eqref{diflw} and \eqref{Xtilda} with
\begin{align*}\tilde{w}^{(0)}=0, && \tilde{q}^{(0)}_w=0, && \tilde{X}^{(0)}=\alpha+A\tv_0t,\end{align*}
where $\tilde{q}_\phi$ is given by \eqref{q0} and \eqref{q0b}, $\phi$ is given by \eqref{fi} and $\tilde{v}^{n}=\tilde{w}^{(n)}+\phi$.
Then $$\left(\{\tilde{w}^{(n)}\}_{n=0}^\infty,\,\,\{\tilde{q}_w^{(n)}\}_{n=0}^\infty,\,\,\{\tilde{X}^{(n)}\}_{n=0}^\infty-\alpha-Av_0t\right)$$ is a Cauchy sequence in
$$H_{(0)}^{ht, s+1}\left([0,T],\,\Om_0\right)\times H^{ht, s}_{pr\, (0)}\left([0,T],\,\Om_0\right)\times F^{s+1}\left([0,T],\,\Omega_0\right)$$
for $1<\gamma<s-1$, $2<s<2.5$ and
\begin{align*} ||\tilde{u}_0||_{H^{100}(\tilde{\Om}_0)}<C && |\pa\tilde{\Om}_0|_{C^{100}}< C,
\end{align*}

for $T$ sufficiently small \footnote{We assume a large number of derivatives mostly to simplify the exposition. However one can likely reduce this space to, say, $H^{10}$
by a slightly more careful analysis.}.
\end{thm}

In order to prove this result we use propositions \ref{prop1} and \ref{prop2} presented below concerning the system (\ref{NSOflw}-\ref{diflw}), together with \eqref{Xtilda}. We start by writing this system in the more concise form
\begin{equation}\label{conciso}
L (w^{(n+1)},q_w^{(n+1)})= \left(f^{(n)},g^{(n)}, h^{(n)}\right)+\left(f^{L}_{\phi},g^{L}_{\phi}, h^{L}_{\phi}\right).
\end{equation}
where
\begin{align}
 f^L_{\phi} & = - \pa_{t} \phi + Q^{2} \Delta \phi- A^{*} \nabla q_\phi,\nonumber\\
g^L_{\phi} & =  -Tr(\nabla \phi A),\label{fghL}\\
h^L_{\phi} & = -q_{\phi} A^{-1} n_0 + (\nabla \phi A + (\nabla \phi A)^{*}) A^{-1} n_0. \nonumber
\end{align}

For technical reasons we rewrite the right hand side in a different way.  First we notice that
\begin{align*}
\pa_t \zn=\pa_t \left[\left(\nabla \xn\right)^{-1}\right]=-\left(\nabla \xn\right)^{-1}\nabla \pa_t \xn\left(\nabla \xn\right)^{-1},
\end{align*}
and therefore
\begin{align*}
\pa_t \zn|_{t=0}= -\nabla\left(Av_0\right).
\end{align*}
We define $\zeta_\phi$, independent of $n$ by the following expression
$$\zeta_\phi= \I+t\left(\pa_t\zn|_{t=0}\right),$$
and also we define $A_\phi$ as the matrix with entries
$$\left(A_\phi\right)_{ij}=A_{ij}+t\left.\left(\frac{d}{dt}\left(A^{ij}\circ X\right)\right)\right|_{t=0}=A_{ij}+t\pa_kA_{ij}A_{kl}v_{0l}.$$
By using $\zeta_\phi$, $A_\phi$, 
we will write the system \eqref{conciso} in the following way
$$L(w^{(n+1)},q^{(n+1)}_w)=(f^{(n)},\overline{g}^{(n)},h^{(n)})+(f_{\phi}^{L},\overline{g}^{L}_{\phi},h^{L}_{\phi}),$$
with
\begin{align}\label{overlinegn}
\overline{g}^{(n)}=&g^{(n)}+Tr\left(\nabla \phi \zeta_\phi A_\phi\right)-Tr\left(\nabla\phi A\right)
\\ = &-Tr\p{\nabla \wn \zn A\circ \xn}+Tr\p{\nabla \wn A}-Tr\p{\nabla \phi \zn A\circ \xn}+Tr\p{\nabla \phi \zeta_\phi A_\phi}
\end{align}
and
\begin{align}\label{gLphi}
\overline{g}^{L}_{\phi}=&g^{L}_{\phi}-Tr\left(\nabla \phi\zeta_\phi A_\phi\right)+Tr\left(\nabla\phi A\right)=-Tr\p{\nabla \phi \zeta_\phi A_{\phi}}.
\end{align}

In this way we have that $\pa_t \overline{g}^L_\phi|_{t=0}=\overline{g}^L_\phi|_{t=0}=0$.

 In this situation we have the following result.

\begin{proposition}\label{prop1}\begin{enumerate}
\item Let $\xn-\alpha-Av_0t\in F^{s+1}$, $q_w^{(n)}\in H^{ht,\,s}_{pr\, (0)}$
 and $\wn\in H_{(0)}^{ht, s+1}$ and such that
\begin{align*}
\xn-\alpha-Av_0t\in & \left\{X-\alpha-A v_0t\in F^{s+1}\,:\, \fn{X-\alpha-\int_0^tA\phi\, d\tau}\leq N \right\}\\
&\equiv B_{A\psi} \\
(\wn, q_w^{(n)})\in &  \left\{(w,q)\in H_{(0)}^{ht,\, s+1}\times H^{ht,s}_{pr\, (0)}\,:\, w|_{t=0}=0, \pa_t w|_{t=0} = 0,\right.\\ &\left.\left|\left|(w,q)-L^{-1}(f_{\phi},\overline{g}_\phi,h_\phi)\right|\right|_{H_{(0)}^{ht,s+1}\times H^{ht,s}_{pr\, (0)}}\right.\left.\leq N\right\} \\
& \equiv B_{L^{-1}(f_\phi,\overline{g}_\phi,h_\phi)}
\end{align*}
where
\begin{align*}
N\equiv \max\left\{\fn{\int_0^t\tau A\psi\, d\tau},\,\left|\left|L^{-1}(f_\phi,\overline{g}_\phi,h_\phi)\right|\right|_{H_{(0)}^{ht,s+1}\times H^{ht,s}_{pr \, (0)}}, ||v_0||_{H^{100}}\right\}
\end{align*}
Then, for small enough $T>0$, depending only on $v_0$.
\begin{align*}
X^{(n+1)}-\alpha-Av_0t\in & B_{A\psi}
\end{align*}
\item Let $X^{(n)}-\alpha,X^{(n-1)}-\alpha\in B_{A\phi}$ and  $\left(w^{(n)},q^{(n)}\right),\left(w^{(n-1)},q^{(n-1)}\right)\in B_{L^{-1}(f_\phi,\overline{g}_\phi,h_\phi)}$  Then
\begin{align*}
||X^{(n+1)}-X^{(n)}||_{F^{s+1}}\leq C[v_0]T^\ep\left(\htn{w^{(n)}-w^{(n-1)}}{s+1}+||X^{(n)}-X^{(n-1)}||_{F^{s+1}}\right)
\end{align*}
for a small enough $\ep$.
\end{enumerate}
\end{proposition}

\begin{proof}
We refer the reader to Appendix \ref{appendixprop1} for details about the proof.
\end{proof}

\begin{proposition}
\label{prop2}
\begin{enumerate}

\item \;Let $\xn-\alpha-Av_0t\in F^{s+1}$, $q_w^{(n)}\in H^{ht,\,s}_{pr \, (0)}$ and $\wn\in H_{(0)}^{ht,s+1}$,  and such that
\begin{align*}
\xn-\alpha-Av_0t\in & \left\{X-\alpha-A v_0t\in F^{s+1}\,:\, \fn{X-\alpha-\int_0^tA\phi\, d\tau}\leq N \right\}\\
&\equiv B_{A\psi} \\
(\wn, q_w^{(n)})\in &  \left\{(w,q)\in H_{(0)}^{ht,\, s+1}\times H^{ht,s}_{pr\, (0)}\,:\, w|_{t=0}=0, \pa_t w|_{t=0} = 0, q_w|_{t=0} = 0,\right.\\ &\left.\left|\left|(w,q)-L^{-1}(f_{\phi},\overline{g}_\phi,h_\phi)\right|\right|_{H_{(0)}^{ht,s+1}\times H^{ht,s}_{pr\, (0)}}\right.\left.\leq N\right\} \\
& \equiv B_{L^{-1}(f_\phi,\overline{g}_\phi,h_\phi)}
\end{align*}
where
\begin{align*}
N\equiv \max\left\{\fn{\int_0^tA\psi\,\tau d\tau},\,\left|\left|L^{-1}(f_\phi,\overline{g}_\phi,h_\phi)\right|\right|_{H_{(0)}^{ht,s+1}\times H^{ht,s}_{pr \, (0)}},||v_0||_{H^{100}}\right\}
\end{align*}

Then
\begin{align*}
(w^{(n+1)}, q_w^{(n+1)})\in &  B_{L^{-1}(f_\phi,\overline{g}_\phi,h_\phi)}.
\end{align*}
\item Let $X^{(n)}-\alpha,X^{(n-1)}-\alpha\in B_{A\phi}$ and  $\left(w^{(n)},q^{(n)}\right),\left(w^{(n-1)},q^{(n-1)}\right)\in B_{L^{-1}(f_\phi,\overline{g}_\phi,h_\phi)}$
Then
\begin{align*}
&\htn{w^{(n+1)}-\wn}{s+1}+||q^{(n+1)}-\qn||_{H^{ht,s}_{pr \, (0)}}\\ &\leq C[v_0]T^\ep \left(\htn{\wn-\wm}{s+1}+\fn{\xn-\xm}+||\qn-\qm||_{H^{ht,s}_{pr \, (0)}}\right),
\end{align*}
for $\ep$ small enough.
\end{enumerate}

\end{proposition}

\begin{proof}
We refer the reader to Appendix \ref{appendixprop2} for details about the proof.
\end{proof}

\begin{cor}\label{corcontraction}
By the contraction mapping principle, there is a unique fixed point $(X-\alpha-Av_0t, w, q_w)\in F^{s+1}\times H_{(0)}^{ht,s+1}\times H^{ht,\,s}_{pr \, (0)}$, with $2<s<2.5$, for $T$ small enough which is the solution of (\ref{NSOf}-\ref{dif}).
\end{cor}

\section{Structural stability}\label{sectionstability}
Let us assume that $(v,q,X)$ is a solution with initial data $v_0$ and $\Om_0$, where $\Om_0$ is the projection by $P$ of a splash domain. We choose $v_0$ so that the normal component of $v_0$ in a neighbourhood of the splash points is directed towards $P(\Gamma)$, as shown in the next section. We also assume that $(v'_\ep, q'_\ep,X'_\ep)$ is a solution with initial data $v'_{\ep\,0}$ and $\Om_\ep$ where $\Om_\ep$ is a translation of $\Om_0$ of size $\ep$, i.e.,
$$\Om_\ep=\Om_0+\ep b$$
where $b$ is a constant vector, $|b|=1$ and such that $P^{-1}(\Omega_\ep)$ is a good domain. We define $(v_\ep, q_\ep, X_\ep)$ in the following way
\begin{align*}
v_\ep(\alpha,t)= v'_\ep(\alpha+\ep b,t),\qquad q_\ep(\alpha,t)= q'_\ep(\alpha+\ep b, t),\quad X_\ep(\alpha,t)= & X'_\ep(\alpha +\ep b,t),
\end{align*}
for $\alpha\in \Omega_0$.

For $\ep > 0$, $P^{-1}(\Omega_{\ep})$ is a good domain without self-intersections (as opposed to $\ep \leq 0$). By Corollary \ref{corcontraction}, we have local existence of solutions, and we can find a time of existence which is uniform in $\ep$. The existence of a splash singularity follows from the perturbative argument explained in the introduction.

 We choose $v'_{\ep 0}$ to have $v_\ep(\alpha,0)=v_0(\alpha)$. Since the change of variables from the variables with prime to the variables without prime is just a constant translation, the functions $(v_\ep, q_\ep, X_\ep)$ satisfy \eqref{NSOf}-\eqref{dif}
and \eqref{dtx}. However instead of \eqref{dtx0} we have that $X_\ep(\alpha,0)=\alpha+\ep b$.

We denote by $Q_{\ep}^ {2}(\al)$ and $A_{\ep}(\al)$ the following functions:

\begin{align*}
 Q_{\ep}^{2}(\al) = Q^{2}(\al + \ep b), \quad A_{\ep}(\al) = A(\al + \ep b).
\end{align*}

Therefore, we have

\begin{align*}
 \pa_{t} v_{\ep} - Q_{\ep}^{2} \Delta v_{\ep} + A_{\ep}^{*} \nabla q_{\ep} & = f_{\ep} \text{ in } \Omega_{0} \\
Tr(\nabla v_{\ep} A_{\ep}) & = g_{\ep} \text{ in } \Omega_0 \\
(q_{\ep} \I - ((\nabla v_{\ep} A_{\ep}) + (\nabla v_{\ep} A_{\ep})^{*}))A_{\ep}^{-1} n_0 & = h_{\ep} \text{ in } \pa \Omega_{0},
\end{align*}

where

\begin{align*}
 f_{\ep,i} & = Q^{2} \circ X_{\ep} (\zeta_{\ep})_{kj} \pa_{k} ((\zeta_{\ep})_{lj} \pa_{l} v_{i}) - Q^{2}_\ep \Delta v_{\ep,i} + (A_{\ep})_{ki}\pa_{k} q_{\ep} - A_{ki} \circ X_{\ep}(\zeta_{\ep})_{jk} \pa_{j} q_{\ep} \\
g_{\ep} & = Tr(\nabla v_{\ep} A_{\ep}) - Tr(\nabla v_{\ep} \zeta_{\ep} A \circ X_{\ep}) \\
h_{\ep} & = q_{\ep} A_{\ep}^{-1} n_0 - q_{\ep} A^{-1} \circ X_{\ep} \nabla_{J} X_{\ep} n_0 + (\nabla v_{\ep} \zeta_{\ep} A \circ X_{\ep} + (\nabla v_{\ep} \zeta_{\ep} A \circ X_{\ep})^{*}) A^{-1} \circ X_{\ep} \nabla_{J} X_{\ep} n_0 \\
& - (\nabla v_{\ep} A_{\ep} + (\nabla v_{\ep} A_{\ep})^{*}) A^{-1}_{\ep} n_0
\end{align*}

We construct $\phi_{\ep}$ in an analogous way: ensuring that $v_{\ep} = w_{\ep} + \phi_{\ep}$, with $w_{\ep} = \pa_{t} w_{\ep} = 0$. This yields:

\begin{align*}
 \phi_{\ep} & = v_0 + t  (Q^{2}_{\ep} \Delta v_0 - A_{\ep}^{*} \nabla q_{\phi,\ep})\equiv v_0+t\psi_\ep,
\end{align*}

where

\begin{align*}
 -Q^{2}_{\ep} \Delta q_{\phi,\ep} & = Tr(\nabla v_0 A_{\ep} \nabla v_0 A_{\ep}) \text{ in } \Omega_{0} \\
q_{\phi,\ep}(A_{\ep}^{-1} n_0 \cdot A_{\ep}^{-1} n_0) & = A_{\ep}^{-1}n_0(\nabla v_0 A_{\ep} + (\nabla v_0 A_{\ep})^{*}) A_{\ep}^{-1} n_0 \text{ in } \pa \Omega_{0},
\end{align*}
and we also take $q_{w,\ep}=q_\ep-q_{\phi,\ep}.$

Thus, we have the following system:

\begin{align*}
 \pa_t w_{\ep} - Q^{2}_{\ep} \Delta w_{\ep} + A_{\ep}^{*} \nabla q_{w,\ep} & = f_{\ep} + f^L_{\phi,\ep} \text{ in } \Omega_0, \\
Tr(\nabla w_{\ep} A_{\ep}) & = g_{\ep} + g^L_{\phi,\ep} \text{ in } \Omega_{0}, \\
(q_{w,\ep}I - ((\nabla w_{\ep} A_{\ep}) + (\nabla w_{\ep} A_{\ep})^{*})) A_{\ep}^{-1} n_0 & = h_{\ep} + h^L_{\phi,\ep} \text{ on } \pa \Omega_{0},
\end{align*}

where

\begin{align}
 f^L_{\phi,\ep} & = - \pa_{t} \phi_{\ep} + Q_{\ep}^{2} \Delta \phi_{\ep} - A_{\ep}^{*} \nabla q_{\phi_\ep}\nonumber\\
g^L_{\phi,\ep} & =  -Tr(\nabla \phi_{\ep} A_{\ep})\label{fghLe}\\
h^L_{\phi,\ep} & = -q_{\phi,\ep} A_{\ep}^{-1} n_0 + (\nabla \phi_{\ep} A_{\ep} + (\nabla \phi_{\ep} A_{\ep})^{*}) A_{\ep}^{-1} n_0. \nonumber
\end{align}

The next step will be to compare both solutions $(w,q_{w},X)$ and $(w_{\ep},q_{w,\ep},X_{\ep})$. Subtracting one equation from the other:

\begin{align*}
 \pa_t(w-w_{\ep}) - Q^{2} \Delta(w-w_{\ep}) + A^{*} \nabla (q_{w} - q_{w,\ep}) & = F_{\ep} \\
Tr(\nabla (w-w_{\ep}) A) & = G_{\ep} \\
((q_{w} - q_{w,\ep})I - ((\nabla(w-w_{\ep})A) + (\nabla(w-w_{\ep})A)^{*}))A^{-1} n_0 & = H_{\ep},
\end{align*}

where $F_{\ep}, G_{\ep}, H_{\ep}$ are given by

\begin{align}
 F_{\ep}  =& f - f_{\ep} + f^L_{\phi} - f^L_{\phi,\ep} - (Q^{2} - Q^{2}_{\ep}) \Delta w_{\ep} + (A - A_{\ep})^{*} \nabla q_{w,\ep} \label{Fe}\\
G_{\ep}  =& g - g_{\ep} + g^L_{\phi} - g^L_{\phi,\ep} - Tr(\nabla w_{\ep}(A-A_{\ep})) \label{Ge}\\
H_{\ep}  =& h - h_{\ep} + h^L_{\phi} - h^L_{\phi,\ep} - q_{w,\ep}(A^{-1} - A^{-1}_{\ep}) n_0 - ((\nabla w_{\ep} (A - A_{\ep})) + (\nabla w_{\ep} (A - A_{\ep}))^{*}) A_{\ep}^{-1} n_0 \label{He}\\
& - (\nabla w_{\ep} A + (\nabla w_{\ep} A)^{*})(A^{-1} - A^{-1}_{\ep})n_0.\nonumber
\end{align}

This means

\begin{align*}
 (w-w_{\ep},q_{w}-q_{w\,\ep}) = L^{-1}(F_{\ep}, G_{\ep}, H_{\ep}, 0).
\end{align*}

Taking norms as in the previous section, we obtain the following estimates:

\begin{align*}
 \|w-w_{\ep}\|_{H_{(0)}^{ht,s+1}} + \|q_{w}-q_{w,\ep}\|_{H^{ht,s}_{pr \, (0)}} \leq C(\|F_{\ep}\|_{H_{(0)}^{ht,s-1}}+\|G_{\ep}\|_{\overline{H}_{(0)}^{ht,s}}+|H_{\ep}|_{H_{(0)}^{ht,s-\frac12}}).
\end{align*}

Our goal is to prove the next lemmas:

\begin{lemma}
 \label{lemapacoestabilidad1}
We have the following estimate, for $\delta > 0$ small enough:

\begin{align*}
 &\fn{X-X_{\ep}+b\ep+(A_\ep-A)v_0t} \\
 &\leq C[v_0]\ep + C[v_0]T^{\delta}\p{\|w-w_{\ep}\|_{H_{(0)}^{ht,s+1}}+\fn{X-X_{\ep}+b\ep+(A_\ep-A)v_0t}}
\end{align*}

\end{lemma}

\begin{lemma}
 \label{lemapacoestabilidad2}
We have the following estimate, for $\delta > 0$ small enough:

\begin{align*}
&\|F_{\ep}\|_{H_{(0)}^{ht,s-1}} + \|G_{\ep}\|_{\overline{H}_{(0)}^{ht,s}} + |H_{\ep}|_{H_{(0)}^{ht,s-\frac{1}{2}}}  \leq C\ep + CT^{\delta}(\|w-w_{\ep}\|_{H_{(0)}^{ht,s+1}}+ \|q_{w}-q_{w,\ep}\|_{H^{ht,s}_{pr \, (0)}} \\
& + \|X-X_{\ep}+b\ep+(A_\ep-A)v_0t\|_{F^{s+1}}).
\end{align*}

\end{lemma}

We remark that the constants in the previous lemmas are uniformly bounded if $T$ is bounded above. Combining both inequalities, for $0 < T \leq \frac{1}{(2C)^{1/\delta}}$:

\begin{align*}
& \|w-w_{\ep}\|_{H_{(0)}^{ht,s+1}}+ \|q_{w}-q_{w,\ep}\|_{H^{ht,s}_{pr \, (0)}} + \|X-X_{\ep}+b\ep+(A_\ep-A)v_0t\|_{F^{s+1}} \leq 2C\ep
\end{align*}

Thus, for $0 < T \leq \frac{1}{(2C)^{1/\delta}}$:

\begin{align*}
 \|X-X_{\ep}+b\ep+(A_\ep-A)v_0t\|_{L^{\infty}H^{s+1}} \leq 2C\ep\quad\Rightarrow\quad \|X-X_{\ep}\|_{L^{\infty}H^{s+1}} \leq C\ep
\end{align*}
and therefore both functions are as close as we want for a time that only depends on the local existence time of the solution, but it does not depend on $\ep$. We therefore have

\begin{align*}
 X(\Omega_0,t) \approx X_{\ep}(\Omega_{0,\ep}, t)
\end{align*}

and for $\ep$ small enough there exists a $t_s$ such that

\begin{align*}
 P^{-1}(X_{\ep}(\Omega_{0,\ep},t_s)) = P^{-1}(X_{\ep}(\Omega_{0} + \ep b, t_s))
\end{align*}

is a splash domain. Next, we will show the proofs of the lemmas:

\begin{prooflemma}{lemapacoestabilidad1}
 We recall that

\begin{align*}
 \frac{dX_{\ep}}{dt}  = A \circ X_{\ep} v_{\ep}, \qquad &\qquad X_{\ep}(\al,0)  = \al + \ep b.
\end{align*}

and that,  for $T<T_0$, $T_0$ small enough, we control the following norms,

\begin{align*}
\fn{X-\alpha-Av_0t}\leq C[v_0], && \fn{X_\ep-\alpha-b\ep -A_\ep v_0t}\leq C[v_0],\\
\htn{w}{s+1}\leq C[v_0] && \htn{w_\ep}{s+1}\leq C[v_0],\\
||q||\leq C[v_0] && ||q_\ep||\leq C[v_0].
\end{align*}

We first deal with the $L^\infty_{\frac{1}{4}}H^{s+1}$-norm. We can write

\begin{align*}
 \|X-X_{\ep}+b\ep-(A-A_\ep)v_0t\|_{H^{s+1}}(t) & \leq \int_{0}^{t} \|A \circ X v - A \circ X_{\ep} v_{\ep}-(A-A_\ep)v_0\|_{H^{s+1}} d\tau
\end{align*}
In order to bound this norm we will split in the following way, $A \circ X v - A \circ X_{\ep} v_{\ep}-(A-A_\ep)v_0=A\circ X w-A\circ X_\ep w_\ep+(A\circ X-A\circ X_\ep -A+A_\ep)v_0+ t(A\circ X\psi-A\circ X_\ep \psi_\ep)=d_1+d_2+d_3$. For $d_1$ we write $d_1=(A\circ X-A-A\circ X_\ep+A_\ep)w+(A-A_\ep)w+A\circ X_\ep (w-w_\ep)$ and estimate $\int_{0}^t ||d_1||_{H^{s+1}}\leq ||A\circ X-A-A\circ X_\ep+A_\ep||_{L^\infty H^{s+1}}t^\frac{1}{2}||w||_{L^2H^{s+1}}+||A-A_\ep||_{H^{s+1}}t^\frac{1}{2}||w||_{L^2H^{s+1}}+||A\circ X_\ep||_{L^\infty H^{s+1}}t^\frac{1}{2}||w-w_\ep||_{L^2H^{s+1}}$. The second term on the right hand side of the this last inequality is bounded by $C[v_0]\ep t^\frac{1}{2}$ and the third one, by using lemma \ref{composicion1}, by $C[v_0]t^\frac{1}{2}||w-w_\ep||_{L^2H^{s+1}}$. For the first one we have that
\begin{align*}
||A\circ X-A-A\circ X_\ep+A_\ep||_{L^\infty H^{s+1}}\leq&||A\circ (X+b\ep)-A\circ X_\ep ||_{L^\infty H^{s+1}}\\
&+||A\circ X-A-A\circ (X+b\ep)+A_\ep||_{L^\infty H^{s+1}}
\end{align*}
and therefore
\begin{align*}
||A\circ X-A-A\circ X_\ep+A_\ep||_{L^\infty H^{s+1}}\leq& C[v_0] (||X-X_\ep+b\ep||_{L^\infty H^{s+1}}+\ep)\\
\leq& C[v_0](||X-X_\ep+b\ep-(A-A_\ep)v_0t||_{F^{s+1}}+\ep),
\end{align*} where this last inequality is proven in the same way that lemma \ref{composicion3}. This is enough to bound $||\int_{0}^td_1||$. The term $d_2$ can be bounded in a similar way after estimating $||\int_{0}^t d_2||_{H^{s+1}}\leq ||A\circ X-A-A\circ X_\ep+A_\ep||_{L^\infty H^{s+1}}tC[v_0]$. To bound $d_3$ we proceed in a similar way, the main difference is that we have to get the estimate $||\psi-\psi_\ep||_{H^{s+1}}\leq C[v_0]\ep$. Let us  prove this estimate. We have that
\begin{align}\label{estrella}
 \|A-A_{\ep}\|_{H^{r}} \leq C\ep, \quad \|Q^{2} - Q^{2}_{\ep}\|_{H^{r}} \leq C\ep, \text{ for all } r,
\end{align}
since $Q$ and $A$ are $C^{\infty}$ functions in $\Omega$. Thus
\begin{align*}
 -\Delta (q_{\phi} - q_{\phi,\ep}) & = \frac{1}{Q^{2}}Tr(\nabla v_0 A \nabla v_0 A) - \frac{1}{Q^{2}_{\ep}}Tr(\nabla v_0 A_{\ep} \nabla v_{0} A_{\ep})\qquad\mbox{in }\Omega_0 \\
q_{\phi} - q_{\phi,\ep} & = \frac{A^{-1} n_0(\nabla v_0 A + (\nabla v_0 A)^{*})A^{-1}n_0}{|A^{-1} n_0|^{2}} - \frac{A_{\ep}^{-1}n_0((\nabla v_0 A_{\ep}) + (\nabla v_0 A_{\ep})^{*})A_{\ep}^{-1}n_0}{|A_{\ep}^{-1}n_0|^{2}}\qquad\mbox{on }\partial\Omega_0.
\end{align*}
This implies:

\begin{align}\label{qmqe}
 \|q_{\phi} - q_{\phi,\ep}\|_{H^{r+1}} \leq C(\|\Delta(q_{\phi} - q_{\phi,\ep})\|_{H^{r-1}} + |q_{\phi}-q_{\phi,\ep}|_{H^{r+\frac12}}) \leq C\ep, \text{ for all $r\geq 0$} .
\end{align}

Using the definition of $\psi, \psi_{\ep}$:

\begin{align*}
 \psi - \psi_{\ep} & = ((Q-Q_{\ep})\Delta v_{0} - (A^{*} \nabla q_{\phi} - A_{\ep}^{*} \nabla q_{\phi,\ep}))
\end{align*}

yields

\begin{align*}
 \|\psi - \psi_{\ep}\|_{H^{s+1}} \leq C \ep,
\end{align*}

for sufficiently smooth $v_0$.

To get the $H^{2}H^{\gamma}$ estimate, we find that for all $t \leq 1$:

\begin{align*}
 \|X-X_{\ep}+b\ep-(A-A_\ep)v_0t\|_{H_{(0)}^{2}H^{\gamma}} &   \leq \left\|\int_{0}^{t}(A \circ X v - A \circ X_{\ep} v_{\ep}) - (A-A_{\ep})v_0 d\tau\right\|_{H_{(0)}^{2}H^{\gamma}}.
\end{align*}
We will need to make the same splitting as before $A \circ X v - A \circ X_{\ep} v_{\ep}-(A-A_\ep)v_0 = A\circ X w-A\circ X_\ep w_\ep+\p{A\circ X-A\circ X_\ep-(A-A_\ep)}v_0+t\p{A\circ X\psi-A\circ X_\ep\psi_\ep }\equiv d_1+d_2+d_3.$

In addition we split $d_1=(A\circ X-A\circ X_\ep-(A-A_\ep))w+(A-A_\ep)w+ (A\circ X_\ep-A_\ep)(w-w_\ep)+A_\ep(w-w_\ep)=d_{11}+d_{12}+d_{13}+d_{14}$. We have that $$\hhn{\int_{0}^td_{11}}{2}{\gamma}\leq \hhn{d_{11}}{1}{\gamma}\leq \hhn{A\circ X-A\circ X_\ep-(A-A_\ep)}{1}{\gamma}\hhn{w}{1}{\gamma}.$$ Analogous to the proof of lemma \ref{composicion3} we can prove that $\hhn{A\circ X-A\circ X_\ep-(A-A_\ep)}{1}{\gamma}\hhn{w}{1}{\gamma}\leq C[v_0](\hhn{X-X_\ep+b\ep}{1}{\gamma}+\ep)\leq C[v_0](T^\delta \hhn{X-X_\ep+b\ep}{1+\delta}{\gamma}+\ep)$. In addition, we have that  $\hhn{X-X_\ep+b\ep}{1+\delta}{\gamma}\leq \hhn{X-X_\ep+b\ep+(A_\ep-A)v_0t}{1+\delta}{\gamma}+C[v_0]\ep T^{\frac{1}{2}-\delta}$. For $\hhn{\int_{0}^td_{12}}{2}{\gamma}$ we find the bound $C[v_0]\ep$.  For $\hhn{\int_{0}^td_{13}}{2}{\gamma}$ we have the bound $C[v_0]T^\delta$ and  $\hhn{\int_{0}^td_{14}}{2}{\gamma}\leq C[v_0]T^\delta.$ To bound $\hhn{\int_{0}^td_2}{2}{\gamma}$ we use that $\hhn{A\circ X-A\circ X_\ep -(A-A_\ep)}{1}{\gamma}\leq C[v_0]( \hhn{X-X_\ep+b\ep}{1}{\gamma}+\ep)$, thus we can bound finally by $C[v_0]\p{T^\delta\fn{X-X_\ep+b\ep +(A-A_\ep)v_0t}+\ep}$. In order to be done for $d_{3}$ we have that $d_3=(A\circ X-A)t\psi+ t(A\psi-A_\ep\psi)-(A\circ X_\ep-A_\ep)t\psi_\ep$. By using this splitting and lemma \ref{littlelemma} we find that $\hhn{\int_{0}^td_3}{2}{\gamma}\leq C[v_0]\ep T$. This concludes the estimation of $\|X-X_{\ep}+b\ep-(A-A_\ep)v_0t\|_{H_{(0)}^{2}H^{\gamma}}$.

\end{prooflemma}

\begin{prooflemma}{lemapacoestabilidad2}\\

\underline{Part I: Estimates for $F_{\ep}$:}\\

We consider first in \eqref{Fe} the terms

\begin{align*}
 -(Q^{2} - Q^{2}_{\ep})\Delta w_{\ep} + (A - A_{\ep})^{*} \nabla q_{w,\ep}.
\end{align*}

Using \eqref{estrella} it is easy to check that

\begin{align*}
 \|(Q^{2} - Q^{2}_{\ep})\Delta w_{\ep}\|_{H_{(0)}^{ht,s-1}} + \|(A-A_{\ep})^{*} \nabla q_{w,\ep}\|_{H_{(0)}^{ht,s-1}} \leq C[v_0] \ep.
\end{align*}

Concerning $f$, we split as in \eqref{fsplit} (ignoring the superindices), with $f = f_{w} + f_{\phi} + f_{q}$. In an analogous way, in \eqref{Fe} we consider $f_{\ep} = f_{w,\ep} + f_{\phi,\ep} + f_{q,\ep}$.

We then split as:

\begin{align*}
 f_{w} - f_{w,\ep} & = Q^{2} \circ X \zeta \pa(\zeta \pa w) - Q^{2} \Delta w - (Q^{2} \circ X_{\ep} \zeta_{\ep} \pa(\zeta_{\ep} \pa w_{\ep}) - Q^{2}_{\ep} \Delta w_{\ep}) \\
& = d^{e}_{1} + d^{e}_{2} + d^{e}_{3} + d^{e}_{4} + d^{e}_{5} + d^{e}_{6} + d^{e}_{1,\ep},
\end{align*}

with

\begin{align*}
&d^{e}_{1} = (Q^{2} \circ X - Q^{2} \circ X_{\ep}) \zeta \pa (\zeta \pa w),
&d^{e}_{2} = Q^{2} \circ X_{\ep} (\zeta -\zeta_{\ep})\pa (\zeta \pa w) \\
&d^{e}_{3} = Q^{2} \circ X_{\ep} \zeta_{\ep}\pa ((\zeta -\zeta_{\ep})\pa w),
&d^{e}_{4} = (Q^{2} \circ X_{\ep} - Q^{2}_{\ep})\zeta_{\ep}\pa (\zeta_{\ep}\pa (w-w_{\ep})) \\
&d^{e}_{5} = Q^{2}_{\ep}(\zeta_{\ep}-I)\pa (\zeta_{\ep}\pa (w-w_{\ep})),
&d^{e}_{6} = Q^{2}_{\ep}\pa((\zeta_{\ep}-I)\pa (w-w_{\ep})) \\
& d^{e}_{1,\ep}  = (Q^{2}_{\ep} - Q^{2}) \Delta w &
\end{align*}

As before, \eqref{estrella} yields

\begin{align*}
 \|d^{e}_{1,\ep}\|_{H_{(0)}^{ht,s-1}} \leq C[v_0] \ep.
\end{align*}

To estimate $d^{e}_{j}$, $1 \leq j \leq 6$, we will compare the procedure with the one for $d_j$ \eqref{diferenciaf}. It is easy to see that they are similar, by identifying $X^{n}, \zeta^{n}, w^{n}$ with $X, \zeta, w$ and $X^{n-1}, \zeta^{n-1}, w^{n-1}$ with $X_{\ep}, \zeta_{\ep}, w_{\ep}$. As an illustration, we will discuss $d^{e}_{1}$ with detail. We first split $d^e_1=(Q^2\circ X-Q^2\circ X_\ep-(Q^2-Q^2_\ep))\zeta\pa(\zeta\pa w)+(Q^2-Q^2_\ep)\zeta\pa(\zeta\pa w).$ Then $\lhn{(Q^2-Q^2_\ep)\zeta\pa(\zeta\pa w)}{s-1}\leq C[v_0]\ep$ and, in addition, the $\lhn{\cdot}{s-1}$-norm of the first one is bounded by $C[v_0]||X-X_\ep+b\ep||_{L^\infty H^{s-1}}\leq C[v_0]\p{\ep+T^\frac{1}{4}\fn{X-X_\ep+b\ep-(A-A_\ep)v_0t}}$.  Concerning the $H^{\frac{s-1}{2}}L^{2}$ the same splitting yields
\begin{align*}
 \|d_1^{e}\|_{H_{(0)}^{\frac{s-1}{2}}L^{2}}  & \leq C\ep + CT^{\delta}(\|w-w_{\ep}\|_{H_{(0)}^{ht,s+1}} + \|X-X_{\ep}+b\ep-(A-A_\ep)v_0t\|_{F^{s+1}} ,
\end{align*}
Analogous estimates can be deduced for $d_{j}^{e}, 2 \leq j \leq 6$, achieving
\begin{align*}
 \sum_{j=1}^{6} \|d_{j}^{e}\|_{H_{(0)}^{ht,s-1}} \leq C\ep + CT^{\delta}(\|w-w_{\ep}\|_{H_{(0)}^{ht,s+1}} + \|X-X_{\ep}+b\ep-(A-A_\ep)v_0 t\|_{F^{s+1}})
\end{align*}
We next consider
\begin{align*}
 f_{\phi} - f_{\phi,\ep} & = Q^{2} \circ X \zeta \pa(\zeta \pa \phi) - Q^{2} \Delta \phi - Q^{2} \circ X_{\ep} \zeta_{\ep} \pa(\zeta_{\ep} \pa \phi_{\ep}) + Q^{2}_{\ep} \Delta \phi_{\ep}.
\end{align*}
We need to make the following splitting
\begin{align*}
f_{\phi}-f_{\phi, \ep}  =& ((Q^2\circ X-Q^2\circ X_\ep)\zeta\pa(\zeta\pa \phi)-Q^2 \Delta \phi +Q^2_\ep \Delta \phi)\\
&+ Q^2\circ X_\ep (\zeta-\zeta_\ep)\pa\p{\zeta\pa \phi}\\
&+ Q^2\circ X_\ep \zeta_\ep\pa\p{(\zeta-\zeta_\ep)\pa\phi}\\
&+Q^2\circ X_\ep \zeta_\ep \pa\p{\zeta_\ep \pa(\phi-\phi_\ep)}\equiv d^\phi_1+d^\phi_2+d^\phi_3+d^\phi_4.
\end{align*}
In addition
\begin{align*}
d_1^\phi =& (Q^2\circ X-Q^2\circ X_\ep -Q^2+Q^2_\ep)\zeta\pa(\zeta\pa\phi)+(Q^2-Q^2_\ep)(\zeta-\I)\pa(\zeta\pa\phi)\\&
+(Q^2-Q^2_\ep)\pa((\zeta-\I)\pa\phi)+(Q^2-Q^2_\ep)\Delta \phi-Q^2\Delta \phi+Q^2_\ep\Delta\phi_\ep.
\end{align*}
where we notice that $(Q^2-Q^2_\ep)\Delta \phi-Q^2\Delta \phi+Q^2_\ep\Delta\phi_\ep = tQ^2_\ep\Delta(\psi_\ep-\psi)$. This splitting allows us to prove a suitable bound for $d_1^\phi$. The rest of the term $d^\phi_i$ needs of similar splitting in order to be bounded.

We now estimate $f_{q} - f_{q,\ep}$. We  split $f_{q} - f_{q,\ep}=f_{q_w}-f_{q_{w\ep}}+f_{q_\phi}-f_{q_{\phi\ep}}$ where
\begin{align*}
&f_{q_{w}}-f_{q_{w\ep}}=\p{A\circ X_\ep\zeta_\ep}^*\nabla q_{w\ep}-(A\circ X\zeta)^*\nabla q_w-A_\ep^*\nabla q_{w\ep}+A^*\nabla q_w\\
&f_{q_\phi}-f_{q_{\phi\ep}}=\p{A\circ X_\ep\zeta_\ep}^*\nabla q_{\phi\ep}-(A\circ X\zeta)^*\nabla q_\phi-A_\ep^*\nabla q_{\phi\ep}+A^*\nabla q_\phi
\end{align*}
For $f_{q_{w}}-f_{q_{w\ep}}$ we can write
\begin{align*}
&f_{q_{w}}-f_{q_{w\ep}}=\p{(A\circ X_\ep-A\circ X-A_\ep+A)\zeta_\ep}^*\nabla q_{w_\ep}+\p{(A_\ep-A)(\zeta_\ep-\I)}^*\nabla q_{w_\ep}+\p{(A\circ X-A)(\zeta-\zeta_\ep)}^*\nabla q_{w_\ep}\\
&\p{A(\zeta-\zeta_\ep)}^*\nabla q_w+\p{(A\circ X-A)(\zeta-\I)}^*\nabla(q_{w\ep}-q_w)+\p{(A\circ X-A)}^*\nabla(q_{w_\ep}-q_w)+A^*\nabla(q_{w_\ep}-q_w).
\end{align*}
After this splitting the way to control $||f_{q_{w}}-f_{q_{w\ep}}||_{H_{(0)}^{ht,\, s-1}}$ is similar to the way we control $f_{q_w}^{(n)}-f_{q_w}^{(n-1)}$ in proposition \ref{prop2}.

To control $||f_{q_{\phi}}-f_{q_{\phi\ep}}||_{H_{(0)}^{ht,\, s-1}}$ we make a similar splitting
\begin{align*}
&f_{q_{\phi}}-f_{q_{\phi\ep}}=\p{(A\circ X_\ep-A\circ X-A_\ep+A)\zeta_\ep}^*\nabla q_{\phi_\ep}+\p{(A_\ep-A)(\zeta_\ep-\I)}^*\nabla q_{\phi_\ep}+\p{(A\circ X-A)(\zeta-\zeta_\ep)}^*\nabla q_{\phi_\ep}\\
&\p{A(\zeta-\zeta_\ep)}^*\nabla q_\phi+\p{(A\circ X-A)(\zeta-\I)}^*\nabla(q_{\phi\ep}-q_\phi)+\p{(A\circ X-A)}^*\nabla(q_{\phi_\ep}-q_\phi)+A^*\nabla(q_{\phi_\ep}-q_\phi).
\end{align*}
And we just need to use the definitions of $q_\phi$ and $q_{\phi_\ep}$ to obtain a suitable estimate.

With this estimate we finish the control of $f - f_\ep$. We are left to estimate $f^{L}_{\phi} - f^{L}_{\phi,\ep}$ given in \eqref{fghL} and \eqref{fghLe}. Here we notice that $f^L_{\phi}=tQ^2\Delta \psi$, $f^L_{\phi_\ep}=tQ^2_\ep \Delta \psi_\ep$, thus,  proceeding as before, the control for $\psi - \psi_{\ep}, Q - Q_{\ep}, A - A_{\ep}$ makes us conclude that
\begin{align*}
 \|f^{L}_{\phi} - f^{L}_{\phi,\ep}\|_{H_{(0)}^{ht,s-1}} \leq C\ep,
\end{align*}
and we are done with $F_{\ep}$.\\

\underline{Part II: Estimates for $G_{\ep}$:}\\

We consider first the following splitting
$$
G_\ep=-Tr(\nabla w_{\ep}(A-A_{\ep}))+(g-g_\ep)+(g_\phi^L-g_{\phi,\ep}^L).
$$
The first term can be estimated using \eqref{estrella} in a way such that
\begin{align*}
	\|Tr(\nabla w_{\ep}(A-A_{\ep}))\|_{\overline{H}_{(0)}^{ht,s}} \leq C[v_0]\ep,
\end{align*}
due to $w_\ep=\pa_t w_\ep=0$ and $A-A_\ep$ does not depend on time.

Next we consider the $L^{2}H^{s}$ norm for the two terms left. We consider $g - g_{\ep} = \sum_{j=1}^{5} d_{j}^{e}$
where
\begin{align*}
	d_{1}^{e}  = -Tr(\nabla(v-v_{\ep})(\zeta -I) A \circ X),&\qquad
	d_{2}^{e} = -Tr(\nabla v_{\ep}(\zeta - \zeta_{\ep})A \circ X), \\
	d_{3}^{e}  = -Tr(\nabla v_{\ep} \zeta_{\ep}(A \circ X - A \circ X_{\ep})),&\qquad
	d_{4}^{e}  = -Tr(\nabla(v-v_{\ep})(A \circ X - A), \\
	d_{5}^{e}  = -Tr(\nabla v_{\ep}(A_{\ep} - A)).&\qquad
\end{align*}

To estimate $d_{1}^{e}$ we compare it with \eqref{d1g}. In an analogous way we get

\begin{align*}
	\|d_{1}^{e}\|_{L^{2}H^{s}} &  \leq \|\nabla(w-w_{\ep})(\zeta - I)A \circ X\|_{L^{2}H^{s}}
	+ \|\nabla(\phi - \phi_{\ep})(\zeta - I)A \circ X\|_{L^{2}H^{s}} \\
	& \leq  C[v_0](\|X-\al-Av_0t\|_{L^\infty H^{s+1}}+C[v_0]T)\|w-w_{\ep}\|_{L^{2}H^{s+1}}+C[v_0]\ep\\
	& \leq C[v_0]\ep+C[v_0]T^{\frac14}\|w-w_{\ep}\|_{L^{2}H^{s+1}}.
\end{align*}
For $d_{2}^{e}$, we have:
\begin{align*}
	\|d_{2}^{e}\|_{L^{2}H^{s}} & \leq C[v_0]\|\zeta - \zeta_{\ep}\|_{L^{\infty}H^{s}} \leq C[v_0]\|X-X_{\ep}+\ep b-(A-A_\ep)v_0t\|_{L^{\infty}H^{s+1}}+C[v_0]\ep \\
	& \leq C[v_0]\ep + C[v_0]T^{\frac14}\|X-X_{\ep}+\ep b-(A-A_\ep)v_0t\|_{F^{s+1}},
\end{align*}
as before. The same procedure applied to $d_{3}^{e}$ yields:
\begin{align*}
	\|d_{3}^{e}\|_{L^{2}H^{s}} \leq C[v_0]\|X-X_{\ep}\|_{L^{\infty}H^{s}} \leq C[v_0]\ep + C[v_0]T^{\frac14}\|X-X_{\ep}+\ep b-(A-A_\ep)v_0t\|_{F^{s+1}}.
\end{align*}
In an analogous way to $d_1^e$ we get for $d_{4}^{e}$:
\begin{align*}
	\|d_{4}^{e}\|_{L^{2}H^{s}} &\leq  C[v_0]\ep+C[v_0]T^{\frac14}\|w-w_{\ep}\|_{L^{2}H^{s+1}}.
\end{align*}
Finally, it is easy to check using \eqref{estrella} that $$\|d_{5}^{e}\|_{L^{2}H^{s}} \leq C[v_0]\ep.$$ We are then done with $g-g_\ep$. It remains to control $g^L_\phi-g_{\phi,\ep}^L$, but in a similar manner to $d_5^e$ we find
$$\|g^L_\phi-g_{\phi,\ep}^L\|_{L^{2}H^{s}} \leq C[v_0]\ep,$$
due to the formulas for $\phi$ and $\phi_\ep$.
We now move on to the $H^{\frac{s+1}{2}}H^{-1}$ norm. In order to handle it, we consider a different splitting taking
$$
g-g_\ep+g^L_\phi-g^L_{\phi,\ep}=(\overline{g}_{w}-\overline{g}_{w,\ep})+(\overline{g}_{\phi}-\overline{g}_{\phi,\ep})+(\overline{g}^L_\phi-\overline{g}^L_{\phi,\ep}),
$$
where
$$
\overline{g}_{w,\ep}=-Tr(\nabla w_\ep \zeta_\ep A\circ X_\ep)+Tr(\nabla w_\ep A_\ep),\quad \overline{g}_{\phi,\ep}=-Tr(\nabla \phi_\ep \zeta_\ep A\circ X_\ep)+Tr(\nabla \phi_\ep \zeta_{\phi,\ep} A_{\phi,\ep}),
$$
and
$$
\overline{g}^L_{\phi,\ep}=-Tr(\nabla \phi_\ep \zeta_{\phi,\ep}A_{\phi,\ep}),
$$
are defined analogously as $\overline{g}^{(n)}$ and $\overline{g}^{L}_{\phi}$ in \eqref{overlinegn} and \eqref{gLphi} respectively. Above
$$\zeta_{\phi,\ep}= \I+t(\partial_t \zeta_\ep)_{|_{t=0}}=\I-t\nabla(A_\ep v_0),
\qquad\mbox{ and }\qquad
A_{\phi,\ep}=A_\ep+t(\pa_t A\circ X_\ep)_{|_{t=0}}=A_\ep+t\nabla A_\ep A_\ep v_0.$$
The terms $\overline{g}$, $\overline{g}_\phi$ and $\overline{g}^L_{\phi}$ are defined similarly removing the epsilons everywhere as it was done before.

We consider first the following splitting $\overline{g}_{w}-\overline{g}_{w,\ep}=\sum_{j=1}^{14}D^e_{j}$ where
$$
D^e_{1}=-Tr(\nabla(w-w_\ep)(\zeta-\zeta_\phi)(A\circ X-A_{\phi})),\qquad D^e_{2}=-Tr(\nabla(w-w_\ep)(\zeta-\zeta_\phi)A_{\phi}),
$$
$$
D^e_{3}=-Tr(\nabla(w-w_\ep)\zeta_\phi(A\circ X-A_{\phi})),\qquad D^e_{4}=-Tr(\nabla(w-w_\ep)(\zeta_\phi -\I)A_{\phi}),
$$
$$
D^e_{5}=-Tr(\nabla w_\ep(\zeta-\zeta_\phi-(\zeta_\ep-\zeta_{\phi,\ep}))(A\circ X-A_{\phi})),\qquad D^e_{6}=-Tr(\nabla w_\ep(\zeta-\zeta_\phi-(\zeta_\ep-\zeta_{\phi,\ep}))A_{\phi}),
$$
\begin{equation}\label{DeHs1p2}
	D^e_{7}=-Tr(\nabla w_\ep(\zeta_\phi-\zeta_{\phi,\ep})(A\circ X-A_{\phi})),\qquad D^e_{8}=-Tr(\nabla w_\ep(\zeta_\phi-\zeta_{\phi,\ep})A_{\phi}),
\end{equation}
$$
D^e_{9}=-Tr(\nabla w_\ep(\zeta_\ep\!-\!\zeta_{\phi,\ep})(A\circ X\!-\!A_{\phi}\!-\!(A\circ X_\ep\!-\!A_{\phi,\ep}))),\quad D^e_{10}=-Tr(\nabla w_\ep(\zeta_\ep-\zeta_{\phi,\ep})(A_{\phi}-A_{\phi,\ep})),
$$
$$
D^e_{11}=-Tr(\nabla w_\ep\zeta_{\phi,\ep}(A\circ X\!-\!A_{\phi}\!-\!(A\circ X_\ep\!-\!A_{\phi,\ep}))),\qquad D^e_{12}=-Tr(\nabla w_\ep\zeta_{\phi,\ep}(A_{\phi}-A_{\phi,\ep})),
$$
$$
D^e_{13}=Tr(\nabla (w-w_\ep)(A-A_\phi)),\qquad D^e_{14}=Tr(\nabla w_\ep(A-A_{\ep})),
$$
Next we decompose further $D_{1}^e=D_{1,1}^e+D_{1,2}^e+D_{1,3}^e$ so that
\begin{align*}
	D_{1,1}^{e} & = -\int_{0}^{t}Tr(\nabla(\pa_t(w-w_\ep))(\zeta-\zeta_\phi)(A\circ X-A_{\phi}))d\tau, \\
	D_{1,2}^{e} & = -\int_{0}^{t} Tr(\nabla(w-w_\ep)\pa_t(\zeta-\zeta_\phi)(A\circ X-A_{\phi}))d\tau,\\
	D_{1,3}^{e} & = -\int_{0}^{t} Tr(\nabla(w-w_\ep)(\zeta-\zeta_\phi)\pa_t(A\circ X-A_{\phi}))d\tau.
\end{align*}
One obtains
\begin{align*}
	\|D_{1,1}^{e}\|_{H_{(0)}^{\frac{s+1}{2}}H^{-1}} & \leq \|\nabla(\pa_t(w-w_\ep))(\zeta-\zeta_\phi)(A\circ X-A_{\phi})\|_{H_{(0)}^{\frac{s-1}{2}}H^{-1}} \\
	& \leq C\|\nabla(\pa_t(w-w_{\ep})\|_{H_{(0)}^{\frac{s-1}{2}}H^{-1}}\|\zeta-\zeta_\phi\|_{H_{(0)}^{\frac{s-1}{2}}H^{1+\delta}}\|A\circ X-A_{\phi}\|_{H_{(0)}^{\frac{s-1}{2}}H^{1+\delta}} \\
	& \leq C[v_0]T^{\delta}\|w-w_{\ep}\|_{H_{(0)}^{\frac{s+1}{2}}L^2}\|X-\al-Av_0t\|_{H_{(0)}^{\frac{s-1}{2}+\delta}H^{2+\delta}} \leq C[v_0]T^{\delta}\|w-w_{\ep}\|_{H_{(0)}^{ht,s+1}}.
\end{align*}
We share derivatives in a different way to estimate $D_{1,2}^{e}$:
\begin{align*}
	\|D_{1,2}^{e}\|_{H_{(0)}^{\frac{s+1}{2}}H^{-1}} & \leq  C\|\nabla(w-w_{\ep})\|_{H_{(0)}^{\frac{s-1}{2}}H^{1}}\|\pa_t(\zeta-\zeta_\phi)\|_{H_{(0)}^{\frac{s-1}{2}}L^2}\|A\circ X-A_{\phi}\|_{H_{(0)}^{\frac{s-1}{2}}H^{1+\delta}}  \\
	& \leq C[v_0]T^{\delta}\|w-w_{\ep}\|_{H_{(0)}^{\frac{s-1}{2}}H^2}\|X-\al-Av_0t\|_{H_{(0)}^{\frac{s+1}{2}+\delta}H^{1}} \leq C[v_0]T^{\delta}\|w-w_{\ep}\|_{H_{(0)}^{ht,s+1}}.
\end{align*}
The term $D_{1,3}^{e}$ is estimated as $D_{1,2}^{e}$ yielding:
\begin{align*}
	\|D_{1,3}^{e}\|_{H_{(0)}^{\frac{s+1}{2}}H^{-1}} & \leq  C[v_0]\|\nabla(w-w_{\ep})\|_{H_{(0)}^{\frac{s-1}{2}}H^{1}}\|X-\al-Av_0t\|_{H_{(0)}^{\frac{s-1}{2}}H^{2+\delta}}\|\pa_t(A\circ X-A_{\phi})\|_{H_{(0)}^{\frac{s-1}{2}}L^2}  \\
	& \leq C[v_0]T^{\delta}\|w-w_{\ep}\|_{H_{(0)}^{ht,s+1}}.
\end{align*}
We are done with $D_{1}^e$. In order to deal with $D_2^e$ we split it further into
$$
D_{2,1}^e=-Tr(\nabla(w-w_\ep)(\zeta-\zeta_\phi)A),\qquad D_{2,2}^e=-Tr(t\nabla(w-w_\ep)(\zeta-\zeta_\phi)\nabla A A v_0).
$$
The term $D_{2,1}^e$ is estimated as $D_1^e$ and $D_{2,2}^e$ is controlled using lemma \ref{littlelemma}. The same approach works to bound $D_3^e$ and $D_4^e$.

In order to deal with $D_5^e$ one could consider the following splitting $D_5^e=D_{5,1}^{e}+D_{5,2}^{e}+D_{5,1}^{e}$
where
\begin{align*}
	D_{5,1}^{e} & = -\int_{0}^{t}Tr(\nabla \pa_tw_\ep(\zeta-\zeta_\phi-(\zeta_\ep-\zeta_{\phi,\ep}))(A\circ X-A_{\phi}))d\tau, \\
	D_{5,2}^{e} & = -\int_{0}^{t}Tr(\nabla w_\ep\pa_t(\zeta-\zeta_\phi-(\zeta_\ep-\zeta_{\phi,\ep}))(A\circ X-A_{\phi}))d\tau,\\
	D_{5,3}^{e} & = -\int_{0}^{t}Tr(\nabla w_\ep(\zeta-\zeta_\phi-(\zeta_\ep-\zeta_{\phi,\ep}))\pa_t(A\circ X-A_{\phi}))d\tau.
\end{align*}
We share derivatives as for $D_{1,1}^e$ to get
\begin{align*}
	\|D_{5,1}^{e}\|_{H_{(0)}^{\frac{s+1}{2}}H^{-1}} &  \leq C\|\nabla\pa_tw_{\ep}\|_{H_{(0)}^{\frac{s-1}{2}}H^{-1}}\|\zeta-\zeta_\phi-(\zeta_\ep-\zeta_{\phi,\ep})\|_{H_{(0)}^{\frac{s-1}{2}}H^{1+\delta}}\|A\circ X-A_{\phi}\|_{H_{(0)}^{\frac{s-1}{2}}H^{1+\delta}} \\
	& \leq C[v_0]T^{\delta}\|w_{\ep}\|_{H_{(0)}^{\frac{s+1}{2}}L^2}\|X-X_\ep+\ep b-(A-A_\ep)v_0t\|_{H_{(0)}^{\frac{s-1}{2}+\delta}H^{2+\delta}}\\ &\leq C[v_0]T^{\delta}\|X-X_\ep+\ep b-(A-A_\ep)v_0t\|_{F^{s+1}}.
\end{align*}
Next we consider $D_{5,2}^{e}$ as $D_{1,2}^{e}$ to find:
\begin{align*}
	\|D_{5,2}^{e}\|_{H_{(0)}^{\frac{s+1}{2}}H^{-1}} & \leq  C\|\nabla w_{\ep}\|_{H_{(0)}^{\frac{s-1}{2}}H^{1}}\|\pa_t(\zeta-\zeta_\phi-(\zeta_\ep-\zeta_{\phi,\ep}))\|_{H_{(0)}^{\frac{s-1}{2}}L^2}\|A\circ X-A_{\phi}\|_{H_{(0)}^{\frac{s-1}{2}}H^{1+\delta}}  \\
	& \leq C[v_0]T^{\delta}\|X-X_\ep+\ep b-(A-A_\ep)v_0t\|_{F^{s+1}}.
\end{align*}
As for $D_{1,3}^{e}$ it is possible to get
\begin{align*}
	\|D_{5,3}^{e}\|_{H_{(0)}^{\frac{s+1}{2}}H^{-1}} & \leq C[v_0]T^{\delta}\|X-X_\ep+\ep b-(A-A_\ep)v_0t\|_{F^{s+1}}.
\end{align*}
We are done with $D_{5}^e$. To bound $D_{6}^e$ it is possible to split further using identity $A_\phi=A+t\grad A A v_0$ and lemma \ref{littlelemma} getting
$$
\|D_{6}^{e}\|_{H_{(0)}^{\frac{s+1}{2}}H^{-1}} \leq C[v_0]T^{\delta}\|X-X_\ep+\ep b-(A-A_\ep)v_0t\|_{F^{s+1}}.
$$
Sharing the time derivative as before, the terms $D_{7}^e$ and $D_{8}^e$ are bounded by
$$
\|D_{7}^{e}\|_{H_{(0)}^{\frac{s+1}{2}}H^{-1}} \leq C[v_0]\ep,\qquad \|D_{8}^{e}\|_{H_{(0)}^{\frac{s+1}{2}}H^{-1}} \leq C[v_0]\ep,
$$
using \eqref{estrella}. An analogous approach to the composition lemma \ref{composicionAfi} provides
$$
\|D_{9}^{e}\|_{H_{(0)}^{\frac{s+1}{2}}H^{-1}} \leq C[v_0]T^{\delta}\|X-X_\ep+\ep b-(A-A_\ep)v_0t\|_{F^{s+1}}
$$
by a similar splitting as for $D_5^e$. We control the terms $D_{10}^e$ and $D_{12}^e$ as $D_{7}^e$ to obtain
$$
\|D_{10}^{e}\|_{H_{(0)}^{\frac{s+1}{2}}H^{-1}} \leq C[v_0]\ep,\qquad \|D_{12}^{e}\|_{H_{(0)}^{\frac{s+1}{2}}H^{-1}} \leq C[v_0]\ep.
$$
As for $D_{6}^e$, the use of lemmas \ref{2.4}, \ref{bestiario} and \ref{littlelemma} provides
$$
\|D_{11}^{e}\|_{H_{(0)}^{\frac{s+1}{2}}H^{-1}} \leq C[v_0]T^{\delta}\|X-X_\ep+\ep b-(A-A_\ep)v_0t\|_{F^{s+1}}.
$$
At this point, it is easy to find as before
$$
\|D_{13}^{e}\|_{H_{(0)}^{\frac{s+1}{2}}H^{-1}} \leq C[v_0]T^\delta\|w-w_{\ep}\|_{H_{(0)}^{ht,s+1}},\qquad \|D_{14}^{e}\|_{H_{(0)}^{\frac{s+1}{2}}H^{-1}} \leq C[v_0]\ep.
$$
We are therefore done with $\overline{g}_{w}-\overline{g}_{w,\ep}$. The same approach can be used to handle $\overline{g}_{\phi}-\overline{g}_{\phi,\ep}$ but replacing $w$ by $\phi$ and $w_\ep$ by $\phi_\ep$. This provides
$$
\|\overline{g}_{\phi}-\overline{g}_{\phi,\ep}\|_{H_{(0)}^{\frac{s+1}{2}}H^{-1}} \leq C[v_0]\ep+C[v_0]T^{\delta}\|X-X_\ep+\ep b-(A-A_\ep)v_0t\|_{F^{s+1}}.
$$
It remains to deal with $\overline{g}^L_\phi-\overline{g}^L_{\phi,\ep}$. Using that $\overline{g}^L_\phi=O(t^2)=\overline{g}^L_{\phi,\ep}$, \eqref{estrella} together with lemma \ref{littlelemma} we finally obtain
$$
\|\overline{g}^L_\phi-\overline{g}^L_{\phi,\ep}\|_{H_{(0)}^{\frac{s+1}{2}}H^{-1}} \leq C[v_0]\ep.
$$

\underline{Part III: Estimates for $H_{\ep}$:}\\

We first consider in formula \eqref{He} the splitting
$$
H_{\ep}  = h - h_{\ep} + h^L_{\phi} - h^L_{\phi,\ep}+\overline{H}_{\ep}
$$
where the term $\overline{H}_{\ep}$ is given by
\begin{align*}
	\overline{H}_{\ep} = &  -q_{w,\ep}(A^{-1} - A^{-1}_{\ep}) n_0 - ((\nabla w_{\ep}(A-A_{\ep})) + (\nabla w_{\ep}(A-A_{\ep}))^{*}) A^{-1}_{\ep} n_0 \\
	& - ((\nabla w_{\ep} A) + (\nabla w_{\ep} A)^{*})(A^{-1} - A_{\ep}^{-1}) n_{0}
\end{align*}

The first estimate in \eqref{estrella} yields
\begin{align*}
	|\overline{H}_{\ep}|_{H_{(0)}^{ht,s-\frac12}} &\leq C(|q_{w,\ep}|_{H_{(0)}^{ht,s-\frac12}}|A^{-1}\!-\!    A^{-1}_{\ep}|_{H^{s-\frac12}}+|\nabla w_\ep|_{H_{(0)}^{ht,s-\frac12}}(|A- A_{\ep}|_{H^{s-\frac12}}+|A^{-1}\!-\!A^{-1}_{\ep}|_{H^{s-\frac12}})\\
	&\leq C[v_0]\ep
\end{align*}

Next, the use of \eqref{estrella}, \eqref{qmqe} and the smallness of $\phi - \phi_{\ep}$ allows us to obtain
\begin{align*}
	|h^{L}_{\phi} - h^{L}_{\phi,\ep}|_{L^2H^{s-\frac12}}\leq C 	\|h^{L}_{\phi} - h^{L}_{\phi,\ep}\|_{L^2H^s}  \leq C[v_0]\ep.
\end{align*}
The compatibility condition and the formulas for $q_{\phi,\ep}$, $q_{\phi}$ provide
\begin{align*}
	|h^{L}_{\phi} - h^{L}_{\phi,\ep}|_{H_{(0)}^{\frac s2-\frac14}L^2}&=|t\exp(-t^2) ((\nabla \psi A + (\nabla \psi A)^{*}) A^{-1} - (\nabla \psi_\ep A_{\ep} + (\nabla \psi_\ep A_{\ep})^{*}) A_{\ep}^{-1}) n_0|_{H_{(0)}^{\frac s2-\frac14}L^2}.
\end{align*}
Together with \eqref{estrella}, \eqref{qmqe} yield
\begin{align*}
	|h^{L}_{\phi} - h^{L}_{\phi,\ep}|_{H_{(0)}^{\frac s2-\frac14}L^2}&\leq C[v_0]\ep,\qquad  \mbox{ and finally }\qquad
	|h^{L}_{\phi} - h^{L}_{\phi,\ep}|_{H_{(0)}^{ht,s-\frac12}}\leq C[v_0]\ep.
\end{align*}
It remains to deal with $h - h_{\ep}$. As we did in \eqref{hsplit}, we split $h = h_v  + h_{v^{*}} + h_q$, and similarly $h_{\ep} = h_{v,\ep} + h_{v^{*}\!,\ep} + h_{q,\ep}$. We estimate first $h_{v} - h_{v,\ep}$ using the splitting

\begin{align*}
	h_{v} - h_{v,\ep} = d_{1}^{e} + d_{2}^{e} + d_{3}^{e} + d_{4}^{e},
\end{align*}
where
\begin{align*}
	d_{1}^{e} = \nabla v(\zeta - \zeta_{\ep}) \nabla_{J} X n_0, &\qquad
	d_{2}^{e} = \nabla v\zeta_{\ep}( \nabla_{J} X - \nabla_{J} X_{\ep}) n_0, \\
	d_{3}^{e}  = \nabla (v - v_{\ep}) (\zeta_{\ep} - \I) \nabla_{J} X_{\ep} n_0, &\qquad
	d_{4}^{e}  = \nabla(v-v_{\ep})(\nabla_{J} X_{\ep} - \I) n_0,
\end{align*}
in a similar way to \eqref{diferenciahv}. As before, using the splitting $v=w+v_0+t\psi$, we are able to bound as follows:

\begin{align*}
	|d_{1}^{e}|_{H_{(0)}^{ht,s-\frac12}} + |d^{e}_{2}|_{H_{(0)}^{ht,s-\frac12}}
	& \leq C[v_0]\ep+C[v_0](\|X-X_{\ep}+\ep b-(A-A_{\ep})v_0t\|_{L^{\infty}H^{s+1}}\\
	& \qquad\qquad\qquad\qquad\qquad\qquad+\|X-X_{\ep}+\ep b-(A-A_{\ep})v_0t\|_{H_{(0)}^{\frac s2-\frac14}H^{2+\eta}}) \\
	& \leq C[v_0]\ep + C[v_0]T^{\delta}\|X-X_{\ep}+\ep b-(A-A_{\ep})v_0t\|_{F^{s+1}},
\end{align*}
where $\eta>0$ small enough and  we have used Lemma \ref{lemapacoestabilidad1}. Thus, we have obtained the appropriate estimate. Repeating the procedure in the splitting \eqref{diferenciahv} to $d_3^e$ and $d_4^e$, we find
\begin{align*}
	|d_{3}^{e}|_{H_{(0)}^{ht,s-\frac12}} + |d_{4}^{e}|_{H_{(0)}^{ht,s-\frac12}} & \leq  C[v_0]\ep+C[v_0]T^{\delta}\|w-w_{\ep}\|_{H_{(0)}^{ht,s+1}}.
\end{align*}

We split further for $h_{v^{*}} - h_{v^{*}\!,\ep} = \sum_{j=1}^{5} d_{j}^{*,e}$, where

$$
d_{1}^{*,e}  = [(\nabla v\zeta (A \circ X\!-\!A))^{*}(A^{-1} \circ X\!-\!A^{-1}) \nabla_J X-(\nabla v_\ep\zeta_\ep (A \circ X_\ep\!-\!A_\ep))^{*}(A^{-1}\circ X_\ep\!-\!A^{-1}_\ep)  \nabla_J X_\ep] n_0,
$$
$$
d_{2}^{*,e}  = [(\nabla v\zeta (A \circ X-A))^{*}A^{-1}  \nabla_J X-(\nabla v_\ep\zeta_\ep (A \circ X_\ep-A_\ep))^{*}A^{-1}_\ep  \nabla_J X_\ep)] n_0,
$$
$$
d_{3}^{*,e}  = [(\nabla v\zeta A)^{*}(A^{-1}\circ X-A^{-1})\nabla_J X-(\nabla v_\ep\zeta_\ep A_\ep)^{*} (A^{-1}\circ X_\ep-A^{-1}_\ep)\nabla_J X_\ep] n_0,
$$
$$
d_{4}^{*,e}  = [(\nabla v (\zeta-\I) A)^{*} A^{-1} \nabla_J X -(\nabla v_\ep (\zeta_\ep-\I) A_\ep)^{*} A_\ep^{-1} \nabla_J X_\ep] n_0,
$$
and
$$
d_{5}^{*,e}  = [(\nabla v A)^*A^{-1}(\nabla_{J} X_{\ep}-I)-(\nabla v_\ep A_\ep)^*A_\ep^{-1}(\nabla_{J} X_{\ep}-I)]n_0.
$$
Further decomposing provides $d_1^{*,e}=\sum_{j=1}^{5} d_{1,j}^{*,e}$ with
$$
d_{1,1}^{*,e}  = (\nabla v(\zeta-\zeta_\ep) (A \circ X\!-\!A))^{*}(A^{-1} \circ X\!-\!A^{-1}) \nabla_J Xn_0,
$$
$$
d_{1,2}^{*,e}  = (\nabla v\zeta_\ep) (A \circ X\!-\!A-A \circ X_\ep\!+\!A_\ep))^{*}(A^{-1} \circ X\!-\!A^{-1}) \nabla_J Xn_0,
$$
$$
d_{1,3}^{*,e}  = (\nabla v\zeta_\ep (A \circ X_\ep\!-\!A_\ep))^{*}(A^{-1} \circ X\!-\!A^{-1}-A^{-1} \circ X_\ep\!+\!A^{-1}_\ep) \nabla_J Xn_0,
$$
$$
d_{1,4}^{*,e}  = (\nabla v\zeta_\ep (A \circ X_\ep\!-\!A_\ep))^{*}(A^{-1} \circ X_\ep\!-\!A^{-1}_\ep)(\nabla_J X-\nabla_J X_\ep)n_0,
$$
and
$$
d_{1,5}^{*,e}  = (\nabla (v-v_\ep)\zeta_\ep (A \circ X_\ep\!-\!A_\ep))^{*}(A^{-1} \circ X_\ep\!-\!A^{-1}_\ep)\nabla_J X_\ep n_0.
$$
As for $d_1^e$ and $d_2^e$ we can find
\begin{align*}
	\sum_{j=1}^4|d_{1,j}^{*,e}|_{H_{(0)}^{ht,s-\frac12}}
	& \leq C[v_0]\ep+C[v_0](\|X-X_{\ep}+\ep b-(A-A_{\ep})v_0t\|_{L^{\infty}H^{s+1}}\\
	& \qquad\qquad\qquad\qquad\qquad\qquad+\|X-X_{\ep}+\ep b-(A-A_{\ep})v_0t\|_{H_{(0)}^{\frac s2-\frac14}H^{2+\eta}}) \\
	& \leq C[v_0]\ep + C[v_0]T^{\delta}\|X-X_{\ep}+\ep b-(A-A_{\ep})v_0t\|_{F^{s+1}},
\end{align*}
for $\eta>0$ small enough. As we did for $d_3^e$ and $d_4^e$ it is possible to get
\begin{align*}
	|d_{1,5}^{*,e}|_{H_{(0)}^{ht,s-\frac12}} & \leq  C[v_0]\ep+C[v_0]T^{\delta}\|w-w_{\ep}\|_{H_{(0)}^{ht,s+1}}.
\end{align*}
In an analogous manner, we estimate $d_{j}^{*,e}$ for $j=2,...,5$ so that
\begin{align*}
	\sum_{j=2}^5|d_{j}^{*,e}|_{H_{(0)}^{ht,s-\frac12}}&\leq
	C[v_0]\ep + C[v_0]T^{\delta}(\|w-w_{\ep}\|_{H_{(0)}^{ht,s+1}}+\|X-X_{\ep}+\ep b-(A-A_{\ep})v_0t\|_{F^{s+1}}).
\end{align*}
The estimates for $h_{v^*}-h_{v^*\!,\ep}$ are done. To finish, we consider $h_q - h_{q,\ep} = d_{1}^{q,e} + d_{2}^{q,e}$, where

\begin{align*}
	d_{1}^{q,e} & = [q_{\ep}(A^{-1} \circ X_{\ep} - A^{-1}_\ep)\nabla_{J} X_{\ep}-q(A^{-1} \circ X - A^{-1})\nabla_{J} X] n_0, \\
	d_{2}^{q,e} & = [q_{\ep} A^{-1}_\ep(\nabla_{J} X_{\ep}-\I)-q A^{-1}(\nabla_{J} X-\I)] n_0.
\end{align*}
The last detailed splitting $d_{1}^{q,e}=d_{1,1}^{q,e}+d_{1,2}^{q,e}+d_{1,3}^{q,e}$ provides
$$
d_{1,1}^{q,e}=q_{\ep}(A^{-1} \circ X_{\ep} - A^{-1}_\ep-A^{-1} \circ X+ A^{-1})\nabla_{J} X_{\ep}n_0,
$$
$$
d_{1,2}^{q,e}=q_{\ep}(A^{-1} \circ X-A^{-1})(\nabla_{J} X_{\ep}-\nabla_{J} X) n_0,
$$
$$
d_{1,3}^{q,e}=(q_{\ep}-q)(A^{-1} \circ X-A^{-1})\nabla_{J} X n_0,
$$
which allows us to estimate as before
\begin{align*}
	\sum_{j=1}^3|d_{1,j}^{q,e}|_{H_{(0)}^{ht,s-\frac12}}
	& \leq C[v_0]\ep+C[v_0]T^{\delta}(
	\|q_{w}-q_{w,\ep}\|_{H^{ht,s}_{pr \, (0)}}+\|X-X_{\ep}+\ep b-(A-A_{\ep})v_0t\|_{F^{s+1}}).
\end{align*}
We end the bounds getting
\begin{align*}
	|d_{2}^{q,e}|_{H_{(0)}^{ht,s-\frac12}} & \leq C[v_0]\ep+C[v_0]T^{\delta}(
	\|q_{w}-q_{w,\ep}\|_{H^{ht,s}_{pr \, (0)}}+\|X-X_{\ep}+\ep b-(A-A_{\ep})v_0t\|_{F^{s+1}}).
\end{align*}
by a similar splitting. We are done with $h_q-h_{q,\ep}$ and therefore with $H_\ep$.

\end{prooflemma}

\section{Setting the right initial normal velocity}\label{sectioninitialvel}

We consider the following parametrization of the boundary of $\Omega$:
$$z(s),\quad |z_s(s)|=1,$$
we also consider a small enough neighborhood of the boundary, $U$. In $U$ one can use the coordinates $(s,\la)$ given by
$$x(s,\lambda)= z(s)+\lambda z_s^\perp(s).$$
The stream function $\psi$, in $U$,will be given by
 \begin{align*}
 \overline{\psi}(s,\lambda)=&\psi_0(s)+\psi_1(s)\lambda+ \frac{1}{2}\psi_2(s)\lambda^2\\
 \psi(x(s,\la))=& \overline{\psi}(s,\la).
\end{align*}
Then we can extend in a smooth way $\psi$ to the rest of the domain $\Omega$ and take $v_{0} = \nabla^{\perp} \psi$. $v_{0}$ is clearly divergence free.

The initial velocity $v_{0}(x)$ must satisfy
$$t\left.\left(\nabla v_{0} + \nabla v_{0}^*\right)\right|_{\partial \Omega}n=0,$$
where $t$ and $n$ are the tangential and normal vectors to the boundary of $\Omega$ respectively.

If $T$ and $N$ are an extension of $t$ and $n$ to $U$ respectively, we can write
$$\left(T\left(\nabla v_{0} + \nabla v_{0}^*\right)N\right)|_{\partial \Omega}=0.$$

We will take
\begin{align*}
T(s,\lambda)= & x_s(s,\lambda)=z_s(s)+\lambda z_{ss}^\perp(s)=(1-\lambda\kappa(s))z_s(s)\\
N(s,\lambda)= & x_\lambda(s,\la)= z_s(s)^\perp,
\end{align*}
where $$\kappa(s)= z_{ss}(s)\cdot z_s^\perp,$$
and we notice that
$$\nabla \psi\circ x(s,\la)=\frac{1}{1-\kappa\la}\overline{\psi}_s\, z_s+\overline{\psi}_\la\,z_s^\perp.$$

By defining  $$\overline{v_{0}}(s,\la)= v_{0} \circ x(s,\la)$$
we have that
$$\overline{v_{0}}(s,\lambda)=\frac{1}{1-\kappa\la}\overline{\psi}_s \, z^\perp_s- \overline{\psi}_\la\,z_s.$$
Now we can compute that
\begin{align*}
 T^i\pa_i v_{0}^j\circ x(s,\lambda)=&\pa_s \overline{v_{0}}^{j} (s,\lambda)\\
 N^j\pa_j v_{0}^i \circ x(s,\lambda)=& \pa_\lambda \overline{v_{0}}^{i}(s,\lambda).
 \end{align*}

And then
\begin{align*}
T^i\pa_i v_{0}^j\circ x(s,\lambda)N^j=\pa_s\left(\overline{v_{0}}\cdot N\right)-\overline{v_{0}}\cdot N_s\\
N^j\pa_j v_{0}^i\circ x(s,\lambda)T^i=\pa_\lambda\left(\overline{v_{0}}\cdot T\right)-\overline{v_{0}}\cdot T_\lambda.
\end{align*}
But
\begin{align*}
T_\lambda=& z_{ss}^\perp(s)\\
N_s= & z_{ss}^\perp(s)=-\kappa(s)z_s(s).
\end{align*}
In addition
\begin{align*}
\overline{v_{0}}\cdot N = \frac{1}{1-\lambda\kappa}\overline{\psi}_s.
\end{align*}
Then
$$\pa_s(\overline{v_{0}}\cdot N)|_{\la=0}=\pa^2_s\overline{\psi}(s,0).$$
Also
$$\overline{v_{0}}\cdot T=-(1-\lambda \kappa)\overline{\psi}_\lambda.  $$
Therefore
$$\pa_\lambda(\overline{v_{0}}\cdot T)|_{\lambda=0}=\kappa(s)\pa_\lambda\overline{\psi}(s,0)-\pa^2_\lambda\overline{\psi}(s,0).$$
Finally
$$(\overline{v_{0}}\cdot T_\lambda)|_{\lambda=0}= \overline{v_{0}}\cdot N_s|_{\la=0} = \kappa \overline{\psi}_\lambda(s,0).$$
Thus, we have that
$$\pa^2_s \overline{\psi}(s,0)-\kappa\pa_\lambda \overline{\psi}(s,0)-\pa^2_\lambda\overline{\psi}(s,0)=0.$$
Taking $\psi_1=0$ yields
$$\pa^2_s\psi_0=\psi_2(s).$$

Just to conclude we notice that $v_{0} \cdot n|_{\pa\Omega}=\pa_s\psi_0(s)$. We first pick up $\psi_0$ in order to choose the normal component of the velocity to be strictly positive and outward at the interface to guarantee a splash. Then $\psi_2$ is taken to satisfy the continuity of the stress tensor.

\section*{Acknowledgments}

AC, DC, CF, FG and JGS were partially supported by the grants MTM2011-26696, MTM2014-59488-P (Spain) and ICMAT Severo Ochoa projects SEV-2011-0087 and SEV-2015-0554. AC was partially supported by the Ram\'on y Cajal program RyC-2013-14317 and ERC grant 307179-GFTIPFD. CF was partially supported by NSF grant DMS-1265524. FG was partially supported by the Ram\'on y Cajal program RyC-2010-07094, by the P12-FQM-2466 grant from Junta de Andaluc\'ia, Spain and by the ERC Grant H2020-EU.1.1.-639227. JGS was partially supported by the Simons Collaboration Grant 524109 and by the NSF grant DMS-1763356. We thank the anonymous referee for his or her useful comments on previous versions of this manuscript.

\appendix

\section{Results on the linear system}

\label{sectionlineal}

This appendix is devoted to prove Theorem \ref{Lmenos1}. The proof is an adaptation
of \cite[Theorem 4.3]{Beale:initial-value-problem-navier-stokes} to the tilde domain.

\subsection{Case $g=h=v_0=0$}
\label{subsectionlinealceros}

We would like to solve the following system:

\begin{align}
 v_t - Q^{2} \Delta v + A^{*} \nabla q & = f \nonumber \\
Tr(\nabla v A) & = 0 \nonumber \\
(qI - ((\nabla v A) +(\nabla v A)^{*})B_1 n & = 0 \nonumber \\
\left.v\right|_{t=0} & = 0 \label{uno},
\end{align}
where
\begin{align*}
B_1 = -JA^{-1}J=\frac{1}{Q^2}A^*
\end{align*}
and
$$A_{ij}=\pa_{j}P^i\circ P^{-1}.$$

Our first purpose will be to obtain a weak formulation of the time independent part of this system. In order to do it we will use the following identities
\begin{align*}
 & AA^{*} = Q^{2} I\\
 &\Delta v = \text{div}\left(AA^{*}\frac{1}{Q^{2}}\nabla v\right)
\end{align*}
and also
%

\begin{align*}
Q^{2} \pa_{l}\left(\frac{1}{Q^{2}}A_{lj}(A_{kj} \pa_{k} v^{i} + A_{ki} \pa_{k} v^{j})\right) = Q^{2} \Delta v
\end{align*}

Using this relation it is easy to arrive to the following identity:

\begin{align*}
 &\int_{\Om} \pa_{l}\left(\frac{1}{Q^{2}}A_{lj}(A_{kj} \pa_{k} v^{i} + A_{ki} \pa_{k} v^{j})\right)\phi^{i}  dx\\ &
 = -\int_{\Om}\frac{1}{Q^{2}}A_{lj}(A_{kj} \pa_{k} v^{i} + A_{ki} \pa_{k} v^{j})\pa_{l}\phi^{i}  dx
+ \int_{\pa \Om} n^{l} A_{lj}(A_{kj} \pa_{k} v^{i} + A_{ki} \pa_{k} v^{j}) \phi^{i}\frac{1}{Q^{2}}d\sigma \\
& = -\int_{\Om}\frac{1}{Q^{2}}(A_{kj} \pa_{k} v^{i} + A_{ki} \pa_{k} v^{j})A_{lj}\pa_{l}\phi^{i}  dx
+ \int_{\pa \Om} n^{l} A_{lj}(A_{kj} \pa_{k} v^{i} + A_{ki} \pa_{k} v^{j}) \phi^{i} \frac{1}{Q^2}d\sigma \\
& = -\frac12\int_{\Om}\frac{1}{Q^{2}}(A_{kj} \pa_{k} v^{i} + A_{ki} \pa_{k} v^{j})(A_{lj}\pa_{l}\phi^{i}+A_{li}\pa_{l} \phi^{j})  dx
+ \int_{\pa \Om} n^{l} A_{lj}(A_{kj} \pa_{k} v^{i} + A_{ki} \pa_{k} v^{j}) \phi^{i}\frac{1}{Q^2} d\sigma \\
\end{align*}

for $\phi^{i} \in C^{\infty}(\overline{\Omega})$.

We also have that

\begin{align*}
 \int_{\Om} A_{ki} \pa_{k} q \phi^{i} \frac{1}{Q^{2}} dx = - \int_{\Om} q \pa_{k} \left(A_{ki} \phi^{i} \frac{1}{Q^{2}}\right) dx + \int_{\pa \Om} q A_{ki} n^{k} \phi^{i}\frac{d\sigma}{Q^2}
\end{align*}

The following identities hold:

\begin{align*}
 \pa_{k}\left(A_{ki} \phi^{i} \frac{1}{Q^{2}}\right) = A_{ki} \pa_{k} \phi^{i} \frac{1}{Q^{2}}, \quad \pa_{k} \left(\frac{A_{ki}}{Q^2}\right) = 0
\end{align*}



The energy identity for the time independent version of \eqref{uno} reads:

\begin{align*}
 & \frac12\int_{\Om}\frac{1}{Q^{2}}(A_{kj} \pa_{k} v^{i} + A_{ki} \pa_{k} v^{j})(A_{lj}\pa_{l}\phi^{i}+A_{li}\pa_{l} \phi^{j})  dx - \int_{\Om} q \pa_{k} \left(A_{ki} \frac{1}{Q^{2}} \phi^{i}\right) dx \\
& - \int_{\pa \Om} n^{l} A_{lj}(A_{kj} \pa_{k} v^{i} + A_{ki} \pa_{k} v^{j}) \phi^{i} \frac{1}{Q^{2}} d\sigma  + \int_{\pa \Om} \phi^{i} q A_{ki} n^{k} \frac{d\sigma}{Q^{2}} = \int_{\Om} f \cdot \phi \frac{1}{Q^{2}} dx
\end{align*}
and therefore
\begin{align*}
 & \frac12\int_{\Om}\frac{1}{Q^{2}}(A_{kj} \pa_{k} v^{i} + A_{ki} \pa_{k} v^{j})(A_{lj}\pa_{l}\phi^{i}+A_{li}\pa_{l} \phi^{j})  dx - \int_{\Om} q \pa_{k} \left(A_{ki} \frac{1}{Q^{2}} \phi^{i}\right) dx \\
& = -\int_{\pa \Om} (q \delta^{ij} - (A_{kj} \pa_{k} v^{i} + A_{ki} \pa_{k} v
^{j}))A_{lj} n^{l} \phi^{i} \frac{d\sigma}{Q^{2}} + \int_{\Om} f \cdot \phi \frac{1}{Q^{2}} dx, \\
\end{align*}
where

\begin{align*}
 \pa_{k} \left(A_{ki} \frac{1}{Q^{2}} \phi^{i}\right) = \frac{A_{ki}}{Q^{2}} \pa_{k} \phi^{i} =\frac{1}{Q^2} Tr(\nabla \phi A),
\end{align*}

so that we finally write as follows

\begin{align}\label{energy}
&\frac{1}{2}\int_{\Om}Tr\left(\left(\nabla v A+A^*\nabla v^*\right)\left(\nabla\phi A +A^*\nabla\phi^*\right)\right)\frac{1}{Q^2}dx-\int_{\Om}q \,Tr(\nabla\phi A) \frac{1}{Q^2} dx\\
&=\int_{\Om} f\cdot \phi \frac{dx}{Q^2} -\int_{\pa\Omega} \left(q\I -\left(\nabla v A+ A^*\nabla v^*\right)\right)A^* n \cdot \phi \frac{1}{Q^2}d\sigma.\nonumber
\end{align}
Therefore \eqref{energy} is the time independent weak formulation  of our system.

Next we will define a kind of Leray projector. Let $H^{0}_{d}$ the subspace of $H^{0}$ formed by vectors $A^{*} \nabla \phi$ such that $\phi \in H^{1}_{0}$. Let $H^{0}_{\sigma}$ the orthogonal complement of $H^{0}_{d}$ with the following vector valued $H^{0}$ scalar product:

\begin{align*}
 ( f,g )_{H^{0}} = \int_{\Om} f \cdot g \frac{1}{Q^{2}} dx.
\end{align*}
Then it is easy to check that if $v\in H^1\cap H^{0}_\sigma$, then $v$ must satisfy
$$Tr(\nabla v A) = 0.$$
For $v\in L^2$ we define $Rv$ to be  the orthogonal projection of $v$ onto $H^{0}_{\sigma}$ . For $v\in H^1$ we have that
\begin{align}\label{defR}
 Rv = v - A^{*} \nabla \psi,
\end{align}
where
\begin{align}
 Q^{2} \Delta \psi  = Tr(\nabla v A)\quad&\text{in }\Omega_0 \nonumber\\
\psi = 0\quad&\text{on }\partial\Omega_0,\label{vmRv}
\end{align}

The next lemma will deal with some properties about this projector $R$. Note that we have defined $R$ for functions of $x$. For functions of $(x,t)$ we apply $R$ for every $t$.
\begin{lemma}
 \label{lema31}
Let $0 \leq s \leq 4$. We have:
\begin{enumerate}
 \item[i)] $R$ is a bounded operator on $H^{s}$.
 \item[ii)] $R$ is a bounded operator on $H_{(0)}^{ht,s}$, with norm bounded uniformly if $T$ is bounded above.
 \item[iii)] If $\phi \in H^{s+1}$, then $R(A^{*} \nabla \phi) = A^{*} \nabla \phi_1$, with $Q^{2} \Delta \phi_1 = 0$ in $\Omega_0$, $\phi_1 = \phi$ on $\partial\Omega_0$.
\end{enumerate}
\end{lemma}

\begin{proof}
\begin{enumerate}
\item[i)] For $v \in H^{s}, 0 \leq s \leq 4$, it is easy to see that $Tr(\nabla v A) \in H^{s-1}$. Therefore the solution of the system \eqref{vmRv} satisfies:

\begin{align*}
\|\psi\|_{H^{s+1}} \leq C \left\|\frac{Tr(\nabla v A)}{Q^{2}}\right\|_{H^{s-1}} \leq \|v\|_{H^{s}},
\end{align*}

by elliptic theory, since both $A$ and $Q^{2}$ are regular. The identity \eqref{defR} provides

\begin{align*}
\|R v \|_{H^{s}} \leq \|v\|_{H^{s}} + \|\nabla \psi\|_{H^{s}} \leq C \|v\|_{H^{s}}.
\end{align*}

\item[ii)] It is easy to check that $\pa_{t}$ commutes with $R$, since

\begin{align*}
\pa_{t}^{k} Rv = \pa_{t}^{k} v - A^{*} \nabla \pa_{t}^{k} \psi = R \pa_{t}^{k} v.
\end{align*}

This proves the result for an integer number of derivatives. By interpolation we get the result for fractional derivatives.

\item[iii)] For $\phi \in H^{s+1}$, if $v = A^{*} \nabla \phi$, then

\begin{align*}
R(A^{*} \nabla \phi) = A^{*} \nabla \phi - A^{*} \nabla \psi = A^{*} \nabla (\phi - \psi).
\end{align*}

Thus, we have that

\begin{align*}
0 = Tr(\nabla (R(A^{*} \nabla \phi)) A ) = Q^{2} \Delta (\phi - \psi), \quad \left.\psi\right|_{\pa \Omega_{0}} =  0,
\end{align*}

which implies that $\left.\phi - \psi \right|_{\pa \Omega_{0}} = \phi$ and we can take $\phi_1 = \phi - \psi$.

\end{enumerate}
\end{proof}

Once we have obtained the energy identity \eqref{energy} and Lemma \ref{lema31} we pass to announce the main theorem of this section:

\begin{thm}\label{thm3.2}

Let $f \in H_{(0)}^{ht,s-1}$, let $v,q$ solve \eqref{uno} and $2<s\leq 3$. Then

\begin{align*}
 \|v\|_{H_{(0)}^{ht,s+1}} + \|\nabla q\|_{H_{(0)}^{ht,s-1}} + |q|_{H_{(0)}^{ht,s-\frac12}} \leq C \|f\|_{H_{(0)}^{ht,s-1}}.
\end{align*}
The constant is independent of $T$.
\end{thm}

The rest of this subsection is devoted to prove this theorem.

First of all we will  modify  the equation by considering the new variables

\begin{align*}
 u = e^{-t} v; \quad p = e^{-t} q; \quad \overline{f} = e^{-t} f
\end{align*}

We should remark that $f \in H_{(0)}^{ht,s-1} \Leftrightarrow \overline{f} \in H_{(0)}^{ht,s-1}$. Then, the equation reads

\begin{align*}
 \pa_t u & = -e^{-t} v + e^{-t} v_t = -u + e^{-t}(Q^{2} \Delta v - A^{*} \nabla q + f) \\
& = -u + Q^{2} \Delta u - A^{*} \nabla p + \overline{f}.
\end{align*}

We will solve therefore

\begin{align*}
 \pa_t u +u - Q^{2} \Delta u + A^{*} \nabla p = \overline{f}.
\end{align*}

%
%
%

Let's start projecting onto $H^0_\sigma$ to obtain
$$\pa_t u +u-Q^2 \Delta u +A^*\nabla q_1 = R\overline{f},$$
since $Tr(\nabla (Q^{2} \Delta \tilde{v}) A) = 0$ and therefore $R(Q^{2} \Delta u) = Q^{2} \Delta u$, where $A^*\nabla q_1=RA^*\nabla p$.

We now introduce the operator $$S_A: V^{s+1}(\Om) \rightarrow RH^{s-1}$$ defined via:

\begin{align*}
 S_Au = -Q^{2} \Delta u + u + A^{*} \nabla q_1, \quad A^{*} \nabla q_1 \equiv RA^{*} \nabla p,
\end{align*}

where $V^{r}(\Om) = \{u \in RH^{r}, A^*t_0\left(\nabla u A+A^*\nabla u^*\right)A^* n_0=0 \text{ on $\pa\Om$}\}$ and $RH^{s-1} = \{Rf, f \in H^{s-1}\}$.

The following lemma deals with the invertibility of $S_A$.
\begin{lemma}
 \label{lema33}
$S_A$ has a bounded inverse for $1\leq s \leq 3$, and

\begin{align*}
\|S_{A}^{-1} f\|_{H^{s+1}} \leq C \|f\|_{H^{s-1}}.
\end{align*}
\end{lemma}

\begin{proof}
 Let $f \in RH^{s-1}$. We will show that there exists $u \in V^{s+1}$ such that $S_Au = f$, i.e. $u - Q^{2} \Delta u + A^{*} \nabla q_1 = Rf$ (we will keep in the notation $Rf$ instead of $f$, although $Rf=f$, to keep in mind this fact). Using the energy identity \eqref{energy}, we observe that

\begin{align*}
 (u,\phi) + \langle u,\phi \rangle = - \int_{\pa \Om}h\cdot \phi \frac{d\sigma}{Q^{2}} + \int_{\Om} R\overline{f} \cdot \phi \frac{dx}{Q^{2}} + \int_{\Om} q_1 Tr(\nabla \phi A) \frac{d\tilde{x}}{Q^{2}}
\end{align*}
is the weak formulation of

\begin{align*}
 u - Q^{2} \Delta u + A^{*} \nabla q_1 & = R \overline{f} \\
\left(q_1-\left(\nabla u A+A^*\nabla u^*\right)\right)A^* n_0 & = h,
\end{align*}

where
\begin{align}
 (u,\phi) & = \int u \phi \frac{dx}{Q^{2}} \\
\langle u, \phi \rangle & = \int_{\Om} Tr((\nabla u A + A^{*} \nabla u^{*})(\nabla \phi A + A^{*} \nabla \phi^{*})) \frac{dx}{Q^{2}}
\end{align}

If we look for $h = 0$ we then have to solve

\begin{align}\label{weakstokes}
 (u,\phi) + \langle u, \phi \rangle = \int_{\Om} R \overline{f} \cdot \phi \frac{dx}{Q^{2}} + \int_{\Om} q_1 Tr(\nabla \phi A) \frac{d\tilde{x}}{Q^{2}}.
\end{align}
We will find a solution to this equation in $RH^1$. By this we mean that there exists $u\in RH^1$ such that \eqref{weakstokes} holds for all $\phi\in RH^1$. We notice that the last term vanishes since $\phi \in RH^{1}$, and henceforth, in $RH^1$, the equation \eqref{weakstokes} is equivalent to
\begin{align}\label{weakstokes2}
 (u,\phi) + \langle u, \phi \rangle = \int_{\Om} R \overline{f} \phi \frac{dx}{Q^{2}}.
\end{align}
for all $\phi\in RH^1$.

Via \cite[Corollary 4.7]{Neff:Korn-first-inequality-nonconstant-coefficients}, $$(u,u) + \langle u, u \rangle \geq C ||u||_{H^1}^2$$
and therefore it is a coercive bilinear form. That implies   Lax-Milgram's Theorem can be applied in $RH^{1}$ to obtain a solution $u \in RH^{1}$ of \eqref{weakstokes2}.

The next step is to show that there exists $p \in L^{2}$ such that

\begin{align*}
 (u,\phi) + \langle u, \phi \rangle = \int_{\Om} Rf \phi + \int_{\Om} p\, Tr(\grad\phi A) \quad \forall \phi \in H^{1}.
\end{align*}

To do this, we decompose $\phi = R\phi + A^{*} \nabla \pi_{\phi}$, and then

\begin{align*}
 (u,\phi) + \langle u, \phi \rangle & = (u,R\phi) + \langle u, R\phi \rangle + \langle u, A^{*} \nabla \pi_{\phi} \rangle + \underbrace{(u, A^{*} \nabla \pi_{\phi})}_{0} \\
& = (Rf, R\phi) + \langle u, A^{*} \nabla \pi_{\phi} \rangle = (Rf,\phi) + \langle u, A^{*} \nabla \pi_{\phi} \rangle.
\end{align*}

Therefore we have to show that there exists $p\in L^2$ such that

\begin{align*}
 (p,Tr(\nabla \phi A)) = \langle u, A^{*} \nabla \pi_{\phi} \rangle \quad \forall \phi \in H^{1}.
\end{align*}

Let us assume that $u$ is smooth and suppose we take $p$ satisfying:

\begin{align}
 Q^{2} \Delta p & = 0 \nonumber \\
\left. p \right|_{\pa \Om} & = (A^{*}n, ((\nabla u A) + (\nabla u A)^{*}) A^{*}n) \\
& = Q^{2}(n (A \nabla u + \nabla u^{*}A^* ) n). \label{eq44plus}
\end{align}

Then, on the one hand

\begin{align}\label{comparing}
 \int p Tr(\nabla \phi A) \frac{dx}{Q^{2}} & = \int p \Delta \pi_{\phi} dx = \int p \text{div}\left(\frac{AA^{*}}{Q^{2}}\nabla \pi_{\phi}\right) dx \\
& = -\int_{\Om} \nabla p \cdot AA^{*} \nabla \pi_{\phi} \frac{dx}{Q^{2}} + \int_{\pa \Om} p \cdot n AA^{*} \nabla \pi_{\phi} \frac{dx}{Q^{2}} \nonumber \\
& = - \underbrace{\int_{\Om} \nabla p Q^{2} \nabla \pi_{\phi} \frac{dx}{Q^{2}}}_{0} + \int_{\pa \Om} p (A^{*}n) \cdot A^{*} \nabla \pi_{\phi} \frac{dx}{Q^{2}}.\nonumber
\end{align}

On the other hand

\begin{align*}
 \langle u, A^{*} \nabla \pi_{\phi}\rangle = \frac12 \int Tr((\nabla u A + A^{*} \nabla u^{*})(\nabla (A^{*} \nabla \pi_{\phi}) A + A^{*} \nabla (A^{*} \nabla \pi_{\phi}))^{*}) \frac{dx}{Q^{2}},
\end{align*}

where

\begin{align*}
 (\nabla(A^{*} \nabla \pi_{\phi})A)_{ij} & = \pa_{l}(A^{*} \nabla \pi_{\phi})^{i} A_{lj} = \pa_{l}(A_{ki} \pa_{k} \pi_{\phi}) A_{lj} \\
& = A_{lj} \pa_{l} A_{ki} \pa_{k} \pi_\phi + A_{ki} A_{lj} \pa_{lk}^{2} \pi_\phi
\end{align*}
\begin{align*}
 \langle u, A^{*} \nabla \pi_{\phi}\rangle & = \int (\pa_{l} u^{i} A_{lj} + \pa_{l} u^{j} A_{li}) A_{mj} \pa_{m}(\pa_{k} \pi_\phi A_{ki}) \frac{dx}{Q^{2}} \\
& = -\int \pa_{m}\left(\left(A_{mj} \pa_{l} u^{i} A_{lj} + A_{mj} \pa_{l} u^{j} A_{li}\right)\frac{1}{Q^{2}}\right) \pa_{k} \pi_\phi A_{ki} dx \\
& +\int_{\pa\Om} (n_{0})_m\left(A_{mj} \pa_{l} u^{i} A_{lj} + A_{mj} \pa_{l} u^{j} A_{li}\right) \pa_{k} \pi_\phi A_{ki} \frac{dx}{Q^{2}} \\
& = -\int \pa_{m}\left(\left(\delta^{ml} \pa_{l} u^{i} + A_{mj} \pa_{l} u^{j} A_{li}\right)\frac{1}{Q^{2}}\right) \pa_{k} \pi_\phi A_{ki} dx \\
& +\int_{\pa\Om} A^{*}n_0 \cdot (\nabla u A + A^{*} \nabla u) A^{*} \nabla \pi_\phi \frac{dx}{Q^{2}}. \\
\end{align*}

The first term is zero because of the orthogonality and  because of the condition $Tr(\nabla u A)=0$. We do the calculations for the second term:

\begin{align*}
 \int_{\pa \Om} (A^{*} n_0)\cdot (\nabla u A + A^{*} \nabla u^{*}) A^{*} \nabla \pi_\phi \frac{d\sigma}{Q^{2}}
= \int n_0(A \nabla u + \nabla u^{*} A^{*}) \nabla \pi_\phi = \int n_0(A \nabla u + \nabla u^{*} A^{*}) n_0 \pa_{n} \pi_\phi,
\end{align*}
since  $t_0\cdot\nabla \pi_\phi = 0$. Comparing with \eqref{eq44plus} we have that
\begin{align*}
 (p,Tr(\nabla \phi A)) = \langle u, A^{*} \nabla \pi_\phi \rangle, \quad \forall \phi \in H^{1}.
\end{align*}
for smooth $u$. Let $u \in H^{1}$ and $\{u_m\}_{m=1}^{\infty}$ such that $u_m \in H^{\infty}$, and $tr(\nabla u_m A) = 0$ for every $m$, and $u_m\to u$ strongly in $H^1$ (for example we extend $u$ by zero to $\R^2$, we make the convolution with a mollifier $\rho_{\frac{1}{m}}$ and finally we project onto $H_\sigma$).  Let $p_m$ be given by
\begin{align*}
 \Delta p_m & = 0\\
\left.p_m\right|_{\pa \Om} & =A^{-1} n_0 (\nabla u_{m} A + A^{*} \nabla (u_{m})^{*})A^{-1} n_0.
\end{align*}

Then, we have that
\begin{align*}
 (p_m, Tr(\nabla \phi A)) = \langle u_m, A^{*} \nabla \pi_\phi\rangle \quad \forall \phi \in H^{1}.
\end{align*}

In particular, we take $\phi_{m}$ such that $Tr(\nabla \phi_{m}A) = p_m$. This implies that $Q^{2} \Delta \pi_{\phi_{m}} = p_{m}, \left.\pi_{\phi_m}\right|_{\pa \Om} = 0$. Showing the existence of such $\phi_{m}$ is trivial since one can choose $\phi_{m} = A^{*} \nabla \psi$, with $Q^{2} \Delta \psi = p_m$. Then, we can bound the $L^{2}$ norm of $p_m$ in the following way:

\begin{align*}
 \|p_m\|_{L^{2}}^{2} \leq C \langle u_{m}, A^{*} \nabla \pi_{\phi_m}\rangle  \leq C \|u_{m}\|_{H^{1}} \|\pi_{\phi_m}\|_{H^{2}} \leq C \|u\|_{H^{1}} \|p_m\|_{L^{2}},
\end{align*}
which shows that $p_m$ is bounded in $L^{2}$. Therefore there exists a subsequence $p_{m_i}$ which converges weakly to  a function $p$ in $L^{2}$ and
\begin{align*}
 (p, Tr(\nabla \phi A)) = \langle u, \nabla \pi_\phi\rangle \quad \forall \phi \in H^{1}.
\end{align*}
We have shown that there exist $(u,p) \in H^{1} \times L^{2}$ such that
\begin{align*}
 (u,\phi) + \langle u,\phi \rangle = (Rf,\phi) + \int p Tr(\nabla \phi A)  \frac{dx}{Q^{2}} \quad \forall \phi \in H^{1}.
\end{align*}
Indeed, $u \in H^{2}, p \in H^{1}$. We now show that improvement on the regularity.

For every $\Omega^{\flat} \Subset \Omega$, it is easy to obtain the interior regularity estimate $\|u\|_{H^{2}(\Omega^{\flat})} \leq C, \|p\|_{H^{1}(\Omega^{\flat})} \leq C$. We focus here on the boundary estimates. We perform the following change of coordinates in $\Omega'$, where $\Omega'$ is a tubular neighborhood of $\partial \Omega$:

\begin{align*}
 x(s,\lambda) = z(s) + \lambda z_s^{\perp}(s)
\end{align*}

We would like to check that it is indeed a diffeomorphism. We have following:

\begin{align*}
 |z_s|^{2} & = 1 \\
x_s(s) & = z_{s}(s) + \lambda z_{ss}^{\perp}(s) \\
 z_{ss}^{\perp} & = \langle z_{ss}^{\perp}, z_s(s) \rangle z_s(s) = -\kappa(s) z_s(s) \\
x_s(s) & = (1 - \lambda \kappa(s))z_s(s) \\
x_{\lambda} & = z_s^{\perp}(s)
\end{align*}

Computing more,
\begin{align*}
\text{det}
\left(
\begin{array}{cc}
 x^{1}_{s} & x^{2}_{s} \\
x^{1}_{\lambda} & x^{2}_{\lambda}
\end{array}
\right)
= x_{s}^{1} x_{\lambda}^{2} - x_{s}^{2} x_{\lambda}^{2} = x_{s}^{\perp} x_{\lambda} = 1 - \lambda \kappa(s)
\end{align*}

This fixes the width of the tubular neighborhood to be $\lambda_0 < \sup_{s} \frac{1}{\kappa(s)}$. Under these assumptions $x(s,\lambda)$ is a diffeomorphism. Fix $x_{0} \in \pa \Omega$, and consider the following cutoff function $\psi$, defined by

\begin{align*}
\psi(x) =
\left\{
\begin{array}{cc}
 1 & B\left(x_0,\frac{\lambda_0}{2}\right) \cap \Omega' \\ 
0 & \left(B\left(x_0,\frac{3\lambda_0}{4}\right) \cap \Omega'\right)^{c} 
\end{array}
\right\}, \quad 0 \leq \psi(x) \leq 1, \quad \psi(x) \in C^{\infty}.
\end{align*}

We extend $\psi$ to $\Omega$ by zero. 
We define the set $C$ as $C = x^{-1}(B(x_0,\lambda_0) \cap \Omega')$, and the set $C/2 = x^{-1}(B(x_0,\frac{\lambda_0}{2}) \cap \Omega')$.

The energy identity can be written down as:

\begin{align*}
\frac12 \int_{\Omega} (\pa_{l} u^{i} A_{lj} + \pa_{l} u^{j} A_{li}) (\pa_{k} \phi^{i} A_{kj} + \pa_{k} \phi^{j} A_{ki}) \frac{dx}{Q^{2}}
+ \int_{\Omega} u^{i} \phi^{i} \frac{dx}{Q^{2}} = \int_{\Omega} f^{i} \phi^{i} \frac{dx}{Q^{2}} + \int_{\Omega} p \pa_{l} \phi^{i} A_{li} \frac{dx}{Q^{2}}
\end{align*}

We choose as test function $\phi = \va \psi$. Then:

\begin{align*}
\int_{\Omega} (\pa_{l} u^{i} A_{lj} + \pa_{l} u^{j} A_{li}) (\pa_{k} (\va^{i} \psi) A_{kj} + \pa_{k} (\va^{j} \psi)A_{ki}) \frac{dx}{Q^{2}}
+ \int_{\Omega} u^{i} \phi^{i} \frac{dx}{Q^{2}} & = \int_{\Omega} f^{i} \phi^{i} \frac{dx}{Q^{2}} + \int_{\Omega} p \pa_{l} \phi^{i} A_{li} \frac{dx}{Q^{2}} \\
M_1 + M_2 & = M_3 + M_4
\end{align*}

We start developing each of the terms one by one

\begin{align*}
 M_1 & = \int_{\Omega'} (\pa_{l} u^{i} A_{lj} + \pa_{l} u^{j} A_{li}) (\pa_{k} (\va^{i}) A_{kj} + \pa_{k} (\va^{j})A_{ki})  \psi \frac{dx}{Q^{2}} \\
& + \int_{\Omega'} (\pa_{l} u^{i} A_{lj} + \pa_{l} u^{j} A_{li}) (\va^{i} \pa_{k} \psi A_{kj} + \va^{j} \pa_{k} \psi A_{ki}) \frac{dx}{Q^{2}} \\
& = \int_{\Omega'} (\pa_{l} u^{i} A_{lj} + \pa_{l} u^{j} A_{li}) \psi (\pa_{k} (\va^{i}) A_{kj} + \pa_{k} (\va^{j})A_{ki}) \frac{dx}{Q^{2}} \\
& + \int_{\Omega'} (\pa_{l} u^{i} A_{lj} + \pa_{l} u^{j} A_{li})\pa_{k} \psi (\va^{i}  A_{kj} + \va^{j}  A_{ki}) \frac{dx}{Q^{2}} \\
\end{align*}

We now do the change of variables

\begin{align*}
 \underline{dx} = \left(\frac{dx}{dsd\lambda}\right) ds d\lambda = \underbrace{(1+\kappa(s) \lambda)}_{J(s,\lambda)} ds d\lambda,
\end{align*}

and define

\begin{align*}
 \overline{A}(s,\lambda) = A \circ x(s,\lambda),&\quad \overline{Q}(s,\lambda) = Q \circ x(s,\lambda),\quad
 \overline{u}(s,\lambda) = u \circ x(s,\lambda),\quad \overline{p}(s,\lambda)=p\circ x(s,\lambda),\\
 \overline{\varphi}(s,\lambda) = \va \circ x(s,\lambda),&\quad \overline{\psi}(s,\lambda) = \psi \circ x(s,\lambda),\quad \overline{f}(s,\lambda) = f\circ x(s,\lambda). \\
\end{align*}

We can compute the derivatives in the bar-coordinates

\begin{align*}
 u(x) = \overline{u} \circ s(x), \quad \pa_{l} u^{i} = \pa_{k} \overline{u}^{i} \circ s(x) \pa_{l} s^{k} \\
\pa_{l} u^{i} \circ x(s,\lambda) = \pa_{k} \overline{u}^{i}(s,\lambda) \pa_{l} s^{k} \circ x(s,\lambda)
\end{align*}

Setting $x = x(s,\lambda), \quad x^{j} = x^{j}(s,\lambda)$, we obtain:

\begin{align*}
 \pa_{l} x^{j} & = \pa_{l} x^{j}(s,\lambda) \frac{\pa s^{l}}{\pa x^{i}} \\
\delta^{ij} & = \pa_{l} x^{j}(s,\lambda) \frac{\pa s^{l}}{\pa x^{i}} \\
\delta^{ij} & = \pa_{l} x^{j} \circ s(x) \frac{\pa s^{l}}{\pa x^{i}}(x) \\
\delta^{ij} & = (\nabla x)_{jl} (\nabla s)_{li} \circ x(s,\lambda) = (\nabla x \nabla s)_{ji}\\
I & = \nabla x \nabla s \circ x(s,\lambda), \quad \nabla s \circ x(s,\lambda) = (\nabla x)^{-1} \equiv \Sigma
\end{align*}

Plugging this result into the equation for the derivatives of $u$, we get

\begin{align*}
 \pa_{l} u^{i} \circ x(s,\lambda) & = \Sigma_{kl} \pa_{k} \overline{u}^{i} \\
\Sigma & = \frac{1}{1 - \kappa(s)\lambda}
\left(
\begin{array}{cc}
 x^{2}_{\lambda} & -x^{1}_{\lambda} \\
-x^{2}_{s} & x^{1}_{s}
\end{array}
\right)
\end{align*}

Plugging this expression into $M_1$, and letting $\Sigma \overline{A} = B$, we obtain

\begin{align*}
2 M_1 & = \int_{C} (\pa_{l} u^{i} \circ x(s,\lambda) \overline{A}_{lj} + \pa_{l} u^{j} \circ x(s,\lambda) \overline{A}_{li}) \left(\psi (\pa_{k} \overline{\va}^{i} \circ x(s,\lambda) \overline{A}_{kj} + \pa_{k} \overline{\va}^{j} \circ x(s,\lambda) \overline{A}_{ki}\right) \frac{J}{\overline{Q}^{2}} ds d\lambda \\
& + \int_{C} \pa_{a} \overline{\psi} (\pa_{m} \overline{u}^{i} B_{mj} + \pa_{m} \overline{u}^{j} B_{mi}) (\overline{\va}^{i}  B_{aj} + \overline{\va}^{j}  B_{ai}) \frac{J}{\overline{Q}^{2}}dsd\lambda \\
& = \int_{C} (\pa_{m} \overline{u}^{i} \Sigma_{ml} \overline{A}_{lj} + \pa_{m} \overline{u}^{j} \Sigma_{ml} \overline{A}_{li}) \overline{\psi} (\pa_{a} \overline{\va}^{i} \Sigma_{ak} \overline{A}_{kj} + \pa_{a} \overline{\va}^{j} \Sigma_{ak} \overline{A}_{ki}) \frac{J}{\overline{Q}^{2}} ds d\lambda \\
& + \int_{C} \pa_{a} \overline{\psi} (\pa_{m} \overline{u}^{i} B_{mj} + \pa_{m} \overline{u}^{j} B_{mi}) (\overline{\va}^{i}  B_{aj} + \overline{\va}^{j}  B_{ai}) \frac{J}{\overline{Q}^{2}}dsd\lambda \\
& = \int_{C} \overline{\psi} (\pa_{m} \overline{u}^{i} B_{mj} + \pa_{m} \overline{u}^{j} B_{mi})  (\pa_{a} \overline{\va}^{i} B_{aj} + \pa_{a} \overline{\va}^{j} B_{ai}) \frac{J}{\overline{Q}^{2}} ds d\lambda \\
& + \int_{C} \pa_{a} \overline{\psi} (\pa_{m} \overline{u}^{i} B_{mj} + \pa_{m} \overline{u}^{j} B_{mi}) (\overline{\va}^{i}  B_{aj} + \overline{\va}^{j}  B_{ai}) \frac{J}{\overline{Q}^{2}}dsd\lambda
\end{align*}



Moreover,

\begin{align*}
 M_2 & = \int_{\Omega'} u^{i} \varphi^{i} \psi \frac{dx}{Q^{2}} = \int_{C} \overline{u}^{i} \overline{\varphi}^{i} \overline{\psi} \frac{J}{\overline{Q}^{2}} ds d\lambda \\
M_3 & = \int_{C} \overline{f}^{i} \overline{\varphi}^{i} \overline{\psi} \frac{J}{\overline{Q}^{2}} ds d\lambda \\
M_4 & = \int_{C} \overline{p} \pa_{a}(\overline{\varphi}^{i} \overline{\psi}) B_{ai} \frac{J}{\overline{Q}^{2}} ds d\lambda,
\end{align*}


with

\begin{align*}
 M_1 + M_2 = M_3 + M_4 \Rightarrow M_1 = -M_2 + M_3 + M_4.
\end{align*}

We start computing $M_1$. We have that

\begin{align*}
 2M_1 & = \int_{C} (\pa_{l} \overline{u}^{i} B_{lj} + \pa_{l} \overline{u}_{j} B_{li})(\pa_{k} \overline{\va}^{i} B_{kj} + \pa_{k} \overline{\va}^{j} B_{ki}) \overline{\psi} \frac{J}{\overline{Q}^{2}} \\
 & + \int_{C} (\pa_{l} \overline{u}^{i} B_{lj} + \pa_{l} \overline{u}_{j} B_{li})( \overline{\va}^{i} B_{kj} + \overline{\va}^{j} B_{ki}) \pa_{k} \overline{\psi} \frac{J}{\overline{Q}^{2}} \\
 & = m + l_1.
\end{align*}

We will use the convention that the $m$ terms will denote bad terms from now on. We further split $m$ into four terms:

\begin{align*}
m & = \int_{C} \pa_{l} \overline{u}^{i} B_{lj} \pa_k \overline{\va}^{i}B_{kj} \overline{\psi}\frac{J}{\overline{Q}^{2}} + \int_{C} \pa_{l} \overline{u}^{j} B_{li} \pa_k \overline{\va}^{i}B_{kj} \overline{\psi}\frac{J}{\overline{Q}^{2}} \\
& + \int_{C} \pa_{l} \overline{u}^{i} B_{lj} \pa_k \overline{\va}^{j}B_{ki} \overline{\psi}\frac{J}{\overline{Q}^{2}} + \int_{C} \pa_{l} \overline{u}^{j} B_{li} \pa_k \overline{\va}^{j}B_{ki} \overline{\psi}\frac{J}{\overline{Q}^{2}} \\
& = m_{1} + m_{2} + m_{3} + m_{4}.
\end{align*}

If we are able to estimate any one of them we can estimate all of them.

We will use the following formulas related to finite differences:

\begin{align*}
 D_{s}^{h}(fg) & = \frac{f(s+h)g(s+h) - f(s)g(s)}{h} = \frac{(f(s+h)-f(s))g(s+h)}{h} + f(s)\frac{g(s+h)-g(s)}{h} \\
& = g(s+h)D_{s}^{h}f + f(s)D_{s}^{h}g, \\
D_{s}^{-h}D_{s}^{h}(fg) & = D_{s}^{-h}(D_{s}^{h}f g(s)) + D_{s}^{-h}(f(s)D_{s}^{h}g) \\
& = (D_{s}^{-h}D_{s}^{h}f) g(s) + D_{s}^{h}fD_{s}^{-h}(g(s+h)) + D_{s}^{-h}f D_{s}^{h}g(s-h) + f(s) D_{s}^{-h} D_{s}^{h} g, \\
D_{s}^{-h}(g(s+h)) & = \frac{g(s)-g(s+h)}{-h} = D_{s}^{h} g, \\
D_{s}^{h} (g(s-h)) & = \frac{g(s)-g(s-h)}{h} = D_{s}^{-h} g, \\
\Rightarrow D_{s}^{-h} D_{s}^{h} (fg) & = (D_{s}^{-h} D_{s}^{h} f) g(s) + D_{s}^{h} f D_{s}^{h}g + D_{s}^{-h} f D_{s}^{-h} g + f(s) D_{s}^{-h} D_{s}^{h} g.
\end{align*}

We take $\overline{\va} = D_{s}^{-h} D_{s}^{h} \overline{u}$. Since divergence free condition, we have that
\begin{align*}
 D_{s}^{-h} D_{s}^{h}(B_{lj} \pa_{l} \overline{u}^{j}) & = 0 = (D_{s}^{-h} D_{s}^{h} B_{lj}) \pa_{l} \overline{u}^{j}(s) + D_{s}^{h} B_{lj} D_{s}^{h} \pa_{l} \overline{u}^{j} + D_{s}^{-h} B_{lj} D_{s}^{-h} \pa_{l} \overline{u}^{j} + B_{lj} \pa_{l} D_{s}^{-h} D_{s}^{h} \overline{u}^{j}.
\end{align*}

Therefore

\begin{align*}
 B_{lj} \pa_{l} D_{s}^{-h} D_{s}^{h} \overline{u}^{j}  =& - (D_{s}^{-h} D_{s}^{h} B_{lj}) \pa_{l} \overline{u}^{j}(s) - D_{s}^{h} B_{lj} D_{s}^{h} \pa_{l} \overline{u}^{j} - D_{s}^{-h} B_{lj} D_{s}^{-h} \pa_{l} \overline{u}^{j} \\
B_{lj} \pa_{l} D_{s}^{-h} D_{s}^{h}(\overline{u}^{j} \overline{\psi})  =& D_{s}^{-h} D_{s}^{h}(B_{lj} \pa_{l}(\overline{u}^{j} \overline{\psi})) - D_{s}^{-h} D_{s}^{h} B_{lj} \pa_{l}(\overline{u}^{j}(s)\overline{\psi}(s)) \\
& - D_{s}^{h} B_{lj} D_{s}^{h} \pa_{l} (\overline{u}^{j} \overline{\psi}) - D_{s}^{-h} B_{lj} D_{s}^{-h}(\pa_{l}(\overline{u}^{j} \overline{\psi})),
\end{align*}

and

\begin{align*}
 B_{lj} \pa_{l}(\overline{\psi} \overline{u}^{j}) = B_{lj} \pa_{l} \overline{\psi} \overline{u}^{j}.
\end{align*}

Expanding the calculations, we obtain

\begin{align*}
 B_{lj} \pa_{l} D_{s}^{-h} D_{s}^{h}(\overline{u}^{j} \overline{\psi}) = & D_{s}^{-h} D_{s}^{h}(B_{lj} \pa_{j} \overline{\psi} \overline{u}^{j}) - D_{s}^{-h} D_{s}^{h} B_{lj} \pa_{l}(\overline{u}^{j}(s) \overline{\psi}(s)) \\
& - D_{s}^{h} B_{lj} D_{s}^{h} \pa_{l}(\overline{u}^{j} \overline{\psi}) - D_{s}^{-h} B_{lj} D_{s}^{-h}(\pa_{l} \overline{u}^{j} \overline{\psi}))  \\
\overline{\psi} B_{lj} \pa_{l} D_{s}^{-h} D_{s}^{h} (\overline{u}^{j} \overline{\psi}) = & \overline{\psi} D_{s}^{-h} D_{s}^{h}(B_{lj} \pa_{l} \overline{\psi} \overline{u}^{j}) - \overline{\psi} D_{s}^{-h} D_{s}^{h} B_{lj} \pa_{l} (\overline{u}^{j} \overline{\psi}) \\
& - \overline{\psi} D_{s}^{h} B_{lj} D_{s}^{h} \pa_{l}(\overline{u}^{j} \overline{\psi}) - \overline{\psi} D_{s}^{-h} B_{lj} D_{s}^{-h}(\pa_{l}(\overline{u}^{j} \overline{\psi})).
\end{align*}

Finally, we get to

\begin{align*}
 &\overline{\psi} D_{s}^{-h} D_{s}^{h}(B_{lj} \pa_{l} \overline{\psi} \overline{u}^{j})  = D_{s}^{-h} D_{s}^{h}(B_{lj} \pa_{l} \overline{\psi} \overline{u}^{j}\overline{\psi}) \\&- D_{s}^{-h}(B_{lj}\pa_l \overline{\psi}\overline{u}^j)D_{s}^{-h}\overline{\psi}-D_{s}^{h}(B_{lj}\pa_l \overline{\psi}\overline{u}^j)D_{s}^{h}\overline{\psi}-B_{lj}\pa_l\overline{\psi}\overline{u}^jD_s^{-h}D_s^{h}\overline{\psi}.
\end{align*}

Therefore

\begin{align*}
\overline{\psi} B_{lj} \pa_{l} D_{s}^{-h} D_{s}^{h}(\overline{u}^{j} \overline{\psi}) & = D_{s}^{-h} D_{s}^{h}(B_{lj} \pa_{l} \overline{\psi} \overline{u}^{j}\overline{\psi})
- \overline{\psi} D_{s}^{h} B_{lj} D_{s}^{h} \pa_{l}(\overline{u}^{j} \overline{\psi}) - \overline{\psi} D_{s}^{h} B_{lj} D_{s}^{h} \pa_{l}(\overline{u}^{j} \overline{\psi}) + \text{LOW} \\
& = B_{lj} \pa_{l} \overline{\psi} D_{s}^{-h} D_{s}^{h} (\overline{u}^{j} \overline{\psi}) - \overline{\psi} D_{s}^{h} B_{lj} D_{s}^{-h} \pa_{l}(\overline{u}^{j} \overline{\psi}) - \overline{\psi} D_{s}^{-h} B_{lj} D_{s}^{h} \pa_{l}(\overline{u}^{j} \overline{\psi}) + \text{LOW},
\end{align*}

where we say that a term $T$ is LOW is $\|T\|_{L^{2}} \leq C\|\overline{u}\|_{H^{1}}$. We also say that a term $T$ is SAFE if, for any $\delta > 0$,

\begin{align*}
 \|T\|_{L^{2}} \leq C_{\delta} + \delta(\|\nabla D_{s}^{-h}(\overline{u} \overline{\psi})\|_{L^{2}}^{2} + \|\nabla D_{s}^{h}(\overline{u} \overline{\psi})\|_{L^{2}}^{2}),
\end{align*}

where $C_{\delta}$ may depend on $\|\overline{u}\|_{H^{1}}, \|\overline{p}\|_{L^{2}}, \|\overline{f}\|_{L^{2}}$ and $\overline{\psi}, B, \overline{Q}$ or $J$. We now look at what the terms $M_1, \ldots, M_4$ look like. We have that

\begin{align*}
 M_2 & = \int_{C} \overline{\psi} \overline{u}^{i} D_{s}^{-h} D_{s}^{h} \overline{u}^{i} \frac{J}{\overline{Q}^{2}} = \int_{C} D_{s}^{h}(\overline{\psi} \overline{u}^{i}) D_{s}^{h} \overline{u}^{i} \frac{J}{Q^{2}} + \text{LOW}, \\
M_3 & = \int_{C} \overline{f}^{i} D_{s}^{-h} D_{s}^{h} \overline{u}^{i} \frac{J}{\overline{Q}^{2}} \leq C_{\delta}\|\overline{f}^{i}\|_{L^{2}}^{2} + \delta \|D_{s}^{-h}D_{s}^{h} (\overline{u} \overline{\psi})\|_{L^{2}}^{2}+\text{LOW}, \\
M_4 & = \int_{C} \overline{p} B_{lj} \pa_{j} \overline{\psi} D_{s}^{-h} D_{s}^{h} \overline{u}^{i} \frac{J}{Q^{2}}\leq C_{\delta}\|\overline{p}\|_{L^{2}}^{2} + \delta \|D_{s}^{-h} D_{s}^{h}(\overline{u} \overline{\psi})\|_{L^{2}}^{2}, \\
l_1 & = \int_{C} \pa_{l} \overline{u}^{i} B_{lj} \phi^{j} \pa_{k} \overline{\psi} B_{kj} \frac{J}{\overline{Q}^{2}} \leq C_{\delta} \|\overline{u}\|_{H^{1}}^{2} + \delta \|D_{s}^{-h} D_{s}^{h}(\overline{u} \overline{\psi})\|_{L^{2}}^{2}, \\
M_1 & = \int_{C} \pa_{l} \overline{u}^{i} B_{lj} B_{kj} \pa_{k} D_{s}^{-h} D_{s}^{h}(\overline{u}^{i} \overline{\psi}) \overline{\psi} \frac{J}{\overline{Q}^{2}}
= \int_{C} D_{s}^{h}\left(\overline{\psi} \pa_{l} \overline{u}^{i} B_{lj} B_{kj}\frac{J}{\overline{Q}^{2}}\right) D_{s}^{h}(\pa_{k}(\overline{u}^{i} \overline{\psi})) \\
& = \int_{C} B_{lj} B_{kj} D_{s}^{h}(\overline{\psi} \pa_{l} \overline{u}^{i}) D_{s}^{h} \pa_{k}(\overline{u}^{i} \overline{\psi}) \frac{J}{\overline{Q}^{2}} + \text{SAFE} \\
& = \int_{C} B_{lj} B_{kj} D_{s}^{h}(\pa_{l} \overline{\psi} \overline{u}^{i})D_{s}^{h} \pa_{k}(\overline{u}^{i} \overline{\psi}) \frac{J}{\overline{Q}^{2}} + \text{SAFE}.
\end{align*}

We can then bound

\begin{align*}
 \int_{C} |\nabla D_{s}^{h}(\overline{\psi} \overline{u})|^2 & \leq M_1 \leq C_{\delta} + \delta \|\nabla D_{s}^{-h}(\overline{u} \overline{\psi})\|_{L^{2}}^{2} + \delta \|\nabla D_{s}^{h}(\overline{u} \overline{\psi})\|_{L^{2}}^{2}, \\
 \int_{C} |\nabla D_{s}^{-h}(\overline{\psi} \overline{u})|^2 & \leq M_1 \leq C_{\delta} + \delta \|\nabla D_{s}^{-h}(\overline{u} \overline{\psi})\|_{L^{2}}^{2} + \delta \|\nabla D_{s}^{h}(\overline{u} \overline{\psi})\|_{L^{2}}^{2}.
\end{align*}

Similar computations show that we can control $\|\nabla \pa_{s} u \|_{L^{2}}$. This implies that

\begin{align*}
\|\pa_{s} u \|_{H^{1}(C/2)} \leq C.
\end{align*}

 We proceed to control the pressure terms. From the energy identity:

\begin{align*}
 & \frac12\int_{C}(\pa_{l} \overline{u}^{i} B_{lj} + \pa_{l} \overline{u}^{j} B_{li})(\pa_{k} \overline{\va}^{i} B_{kj} + \pa_{k} \overline{\va}^{j} B_{ki}) \frac{J}{\overline{Q}^{2}} + \int_{C}(\pa_{l} \overline{u}^{i} B_{lj} + \pa_{l} \overline{u}^{j} B_{li})(\overline{\va}^{i} B_{kj} + \overline{\va}^{j} B_{ki}) \pa_{k} \overline{\psi} \frac{J}{\overline{Q}^{2}} + \int_{C} \overline{u}^{i} \overline{\pa}^{i} \psi \frac{J}{\overline{Q}^{2}} \\
& = \int_{C} \overline{f}^{i} \overline{\va}^{i} \overline{\psi} \frac{J}{\overline{Q}^{2}} + \int_{C} \overline{p} Tr(\nabla \overline{\va} B) \overline{\psi} \frac{J}{\overline{Q}^{2}} + \overline{p} B \nabla \overline{\psi} \overline{\va} \frac{J}{\overline{Q}^{2}} \\
& P_{1} + P_{2} = P_{3} + P_{4} + P_{5}
\end{align*}

We choose $\overline{\va} = B^{*} \nabla \overline \Pi$, with

\begin{align*}
 \frac{J}{\overline{Q}^{2}} D_{s}^{-h} D_{s}^{h}(\overline{p} \overline{\psi}) = \text{div}(J\Sigma \Sigma^{*} \nabla \overline{\Pi}).
\end{align*}

Then:

\begin{align*}
 P_{4} & = \int_{C} \overline{p} Tr(\nabla \overline{\va} B) \overline{\psi} \frac{J}{\overline{Q}^{2}} = \int_{C} \overline{\psi} \overline{p} D_{s}^{-h} D_{s}^{h}(\overline{p} \overline{\psi}) \frac{J}{\overline{Q}^{2}} \\
& = -\int_{C}(D_{s}^{h}(\overline{\psi} \overline{p}))^{2} \frac{J}{\overline{Q}^{2}} + \int_{C}\left(D_{s}^{h}\left(\frac{J}{\overline{Q}^{2}} \overline{\psi} \overline{p}\right) - \frac{J}{\overline{Q}^{2}} D_{s}^{h}(\overline{\psi} \overline{p})\right) \\
& = m_{5} + l_{1},
\end{align*}

where

\begin{align*}
 |l_{1}| \leq c \|p\|_{L^{2}}.
\end{align*}

So, if we control $m_{5}$ we can control $\|\pa_{s} \overline{p}\|_{L^{2}\left(C/2\right)}$. It is not hard to see that

\begin{align*}
 |P_1| + |P_2| + |P_3| + |P_5| \leq C_{\delta}(\|\overline{u}\|_{L^{2}}^{2} + \|\nabla \pa_{s} \overline{u}\|_{L^{2}}^{2} + \|\nabla \overline{u}\|_{L^{2}}^{2} + \|\overline{p}\|_{L^{2}}^{2}) + \delta\|D_{s}^{h}(\overline{p} \overline{\psi})\|_{L^{2}}^{2}
\end{align*}

which implies that $\|D_{s}^{h}(\overline{p} \overline{\psi})\|_{L^2} \leq C$ independent of $h$. Because of the interior regularity estimate we have that the solution is strong in the interior of $\Omega$ and
we can write

\begin{align*}
 \frac{\overline{Q}^{2}}{J}\text{div}\left(\frac{1}{J}\left(\begin{array}{cc}1 & 0 \\ 0 & J^{2}\end{array}\right)\nabla \overline{u}^{i}\right) + \overline{u}^{i} + (B^{*} \nabla \overline{p})^{i} = \overline{f}^{i}, \quad Tr(\nabla \overline{u} B) = 0
\end{align*}

where we recall that $J = (1 + \lambda \kappa(s))$ and that we have used
 \begin{align*}
  \Delta u \circ x(s,\lambda) = \frac{1}{J(s,\lambda)} \text{div}\left(J(s,\lambda)\Sigma \Sigma^{*} \nabla \overline{u}\right)(s,\lambda)
 \end{align*}

 \begin{align*}
  \Sigma \Sigma^{*} = \frac{1}{(1+\kappa(s)\lambda)^{2}}\left(\begin{array}{cc}x_{\lambda}^{2} & -x_{\lambda}^{1} \\ -x_{s}^{2} & x_{s}^{1} \end{array}\right)\left(\begin{array}{cc}x_{\lambda}^{2} & -x_{s}^{2} \\ -x_{\lambda}^{1} & x_{s}^{1} \end{array}\right)
 = \frac{1}{(1+\kappa(s)\lambda)^{2}}\left(\begin{array}{cc}|x_{\lambda}|^{2} & -x_{\lambda}\cdot x_{s} \\ -x_{\lambda} \cdot x_{s} & |x_{s}|^{2} \end{array}\right)
 \end{align*}

 This implies that

 \begin{align*}
  \frac{1}{J(s,\lambda)} \text{div}\left(\frac{1}{(1+\kappa(s)\lambda)^{2}}\left(\begin{array}{cc}|x_{\lambda}|^{2} & -x_{\lambda}\cdot x_{s} \\ -x_{\lambda} \cdot x_{s} & |x_{s}|^{2} \end{array}\right)\nabla \overline{u}^{i}\right)
 = \frac{1}{J(s,\lambda)} \text{div}\left(\frac{1}{(1+\kappa(s)\lambda)^{2}}\left(\begin{array}{cc}1 & 0 \\ 0 & (1+\kappa(s)\lambda)^{2} \end{array}\right)\nabla \overline{u}^{i}\right).
 \end{align*}

Let us define

\begin{align*}
 \beta & = \left(\begin{array}{cc}1/J & 0 \\ 0 & J\end{array}\right) \\
 \text{div}(\beta \nabla \overline{u}^{i}) & = \pa_{k}(\beta_{kl} \pa_{l} \overline{u}^{i}) = \underbrace{\pa_{k} B_{kl}}_{\Gamma_{l}} \pa_{l} \overline{u}^{i} + \beta_{kl} \pa_{k} \pa_{l} \overline{u}^{i} \\
& = \Gamma \nabla \overline{u}^{i} + \frac{1}{J} \pa_{s}^{2} \overline{u}^{i} + J \pa_{\lambda}^{2} \overline{u}^{i}.
\end{align*}

And then:

\begin{align*}
  \frac{\overline{Q}^{2}}{J} \pa_{s}^{2} \overline{u}^{i} + \overline{Q}^{2} \pa_{\lambda}^{2} \overline{u}^{i} + \Gamma \nabla \overline{u}^{i} + B_{li} \pa_{l} \overline{p} = \overline{f}^{i} \\
\Rightarrow   \overline{Q}^{2} \pa_{\lambda}^{2} \overline{u}^{i}  + B_{li} \pa_{l} \overline{p} = -\frac{\overline{Q}^{2}}{J} \pa_{s}^{2} \overline{u}^{i}  - \Gamma \nabla \overline{u}^{i} + \overline{f}^{i} = g^{i},
\end{align*}

and we know that $g^{i} \in L^{2}\left(C/2\right)$. We also have that

\begin{align*}
 (\pa_{l} \overline{u}^{i} B_{li}) & = 0 \\
\pa_{s} \overline{u}^{1} B_{11} + \pa_{\lambda} \overline{u}^{i} B_{21} + \pa_{s} \overline{u}^{2} B_{12} + \pa_{\lambda} \overline{u}^{2} B_{22} & = 0
\end{align*}

Thus

\begin{align*}
 \pa_{\lambda}^{2} \overline{u}^{1} B_{21} + \pa_{\lambda}^{2} \overline{u}^{2} B_{22} & = g^{3}, \text{ where } g^{3} \in L^{2}
\end{align*}

Plus

\begin{align*}
 \overline{Q}^{2} \pa_{\lambda}^{2} \overline{u}^{1}  + B_{11} \pa_{s} \overline{p} + B_{21} \pa_{\lambda} \overline{p} & =  g^{1} \\
 \overline{Q}^{2} \pa_{\lambda}^{2} \overline{u}^{2}  + B_{12} \pa_{s} \overline{p} + B_{22} \pa_{\lambda} \overline{p} & =  g^{2} \\
 \overline{Q}^{2} B_{21} \pa_{\lambda}^{2} \overline{u}^{1}  + B_{12} B_{11} \pa_{s} \overline{p} + B_{12}^{2} \pa_{\lambda} \overline{p} & =  B_{12} g^{1} \\
 \overline{Q}^{2} B_{22} \pa_{\lambda}^{2} \overline{u}^{2}  + B_{22} B_{12} \pa_{s} \overline{p} + B_{22}^{2} \pa_{\lambda} \overline{p} & =  B_{22} g^{2} \\
\end{align*}

This implies

\begin{align*}
 (B_{21}^{2} + B_{22}^{2}) \pa_{\lambda} \overline{p} & = g^{4}, \text{ where } g^{4} \in L^{2}\left(C/2\right),
\end{align*}

and since $(B_{21}^{2} + B_{22}^{2}) > 0$ for a small enough $\lambda_0$, and we get that $\pa_{\lambda} \overline{p} \in L^{2}\left(C/2\right)$. Finally, we use this to get that

\begin{align*}
 \overline{Q}^{2} \pa_{\lambda}^{2} \overline{u}^{1} & = g^{5}, \text{ where } g^{5} \in L^{2}\left(C/2\right). \\
 \overline{Q}^{2} \pa_{\lambda}^{2} \overline{u}^{2} & = g^{6}, \text{ where } g^{6} \in L^{2}\left(C/2\right).
\end{align*}

This completes the regularity proof, since we can cover a $\frac{\lambda_0}{2}$ neighborhood by a finite number of sets of type $C/2$.

%
%
%
%
%
We are only left to show that we can apply \cite[Theorem 10.5, p.78]{Agmon-Douglis-Nirenberg:estimates-boundary-elliptic-pde-II} to our system in order to obtain higher regularity. Indeed we can apply this theorem to show:

\begin{align}\label{restimate}
 \|u\|_{H^{s+1}} + \|p\|_{H^{s}} \leq \|f\|_{H^{s-1}}.
\end{align}

What follows is a confirmation that our problem fulfills the elliptical conditions of \cite{Agmon-Douglis-Nirenberg:estimates-boundary-elliptic-pde-II}. To adapt our notation to the one in \cite{Agmon-Douglis-Nirenberg:estimates-boundary-elliptic-pde-II} we will write $(u^{1}, u^{2}, u^{3}) = (u^{1}, u^{2}, p)$.

In \cite[pp.38, 42]{Agmon-Douglis-Nirenberg:estimates-boundary-elliptic-pde-II} the system is written like
\begin{align*}
l_{ij}(x,\pa)u^j=F_i,
\end{align*}
with the boundary condition
\begin{align*}
B_{h,j}(x,\pa)u_j= \phi_h.
\end{align*}
In our case we have the correspondence
\begin{align*}
 l_{11} = 1- Q^{2} \Delta, \quad l_{12} = 0, \quad l_{13} = A_{k1} \pa_{l}, \quad F_{1} = f_1 \\
 l_{21} = 0, \quad l_{22} =1 - Q^{2} \Delta, \quad l_{23} = A_{k2} \pa_{l}, \quad F_{2} = f_2 \\
 l_{31} = A_{k1} \pa_{k}, \quad l_{32} = A_{k2} \pa_{k}, \quad l_{33} = 0, \quad F_{3} = 0 \\
\end{align*}
And the indices $t_j$ and $s_j$ can be taken as
\begin{align*}
t_1 = 2, \quad t_2 = 2, \quad t_3 = 1 \\
s_1 = 0, \quad s_2 = 0, \quad s_3 = -1
\end{align*}
It can be checked that with this choice $l'_{ij}$ is given by
 \begin{align*}
 l'_{11} = - Q^{2} \Delta, \quad l'_{12} = 0, \quad l'_{13} = A_{k1} \pa_{l},  \\
 l'_{21} = 0, \quad l'_{22} = - Q^{2} \Delta, \quad l'_{23} = A_{k2} \pa_{l},  \\
 l'_{31} = A_{k1} \pa_{k}, \quad l'_{32} = A_{k2} \pa_{k}, \quad l'_{33} = 0,  \\
\end{align*}

 where $l'_{ij}$ is defined in pages 38-39 of \cite{Agmon-Douglis-Nirenberg:estimates-boundary-elliptic-pde-II}. Also
\begin{align*}
B_{2l} & = (t^{i} A_{il} n_{j} + t^{j} A_{il} n_i) \pa_{j}, \quad l = 1,2 \\
B_{23} & = 0
\end{align*}
and the indices $r_1=-1$ and $r_2=-1$. With this choice $B'_{ij}=B_{ij}$ (see \cite[p.42]{Agmon-Douglis-Nirenberg:estimates-boundary-elliptic-pde-II}).

We can write
\begin{align*}
 (l'(x,\xi))_{ij} =
\left(
\begin{array}{ccc}
 -Q^{2}(\xi_{1}^{2} + \xi_{2}^{2}) & 0 & A_{k1} \xi_{k} \\
 0 & -Q^{2}(\xi_{1}^{2} + \xi_{2}^{2}) & A_{k2} \xi_{k} \\
  A_{k1} \xi_{k} & A_{k2} \xi_{k} & 0\\
\end{array}
\right)
\end{align*}
and
\begin{align*}
 B(x,\xi)_{ij} = \left(
\begin{array}{ccc}
-2n^{i} A_{i1}(n \cdot \xi) & - 2n^{i} A_{i2}(n\cdot \xi) & 1 \\
 t^{i} A_{i1}(\xi \cdot n) + n^{i} A_{i1}(t\cdot \xi)& t^{i} A_{i2} (\xi \cdot n) + n^{i} A_{i2}(t \cdot \xi) & 0
\end{array}
\right)
\end{align*}

Let $L= \det(l'_{ij}(x,\xi))$, i.e.

\begin{align*}
 L(x,\xi) & = \det(l'(x,\xi))
= Q^{2}(\xi_{1}^{2} + \xi_{2}^{2})(A_{k1}\xi_{k})^{2} + Q^{2}(\xi_{1}^{2} + \xi_{2}^{2})(A_{k2}\xi_{k})^{2} \\
(A_{11} \xi_{1} + A_{21} \xi_{2})^{2} & = A_{11}^{2} \xi_{1}^{2} + A_{21}^{2} \xi_{2}^{2} + 2A_{11} A_{21} \xi_{1} \xi_{2} \\
(A_{12} \xi_{1} + A_{22} \xi_{2})^{2} & = A_{12}^{2} \xi_{1}^{2} + A_{22}^{2} \xi_{2}^{2} + 2A_{12} A_{22} \xi_{1} \xi_{2} \\
\Rightarrow L(P,\xi) & = Q^{2}(\xi_{1}^{2} + \xi_{2}^{2})Q^{2}(\xi_{1}^{2} + \xi_{2}^{2}) = Q^{4}(\xi_{1}^{2} + \xi_{2}^{2})^{2} = Q^{4}|\xi|^{4}
\end{align*}

To know whether the system is uniformly elliptic, we are left to check that it satisfies
the Supplementary condition (see \cite[p.39]{Agmon-Douglis-Nirenberg:estimates-boundary-elliptic-pde-II}). The degree is 4 and therefore $m = 2$.
We need to compute the solutions of $L(P,\xi+\tau \xi') = 0$, which are the solutions to $|\xi + \tau \xi'|^{2} = 0$.

\begin{align*}
 |\xi + \tau \xi'|^{2} & = |\xi|^{2} + \tau^{2}|\xi'|^{2} + 2 \tau \xi \cdot \xi' = 0\\
\end{align*}

Solving in $\tau$ yields

\begin{align*}
 \tau = \frac{-2 \xi \cdot \xi' + \sqrt{4(\xi \cdot \xi')^{2} - 4|\xi|^{2}|\xi'|^{2}}}{2|\xi'|^{2}}.
\end{align*}

If $\xi$ and $\xi'$ are linearly independent, then the discriminant is strictly negative, which implies that there is a complex root with positive imaginary part. Since the roots have multiplicity 2, the Supplementary condition is satisfied. In addition uniform ellipticity is easy to obtain.
Next we check that the Complementing Boundary Condition is satisfied (\cite[p. 42]{Agmon-Douglis-Nirenberg:estimates-boundary-elliptic-pde-II}).

Let $t_0$ be tangential vector and $n_0$ the normal one.

Since $i$ is a double root of $L(x,t_0+\tau n)$ and $M^+(x,\xi,\tau)$ (\cite[p. 42]{Agmon-Douglis-Nirenberg:estimates-boundary-elliptic-pde-II}) is given by
\begin{align*}
 M^+(x,\xi,\tau) = (\tau - i)^{2}.
\end{align*}
In addition
\begin{align*}
 (l'(x,t_0+\tau n))_{ij} = \left(
\begin{array}{ccc}
-Q^{2}(1+\tau^{2}) & 0 & A_{k1}(t_0^{k} + \tau n^{k}) \\
0 & -Q^{2}(1+\tau^{2}) & A_{k2}(t_0^{k} + \tau n^{k}) \\
A_{k1}(t_0^{k} + \tau n^{k}) & A_{k2}(t_0^{k} + \tau n^{k}) & 0
\end{array}
\right)
\end{align*}
We define $L_{ij}$ as in \cite[p.42]{Agmon-Douglis-Nirenberg:estimates-boundary-elliptic-pde-II}, and we have that.
$L(x,t_0+\tau n) = \overline{l'(x,t+\tau n)}$.
Also, it can be computed that
\begin{align*}
 B(x,t+\tau n) =
\left(
\begin{array}{ccc}
 -2n^{i} A_{i1} \tau & -2n^{i} A_{i2} \tau & 1  \\
(n+\tau t)^{i} A_{i1} & (n+\tau t)^{i} A_{i2} & 0
\end{array}
\right)
\end{align*}

\begin{align*}
 L_{ij}(x,t_0+\tau n) = \left(
\begin{array}{ccc}
-Q^{2}(1+\overline{\tau}^{2}) & 0 & A_{k1}(t+\overline{\tau}n)^{k} \\
0 & -Q^{2}(1+\overline{\tau}^{2}) & A_{k2}(t+\overline{\tau}n)^{k} \\
A_{k1}(t+\overline{\tau}n)^{k} & A_{k2}(t+\overline{\tau}n)^{k} & 0
\end{array}
\right)
\end{align*}

\scriptsize
\begin{align*}
& (B(x,t+\tau n) L(x,t+\tau n))_{ij} \\
& =
\left(
\begin{array}{ccc}
 2n_{k} A_{k1} \tau Q^{2}(1 + \overline{\tau}^{2}) + A_{k1}(t+\overline{\tau}n)^{k} & 2n_{k} A_{k2} \tau Q^{2}(1+\overline{\tau}^{2}) + A_{k2} (t+\overline{\tau}n)^{k}
& n^{i}A_{i1} \tau A_{k1}(t + \overline{\tau} n)^{k} + \tau n^{i} A_{i2}A_{k2}(t+\overline{\tau}n)^{k} \\
-(n+\tau t)^{k}A_{k1} Q^{2}(1+\overline{\tau}^{2}) & -(n+\tau t)^{i} A_{i2} Q^{2} (1+ \overline{\tau}^{2}) & (n + \tau t)^{i} A_{i1} A_{k1} (t+\overline{\tau}n)^{k} + (n+\tau t)^{i} A_{i2} A_{k2}(t+\overline{\tau}n)^{k}
\end{array}
\right)
\end{align*}

\normalsize

If the rows of $(BL)_{ij}$ are linearly independent modulo   $(\tau-i)^2$, that means that the condition
\begin{align*}
C_h B_{hj}L_{jk}= 0 \quad \mod \,\ M^+
\end{align*}
implies that  all $C_h=0$ (\cite[p.43]{Agmon-Douglis-Nirenberg:estimates-boundary-elliptic-pde-II}). If this condition is satisfied in particular
\begin{align*}
C_h B_{hj}L_{jk}|_{\tau=i}= 0.
\end{align*}
Then $c_3$ must be zero and   the following system of equations must be satisfied:
\begin{align*}
c_1 A_{k1} t^{k} + c_2 A_{k2} t^{k} & = 0 \\
c_1 A_{k1} n^{k} + c_2 A_{k2} n^{k} & = 0
\end{align*}

In matrix form:

\begin{align*}
 \left(
\begin{array}{cc}
 A_{k1} t^{k} & A_{k2} t^{k}\\
A_{k1} n^{k} & A_{k2} n^{k}
\end{array}
\right)
\left(
\begin{array}{c}
 c_1 \\
c_2
\end{array}
\right)
 = \left(
\begin{array}{c}
 0 \\
0
\end{array}
\right)
\end{align*}
But the determinant of this matrix satisfies
\begin{align*}
 A_{k1} t^{k} A_{k2} n^{k} - A_{k2} t^{k} A_{k1} n^{k} \\
(A^{*}t)^{1}(A^{*}n)^{2} - (A^{*}t)^{2}(A^{*}n)^{1} \\
(A^{*} t, JA^{*}n) = (t,AJA^{*}n) \neq 0,
\end{align*}
and therefore $c_1=c_2=0$. Thus we have checked that the complementing Boundary Condition is satisfied.

Finally we notice that our system can be written as in \cite[p. 71]{Agmon-Douglis-Nirenberg:estimates-boundary-elliptic-pde-II}, where the coefficients $a_{ij,P}$ are smooth. In our case the index $l_1$ in \cite[p. 77]{Agmon-Douglis-Nirenberg:estimates-boundary-elliptic-pde-II} is $l_1=0$. The index $l$ in \cite{Agmon-Douglis-Nirenberg:estimates-boundary-elliptic-pde-II} coincides with $s-1$. That means $l=0,1,2$. The regularity we ask for the coefficient $b_{hj,\sigma}$ is $C^{l-r_h}$ (cf. \cite[p. 77]{Agmon-Douglis-Nirenberg:estimates-boundary-elliptic-pde-II}). The most we need is therefore $b_{hj,\sigma}\in C^3$. Since these coefficients are one derivative less regular that the boundary, a $C^4$ boundary is enough for our purpose.  This fact finishes the proof of the inequality \eqref{restimate}  if $s$ is an integer. For the rest of the values we proceed by interpolation.

%
%
%
%
%
%
%
%
%
%

This concludes the proof of lemma \ref{lema33}.
\end{proof}

Once we have studied the operator $S_A$ we will solve the time evolution. First we will show the following lemma.
\begin{lemma}\label{lema34}
Let $1 \leq s \leq 3$, $\lambda \in \mathbb{C}$, $\Re(\lambda) \geq 0$. Then the operator $\lambda + S_A: V^{s+1} \rightarrow RH^{s-1}$ is invertible. The inverse satisfies:
\begin{align}
 \label{angelpaco1} \|(\lambda + S_A)^{-1}Rf\|_{H^{s+1}} \leq C(\|Rf\|_{H^{s-1}} + |\lambda|^{\frac{s-1}{2}}\|Rf\|_{L^{2}})
\end{align}
\end{lemma}

\begin{proof}
 As before, we look for a weak solution of $(\lambda + S_A)v = Rf, \quad f \in L^{2}$. Therefore,
\begin{align*}
 (1+\lambda)(v,w) + \langle v,w \rangle = \int_{\Om} Rf\bar{w} \frac{dx}{Q^{2}} \quad \forall w \in RH^{1}.
\end{align*}
The solution is given by Lax-Milgram's Theorem. For $w=v$, by virtue of Korn's inequality and $\Re(\lambda) \geq 0$, one obtains
\begin{align*}
 |(1+\lambda)(v,v) + \langle v,v \rangle| \geq C((1+|\lambda|)\|v\|_{L^{2}}^{2} + \|\nabla v \|_{L^{2}}^{2}),
\end{align*}
with $C$ independent of $\lambda$. Easily, the following bound is obtained:
\begin{align}
\label{angelpaco2} \|v\|_{L^{2}} \leq \frac{1}{1+|\lambda|} \|Rf\|_{L^{2}}.
\end{align}
If $f \in H^{s-1}$, then Lemma \ref{lema33} gives $v = S_{A}^{-1}(Rf - \lambda v) \in H^{s+1}$, and we can get the following bounds:
\begin{align*}
 \|v\|_{H^{s+1}} & \leq C\|S_A v\|_{H^{s-1}} \leq C(\|(S_A + \lambda) v\|_{H^{s-1}} + |\lambda|\|v\|_{H^{s-1}}) \\
& \leq C\|Rf\|_{H^{s-1}} + |\lambda| \|v\|_{H^{s+1}}^{\frac{s-1}{s+1}}\|v\|_{L^{2}}^{\frac{2}{s+1}}.
\end{align*}
The case $s = 1$ is already solved using \eqref{angelpaco2}. Young's inequality provides
\begin{align*}
 \|v\|_{H^{s+1}} \leq C\|Rf\|_{H^{s-1}} + \frac12\|v\|_{H^{s+1}} + |\lambda|^{\frac{s+1}{2}}\|v\|_{L^{2}}.
\end{align*}
and we can get \eqref{angelpaco1} using \eqref{angelpaco2}.

\end{proof}

In order to find the solution of $$v_t + S_A(v) = Rf \quad \text{ in } V^{s+1},$$ we take Fourier transforms in time. Since $f \in H_{(0)}^{ht,s-1}, Rf(0) = 0$, we can extend $Rf$ to a function $\overline{Rf}$ defined in $H^{ht,s-1}(\Omega \times \mathbb{R})$, with $\overline{Rf}(t) = 0$ for all $t < 0$. Since $\frac{s-1}{2} \leq \frac{3}{4} < 1$, using Lemma \ref{lema22} (ii) we get
\begin{align*}
 \|\overline{Rf}\|_{H_{(0)}^{ht,s-1}(\mathbb{R}\times\Omega_0)} \leq C\|Rf\|_{H_{(0)}^{ht,s-1}([0,T]\times\Omega_0)}\leq C\|f\|_{H_{(0)}^{ht,s-1}([0,T]\times\Omega_0)},
\end{align*}
with $C$ independent of $T$. We look for a solution of
\begin{align*}
 v_t + S_A(v) = \overline{Rf}, \quad \forall t \in \mathbb{R}, \quad v(0) = 0.
\end{align*}
By Fourier, $i \tau \hat{v}(\tau) + S_A(\hat{v})(\tau) = \hat{\overline{Rf}}$, and therefore the solution is given by
$ \hat{v}(\tau) = (i \tau + S_A)^{-1}\hat{\overline{Rf}}$.
Using \eqref{angelpaco1} and \eqref{angelpaco2} we can bound
\begin{align*}
 \|v\|^2_{H^{ht,s+1}(\mathbb{R}\times\Omega_0)} & = \int_{\R}(\|\hat{v}\|_{H^{s+1}}^{2}(\tau) + |\tau|^{s+1}\|\hat{v}\|_{L^{2}}^{2}(\tau)) d\tau \leq C\int_{\R}(\|\overline{Rf}\|_{H^{s-1}}^{2}(\tau) + |\tau|^{s-1}\|\overline{Rf}\|_{L^{2}}^{2}) d\tau \\
&\leq C\|\overline{Rf}\|^2_{H^{ht,s-1}} \leq C\|f\|^2_{H_{(0)}^{ht,s-1}}.
\end{align*}

Since $\overline{Rf}(t) = 0$ for every $t < 0$, $\hat{\overline{Rf}}(\tau)$ has an analytic extension in $\tau$ to $\Im(\tau) < 0$. Using Lemma \ref{lema34}, $\hat{v}(\tau)$ also has that extension. Moreover, \eqref{angelpaco2} gives $\|\hat{v}(\tau)\|_{L^{2}} \leq C\|\hat{\overline{Rf}(\tau)}\|_{L^{2}}$. Thus, Paley-Wiener provides $v(t) = 0 \; \forall t < 0$. Since $v \in H^{\frac{s+1}{2}}(\R; L^{2})$ and $ \frac32<\frac{s+1}{2}$ we have continuity in time and hence $v(0) = 0$. Since $\hat{v} \in L^{2}(\R; V^{s+1}(\Om))$ we have that $v \in L^{2}([0,T]; V^{s+1}(\Om))$ and therefore $Tr(\nabla v A) = 0$ and $(q + \nabla v A + (\nabla v A)^{*})A^{-1}n = 0$. We have already solved
\begin{align*}
 v_t - Q^{2}\Delta v+ A^{*} \nabla q_1 = Rf,
\end{align*}
 where $q_1$ satisfies
\begin{align*}
 Q^{2} \Delta q_1 & = 0, \text{ in } \Om_0\times [0,T] \\
q_1 & = A^{-1}n(\nabla v A + (\nabla v A)^{*}) A^{-1}n, \text{ on } \partial\Om \times[0,T].
\end{align*}
We have that $\pa_tv(0)=0$ and then  $||v||_{H_{(0)}^{ht,s+1}}\leq C||f||_{H_{(0)}^{ht,s-1}}$.
Definition
\begin{align*}
 A^{*} \nabla q =(I-R)f + A^{*} \nabla q_1,
\end{align*}
gives us the solution that we were looking for
\begin{align*}
 v_t - Q^{2} \Delta v + A^{*} \nabla q = f.
\end{align*}
The properties of $R$ allow us to obtain
\begin{align*}
\|\nabla q\|_{H_{(0)}^{ht,s-1}} & \leq C\|A^{*} \nabla q\|_{H_{(0)}^{ht,s-1}} \leq \|(I-R)f\|_{H_{(0)}^{ht,s-1}} + \|A^{*} \nabla q_1\|_{H_{(0)}^{ht,s-1}}\\
& \leq C( \|f\|_{H_{(0)}^{ht,s-1}} + \|A^{*} \nabla q_1\|_{H_{(0)}^{ht,s-1}}).
\end{align*}
We have the following bounds:
\begin{align*}
 \|A^{*}\nabla q_1\|_{L^{2}H^{s-1}} & \leq \|\nabla q_1\|_{L^{2}H^{s-1}} \leq |A^{-1}n ((\nabla v A)+(\nabla v A)^{*})A^{-1}n|_{L^{2}H^{s-\frac12}}  \\
&
\leq \|(\nabla v A) + (\nabla v A)^{*}\|_{L^{2}H^{s}}\leq \|\nabla v\|_{L^{2}H^{s}}
\leq \|v\|_{L^{2}([0,T]; H^{s+1})},
\end{align*}
\begin{align*}
 \|A^{*}\nabla q_1\|_{H_{(0)}^{\frac{s-1}{2}}L^{2}}\leq |A^{-1}n ((\nabla v A)+(\nabla v A)^{*})A^{-1}n|_{H_{(0)}^{\frac{s-1}{2}}H^{\frac12}}.
\end{align*}
Decomposing $\nabla v_i$ into the tangential and normal components: $\nabla v_i = (\nabla v_i \cdot n_0) n_0 + (\nabla v_i \cdot t_0) t_0$, we can bound each of them by
\begin{align*}
 |\nabla v_i \cdot t_0|_{H_{(0)}^{\frac{s-1}{2}}H^{\frac12}} & \leq C|v_i|_{H_{(0)}^{\frac{s-1}{2}} H^{\frac{3}{2}}}
\leq C\|v_i\|_{H_{(0)}^{\frac{s-1}{2}}H^{2}} \leq C\|v_i\|_{H_{(0)}^{ht,s+1}}, \\
 |\nabla v_i \cdot n_0|_{H_{(0)}^{\frac{s-1}{2}}H^{\frac12}} & \leq |\nabla v_i \cdot n_0|_{H_{(0)}^{ht,s-\frac12}} \leq C\|v_i\|_{H_{(0)}^{ht,s+1}}.
\end{align*}
Above we have used Lemma \ref{lema21} (i). This yields
\begin{align*}
 \|A^{*} \nabla q_1\|_{H_{(0)}^{\frac{s-1}{2}}L^{2}} \leq C\|v\|_{H_{(0)}^{ht,s+1}}.
\end{align*}
and therefore
\begin{align*}
 \|A^{*} \nabla q\|_{H_{(0)}^{ht,s-1}} \leq C\|f\|_{H_{(0)}^{ht,s-1}}.
\end{align*}
Since $(I-R) f |_{\pa \Om_0} = 0$, we have that $q = q_1$ on $\pa \Om_0$. This implies
\begin{align*}
 |q|_{H_{(0)}^{ht,s-\frac12}} & = |q_1|_{H_{(0)}^{ht,s-\frac12}} = |A^{-1}n((\nabla v A) + (\nabla v A)^{*} )A^{-1}n|_{H_{(0)}^{ht,s-\frac12}} \\
& \leq |A^{-1}n((\nabla v A) + (\nabla v A)^{*} )A^{-1}n|_{L^{2}H^{s-\frac12}} + |A^{-1}n((\nabla v A) + (\nabla v A)^{*} )A^{-1}n|_{H_{(0)}^{\frac{s}{2}-\frac{1}{4}}L^{2}}.
\end{align*}
Decomposing into the tangential and normal components as before, we find finally
\begin{align*}
|q|_{H_{(0)}^{ht,s-\frac12}} & \leq \|\nabla v\|_{L^{2}H^{s}} + |\nabla v_i \cdot t_0|_{H_{(0)}^{\frac{s}{2}-\frac{1}{4}}L^{2}} + |\nabla v_i \cdot n_0|_{H_{(0)}^{\frac{s}{2}-\frac{1}{4}}L^{2}} \\
& \leq \|v\|_{L^{2}H^{s+1}} + |v_i|_{H_{(0)}^{\frac{s}{2}-\frac{1}{4}}H^{1}} + |\nabla v_i \cdot n_0|_{H_{(0)}^{ht,s-\frac12}}\\
& \leq \|v\|_{H_{(0)}^{ht,s+1}} + \|v\|_{H_{(0)}^{\frac{s}{2}-\frac{1}{4}}H^{1+\frac12}} + \|v\|_{H_{(0)}^{ht,s+1}} \leq \|v\|_{H_{(0)}^{ht,s+1}}.
\end{align*}

\subsection{Reduction for arbitrary $g$ and $h$}\label{reduction}

In this section we want to extend the result with $g = 0$ and $h = 0$  to the  case:
\begin{align}
v_t - \nu Q^{2} \Delta v + A^{*} \nabla q & = f & \text{ in } \Omega_0 \times [0,T] \nonumber \\
Tr(\nabla v A) & = g & \text{ in } \Omega_0 \times [0,T] \nonumber \\
(q + (\nabla v A) + (\nabla v A)^{*})A^{-1} n & = h & \text{ on } \pa \Omega_0 \times [0,T] \nonumber \\
 v(x,0) & = 0 & \text{ in } \Omega_0
\end{align}
To get the space where $g$ belongs we proceed as follows. For $\phi \in H^{1}_{0}$ it is easy to find
$$
\int \pa_t^{j}(Tr(\nabla v A)) \phi(x) \frac{dx}{Q^{2}(x)}  = \int Tr(\nabla \pa_{t}^{j} v A) \phi(x) \frac{dx}{Q^{2}(x)}
= \int \pa_{t}^{j} v A^* \nabla \phi \frac{dx}{Q^{2}(x)}.
$$
Then
$$
\left|\int \pa_t^{j}(Tr(\nabla v A)) \phi(x) \frac{dx}{Q^{2}(x)}\right|  \leq \|\pa_{t}^{j} v\|_{L^{2}}(t)\|\nabla \phi\|_{H^{1}},
$$
and duality provides
$$
\|\pa_{t}^{j}(Tr(\nabla v A))\|_{H^{-1}}   \leq \|\pa_{t}^{j} v\|_{L^{2}}(t), \text{ with } H^{-1} = (H^{1}_{0})^*.
$$
For $j = \frac{s+1}{2}$, integration in time yields
\begin{align*}
 \|Tr(\nabla v A)\|_{H^{\frac{s+1}{2}}H^{-1}} \leq \|v\|_{H^{ht,s+1}}.
\end{align*}
Here we remark that we use \eqref{norma1} for the norm of fractional derivatives on time. Also, $Tr(\nabla v A) \in L^{2}([0,T]; H^{s})$, which implies:

\begin{align*}
 Tr(\nabla v A) \in L^{2}([0,T]; H^{s}) \cap H^{\frac{s+1}{2}}([0,T]; H^{-1}),
\end{align*}
 To prove this fact, one can proceed for an integer number of derivatives, then interpolate for fractional ones (see \cite{Lions-Magenes:non-homogeneous-bvp-I}).

We check next the compatibility conditions of the initial data:
\begin{align}
 Tr(\nabla v_0 A) & = g(0) & \text{ in } \Omega_0 \label{paco1} \\
(A^{-1} n)^{\perp} (\nabla v_0 A + (\nabla v_0 A)^{*}) A^{-1} n & = h(0) (A^{-1} n)^{\perp} & \text{ on } \partial\Omega_0\label{paco2}.
\end{align}

Defining the following spaces:
\begin{align*}
 X_0 & := \{(v,q) : v \in H_{(0)}^{ht,s+1}, q \in H^{ht,s}_{pr,\,(0)}\} \\
Y_0 & := \{(f,g,h): f \in H_{(0)}^{ht,s-1}, g \in \overline{H}_{(0)}^{ht,s}, h \in H_{(0)}^{ht,s-\frac12}(\pa \Om \times [0,T]), \}
\end{align*}
(here we remark that in $X_0$ and $Y_0$: $v(0)=\pa_tv(0)=q(0)=f(0)=g(0)=\pa_tg(0)=h(0)=0$), then, for $(f,g,h)\in Y_0$, \eqref{plg} is equivalent to solve:
\begin{align*}
 L(v,q) = (f,g,h,0); \quad L: X_0 \rightarrow Y_0, \quad 2 < s < \frac{5}{2}.
\end{align*}

\begin{thm}
 $L: X_0 \to Y_0$ is invertible for $2<s< \frac{5}{2}$. Moreover, $\|L^{-1}\|$ is bounded uniformly if $T$ is bounded above.
\end{thm}

\begin{proofthm}{Lmenos1}
 We structure the proof in 3 steps, in order to get to the previous case ($g = h = v_0 = 0$).
Let $(f,g,h) \in Y_0$.
\begin{enumerate}
 \item[Step 1:]\underline{$A$-divergence adjustment.} We want to find $(v^{1},q^{1}) \in X$ such that $L(v^{1},q^{1}) = (f^{1},g^{1},h^{1},0)$ with
\begin{align}
\label{paco6} f^{1}(0) = 0; \quad g^{1}(t) = g(t);  \; \forall t \in [0,T]  \quad h^{1}(0) = 0.
\end{align}
We define $\phi$ by solving the following elliptic problem for every $t \in (-\infty,\infty)$, after extending $g$ to the whole real line:
\begin{align*}
 Q^{2} \Delta \phi  &= g(t) \qquad \text{ in } \Omega_0\times\R \\
 \phi  &= 0 \qquad \text{ on } \pa \Omega_0\times\R \\
\end{align*}
This system satisfies $\|\phi\|_{H^{s+2}}(t) \leq C\|g\|_{H^{s}}(t)$. In particular
\begin{align}
 \label{paco7} \|\nabla \phi \|_{H^{s+1}}(t) \leq C\|g\|_{H^{s}}(t).
\end{align}
Taking the Fourier Transform in time, one has, for every $\tau$:
\begin{align*}
 Q^{2} \Delta \hat{\phi}(\tau) & = \hat{g}(\tau) \qquad \text{ in } \Omega_0\times\R, \\
 \hat{\phi}(\tau) & = 0 \qquad \text{ on } \pa \Omega_0\times\R.
\end{align*}
For $\lambda \in H^{1}_{0}$ it is possible to find
\begin{align*}
 \int \hat{\overline{g}}(\tau) \lambda \frac{dx}{Q^{2}}=\int\Delta \hat{\phi}(\tau) \lambda dx = \int \nabla(\nabla \hat{\phi} A)A \lambda \frac{dx}{Q^{2}} = - \int \nabla \hat{\phi} A \underbrace{\nabla \lambda A}_{A^{*} \nabla \lambda} \frac{dx}{Q^{2}} = - \int \nabla \hat{\phi} \nabla \lambda dx
\end{align*}
Therefore:
\begin{align*}
 \int \nabla \hat{\phi} \nabla \lambda dx = -\int \hat{g}(\tau) \lambda \frac{dx}{Q^{2}} \quad \forall \lambda \in H^{1}_{0}.
\end{align*}

This implies:

\begin{align*}
 \|\nabla \hat{\phi}\|_{L^{2}}(\tau) & \leq C\|\hat{g}\|_{H^{-1}}(\tau) \\
\Rightarrow \int |\tau|^{\frac{s+1}{2}}\|\nabla \hat{\phi}\|_{L^{2}}(\tau) d\tau & \leq C \int |\tau|^{\frac{s+1}{2}} \|\hat{g}\|_{H^{-1}}(\tau) d\tau.
\end{align*}

We can conclude that $\nabla \phi \in H_{(0)}^{ht,s+1}$.  We now define
\begin{align*}
 v^{1} & =  A^{*} \nabla \phi \in H_{(0)}^{ht,s+1}, \quad q^{1} = q.
\end{align*}
It is easy to check that
\begin{align*}
f^{1}(0) & =  A^{*} \nabla \phi_t(0)=0.
\end{align*}

By construction, it is obvious that $g^{1}(t) = g(t)$ for every $t \in [0,T]$. Since $\phi(0) = 0$, $h^{1}(0) = 0$. This shows that \eqref{paco6} is satisfied.

  \item[Step 2:] \underline{Adjusting the boundary conditions in the tangential direction without modifying the A-divergence}.

We want to find $(v^{2},q^{2}) \in X_0$ such that $L(v^{2},q^{2}) = (f^{2},g^{2},h^{2}, 0)$ with
\begin{align}
&f^{2}(0) = 0; \quad g^{2}(t) = g(t);  \; \forall t \in [0,T]  \nonumber \\
&\label{paco8}  h^{2}(t) \cdot (A^{-1} n)^{\perp} = h(t) \cdot (A^{-1} n)^{\perp};  \; \forall t \in [0,T], \quad h^{2}(0) = 0.
\end{align}
We will use the following Lemma:

\begin{lemma}
 \label{lemapaco42}
Let $\eta \in H_{(0)}^{ht,s-\frac12}(\pa \Om_0 \times [0,T])$, $2 < s < \frac{5}{2}$, $\eta(0) = 0$. Then there exists $w \in H_{(0)}^{ht,s+1}$ such that $\|w\|_{H_{(0)}^{ht,s+1}} \leq C|\eta|_{H_{(0)}^{ht,s-\frac12}}$, $w(0) = w_t(0) = Tr(\nabla w A) = 0$ and
$$(A^{-1}n)^{\perp}(\nabla w A + (\nabla w A)^{*})A^{-1}n = \eta \text{ on } \pa \Om_0 \times[0,T].$$
\end{lemma}

\begin{proof}

Let $\psi \in H_{(0)}^{ht,s+2}$, with $\psi(0) = \psi_t(0) = 0, \psi(x) = \pa_{n} \psi(x) = 0, \; \forall x \in \pa \Omega_0$, and moreover $\pa_{n}^{2} \psi(x) = \eta(x) \; \forall x \in \pa \Om_0\times[0,T]$. This choice is possible because of the parabolic trace. All that is left is to check that the compatibility conditions from Lemma \ref{lema21}. Defining

\begin{align*}
 w = \nabla_{A}^{\perp} \psi = (- \pa_2 P_1 \circ P^{-1} \pa_1 \psi - \pa_2 P_2 \circ P^{-1} \pa_2 \psi, \pa_1 P_1 \circ P^{-1} \pa_1 \psi + \pa_1 P_2 \circ P^{-1} \pa_2 \psi),
\end{align*}

it is immediate that $w(0) = w_t(0) = 0$. A straightforward, but long calculation gives:

\begin{align*}
 Tr(\nabla (\nabla_{A}^{\perp} \psi) A) = 0 = Tr(\nabla w A).
\end{align*}

We will now show that $(A^{-1}n)^{\perp}(\nabla w A + (\nabla w A)^{*})(A^{-1}n) = \eta$. Let $x_0 \in \pa \Om_0$. We perform an euclidean change of coordinates in a way that $x_0 = 0, n_0 = (0,1)$. Thus, $\psi(0,0) = 0$, as well as $\pa_1 \psi(0,0)$ and $\pa_1^{2} \psi(0,0)$. The condition $\pa_n \psi(0,0) = 0$ implies $\pa_2 \psi(0,0) = 0, \pa_1 \pa_2 \psi(0,0) = 0$. That gives:

\begin{align*}
 \nabla (\nabla_A^{\perp} \psi) = \left(
\begin{array}{cc}
 0 & - \pa_2 P_2 \circ P^{-1} \pa_2^{2} \psi(0,0) \\
0 & \pa_1 P_2 \circ P^{-1} \pa_2^{2} \psi(0,0)
\end{array}\right).
\end{align*}

Computing further:

\begin{align*}
 \nabla (\nabla_A^{\perp} \psi)A = \pa_{2}^{2} \psi(0,0)
\left(
\begin{array}{cc}
 -\pa_2 P_2 \pa_1 P_2 & -(\pa_2 P_2)^2 \\
(\pa_1 P_2)^2 & \pa_1 P_2 \pa_2 P_2
\end{array}\right) \circ P^{-1}
\end{align*}

Thus

\begin{align*}
 ( (\nabla (\nabla_A^{\perp} \psi)A) + ( \nabla (\nabla_A^{\perp} \psi)A)^{*}) =
\pa_{2}^{2} \psi(0,0)
\left(
\begin{array}{cc}
 -2\pa_2 P_2 \pa_1 P_2 & (\pa_1 P_2)^2 -(\pa_2 P_2)^2 \\
(\pa_1 P_2)^2 -(\pa_2 P_2)^2 & 2\pa_1 P_2 \pa_2 P_2
\end{array}\right) \circ P^{-1}.
\end{align*}
We also have that
\begin{align*}
 A^{-1}n_0 = \frac{1}{Q^{2}} \left(\begin{array}{c}-\pa_2 P_1 \\ \pa_1 P_1\end{array}\right)\circ P^{-1}; \quad
(A^{-1}n_0)^{\perp} = \frac{1}{Q^{2}} \left(\begin{array}{c}-\pa_1 P_1 \\ -\pa_2 P_1\end{array}\right)\circ P^{-1};
\end{align*}
Combining everything, we get

\begin{align*}
  ( (\nabla (\nabla_A^{\perp} \psi)A) + ( \nabla (\nabla_A^{\perp} \psi)A)^{*})A^{-1} n_0 & =  \frac{\pa_2 \psi(0,0)}{Q^{2}}(-\pa_1 P_1\circ P^{-1} Q^{2}, \pa_1 P_2\circ P^{-1} Q^{2}) \\
  \Rightarrow (A^{-1}n)^{\perp}( (\nabla (\nabla_A^{\perp} \psi)A) + ( \nabla (\nabla_A^{\perp} \psi)A)^{*})A^{-1} n_0 & =  \pa_2^{2} \psi(0,0)
\end{align*}
Since $ \pa_n^{2} \psi(0,0) = \sum_{i,j}(n_0)_i \pa_i \pa_j \psi (n_0)_j = \pa_2^{2} \psi(0,0),$ we are done.
\end{proof}

 We now apply Lemma \ref{lemapaco42} to $\eta = h(t)(A^{-1}n)^{\perp} - h^{1}(t) (A^{-1}n)^{\perp}$. Equation \eqref{paco6} shows that $\eta(0) = 0$. Then, we can take:

\begin{align*}
 v^{2} = v^{1} + w, \quad q^{2} = q^{1}.
\end{align*}

Since $w(0) = w_t(0) = 0$, we have that $f^2(0)=0$. Also, since $Tr(\nabla w A) = 0$, by \eqref{paco6}, $g^{2}(t) = g(t) \; \forall t \in [0,T]$. By construction, one has that $h^{2}(t) (A^{-1}n)^{\perp} = h(t)(A^{-1}n)^{\perp} $ and since $\eta(0)=0$ we have that $h^2(0)=0$.

  \item[Step 3:] \underline{$h(t) = h^{3}(t)\, \forall t$ without modifying the rest.}
We want to have $(v^{3},q^{3}) \in X_0$.  We have that $L(v^{3},q^{3}) = (f^{3},g^{3},h^{3},0)$ with

\begin{align}
f^{3}(0) = 0; \quad g^{3}(t) = g(t);  \; \forall t \in [0,T]  \nonumber \\
\label{paco9} \quad h^{3}(t) = h(t);  \; \forall t \in [0,T]   .
\end{align}

We first take $v^{3} = v^{2}$ and define $\overline{q}$ by
\begin{align*}
 \overline{q} &= h(t) \cdot \frac{A^{-1}n_0}{|A^{-1}n_0|^{2}} - \underbrace{(q^{2}I + (\nabla v^{2}A) + (\nabla v^{2}A)^{*})A^{-1}n_0}_{h^{2}(t)} \cdot \frac{A^{-1}n_0}{|A^{-1}n_0|^{2}}
\text{ on } \pa \Om_0 \times [0,T],\\
\overline{q}(x,0)&=0,\quad \text{ in }  \Om_0.
\end{align*}

Using once again the parabolic trace, the compatibility condition is satisfied since $\overline{q}(x,0) = 0$ in $\pa \Om \times [0,T]$ by means of \eqref{paco8}. Therefore, we take $q^{3} = q^{2} + \overline{q}$.

Since $\nabla \overline{q}(x,0) = 0$ we have as before $f^{3}(0) = 0$, and $Tr(\nabla v^{3} A) = 0$ because the velocity was not modified. At the boundary we find:
\begin{align*}
 ((q^{3}I + (\nabla v^{3}A) + (\nabla v^{3}A)^{*})A^{-1}n
& = \overline{q}A^{-1}n + (q^{2}I + (\nabla v^{2}A) + (\nabla v^{2}A)^{*}) A^{-1}n\\
 = \overline{q}A^{-1}n + h^{2}(t)&  =  h(t) \cdot \frac{A^{-1}n}{|A^{-1}n|}\frac{A^{-1}n}{|A^{-1}n|} - h^{2}(t) \cdot \frac{A^{-1}n}{|A^{-1}n|}\frac{A^{-1}n}{|A^{-1}n|} + h^{2}(t)\\
& = h(t) \cdot \frac{A^{-1}n}{|A^{-1}n|}\frac{A^{-1}n}{|A^{-1}n|} + h^{2}(t) \cdot \frac{(A^{-1}n)^{\perp}}{|A^{-1}n|}\frac{(A^{-1}n)^{\perp}}{|A^{-1}n|}\\
& \stackrel{\eqref{paco8}}{=}   h(t) \cdot \frac{A^{-1}n}{|A^{-1}n|}\frac{A^{-1}n}{|A^{-1}n|} + h(t) \cdot \frac{(A^{-1}n)^{\perp}}{|A^{-1}n|}\frac{(A^{-1}n)^{\perp}}{|A^{-1}n|} = h(t).
\end{align*}
We find $v^{3}(0) = 0$ trivially. By construction, we conclude that
\begin{align*}
 \|(v^{3},q^{3})\|_{X_0} \leq C \|(f,g,h)\|_{Y_0}.
\end{align*}

If we consider the variables $v_{\tilde{f}} = v - v^{3}, q_{\tilde{f}} = q - q^{3}$, we obtain the following problem:

\begin{align}
\pa_t v_{\tilde{f}} - Q^{2} \Delta v_{\tilde{f}} + A^{*} \nabla q_{\tilde{f}} & = \tilde{f} & \text{ in } \Omega \times [0,T] \nonumber \\
Tr(\nabla v_{\tilde{f}} A) & = 0 & \text{ in } \Omega \times [0,T] \nonumber \\
(q_{\tilde{f}} + (\nabla v_{\tilde{f}} A) + (\nabla v_{\tilde{f}} A)^{*})A^{-1} n & = 0 & \text{ in } \pa \Omega \times [0,T] \nonumber \\
 v_{\tilde{f}}(x,0) & = 0 & \text{ in } \Omega,
\end{align}
with $\tilde{f}(0) = 0$. Using Theorem \ref{thm3.2}, the Theorem is proved.

\end{enumerate}
\end{proofthm}

\section{Proofs of Structural Stability Theorems}

\subsection{Proof of Proposition \ref{prop1}}
\label{appendixprop1}

Part 1:

First we notice that $||\int_{0}^t\tau\psi||_{F^{s+1}}=\frac{1}{2}\fn{t^2\psi}$
thus we need to control $\lit{t^2\psi}{s+1}\leq C[v_0]T^{1.75}$ and $\hhn{t^2\psi}{2}{\gamma}\leq C[v_0]||t^2||_{H_{(0)}^2}\leq C[v_0]$.

In addition, we have that $$f_\phi\equiv -\pa_t \phi+Q^2\Delta \phi -A^*\nabla q_{\phi}=tQ^2\Delta\p{Q^2\Delta v_0-A^*\nabla q_{\phi}}\equiv tQ^2\Delta \psi.$$
Therefore $f_\phi|_{t=0}=0$. By definition \eqref{q0b} of $q_{\phi}$ we have that $ A^{-1} n_0\cdot h_{\phi}|_{t=0}=0$ and the tangential component of $h_{\phi}|_{t=0}=0$ by the choice of the initial data. Then $h_\phi|_{t=0}=0$ and in fact $h_\phi=O(t)$ when $t\to 0$ . Also $g_\phi|=O(t^2)$, when $t\to 0$ by the choice of the initial data (incompressibility condition).  Then we can apply theorem \ref{Lmenos1} to obtain that $\left|\left|L^{-1}(f_\phi,\overline{g}_\phi,h_\phi)\right|\right|_{H_{(0)}^{ht,s+1}\times H^{ht,s}_{pr \, (0)}}$ is bounded independently of $T$ for $T$ small if the norms $||f_\phi||_{H_{(0)}^{ht,s-1}}$, $||g||_{\overline{H}_{(0)}^{ht,s}}$ and $ |h|_{H_{(0)}^{ht,s-\frac{1}{2}}}$ are bounded independently of $T$ for $T$ small. Since $||t^n||_{H_{(0)}^{\frac{s-1}{2}}}\leq C[n]$, for $n=1,2,...$, we see that $f_\phi$ is under control. Since $||t^n||_{H_{(0)}^{\frac{s+1}{2}}}\leq C[n]$, for $n=2,...$, $g_\phi$ is under control. Finally $h_\phi$ is under control because $||t^n||_{H_{(0)}^{\frac{s}{2}-\frac{1}{4}}}\leq C[n]$, for $n=1,2,...$.

Therefore $N$ is a number that does not depend on $T$ for $T$ small. Indeed $N\leq C[v_0]$.

The definition of $F^{s+1}$ and $\vn$ imply
\begin{align*}
&\fn{X^{(n+1)}-\alpha-\int_{0}^tA\phi}\leq \fn{\int_{0}^tA\circ \xn \wn}+\fn{\int_{0}^t\left(A\circ \xn - A\right)\phi d\tau}\\
&\leq \fn{\int_{0}^tA\circ \xn \wn}+\fn{\int_{0}^t\left(A\circ \xn - A\right)v_0 d\tau}+\fn{\int_{0}^t\left(A\circ \xn - A\right)\tau\psi d\tau}\\
& \leq \lit{\int_{0}^tA\circ \xn \wn}{s+1}+\hhn{\int_{0}^tA\circ \xn \wn}{2}{\gamma}\\
& + \lit{\int_{0}^t\left(A\circ \xn - A\right)v_0 d\tau}{s+1}+\hhn{\int_{0}^t\left(A\circ \xn - A\right)v_0 d\tau}{2}{\gamma}\\
& + \lit{\int_{0}^t\left(A\circ \xn - A\right)\tau\psi d\tau}{s+1}+\hhn{\int_{0}^t\left(A\circ \xn - A\right)\tau\psi d\tau}{2}{\gamma}\\
& \equiv I_{1}+I_{2}+I_{3}+I_{4}+I_{5}+I_6.
\end{align*}
To estimate $I_1$ we proceed as follows. Applying H\"{o}lder and Minkowski inequalities yields
\begin{align*}
\left|\left|\int_{0}^t A\circ\xn \wn d\tau\right|\right|_{L^{\infty}_{1/4}H^{s+1}}\leq  T^\frac{1}{4} \lhn{A\circ \xn \wn}{s+1}\leq T^\frac{1}{4}\li{A\circ \xn}{s+1}\hhn{\wn}{0}{s+1},
\end{align*}
In order to bound $\li{A\circ \xn}{s+1}$ we will use lemma \ref{composicion1}.
Therefore
\begin{align*}
I_1\leq T^\frac{1}{4}C[v_0]\lhn{\wn}{s+1}\leq T^\frac{1}{4}C[v_0]||\wn||_{H_{(0)}^{ht,s+1}}\leq T^\frac{1}{4}C[v_0].
\end{align*}

For $I_2$  we have that
\begin{align*}
&I_2\leq \hhn{\int_{0}^t A\circ \xn \wn}{2}{\gamma}\leq C \hhn{A\circ \xn \wn}{1}{\gamma}\\
&\leq C\hhn{\p{A\circ \xn-A}\wn}{1}{\gamma}
+\hhn{A\wn}{1}{\gamma}\\&\leq C \p{\hhn{A\circ\xn-A}{1}{\gamma}\hhn{\wn}{1}{\gamma}+\hhn{A\wn}{1}{\gamma}}\\
\end{align*}
By lemma \ref{composicion2}, we have that $I_2\leq C[v_0]\hhn{w}{1}{\gamma}$.  In addition
\begin{align*}
&\hhn{ \wn}{1}{\gamma}\leq \hhn{\int_{0}^t \pa_t \wn}{1}{\gamma}\leq C\hhn{\int_{0}^t \pa_t \wn}{1+\delta-\ep}{\gamma}  \leq CT^\ep \hhn{\wn}{1+\delta}{\gamma}\leq CT^\ep C[v_0].
\end{align*}
Here $\ep>0$, $\ep<\delta$, $\delta < \frac{s-1-\gamma}{2}$ and we have used lemma \ref{2.3} and \ref{2.4}.

For $I_3$ we have that \begin{align*}
&\lit{\int_{0}^t \p{A\circ X-A}v_0d\tau}{1+s}= \sup_{t\in[0,T]}\frac{1}{t^\frac{1}{4}}\left|\left|\int_{0}^t\p{A\circ\xn-A}v_0d\tau\right|\right|_{H^{s+1}}\\
&\leq C[v_0]T^\frac{3}{4}\li{A\circ X-A}{1+s}\leq C[v_0]T^\frac{3}{4}
\end{align*}
by lemma \ref{composicion1}.

For $I_4$  by applying the second part of Lemma \ref{2.4} with $\ep=0$  we have that
\begin{align*}
I_4\leq \hhn{ (A\circ \xn-A)v_0 }{1}{\gamma}\leq C[v_0] \hhn{ (A\circ \xn-A) }{1}{\gamma}
\end{align*}

Now we apply lemma \ref{composicion2} and we obtain that, for small enough $T$, $I_4\leq C[v_0]||\xn-\alpha||_{H_{(0)}^1H^\gamma}$.  In addition
\begin{align*}
\hhn{\xn-\alpha}{1}{\gamma}\leq \hhn{\xn-\alpha-Av_0t}{1}{\gamma}+\hhn{Av_0t}{1}{\gamma}
\end{align*}
with $\hhn{Av_0t}{1}{\gamma}\leq C[v_0]||t||_{H_{(0)}^1}\leq C[v_0]T^\frac{1}{2}.$ Also
\begin{align*}
&\hhn{ \xn-\alpha-Av_0t}{1}{\gamma}=\hhn{\int_{0}^t \pa_t(\xn-\alpha-Av_0t)}{1}{\gamma}\leq \hhn{\int_{0}^t \pa_t(\xn-\alpha-Av_0t)}{1+\delta-\ep}{\gamma}\\
&\leq C T^\ep \hhn{\xn-\alpha-Av_0t}{1+\delta}{\gamma}\leq C T^\ep \hhn{\xn-\alpha-Av_0t}{1+\delta}{\gamma}\leq CT^\ep C[v_0].
\end{align*}
where we have applied lemma \ref{2.4}, for $\ep>0$ and $\ep<\delta<\frac{1}{2}$. This concludes the estimate for $I_4$. The estimates for $I_5$ and $I_6$ follow in a similar way.
Therefore
\begin{align*}
\fn{X^{(n+1)}-\alpha-\int_{0}^tA\phi}\leq  C[v_0]T^\ep
\end{align*}

This concludes the proof of part 1 of proposition \ref{prop1}.

Part 2:

From \eqref{Xtilda} we have that
\begin{align*}\fn{X^{(n+1)}-\xn}& =\fn{\int_0^t\left(A\circ \xn \vn-A\circ \xm \vm\right)d\tau}\\
&\leq \lit{\int_0^t\left(A\circ \xn \vn-A\circ \xm \vm\right)d\tau}{s+1}\\&+
\hhn{\int_0^t\left(A\circ \xn \vn-A\circ \xm \vm\right)d\tau}{2}{\gamma}\\
& \equiv I_1 +I_2
\end{align*}
In order to bound  $I_1$ we notice that
\begin{align*}
&\left|\left|\int_0^t\left(A\circ \xn \vn-A\circ \xm \vm\right)d\tau \right|\right|_{H^{s+1}}\\&\leq t^\frac{1}{2}\lhn{A\circ \xn \vn-A\circ \xm \vm d\tau}{s+1}
\end{align*}
Therefore
\begin{align*}
I_1& \leq T^\frac{1}{4}\lhn{\left(A\circ \xn \vn-A\circ \xm \vm\right)}{s+1}\\
& \leq T^\frac{1}{4}\lhn{\left(A\circ \xn-A\circ \xm\right)\vn}{s+1}\\
& + T^\frac{1}{4}\lhn{\left(\vn-\vm\right)A\circ\xm}{s+1} \\
& \equiv I_{11} + I_{12}
\end{align*}
In addition

\begin{align*}
I_{11}&\leq T^\frac{1}{4}\li{A\circ \xn-A\circ \xm}{s+1}\lhn{\vn}{s+1}.
\end{align*}
Here we notice that $\lhn{\vn}{s+1}\leq \lhn{\wn}{s+1}+\lhn{v_0}{s+1}+\lhn{t\psi}{s+1}\leq C[v_0].$ Then by applying lemma \ref{composicion3} we have that $I_{11}\leq C[v_0]T^\frac{1}{2}\fn{\xn-\xm}$

Also we have, by applying lemma \ref{composicion1}, that
\begin{align*}
I_{12}&\leq  C[v_0]T^\frac{1}{4}\lhn{\wn-\wm}{s+1}
\end{align*}
Thus $$I_1\leq C[v_0]T^\frac{1}{4}\left(\lit{\xn-\xm}{s+1}+\htn{\wn-\wm}{s+1}\right).$$
It remains to bound $I_2$.
\begin{align*}
I_2=\hhn{\int_{0}^t (A\circ \xn \vn - A\circ \xm \vm)d\tau}{2}{\gamma}
\end{align*}
and applying Lemma \ref{2.4} with $\ep=0$ we have that
\begin{align*}
I_2\leq \hhn{A\circ \xn \vn-A\circ\xm \vm}{1}{\gamma}.
\end{align*}
We will decompose that term in the following way
\begin{align*}
A\circ \xn \vn-A\circ\xm \vm= &(A\circ \xn-A\circ \xm)\wn+ (A\circ \xn-A\circ \xm)\phi\\
&+ A\circ\xm (\wn-\wm).
\end{align*}
Thus
\begin{align*}
I_2\leq I_{21}+I_{22}+I_{23}
\end{align*}
with
\begin{align*}
I_{21}=& \hhn{(A\circ\xn-A\circ\xm)\wn}{1}{\gamma}\\
I_{22}=& \hhn{(A\circ\xn-A\circ\xm)\phi}{1}{\gamma}\\
I_{23}=& \hhn{A\circ \xm(\wn-\wm)}{1}{\gamma}.
\end{align*}
For $I_{21}$, by applying Lemma \ref{2.6} ($1<\gamma<s-1$), we have that
\begin{align*}
& \hhn{(A\circ\xn-A\circ\xm)\wn}{1}{\gamma}\leq  C\hhn{A\circ\xn-A\circ\xm}{1}{\gamma}\hhn{\wn}{1}{\gamma}\\
& \leq C[v_0]\hhn{A\circ \xn-A\circ \xm}{1}{\gamma}\leq C[v_0]\hhn{\xn-\xm}{1}{\gamma},
\end{align*}
because of lemma \ref{composicion3}.
Then
\begin{align*}
I_{21}\leq & C[v_0]\hhn{\xn-\xm}{1}{\gamma}\leq \hhn{\int_0^t(\pa_t\left(\xn-\xm\right)d\tau)}{1+\delta-\ep}{\gamma}\\
& \leq C[v_0]T^\ep \hhn{\xn-\xm}{1+\delta}{\gamma},
\end{align*}
where we have applied Lemma \ref{2.4} with $\ep>0$, $\delta >\ep$ and $\delta<\frac{1}{2}$.

The term $I_{22}$ is quite similar to $I_{21}$. The only difference is that we can not take $\hhn{\phi}{1}{\gamma}$. Instead of that,
\begin{align*}
&\hhn{(A\circ \xn-A\circ \xm)\phi}{1}{\gamma}\leq C[v_0]||v_0||_{H^\gamma}\hhn{A\circ\xn-A\circ\xm}{1}{\gamma}\\
&+C[v_0]\hhn{t\psi}{1}{\gamma}\hhn{A\circ\xn-A\circ\xm}{1}{\gamma}.
\end{align*}
Finally
\begin{align*}
I_{23}\leq C\p{\hhn{A\circ\xm-A}{1}{\gamma}+1} \hhn{\wn-\wm}{1}{\gamma}\leq C[v_0]\hhn{\wn-\wm}{1}{\gamma},
\end{align*}
by applying lemma \ref{composicion1}. Also we can compute
\begin{align*}
I_{23}\leq & C[v_0] \hhn{\wn-\wm}{1+\delta-\ep}{\gamma}\leq C[v_0]T^\ep \hhn{\wn-\wm}{1+\delta}{\gamma}\\
\leq &C[v_0]T^\ep \htn{\wn-\wm}{s+1},
\end{align*}
for $\ep>0$, $\delta>\ep$ and $\delta<\frac{s-1-\gamma}{2}$.

This concludes the proof of part 2 of proposition \ref{prop1}. Therefore proposition \ref{prop1} is shown.

\subsection{Proof of Proposition \ref{prop2}}
\label{appendixprop2}

\subsection*{Proof of part 1:}

We split the proof in three parts, corresponding with the functions $f^{(n)}$, $\overline{g}^{(n)}$ and $h^{(n)}$:\\


\underline{P1.1. Estimate for $f^{(n)}$:}\\

In this section we have to deal with $f^{(n)}$ to estimate
$$
\htn{f^{(n)}}{s-1}=\lhn{f^{(n)}}{s-1}+\hln{f^{(n)}}{\frac{s-1}{2}}.
$$
We will gather terms by writing $f^{(n)}=f^{(n)}_w+f^{(n)}_\phi+f^{(n)}_{q}$ as follows,
$$
f^{(n)}_w= Q^2\circ \xn\zn\pa\left(\zn\pa\wn\right)-Q^2\Delta \wn,
$$
\begin{equation}\label{fsplit}
f^{(n)}_\phi=Q^2\circ \xn\zn\pa\left(\zn\pa\phi\right)-Q^2\Delta \phi,
\end{equation}
and
$$
f^{(n)}_{q}=-A\circ\xn \zn\pa \qn+A\pa\qn.
$$
Above we removed subscripts to alleviate notation. Next we also ignore the superscripts for the same reason. We firstly bound $\lhn{\cdot}{s-1}$:

\begin{align*}
\lhn{f_w}{s-1}\leq & \lhn{(Q^2\circ X-Q^2)\zeta\pa (\zeta \pa w)}{s-1}+\lhn{Q^2(\zeta-\I)\pa(\zeta\pa w)}{s-1}\\
&+\lhn{Q^2\pa((\zeta-\I)\pa w)}{s-1}\equiv I_{1}+I_{2}+I_{3}.
\end{align*}
We can deal with these three terms as follows. For $I_1$ one can get
\begin{align*}
I_{1}\leq&\lhn{(Q^2\circ X -Q^2)\zeta\pa\zeta\pa w}{s-1}+\lhn{(Q^2\circ X-Q^2)\zeta\zeta\pa^2 w}{s-1}\\
\leq& \li{Q^2\circ X-Q^2}{s-1}\li{\zeta}{s-1}\li{\pa \zeta}{s-1}\lhn{\pa w}{s-1}\\
&+ \li{Q^2\circ X-Q^2}{s-1}\li{\zeta}{s-1}^2\lhn{\pa^2w}{s-1},
\end{align*}
By applying lemma \ref{composicionQ1} we have that $\li{Q^2\circ X-Q^2}{s-1}\leq C[v_0]\li{ X-\alpha}{s-1}$ and by applying lemma \ref{composicionz}, $\li{\zeta}{s-1}$ and $\li{\zeta}{s-1}\leq C[v_0]$.
\begin{align*}
I_{1}\leq& C\li{Q^2\circ X-Q^2}{s-1}\li{X}{s+1}^2\htn{w}{s+1}\leq C[v_0]\li{X-\alpha}{s-1}\htn{w}{s+1}\leq C[v_0]T^\frac{1}{4}.
\end{align*}
Also, by lemma \ref{composicionz},
\begin{align*}
I_{2}\leq& \lhn{Q^2(\zeta-\I)\pa\zeta\pa w}{s-1}+\lhn{Q^2(\zeta-\I)\zeta\pa^2 w}{s-1}\\
\leq &C\li{\zeta-\I}{s-1}(\li{\pa\zeta}{s-1}\lhn{\pa w}{s-1}+\li{\zeta}{s-1}\lhn{\pa^2 w}{s-1})\\
\leq & C[v_0] \li{\zeta-\I}{s-1}\htn{w}{s+1}.
\end{align*}
The identity $
\zeta-\I=\zeta\left(\I-\nabla X\right)=\zeta\nabla\left( \alpha-X\right)
$, together with lemma \ref{composicionz},
allows us to get
\begin{align}\label{zetamenosi}
\li{\zeta-\I}{s} \leq \li{\zeta}{s}\li{X-\alpha}{s+1}\leq C[v_0]\p{\li{X-\alpha-Av_0t}{s+1}+T||Av_0||_{s+1}}\leq C[v_0]T^\frac{1}{4},
\end{align}
 Thus
$$I_{2}\leq C[v_0]T^\frac{1}{4}.$$ 
It remains $I_{3}$ for which we compute
\begin{align*}
I_{3}=&\lhn{Q^2\pa((\zeta-\I)\pa w)}{s-1}\leq \lhn{Q^2\pa\zeta\pa w}{s-1}+\lhn{Q^2(\zeta-\I)\pa^2 w}{s-1}\\
\leq & C[v_0](\li{\pa \zeta}{s-1}\lhn{w}{s}+\li{\zeta-\I}{s-1}\lhn{w}{s+1})\\
\leq & C[v_0]\li{\zeta-\I}{s}\htn{w}{s+1}\leq C[v_0]T^\frac{1}{4}.
\end{align*}
We are done with $\lhn{f_w}{s-1}$. At this point it is easy to check that an analogous procedure yields
$$
\lhn{f_\phi}{s-1}\leq C[v_0]T^\frac{1}{4}\lhn{\phi}{s+1}\leq C[v_0]T^\frac{1}{4}.
$$
We notice that in the bound of $\lhn{f_w}{s-1}$ the only bound we really need on $w$ is for $\lhn{w}{s+1}$. In addition $\lhn{\phi}{s+1}\leq C[v_0]$.

Next we deal with $\lhn{f_{q}}{s-1}$. We separate into two terms, $\lhn{f_{q_w}+f_{q_\phi}}{s-1}\leq \lhn{f_{q_w}}{s-1}+\lhn{f_{q_\phi}}{s-1}$Indeed, by lemmas \ref{composicion1}, \ref{composicionz} and expression \eqref{zetamenosi} he have that
\begin{align*}
\lhn{f_{q_w}}{s-1}\leq& \lhn{A\circ X(\zeta-\I)\pa q_w}{s-1}+\lhn{(A-A\circ X)\pa q_w}{s-1}\\
\leq& (\li{A\circ X}{s-1}\li{\zeta-\I}{s-1}+\li{A\circ X-A}{s-1})\lhn{\pa q_w}{s-1}\\
\leq& C[v_0]T^\frac{1}{4}||q_w||_{H^{ht,{s}}_{pr\, (0)}}\leq C[v_0]T^\frac{1}{4}.
\end{align*}
Here we notice that in the previous inequality the only bound we really need on $q_w$ is for the norm $\lhn{\pa q_w}{s-1}$. Therefore we have that
 \begin{align*}
\lhn{f_{q_\phi}}{s-1}\leq C[v_0]T^\frac{1}{4}\lhn{\pa q_\phi}{s-1}\leq C[v_0]T^\frac{1}{4}.
\end{align*}

Thus we are done with $\lhn{f}{s-1}$.

Using the same splitting we now deal with $\hln{f}{\frac{s-1}{2}}$:
\begin{align*}
\hln{f_w}{\frac{s-1}{2}}\leq & \hln{(Q^2\circ X-Q^2)\zeta\pa (\zeta \pa w)}{\frac{s-1}{2}}+\hln{Q^2(\zeta-\I)\pa(\zeta\pa w)}{\frac{s-1}{2}}\\
&+\hln{Q^2\pa((\zeta-\I)\pa w)}{\frac{s-1}{2}}\equiv I_{4}+I_{5}+I_{6}.
\end{align*}
We need to split \begin{align*}
&(Q^2\circ X-Q^2)\zeta\pa\p{\zeta\pa w}\\
&=(Q^2\circ X-Q^2)(\zeta-\I)\pa\p{\zeta\pa w}+ (Q^2\circ X-Q^2)\pa\p{(\zeta-\I)\pa w}+(Q^2\circ X-Q^2)\pa^2 w\\&=(Q^2\circ X-Q^2)(\zeta-\I)\pa\p{(\zeta-\I)\pa w}+(Q^2\circ X-Q^2)(\zeta-\I)\pa^2 w+ (Q^2\circ X-Q^2)\pa\p{(\zeta-\I)\pa w}\\
&+(Q^2\circ X-Q^2)\pa^2 w.\end{align*}
We bound $I_{4}$ by using lemma \ref{2.6} as follows,
\begin{align*}
I_4\leq \hhn{Q^2\circ X-Q^2}{\frac{s-1}{2}}{1+\delta}\p{1+\hhn{\zeta-\I}{\frac{s-1}{2}}{1+\delta}}\p{\hhn{\pa\p{(\zeta-\I)\pa w}}{\frac{s-1}{2}}{0}+\hhn{\pa^2 w}{\frac{s-1}{2}}{0}}.
\end{align*}
Using lemma \ref{2.5m}, with $\frac{1}{q}=\ep$, yields
\begin{align*}
&\hln{\pa(\zeta-\I)\pa w}{\frac{s-1}{2}}\leq  \hhn{\pa(\zeta-\I)}{\frac{s-1}{2}}{\ep}\hhn{\pa w}{\frac{s-1}{2}}{1-\ep}\\
\leq &C[v_0]\hhn{X-\al}{\frac{s-1}{2}}{2+\ep}\hhn{ w}{\frac{s-1}{2}}{2-\ep}\leq  C[v_0]\p{\hhn{X-\alpha-Av_0t}{\frac{s-1}{2}}{2+\ep}+\hhn{tv_0}{\frac{s-1}{2}}{2+\ep}}\leq C[v_0],
\end{align*}
for $\ep>0$ and small enough. Also
\begin{align*}
\hln{(\zeta-\I)\pa^2 w}{\frac{s-1}{2}}\leq \hhn{\zeta-\I}{\frac{s\!-\!1}{2}}{1\!+\!\delta}\hln{\partial^2w}{\frac{s-1}2}\leq C[v_0].
\end{align*}
Finally,
\begin{align}\label{comoQ}
&\hhn{Q^2\circ X\!-\!Q^2}{\frac{s\!-\!1}{2}}{1\!+\!\delta}\leq C[v_0] \p{\hhn{X\!-\!\alpha-Av_0t}{\frac{s\!-\!1}{2}}{1\!+\!\delta}+\hhn{tAv_0}{\frac{s-1}{2}}{1+\delta}}
\end{align}
and we have that $||t||_{H_{(0)}^\frac{s-1}{2}}\leq C\sqrt{T}.$ Also,
\begin{align*}
&\hhn{\pa_t\!\!\! \int_{0}^t\!\!(X\!-\!\alpha-Av_0t)}{\frac{s\!-\!1}{2}}{1\!+\!\delta}\\
&\leq C[v_0]\hhn{ \int_{0}^t (X\!-\!\alpha-Av_0t)}{\frac{s\!-\!1}{2}+\ep +1-\ep}{1\!+\!\delta}
\leq C[v_0]T^\ep \hhn{X-\alpha-Av_0t}{\frac{s-1}{2}+\ep}{1+\delta}\\
&\leq C[v_0]T^\ep \fn{X-\alpha-Av_0t}\leq C[v_0]T^\ep,
\end{align*}
where we have used lemma \ref{2.4}, with $\frac12<\frac{s-1}{2}+\ep<1$ and lemma \ref{bestiario} . With this last inequality we conclude the estimate for $I_{4}$.

For $I_{5}$ we have to make the following splitting $Q^2(\zeta-\I)\pa(\zeta\pa w)=Q^2(\zeta-\I)\pa((\zeta-\I)\pa w)+Q^2(\zeta-\I)\pa^2 w$ and then
\begin{align*}
I_{5}\leq ||Q^2||_{H^{1+\delta}}\hhn{\zeta-\I}{\frac{s-1}{2}}{1+\delta}\p{\hhn{\pa\p{(\zeta-\I)\pa w}}{\frac{s-1}{2}}{0}+\hhn{\pa^2w}{\frac{s-1}{2}}{0}}.
\end{align*}
The terms inside of the parenthesis in the previous expression have already been estimated in $I_{4}$. Using lemma \ref{bestiario} and lemma \ref{2.4} we find that
\begin{align*}
\hhn{\zeta-\I}{\frac{s-1}{2}}{1+\delta}&\leq  C[v_0]T^\ep\hhn{\zeta-\I}{\frac{s-1}{2}+\ep}{1+\delta}\leq C(M)T^\ep\hhn{X-\alpha}{\frac{s-1}{2}+\ep}{2+\delta},
\end{align*}
by lemma \ref{composicionz}. Proceeding as for $I_4$ we have that $I_5\leq C[v_0]T^\ep$.
We are done with $I_5$. For $I_{6}$ we have that
\begin{align*}
&I_{6}\leq \hln{Q^2\pa(\zeta-\I)\pa w}{\frac{s-1}{2}} +\hln{Q^2(\zeta-\I)\pa^2 w}{\frac{s-1}{2}}.
\end{align*}
Both terms can be handled as before. In fact
\begin{align*}
\hln{Q^2(\zeta-\I)\pa^2 w}{\frac{s-1}{2}}\leq & ||Q^2||_{H^{1+\delta}}\hhn{\zeta-\I}{\frac{s-1}{2}}{1+\delta}\hln{\partial^2w}{\frac{s-1}2}\\
\leq &C[v_0] T^\ep \htn{w}{s+1}\leq C[v_0] T^\ep,
\end{align*}
and
\begin{align*}
\hln{Q^2\pa(\zeta-\I)\pa w}{\frac{s-1}{2}}&\leq ||Q^2||_{H^{1+\delta}}\hhn{\pa(\zeta-\I)}{\frac{s-1}{2}}{\ep}\hhn{\partial w}{\frac{s-1}{2}}{1-\ep}\\
&\leq C \hhn{\zeta-\I}{\frac{s-1}{2}}{1+\epsilon}\hhn{w}{\frac{s-1}{2}}{2-\ep}\leq C[v_0] T^\ep \htn{w}{s+1}\\
&\leq C[v_0] T^\ep,
\end{align*}
for $\ep>0$ and small enough. We are done with $I_6$ and therefore with $\hln{f_w}{\frac{s-1}{2}}$. A similar procedure allows us to get
$$\hln{f_\phi}{\frac{s-1}{2}}\leq C[v_0]T^\epsilon.$$
Here we remark the main differences to get it.  We need split $\phi=v_0+t\psi$. For the terms coming from $v_0$ we do not find any problem because $v_0$ does not depend on time, and for the terms coming from $t\psi$ we can use that $\psi$ does not depend on $t$ and $||t||_{H_{(0)}^\frac{s-1}{2}}\leq C\sqrt{T}.$

The estimate of $\hln{f_{q}}{\frac{s-1}{2}}$ is obtained as follows. First we split $f_q=f_{q_w}+f_{q_\phi}$, we  can write
\begin{align*}
f_{q_w}=\p{-A\circ X+A}\p{\zeta-\I}\pa q_{w}+(-A\circ X+A)\pa q_w-A(\zeta-\I)\pa q_w.
\end{align*}
Then we can estimate
\begin{align*}
&\hln{f_{q_w}}{\frac{s-1}{2}}\\ &\leq  \hhn{(A\circ X-A)(\zeta-\I)\pa q_w}{\frac{s-1}{2}}{0}+\hhn{(A\circ X-A)\pa q_w}{\frac{s-1}{2}}{0}+\hhn{A(\zeta-\I)\pa q_w}{\frac{s-1}{2}}{0}
\\& \leq C[v_0]\p{\hhn{(A\circ X-A)}{\frac{s-1}{2}}{0}\hhn{(\zeta-\I)}{\frac{s-1}{2}}{0}+\hhn{(A\circ X-A)}{\frac{s-1}{2}}{0}+\hhn{(\zeta-\I)}{\frac{s-1}{2}}{0}}\\&\leq C[v_0]T^\ep.
\end{align*}
 For $f_{q_\phi}$ we have that \begin{align*}
 f_{q_\phi}=\p{-A\circ X+A}\p{\zeta-\I}\pa q_{\phi}+(-A\circ X+A)\pa q_\phi-A(\zeta-\I)\pa q_\phi.
 \end{align*}
 Here we can not take $\hhn{\nabla q_\phi}{\frac{s-1}{2}}{0}$ because $\nabla q_\phi|_{t=0}\neq 0$. Fortunately we do not need it since $q_\phi$ does not depend on $t$. Similarly to $\hhn{f_{q_w}}{\frac{s-1}{2}}{0}$ we have that $\hhn{f_{q_\phi}}{\frac{s-1}{2}}{0}\leq C[v_0]T^\ep$.
 This finishes the bounds for $\hln{f}{\frac{s-1}2}$ and therefore we are done with $\htn{f}{s-1}$.
 \\

\underline{P1.2. Estimate for $\overline{g}^{(n)}$:}\\

We recall that
\begin{align*}
\kt{\overline{g}^{(n)}}=\lhn{\overline{g}^{(n)}}{s}+\hhn{\overline{g}^{(n)}}{\frac{s+1}{2}}{-1}.
\end{align*}
We will first estimate the $H^0H^s$-norm and after that the $H^{\frac{s+1}{2}}H^{-1}$-norm.
We will split $\overline{g}^{(n)}$ in the following terms,
\begin{align*}
\overline{g}^{(n)}=&-Tr\left(\nabla\vn\zn A\circ\xn\right)+Tr\left(\nabla\vn A\right)\\&
+Tr\left(\nabla \phi \zeta_\phi A_\phi\right)-Tr\left(\nabla\phi A\right)\\
=&-Tr\left(\nabla \wn \zn A\circ\xn\right)-Tr\left(\nabla\phi (\zn A\circ \xn-\zeta_\phi A_\phi)\right)\\
&+ Tr\left(\nabla\wn A\right)+Tr\left(\nabla \phi \zeta_\phi A_\phi\right)\\
=&-Tr\left(\nabla \wn \left(\zn-\I\right) A\circ\xn\right)-Tr\left(\nabla\phi \left(\zn-\zeta_\phi\right) A\circ \xn\right)\\
&+ Tr\left(\nabla\wn \left(A-A\circ\xn\right)\right)+Tr\left(\nabla \phi \zeta_\phi\left(A_\phi-A\circ\xn\right)\right),
\end{align*}
From the partition of $\overline{g}^{(n)}
$ we have that
\begin{align*}
\lhn{\overline{g}^{(n)}}{s}=&\lhn{\nabla \wn \left(\zn-\I\right) A\circ\xn}{s}+\lhn{\nabla\phi \left(\zn-\zeta_\phi\right) A\circ \xn}{s}\\\
&+\lhn{\nabla\wn \left(A-A\circ\xn\right)}{s}+\lhn{\nabla \phi \left(A_\phi-A\circ\xn\right)}{s}\\
& \equiv I_1+I_2+I_3+I_4.
\end{align*}
Since $H^s$ is an algebra for $s>1$ as stated in lemma \ref{2.5} we have that
\begin{align*}
I_1\leq \li{\zn-\I}{s}\li{A\circ\xn}{s}\lhn{\nabla\wn}{s}.
\end{align*}
Therefore, applying lemma \ref{composicion1} yields
\begin{align*}
I_1\leq C[v_0]\li{\zn-\I}{s}\htn{\wn}{s+1}.
\end{align*}
 and using \eqref{zetamenosi} we get $I_1\leq C[v_0]\htn{\wn}{s+1}T^\frac{1}{4}.$ Similarly we obtain that
$I_2\leq C[v_0]T^\frac{1}{4}.$
In addition, by using  lemma \ref{composicion1} it  can be checked that $I_3\leq  C[v_0]T^\frac{1}{4}$, since $\lhn{\nabla \phi}{s}\leq C[v_0]$. And by lemma \ref{composicion1} $I_4\leq C[v_0]T^\frac{1}{4}.$

Then we have proved that
$\lhn{\overline{g}^{(n)}}{s}\leq C[v_0] T^\frac{1}{4}.$

To estimate the $H^{\frac{s+1}{2}}H^{-1}$-norm of $\overline{g}^{(n)}$, we split in a different way
\begin{align}\label{gsplit}
\overline{g}^{(n)}= &-Tr\left(\nabla \wn \zn A\circ\xn\right)+ Tr\left(\nabla\wn A\right)\nonumber\\
&-Tr\left(\nabla\phi (\zn-\zeta_\phi) A\circ \xn\right)+Tr\left(\nabla \phi \zeta_\phi (A_\phi-A\circ \xn)\right)\nonumber\\
&\equiv -\overline{g}_w^{(n)}-\overline{g}_\phi^{(n)},
\end{align}
where
\begin{align*}
\overline{g}^{(n)}_w = & Tr\left(\nabla \wn \zn A\circ\xn\right)- Tr\left(\nabla\wn A\right)\\
= & Tr\left(\nabla \wn (\zn-\zeta_\phi)(A\circ \xn-A_\phi)\right)+ Tr\left(\nabla\wn (\zn-\zeta_\phi) A_\phi\right)\\&+Tr\left(\nabla\wn \zeta_\phi (A\circ \xn-A_\phi)\right)+
  Tr\left(\nabla \wn (\zeta_\phi-\I) A_\phi\right)+Tr\left(\nabla\wn (A_\phi-A)\right)\\
\end{align*}
and
\begin{align*}
&\overline{g}_\phi^{(n)}=Tr\left(\nabla\phi (\zn-\zeta_\phi) A\circ \xn\right)-Tr\left(\nabla \phi \zeta_\phi (A_\phi-A\circ \xn)\right)\\
&=Tr\left(\nabla\phi (\zn-\zeta_\phi) (A\circ \xn-A_\phi)\right)-Tr\left(\nabla \phi \zeta_\phi (A_\phi-A\circ \xn)\right)+Tr\p{\nabla\phi\zeta_\phi A_\phi}.
\end{align*}

First we will bound $\overline{g}_{w}^{(n)}$ and then we will do the same with $\overline{g}_\phi^{(n)}$. Taking $H^{\frac{s+1}{2}}H^{-1}$-norms yields,
\begin{align}\label{gw}
&\hhn{\overline{g}^{(n)}_w}{\frac{s+1}{2}}{-1}\leq \hhn{\nabla \wn (\zn-\zeta_\phi)\p{A\circ \xn-A_\phi}}{\frac{s+1}{2}}{-1}+\hhn{\nabla \wn (\zn-\zeta_\phi)A_\phi}{\frac{s+1}{2}}{-1}
\nonumber\\&+\hhn{\nabla\wn \zeta_\phi (A\circ \xn-A_\phi)}{\frac{s+1}{2}}{-1}+\hhn{\nabla \wn  (A_\phi-A)+\nabla \wn (\zeta_\phi-\I) A_\phi}{\frac{s+1}{2}}{-1}\\ \nonumber
&=I_1+I_2+I_3+I_4. \nonumber
\end{align}
For $I_1$ we have that
\begin{align*}
&I_1=\hhn{\nabla \wn (\zn-\zeta_\phi)(A\circ \xn-A_\phi)}{\frac{s+1}{2}}{-1}=\hhn{\int_{0}^t\pa_t\left(\nabla \wn (\zn-\zeta_\phi)(A\circ \xn-A_\phi)\right)}{\frac{s+1}{2}}{-1}\\
&\leq \hhn{\int_{0}^t\pa_t\nabla\wn (\zn-\zeta_\phi)(A\circ \xn-A_\phi) }{\frac{s+1}{2}}{-1}+\hhn{\int_{0}^t\nabla\wn \pa_t(\zn-\zeta_\phi)(A\circ \xn-A_\phi)}{\frac{s+1}{2}}{-1}\\
&+\hhn{\int_{0}^t\nabla\wn (\zn-\zeta_\phi)\pa_t (A\circ \xn-A_\phi)}{\frac{s+1}{2}}{-1}\\
& \equiv I_{11}+I_{12}+I_{13}.
\end{align*}
And we bound $I_{11}$, $I_{12}$ and $I_{13}$ as follows
\begin{align*}
I_{11}\leq \hhn{\pa_t\nabla\wn (\zn-\zeta_\phi)(A\circ \xn-A_\phi)}{\frac{s-1}{2}}{-1},
\end{align*}
since $0<\frac{s-1}{2}<1$ and we can apply  lemma \ref{2.4} with $\ep=0$. Moreover, applying  lemma \ref{2.6} we obtain that
\begin{align*}
I_{11}\leq C[v_0]\hhn{\nabla\pa_t\wn}{\frac{s-1}{2}}{-1}\hhn{A\circ\xn-A_\phi}{\frac{s-1}{2}}{1+\delta}\hhn{\zn-\zeta_\phi}{\frac{s-1}{2}}{1+\delta}.
\end{align*}
In addition
\begin{align*}
&\hhn{\zeta^{(n)}-\zeta_\phi}{\frac{s-1}{2}}{1+\delta}\leq \hhn{\pa_t\int_{0}^t\zeta^{(n)}-\zeta_\phi}{\frac{s-1}{2}}{1+\delta}\\
&\leq C\hhn{\int_{0}^t\zn-\zeta_\phi}{\frac{s-1}{2}+\ep+1-\ep}{1+\delta}\leq C T^\ep \hhn{\zn-\zeta_\phi}{\frac{s-1}{2}+\ep}{1+\delta}
\end{align*}
by  lemma \ref{2.4}, for $0<\frac{s-1}{2}+\ep<1$. In addition, by lemma \ref{composicionzfi}, $$\hhn{\zn-\zeta_\phi}{\frac{s-1}{2}+\ep}{1+\delta}\leq C[v_0]\hhn{\xn-\alpha-Av_0t}{\frac{s-1}{2}+\ep}{2+\delta}.$$ Then lemmas \ref{composicionAfi} and \ref{bestiario} close the estimate for $I_{11}$.

For $I_{12}$ we have that, applying lemma \ref{2.5} and  lemma \ref{2.6},
\begin{align*}
I_{12}\leq C[v_0] \hhn{\nabla\wn}{\frac{s-1}{2}}{1}\hhn{A\circ\xn-A_\phi}{\frac{s-1}{2}}{1+\delta}\hhn{\pa_t(\zn-\zeta_\phi)}{\frac{s-1}{2}}{0}.
\end{align*}
By lemmas \ref{2.3} and \ref{composicionAfi} we obtain
\begin{align*}
I_{12}\leq C[v_0]\hhn{\pa_t(\zn-\zeta_\phi)}{\frac{s-1}{2}}{0}.
\end{align*}

Again we can apply lemma \ref{2.4} to get

\begin{align*}
I_{12}\leq & C[v_0] T^\ep \hhn{\pa_t(\zn-\zeta_\phi)}{\frac{s-1}{2}+\ep}{0}\leq  C[v_0] T^\ep \hhn{\zn-\zeta_\phi}{\frac{s+1}{2}+\ep}{0}
\end{align*}
for $0<\frac{s-1}{2}+\ep<1$. Therefore lemmas \ref{composicionzfi} and \ref{bestiario} yield a suitable estimate for $I_{12}$. The term $I_{13}$ is bounded in a similar way by using  lemmas \ref{composicionAfi}, \ref{bestiario} and \ref{2.3}.

Next we bound the second term in \eqref{gw}. In order to do it we split $\zeta_\phi = \I-t\nabla(Av_0)$. The terms coming from the identity can be bounded as $I_{13}$. For the terms containing the factor $t\nabla (A v_0)$ we just notice that we can proceed as for $I_{13}$ but putting the factor $t$ together with $\wn$ (here we remark that we can not take $||t||_{H_{(0)}^\frac{s+1}{2}}$ since $\frac{s+1}{2}>1.5$). Indeed
\begin{align*}
I_2= \hhn{t\nabla\wn \nabla(Av_0) (A\circ \xn-A_\phi)}{\frac{s+1}{2}}{-1}\leq & C[v_0]\hhn{t\nabla\wn}{\frac{s+1}{2}}{-1}\hhn{A\circ \xn-A_\phi}{\frac{s+1}{2}}{1+\delta},
\end{align*}
where
\begin{align*}
&\hhn{A\circ \xn -A_\phi}{\frac{s+1}{2}}{1+\delta}=\hhn{\int_0^t\pa_t (A\circ \xn-A_\phi)}{\frac{s-1}{2}+1}{1+\delta}\leq CT^\ep \hhn{\pa_t(A\circ\xn-A_\phi)}{\frac{s-1}{2}+\ep}{1+\delta},
\end{align*}
for $0<\frac{s-1}{2}+\ep<1$. Finally we can apply lemma \ref{composicionAfi}, \ref{bestiario} and \ref{littlelemma}.

For $I_3$ we can proceed in a similar way that for $I_2$. Finally for $I_4$ we just need to use lemma \ref{littlelemma}.

The estimate of $\hhn{\overline{g}_\phi}{\frac{s+1}{2}}{-1}$ follows similar steps. We just notice the need to split $\phi=v_0+t\psi$ and use lemma \ref{littlelemma} and the fact that $Tr\p{\nabla \phi \zeta_\phi A_\phi}=O(t^2)$.\\

\underline{P1.3. Estimate for $h^{(n)}$:}\\

We will show the appropriate estimate for $h^{(n)}$
decomposing $h^{(n)}=h^{(n)}_v+h^{(n)}_{v*}+h^{(n)}_q$ given by
$$
h^{(n)}_v=(\nabla \vn\zn\gradj\xn-\nabla\vn)n_0,
$$
\begin{equation}\label{hsplit}
	h^{(n)}_{v*}=((\nabla \vn\zn A\circ \xn)^*A^{-1}\circ\xn\gradj\xn-(\nabla\vn A)^*A^{-1})n_0,
\end{equation}
$$
h^{(n)}_q=(-\qn A^{-1}\circ\xn \gradj\xn +\qn A^{-1})n_0.
$$
As before, we ignore the superscripts for simplicity. We deal first with the $\lhb{\,\cdot\,}{s-\frac12}$ norm. Then
\begin{align*}
	\lhb{h_v}{s-\frac12}&\leq  \lhb{\nabla v (\zeta-\I)\gradj X }{s-\frac12}+\lhb{\nabla v (\gradj X-\I)}{s-\frac12}\equiv I_1+I_2.
\end{align*}
For $I_1$ we find
\begin{align*}
	I_1&\leq  C \lhb{\nabla v}{s-\frac12} \lib{\zeta-\I}{s-\frac12}\lib{\grad X}{s-\frac12}\\
	&\leq C \lhn{v}{s+1} \li{\zeta-\I}{s}\li{X}{s+1}\\
	&\leq C[v_0]\li{X-\alpha}{s+1}\leq C[v_0]T^\frac14 \lit{X-\alpha}{s+1}\leq C[v_0]T^\frac14.
\end{align*}
For $I_2$ the computation is analogous:
\begin{align*}
	I_2&\leq  C \lhb{\nabla v}{s-\frac12} \lib{\grad X-\I}{s-\frac12}\leq C[v_0]\li{X-\alpha}{s+1}\\
	&\leq C[v_0]T^\frac14 \lit{X-\alpha}{s+1}\leq C[v_0]T^\frac14.
\end{align*}
We are done with $\lhb{h_v}{s-\frac12}$. Next we deal with $\lhb{h_{v*}}{s-\frac12}$. Indeed
\begin{align*}
	&\lhb{h_{v*}}{s-\frac12}\leq   \lhb{(\nabla v(\zeta-\I) A\circ X)^*A^{-1}\circ X \gradj X}{s-\frac12}\\
	&+\!\lhb{(\nabla v (A\!\circ\! X\!-\!A))^* A^{-1}\!\!\circ\! X \gradj X}{s\!-\!\frac12}\!+\!\lhb{(\nabla v A )^* (A^{-1}\!\!\circ\!X\!-\!A^{-1})\gradj X}{s\!-\!\frac12}\\
	&\qquad +\lhb{(\nabla v A)^*A^{-1}(\gradj X-\I)}{s-\frac12}\equiv I_3+I_4+I_5+I_6.
\end{align*}
It is possible to obtain
\begin{align*}
	I_3&\leq  C\lhb{\nabla v}{s-\frac12} \lib{\zeta-\I}{s-\frac12}\lib{X}{s-\frac12}^2\lib{\grad X}{s-\frac12}\\
	&\leq C \lhn{v}{s+1} \li{\zeta-\I}{s}\li{X}{s+1}^3\leq C[v_0]\li{X-\alpha}{s+1}\\
	&\leq C[v_0]T^\frac14.
\end{align*}
Similarly, using Lemma \ref{composicion1}
\begin{align*}
	I_4\!+\!I_5\!+\!I_6&\leq  C[v_0]\lhb{\nabla v}{s-\frac12} \li{X\!-\!\al}{s+1}(\li{X}{s+1}^2\!+\!\li{X}{s+1}\!+\!1)\\
	&\leq C[v_0]T^\frac14 \lit{X-\al}{s+1}\leq C[v_0]T^\frac14.
\end{align*}
It remains to control $\lhb{h_{q}}{s-\frac12}$. We proceed as follows:
\begin{align*}
	\lhb{h_{q}}{s-\frac12}\leq & \lhb{q (A^{-1}\circ X-A^{-1}) \gradj X}{s-\frac12}+\lhb{q A^{-1} (\gradj X-\I)}{s-\frac12}\\
	\leq & C[v_0]\lhb{q}{s-\frac12}\li{X-\alpha}{s+1}(\li{\grad X}{s+1}+1)\\
	\leq & C[v_0]T^{1/4},
\end{align*}
using again Lemma \ref{composicion1} to end with the bounds for $\lhb{h}{s-\frac12}$.

The next step is to deal with $\hlb{h}{\frac s2-\frac14}$.
\begin{align*}
	\hlb{h_v}{\frac s2-\frac14}&\leq \hlb{\nabla v (\zeta-\I)\gradj X }{\frac s2-\frac14}+\hlb{\nabla v (\gradj X-\I)}{\frac s2-\frac14} \equiv K_1+K_2.
\end{align*}
This splitting provides
\begin{align*}
	K_1&\leq \hlb{\nabla w (\zeta-\I)\gradj X }{\frac s2-\frac14}+\hlb{\nabla v_0 (\zeta-\I)\gradj X }{\frac s2-\frac14}\\
	&\qquad\quad+\hlb{t\nabla\psi (\zeta-\I)\gradj X }{\frac s2-\frac14}\equiv K_{11}+K_{12}+K_{13}.
\end{align*}
and therefore
\begin{align*}
	K_{11}&\leq  \hlb{\nabla w (\zeta-\I)(\gradj X-\I)}{\frac s2-\frac14}+\hlb{\nabla w (\zeta-\I)}{\frac s2-\frac14}\\
	&\leq C\hlb{\nabla w}{\frac{s}2-\frac14}\hhb{\zeta-\I}{\frac{s}2-\frac14}{\frac12+\delta}
	(\hhb{\gradj X-\I}{\frac{s}2-\frac14}{\frac12+\delta}+1).
\end{align*}
We remark that the constant above is independent of time due to Lemma \ref{2.6}. Then we use that $\nabla w_i=(\nabla w_i\cdot n_0) n_0+(\nabla w_i\cdot t_0) t_0$ for $i=1,2$
and
\begin{equation*}
	\hlb{(\nabla w_i\cdot t_0)t_0}{\frac{s}2-\frac14}\leq C\hlb{\nabla w_i\cdot t_0}{\frac{s}2-\frac14}\leq C\hhb{w_i}{\frac{s}2-\frac14}{1}\leq C\hhn{w_i}{\frac{s}2-\frac14}{\frac32}\leq
	C\htn{w_i}{s+1},
\end{equation*}
together with
\begin{equation*}
	\hlb{(\nabla w_i\cdot n_0)n_0}{\frac{s}2-\frac14}\leq C\hlb{\nabla w_i\cdot n_0}{\frac{s}2-\frac14}\leq \htb{\nabla w_i\cdot n_0}{s-\frac12}\leq C\htn{w_i}{s+1}.
\end{equation*}
These two yield
\begin{equation}\label{tnYtraza}
	\hlb{\nabla w}{\frac{s}2-\frac14}\leq C\htn{w}{s+1},
\end{equation}
and therefore
\begin{align*}
	K_{11}
	&\leq C \htn{w}{s+1}\hhn{\zeta-\I}{\frac{s}2-\frac14}{1+\delta}
	(\hhn{X-\alpha}{\frac{s}2-\frac14}{2+\delta}+1)\\
	&\leq C[v_0](\hhn{X-\alpha-Av_0t}{\frac{s}2-\frac14}{2+\delta}+T^{\frac12})\leq C[v_0](T^\epsilon \hhn{X-\alpha-Av_0t}{\frac{s}2-\frac14+\epsilon}{2+\delta}+T^{\frac12})\\
	&\leq C[v_0](T^\epsilon\fn{X-\al-Av_0t}+T^{\frac12})\leq C[v_0]T^\epsilon.
\end{align*}
For $K_{12}$ it is easy to find
$$
K_{12}\leq C\| v_0\|_{H^{\frac32}}\hhb{(\zeta-\I)\gradj X }{\frac s2-\frac14}{\frac12+\delta}\leq C[v_0]T^\epsilon,
$$
proceeding as before. For the last term $K_{13}$ we obtain
\begin{align*}
	K_{13}&\leq C\hlb{t\nabla \psi}{\frac{s}2-\frac14}\hhb{(\zeta-\I)\gradj X}{\frac{s}2-\frac14}{\frac12+\delta}\leq CT^{\frac12}\|\psi\|_{H^{\frac32}}
	\hhb{(\zeta-\I)\gradj X}{\frac{s}2-\frac14}{\frac12+\delta}\leq C[v_0]T^\epsilon,
\end{align*}
as before. We are then done with $K_{1}$.

Using that $v=w+v_0+t\psi$, it is possible to estimate $K_2$ similarly as $K_1$. Therefore the appropriate estimate for $\hlb{h_v}{\frac s2-\frac14}$ follows.

Next we deal with $\hlb{h_{v*}}{\frac s2-\frac14}$. Indeed
\begin{align*}
	\hlb{h_{v*}}{\frac s2-\frac14}\leq&\hlb{(\nabla v(\zeta-\I) A\circ X)^*A^{-1}\circ X \gradj X}{\frac s2-\frac14}+\hlb{(\nabla v (A\circ X-A))^* A^{-1}\circ X \gradj X}{\frac s2-\frac14}\\
	&+\hlb{(\nabla v A )^* (A^{-1}\circ X-A^{-1})\gradj X}{\frac s2-\frac14}+\hlb{(\nabla v A)^*A^{-1}(\gradj X-\I)}{\frac s2-\frac14}\\
	&\qquad \equiv K_3+K_4+K_5+K_6.
\end{align*}
Taking $v=w+v_0+t\psi$, we find
\begin{align*}
	K_3\leq & \hlb{(\nabla w(\zeta-\I) A\circ X)^*A^{-1}\circ X \gradj X}{\frac s2-\frac14}+\hlb{(\nabla v_0(\zeta-\I) A\circ X)^*A^{-1}\circ X \gradj X}{\frac s2-\frac14}\\
	&+\hlb{t(\nabla \psi(\zeta-\I) A\circ X)^*A^{-1}\circ X \gradj X}{\frac s2-\frac14}\equiv K_{31}+K_{32}+K_{33}.
\end{align*}
Then using \eqref{tnYtraza} we find
\begin{align*}
	K_{31}\leq &  C\hlb{\nabla w}{\frac{s}2-\frac14}\hhb{((\zeta-\I) A\circ X)^*A^{-1}\circ X \gradj X}{\frac{s}2-\frac14}{\frac12+\delta}\\
	\leq& C\htn{w}{s+1}\hhn{((\zeta-\I) A\circ X)^*A^{-1}\circ X \gradj X}{\frac{s}2-\frac14}{1+\delta},
\end{align*}
where the constant is independent of time due to Lemma \ref{2.6}. The splitting
\begin{align}
	\begin{split}\label{nvchqpem}
		((\zeta-\I)&A\circ X)^*A^{-1}\circ X \gradj X=((\zeta-\I) (A\circ X-A))^*(A^{-1}\circ X-A^{-1})(\gradj X-\I)\\ &+((\zeta-\I) (A\circ X-A))^*(A^{-1}\circ X-A^{-1})+((\zeta-\I) (A\circ X-A))^*A^{-1} (\gradj X-\I)\\
		&+((\zeta-\I) (A\circ X-A))^*A^{-1}+((\zeta-\I)A)^*(A^{-1}\circ X-A^{-1})(\gradj X-\I)\\
		&+((\zeta-\I)A)^*(A^{-1}\circ X-A^{-1})+((\zeta-\I)A)^*A^{-1}(\gradj X-\I)+((\zeta-\I)A)^*A^{-1}
	\end{split}
\end{align}
allows us to bound $K_{31}$ with bounds independent of time due to Lemma \ref{composicion2} to find
\begin{align*}
	K_{31}\leq &  C[v_0]\hhn{\zeta-\I}{\frac s2-\frac14}{1+\delta}\leq C[v_0](\hhn{X-\alpha-Av_0t}{\frac{s}2-\frac14}{2+\delta}+T^{\frac12})\leq C[v_0]T^{\epsilon}.
\end{align*}
We proceed for $K_{32}$ as for $K_{12}$ getting
$$
K_{32}\leq C\|v_0\|_{H^{3/2}}\hhb{((\zeta-\I) A\circ X)^*A^{-1}\circ X \gradj X}{\frac s2-\frac14}{\frac12+\delta},
$$
so that splitting \eqref{nvchqpem} allows us to bound as before $K_{32}\leq C[v_0]T^\epsilon$. We continue by using previous estimates to obtain
$$
K_{33}\leq C\hlb{t\nabla \psi}{\frac{s}2-\frac14}\hhb{((\zeta-\I) A\circ X)^*A^{-1}\circ X \gradj X}{\frac s2-\frac14}{\frac12+\delta}\leq C[v_0]T^{\frac12}.
$$
We are done with $K_3$. Taking $v=w+v_0+t\psi$ it is possible to control $K_4$, $K_5$ and $K_6$ analogously
\begin{align*}
	K_4+K_5+K_6
	&\leq C[v_0]T^\epsilon.
\end{align*}
The term $\hlb{h_{v^*}}{\frac s2-\frac14}$ is then controlled and it remains to handle $\hlb{h_{q}}{\frac s2-\frac14}$. We proceed as follows:
\begin{align*}
	\hlb{h_{q}}{\frac s2-\frac14}\leq & \hlb{q (A^{-1}\circ X-A^{-1})(\gradj X-\I)}{\frac s2-\frac14}+\hlb{q (A^{-1}\circ X-A^{-1})}{\frac s2-\frac14}\\
	&\quad +\hlb{q A^{-1} (\gradj X-\I)}{\frac s2-\frac14}\equiv L_1+L_2+L_3.
\end{align*}
In $L_1$ we split $q=q_{w}+q_{\phi}$ to find
\begin{align*}
	L_1\leq & C (\hlb{q_{w}}{\frac s2-\frac14}+|q_\phi|_{L^2})\hhb{(A^{-1}\circ X-A^{-1})(\gradj X-\I)}{\frac s2-\frac14}{\frac12+\delta}\\
	\leq &C(||q_{w}||_{H^{ht,s}_{pr \, (0)}}+C[v_0])\hhn{(A^{-1}\circ X-A^{-1})(\gradj X-\I)}{\frac s2-\frac14}{1+\delta}\leq C[v_0]T^{\epsilon},
\end{align*}
where the time $T^{\epsilon}$ is found as before using Lemmas \ref{2.6}, \ref{2.4}, \ref{composicionQ1} and \ref{composicionz}. The terms $L_2$ and $L_3$ are estimated analogously to get $L_2+L_3\leq C[v_0]T^{\epsilon}$ and finally
\begin{align*}
	\hlb{h_{q}}{\frac s2-\frac14}\leq &  C[v_0]T^{\epsilon},
\end{align*}
to end with the bounds for $\hlb{h}{\frac s2-\frac14}$. We are done with $h$.\\

\subsection*{Proof of part 2:}

It will be enough to show that
\begin{align*}
\htn{f^{(n)}-f^{(n-1)}}{s-1}\leq&C[v_0]\left(\fn{\xn-\xm}+\htn{\wn-\wm}{s+1}\right.\\
&\qquad\qquad\left.+||\qn-\qm||_{H^{ht,s}_{pr \, (0)}}\right),\\
\kt{g^{(n)}-g^{(n-1)}}\leq&  C[v_0]\left(\fn{\xn-\xm}+\htn{\wn-\wm}{s+1}\right),\\
\htb{h^{(n)}-h^{(n-1)}}{s-\frac{1}{2}}\leq &  C[v_0]\left(\fn{\xn-\xm}+\htn{\wn-\wm}{s+1}\right.\\
&\qquad\qquad\left.+||\qn-\qm||_{H^{ht,s}_{pr \, (0)}}\right).
\end{align*}
Again we split the proof in three parts:\\

\underline{P2.1. Estimate for $f^{(n)}-f^{(n-1)}$}\\

We split as in \eqref{fsplit}: $f^{(j)}=f^{(j)}_w+f^{(j)}_\phi+f^{(j)}_q$.
In $f^{(n)}_w-f^{(n-1)}_w$ we split further $f^{(n)}_w-f^{(n-1)}_w=d_1+...+d_6$ with the following differences
\begin{equation*}
d_1=(Q^2\circ X^{(n)}-Q^2\circ X^{(n-1)})\zeta^{(n)}\pa\left(\zeta^{(n)}\pa w^{(n)}\right),
\end{equation*}
\begin{equation*}
d_2=Q^2\circ X^{(n-1)}(\zeta^{(n)}-\zeta^{(n-1)})\pa\left(\zeta^{(n)}\pa w^{(n)}\right),
\end{equation*}
\begin{equation}\label{diferenciaf}
d_3=Q^2\circ X^{(n-1)}\zeta^{(n-1)}\pa\left((\zeta^{(n)}-\zeta^{(n-1)})\pa w^{(n)}\right),
\end{equation}
\begin{equation*}
d_4=(Q^2\circ X^{(n-1)}-Q^2)\zeta^{(n-1)}\pa\left(\zeta^{(n-1)}\pa (w^{(n)}-w^{(n-1)})\right),
\end{equation*}
\begin{equation*}
d_5=Q^2(\zeta^{(n-1)}-\mathbb{I})\pa\left(\zeta^{(n-1)}\pa (w^{(n)}-w^{(n-1)})\right),
\end{equation*}
\begin{equation*}
d_6=Q^2\pa\left((\zeta^{(n-1)}-\mathbb{I})\pa (w^{(n)}-w^{(n-1)})\right).
\end{equation*}
Above we do not distinguish from coordinates and partial derivatives, as all the cases can be handled in the same manner.
Next we estimate $d_1$. In order to do that we split further
$d_1=d_{11}+d_{12}$ with
\begin{equation*}
d_{11}=(Q^2\circ X^{(n)}-Q^2\circ X^{(n-1)})\zeta^{(n)}\pa\zeta^{(n)}\pa w^{(n)},
\end{equation*}
\begin{equation*}
d_{12}=(Q^2\circ X^{(n)}-Q^2\circ X^{(n-1)})\zeta^{(n)}\zeta^{(n)}\pa^2 w^{(n)}.
\end{equation*}
We take
\begin{align*}
\lhn{d_{11}}{s-1}&\leq  C \li{Q^2\circ X^{(n)}-Q^2\circ X^{(n-1)}}{s-1}\li{\zeta^{(n)}}{s-1}\lhn{\pa\zeta^{(n)}\pa w^{(n)}}{s-1}\\
&\leq  C[v_0]
\li{X^{(n)}-X^{(n-1)}}{s-1}\li{\pa\zeta^{(n)}}{s-1}
\lhn{\pa w^{(n)}}{s-1}\\
&\leq C[v_0]\li{X^{(n)}-X^{(n-1)}}{s-1}\leq C[v_0]T^{1/4}\fn{\xn-\xm},
\end{align*}
and, since $$d_{11}= \p{Q^2\circ \xn-Q^2\circ\xm}\p{\zn-\I}\pa\p{\zn-\I}\pa \wn+\p{Q^2\circ\xn-Q^2\circ \xm}\pa (\zn-\I)\pa \wn,$$
\begin{align*}
\hln{d_{11}}{\frac{s-1}2}&\leq  C \hhn{Q^2\circ X^{(n)}-Q^2\circ X^{(n-1)}}{\frac{s-1}2}{1+\epsilon}\\
&\qquad\qquad\qquad\times(\hhn{\zeta^{(n)}-\I}{\frac{s-1}2}{1+\epsilon}+1)\hln{\pa(\zeta^{(n)}-\I)\pa w^{(n)}}{\frac{s-1}2}\\
&\leq  C[v_0]
\hhn{X^{(n)}-X^{(n-1)}}{\frac{s-1}2}{1+\epsilon}
\hhn{\pa(\zeta^{(n)}-\I)}{\frac{s-1}2}{\epsilon}
\hhn{\pa w^{(n)}}{\frac{s-1}2}{1-\epsilon}.
\end{align*}
We deal with above terms as before to get
\begin{align*}
\hln{d_{11}}{\frac{s-1}2}& \leq C[v_0]T^\epsilon\fn{\xn-\xm}.
\end{align*}
For $d_{12}$ we find
\begin{align*}
\lhn{d_{12}}{s-1}&\leq  C \li{Q^2\circ X^{(n)}-Q^2\circ X^{(n-1)}}{s-1}\li{\zeta^{(n)}}{s-1}^2\lhn{\pa^2 w^{(n)}}{s-1}\\
&\leq  C[v_0]
\li{X^{(n)}-X^{(n-1)}}{s-1}\leq C[v_0]T^{1/4}\fn{\xn-\xm}
\end{align*}
and
\begin{align*}
\hln{d_{12}}{\frac{s-1}2}&\leq  C \hhn{Q^2\circ X^{(n)}-Q^2\circ X^{(n-1)}}{\frac{s-1}2}{1+\epsilon}\\
&\qquad\qquad\times(\hhn{\zeta^{(n)}-\I}{\frac{s-1}2}{1+\epsilon}^2+1)\hln{\pa^2 w^{(n)}}{\frac{s-1}2}\\
&\leq  C[v_0]\hhn{X^{(n)}-X^{(n-1)}}{\frac{s-1}2}{1+\epsilon}\leq C[v_0]T^\epsilon\fn{\xn-\xm}.
\end{align*}
In order to continue we decompose the next term,
$d_2=d_{21}+d_{22}$ where
\begin{equation*}
d_{21}=Q^2\circ X^{(n-1)}(\zeta^{(n)}-\zm)\pa\zeta^{(n)}\pa w^{(n)},
\end{equation*}
\begin{equation*}
d_{22}=Q^2\circ X^{(n-1)}(\zeta^{(n)}-\zm)\zeta^{(n)}\pa^2 w^{(n)}.
\end{equation*}
We take
\begin{align*}
\lhn{d_{21}}{s-1}&\leq  C \li{Q^2\circ X^{(n-1)}}{s-1}\li{\zeta^{(n)}-\zm}{s-1}\lhn{\pa\zeta^{(n)}\pa w^{(n)}}{s-1}\\
&\leq  C[v_0]
\li{X^{(n)}-X^{(n-1)}}{s}\li{\pa\zeta^{(n)}}{s-1}
\lhn{\pa w^{(n)}}{s-1}\\
&\leq C[v_0]T^{1/4}\fn{\xn-\xm}.
\end{align*}

\begin{align*}
\hln{d_{21}}{\frac{s-1}2}&\leq  C (\hhn{Q^2\circ X^{(n-1)}-Q^2}{\frac{s-1}2}{1+\epsilon}+1)\\
&\qquad\qquad\qquad\times \hhn{\zeta^{(n)}-\zm}{\frac{s-1}2}{1+\epsilon}\hln{\pa\zeta^{(n)}\pa w^{(n)}}{\frac{s-1}2}\\
&\leq  C[v_0]\hhn{X^{(n)}-X^{(n-1)}}{\frac{s-1}2}{2+\epsilon}
\hhn{\pa(\zeta^{(n)}-\I)}{\frac{s-1}2}{\epsilon}
\hhn{\pa w^{(n)}}{\frac{s-1}2}{1-\epsilon}\\
&\leq C[v_0]T^\epsilon\fn{\xn-\xm}.
\end{align*}
For $d_{22}$ we find
\begin{align*}
\lhn{d_{22}}{s-1}&\leq  C[v_0] \li{\zn-\zm}{s-1}\lhn{\pa^2 w^{(n)}}{s-1}\\
&\leq  C[v_0]
\li{X^{(n)}-X^{(n-1)}}{s}\leq C[v_0]T^{1/4}\fn{\xn-\xm},
\end{align*}
and
\begin{align*}
\hln{d_{22}}{\frac{s-1}2}&\leq  C[v_0]\hhn{\zeta^{(n)}-\zm}{\frac{s-1}2}{1+\epsilon}\hln{\pa^2 w^{(n)}}{\frac{s-1}2}\\
&\leq  C[v_0]\hhn{X^{(n)}-X^{(n-1)}}{\frac{s-1}2}{2+\epsilon}\leq C[v_0]T^\epsilon\fn{\xn-\xm}.
\end{align*}
In $d_3$ we split as follows $d_3=d_{31}+d_{32}$ with
\begin{equation*}
d_{31}=Q^2\circ X^{(n-1)}\zeta^{(n-1)}\pa(\zeta^{(n)}-\zeta^{(n-1)})\pa w^{(n)},
\end{equation*}
and
\begin{equation*}
d_{32}=Q^2\circ X^{(n-1)}\zeta^{(n-1)}(\zeta^{(n)}-\zeta^{(n-1)})\pa^2 w^{(n)}.
\end{equation*}
For $d_{31}$ it is possible to get
\begin{align*}
\lhn{d_{31}}{s-1}&\leq  C \li{Q^2\circ X^{(n-1)}}{s-1}\li{\zeta^{(n-1)}}{s-1}\lhn{\pa(\zeta^{(n)}-\zm)\pa w^{(n)}}{s-1}\\
&\leq  C[v_0]\li{\pa(\zeta^{(n)}-\zm)}{s-1}\lhn{\pa w^{(n)}}{s-1}\\
&\leq C[v_0]T^{1/4}\fn{\xn-\xm},
\end{align*}
and
\begin{align*}
\hln{d_{31}}{\frac{s-1}2}&\leq  C (\hhn{Q^2\circ X^{(n-1)}-Q^2}{\frac{s-1}2}{1+\epsilon}+1)\\
&\qquad\quad\times\left(\hhn{\zm-\I}{\frac{s-1}2}{1+\epsilon}+1\right)\hln{\pa(\zeta^{(n)}-\zm)\pa w^{(n)}}{\frac{s-1}2}\\
&\leq  C[v_0]\hhn{\partial(\zeta^{(n)}-\zeta^{(n-1)})}{\frac{s-1}2}{\delta}\hhn{\pa w^{(n)}}{\frac{s-1}2}{1-\delta}\\
&\leq C[v_0]\hhn{X^{(n)}-X^{(n-1)}}{\frac{s-1}2}{2+\delta}\leq C[v_0]T^\epsilon\fn{\xn-\xm}.
\end{align*}
Here, in order to bound $\hhn{Q^2\circ X^{(n-1)}-Q^2}{\frac{s-1}2}{1+\epsilon}$ we can do the same than in \eqref{comoQ}.

Next, for $d_{32}$ we obtain
\begin{align*}
\lhn{d_{32}}{s-1}&\leq  C \li{Q^2\circ X^{(n-1)}}{s-1}\li{\zeta^{(n-1)}}{s-1}\lhn{(\zeta^{(n)}-\zm)\pa^2 w^{(n)}}{s-1}\\
&\leq  C[v_0]\li{\zeta^{(n)}-\zm}{s-1}\lhn{\pa^2 w^{(n)}}{s-1}\\
&\leq  C[v_0]\li{X^{(n)}-\xm}{s} \leq C[v_0]T^{1/4}\fn{\xn-\xm}.
\end{align*}
The usual approach also gives
\begin{align*}
\hln{d_{32}}{\frac{s-1}2}&\leq  C \Big(\hhn{Q^2\circ X^{(n-1)}-Q^2}{\frac{s-1}2}{1+\epsilon}+1\Big)\\
&\qquad\times\Big(\hhn{\zm-\I}{\frac{s-1}2}{1+\epsilon}+1\Big)\hln{(\zeta^{(n)}-\zm)\pa^2 w^{(n)}}{\frac{s-1}2}\\
&\leq  C[v_0]\hhn{\zeta^{(n)}-\zeta^{(n-1)}}{\frac{s-1}2}{1+\delta}\hln{\pa^2 w^{(n)}}{\frac{s-1}2}\\
&\leq C[v_0]\hhn{X^{(n)}-X^{(n-1)}}{\frac{s-1}2}{2+\delta}\leq C[v_0]T^\epsilon\fn{\xn-\xm}.
\end{align*}
In $d_{4}$ we proceed by considering $d_{4}=d_{41}+d_{42}$ with
\begin{equation*}
d_{41}=(Q^2\circ X^{(n-1)}-Q^2)\zeta^{(n-1)}\pa\zeta^{(n-1)}\pa (w^{(n)}-w^{(n-1)}),
\end{equation*}
and
\begin{equation*}
d_{42}=(Q^2\circ X^{(n-1)}-Q^2)\zeta^{(n-1)}\zeta^{(n-1)}\pa^2(w^{(n)}-w^{(n-1)}).
\end{equation*}
For $d_{41}$ it is possible to get
\begin{align*}
\lhn{d_{41}}{s-1}&\leq  C \li{Q^2\circ X^{(n-1)}-Q^2}{s-1}\\
&\qquad\qquad\times\li{\zeta^{(n-1)}}{s-1}\lhn{\pa\zm\pa (w^{(n)}-\wm) }{s-1}\\
&\leq  C[v_0]\li{X^{(n-1)}-\alpha}{s-1}\li{\pa\zm}{s-1}\lhn{\pa (w^{(n)}-\wm)}{s-1}\\
&\leq C[v_0]T^{1/4}\htn{\wn-\wm}{s+1},
\end{align*}
and
\begin{align*}
\hln{d_{42}}{\frac{s-1}2}&\leq  C \hhn{Q^2\circ X^{(n-1)}-Q^2}{\frac{s-1}2}{1+\epsilon}\\
&\qquad\qquad\times\Big(\hhn{\zm-\I}{\frac{s-1}2}{1+\epsilon}+1\Big)\hln{\pa\zm\pa (w^{(n)}-\wm)}{\frac{s-1}2}\\
&\leq  C[v_0]\hhn{X^{(n-1)}-\alpha}{\frac{s-1}2}{1+\epsilon}\hhn{\pa(\zm-\I)}{\frac{s-1}2}{\delta}\hhn{\pa (w^{(n)}-\wm)}{\frac{s-1}2}{1-\delta}\\
&\leq C[v_0]T^\epsilon\htn{\wn-\wm}{s+1}.
\end{align*}
For $d_{42}$ it is possible to get
\begin{align*}
\lhn{d_{42}}{s-1}&\leq  C \li{Q^2\circ X^{(n-1)}-Q^2}{s-1}\li{\zeta^{(n-1)}}{s-1}^2\lhn{\pa^2(\wn-\wm)}{s-1}\\
&\leq  C[v_0]\li{\xn-\alpha}{s-1}\lhn{\pa^2 (w^{(n)}-\wm)}{s-1}\\
&\leq C[v_0]T^{1/4}\htn{\wn-\wm}{s+1},
\end{align*}
and
\begin{align*}
\hln{d_{42}}{\frac{s-1}2}&\leq  C \hhn{Q^2\circ X^{(n-1)}-Q^2}{\frac{s-1}2}{1+\epsilon}\\
&\qquad\qquad\times\Big(\hhn{\zm-\I}{\frac{s-1}2}{1+\epsilon}^2+1\Big)\hln{\pa^2(w^{(n)}-\wm)}{\frac{s-1}2}\\
&\leq  C[v_0]\hhn{\xn-\al}{\frac{s-1}2}{1+\epsilon}\hln{\pa^2(w^{(n)}-\wm)}{\frac{s-1}2}\\
&\leq C[v_0]T^{\epsilon}\htn{\wn-\wm}{s+1}.
\end{align*}
Analogously, one could take $d_5=d_{51}+d_{52}$ with
\begin{equation*}
d_{51}=Q^2(\zeta^{(n-1)}-\mathbb{I})\pa\zeta^{(n-1)}\pa (w^{(n)}-w^{(n-1)}),
\end{equation*}
and
\begin{equation*}
d_{52}=Q^2(\zeta^{(n-1)}-\mathbb{I})\zeta^{(n-1)}\pa^2(w^{(n)}-w^{(n-1)}).
\end{equation*}
The splitting yields
\begin{align*}
\lhn{d_{51}}{s-1}&\leq  C \li{\zeta^{(n-1)}-\mathbb{I}}{s-1}\li{\pa\zeta^{(n-1)}}{s-1}\lhn{\pa (w^{(n)}-\wm) }{s-1}\\
&\leq C[v_0]T^{1/4}\htn{\wn-\wm}{s+1},
\end{align*}
and
\begin{align*}
\hln{d_{51}}{\frac{s-1}2}&\leq  C \hhn{\zeta^{(n-1)}-\mathbb{I}}{\frac{s-1}2}{1+\epsilon}\hhn{\partial(\zm-\I)}{\frac{s-1}2}{\delta}
\hhn{\pa (w^{(n)}-\wm)}{\frac{s-1}2}{1-\delta}\\
&\leq C[v_0]T^\epsilon\htn{\wn-\wm}{s+1}.
\end{align*}
For $d_{52}$ we proceed as follows
\begin{align*}
\lhn{d_{52}}{s-1}&\leq  C \li{\zeta^{(n-1)}-\mathbb{I}}{s-1}\li{\zeta^{(n-1)}}{s-1}\lhn{\pa^2 (w^{(n)}-\wm) }{s-1}\\
&\leq C[v_0]T^{1/4}\htn{\wn-\wm}{s+1},
\end{align*}
\begin{align*}
\hln{d_{52}}{\frac{s-1}2}&\leq  C \hhn{\zeta^{(n-1)}-\mathbb{I}}{\frac{s-1}2}{1+\epsilon}\\
&\qquad\qquad\times\Big(\hhn{\zeta^{(n-1)}}{\frac{s-1}2}{1+\epsilon}+1\Big)\hln{\pa^2 (w^{(n)}-\wm)}{\frac{s-1}2}\\
&\leq C[v_0]T^\epsilon\htn{\wn-\wm}{s+1}.
\end{align*}
Finally, writing $d_6=d_{61}+d_{62}$ where
\begin{equation*}
d_{61}=Q^2\pa\zeta^{(n-1)}\pa (w^{(n)}-w^{(n-1)}),\mbox{ and } d_{62}=Q^2(\zeta^{(n-1)}-\mathbb{I})\pa^2(w^{(n)}-w^{(n-1)}),
\end{equation*}
it is possible to find
\begin{align*}
\lhn{d_{61}}{s-1}&\leq  C\li{\pa\zeta^{(n-1)}}{s-1}\lhn{\pa (w^{(n)}-\wm) }{s-1}\\
&\leq C[v_0]T^{1/4}\htn{\wn-\wm}{s+1},
\end{align*}
and
\begin{align*}
\hln{d_{61}}{\frac{s-1}2}&\leq  C \hhn{\partial\zm}{\frac{s-1}2}{\delta}
\hhn{\pa (w^{(n)}-\wm)}{\frac{s-1}2}{1-\delta}\\
&\leq C[v_0]T^\epsilon\htn{\wn-\wm}{s+1}.
\end{align*}
For $d_{62}$ we proceed as for $d_{52}$ to get
\begin{align*}
\lhn{d_{62}}{s-1}&\leq C[v_0]
T^{1/4}\htn{\wn-\wm}{s+1},\\
\hln{d_{62}}{\frac{s-1}2}&\leq C[v_0]T^\epsilon\htn{\wn-\wm}{s+1}.
\end{align*}
We are done with $f^{(n)}_w-f^{(n-1)}_w$. To continue with $f^{(n)}_\phi-f^{(n-1)}_\phi$ we decompose as before but this time using $f^{(n)}_\phi-f^{(n-1)}_\phi=d_7+d_8+d_9$ where
\begin{equation*}
d_7=(Q^2\circ X^{(n)}-Q^2\circ X^{(n-1)})\zeta^{(n)}\pa\left(\zeta^{(n)}\pa \phi\right),
\end{equation*}
\begin{equation*}
d_8=Q^2\circ X^{(n-1)}(\zeta^{(n)}-\zeta^{(n-1)})\pa\left(\zeta^{(n)}\pa \phi\right),
\end{equation*}
\begin{equation*}
d_9=Q^2\circ X^{(n-1)}\zeta^{(n-1)}\pa\left((\zeta^{(n)}-\zeta^{(n-1)})\pa \phi\right).
\end{equation*}
Here we need to split $\phi=v_0+t\psi$. After that  we proceed as for $d_1$, $d_2$ and $d_3$  to find
\begin{equation*}
\htn{d_7+d_8+d_9}{s-1}\leq C[v_0]T^\epsilon\fn{\xn-\xm}.
\end{equation*}
We move to the $f$ term involving $q$. We split $f^{(n)}_{q}=f^{(n)}_{q_w}+f^{(n)}_{q_\phi}$. The splitting $f^{(n)}_{q_w}-f^{(n-1)}_{q_w}=d^q_1+d^q_2+d^q_3+d^q_4$
with
\begin{equation*}
d^q_1=-(A\circ\xn(\zn-\zm))^*\nabla \qn_w,
\end{equation*}
\begin{equation}\label{diferenciafq}
d^q_2=-((A\circ\xn-A\circ\xm)\zm)^*\nabla \qn_w,
\end{equation}
\begin{equation*}
d^q_3=-(A\circ\xm(\zm-\I))^*\nabla (\qn_w-\qm_w),
\end{equation*}
\begin{equation*}
d^q_4=-(A\circ\xm-A)^*\nabla (\qn_w-\qm_w),
\end{equation*}
allows us to do the work. In fact
\begin{align*}
\lhn{d^q_1}{s-1}&\leq  C \li{\zeta^{(n)}-\zm}{s-1}\li{A\circ\xn}{s-1}\lhn{\grad\qn_w}{s-1}\\
&\leq C[v_0]T^{1/4}\fn{\xn-\xm},
\end{align*}
\begin{align*}
\hln{d^q_{1}}{\frac{s-1}2}&\leq  C \hhn{\zeta^{(n)}-\zm}{\frac{s-1}2}{1+\epsilon}\\
&\qquad\qquad\times\Big(\hhn{A\circ\xn-A}{\frac{s-1}2}{1+\epsilon}+1\Big)\hln{\nabla\qn_w}{\frac{s-1}2}\\
&\leq C[v_0]T^\epsilon\fn{\xn-\xm}.
\end{align*}
Sharing norms in the same manner gives
\begin{align*}
\lhn{d^q_2}{s-1}&\leq  C[v_0]T^{1/4}\fn{\xn-\xm},
\end{align*}
and
\begin{align*}
\hln{d^q_2}{\frac{s-1}2}&\leq  C[v_0]T^\epsilon\fn{\xn-\xm}.
\end{align*}
In a similar way
\begin{align*}
\lhn{d^q_3}{s-1}&\leq  C \li{\zm-\I}{s-1}\li{A\circ\xm}{s-1}\lhn{\grad(\qn_w-\qm_w)}{s-1}\\
&\leq C[v_0]T^{1/4}\lhn{\grad(\qn_w-\qm_w)}{s-1}\leq C[v_0]T^{1/4}\|\qn_w-\qm_w\|_{H_{pr \, (0)}^{ht,s}},
\end{align*}
and
\begin{align*}
\hln{d^q_{3}}{\frac{s-1}2}&\leq  C \hhn{\zm-\I}{\frac{s-1}2}{1+\epsilon}\\
&\qquad\times\Big(\hhn{A\circ\xm-A}{\frac{s-1}2}{1+\epsilon}+1\Big)\hln{\nabla(\qn_w-\qm_w)}{\frac{s-1}2}\\
&\leq C[v_0]T^\epsilon\hln{\nabla(\qn_w-\qm_w)}{\frac{s-1}2}\\
&\leq C[v_0]T^{\epsilon}\|\qn_w-\qm_w\|_{H_{pr \, (0)}^{ht,s}}.
\end{align*}
Finally
\begin{align*}
\lhn{d^q_4}{s-1}&\leq  C \li{A\circ\xm-A}{s-1}\lhn{\grad(\qn_w-\qm_w)}{s-1}\\
&\leq C[v_0]T^{1/4}\|\qn_w-\qm_w\|_{H_{pr \, (0)}^{ht,s}},
\end{align*}
and
\begin{align*}
\hln{d^q_4}{\frac{s-1}2}&\leq  C \hhn{A\circ\xm-A}{\frac{s-1}2}{1+\epsilon}\hln{\nabla(\qn_w-\qm_w)}{\frac{s-1}2}\\
&\leq  C[v_0]T^{\epsilon}\|\qn_w-\qm_w\|_{H_{pr \, (0)}^{ht,s}}.
\end{align*}

The estimation for $f^{(n)}_{q_\phi}-f^{(n-1)}_{q_\phi}$ follows similar steps. We only need to take into account that $q_\phi$ does not depend on time.

\underline{P2.2. Estimate for $\overline{g}^{(n)}-\overline{g}^{(n-1)}$:}\\

We are concerned with the estimates of the norms
\begin{align*}
\kt{\overline{g}^{(n)}-\overline{g}^{(n-1)}}\leq\lhn{\overline{g}^{(n)}-\overline{g}^{(n-1)}}{s}+ \hhn{\overline{g}^{(n)}-\overline{g}^{(n-1)}}{\frac{s+1}{2}}{-1},
\end{align*}

We will split $\overline{g}^{(n)}= -\overline{g}^{(n)}_w-\overline{g}^{(n)}_\phi$ (see \eqref{gsplit}). First we will estimate the $\lhn{\cdot}{s}$-norm. For $\overline{g}^{(n)}_w-\overline{g}^{(n-1)}_w$ we have that
\begin{align*}
&\overline{g}^{(n)}_w-\overline{g}^{(n-1)}_w= Tr\p{\p{\nabla \wn -\nabla \wm}\zn A\circ \xn}+Tr\p{\nabla \wm(\zn-\zm)A\circ \xn}\\
&Tr\p{\nabla \wm \zm \p{A\circ \xn-A\circ \xm}}-Tr\p{\p{\nabla \wn-\nabla \wm}A}\\
&= Tr\p{\p{\nabla \wn -\nabla \wm}(\zn-\I) A\circ \xn}+Tr\p{\nabla \wm(\zn-\zm)A\circ \xn}\\
&Tr\p{\nabla \wm \zm \p{A\circ \xn-A\circ \xm}}+Tr\p{\p{\nabla \wn-\nabla \wm}(A\circ \xn-A}\\
&=d_1+d_2+d_3+d_4.
\end{align*}

With $d_1$ we proceed as follows
\begin{align}
\begin{split}\label{d1g}	
\lhn{d_1}{s}=&\lhn{(\nabla\wn-\nabla\wm)(\zn-\I)A\circ \xn}{s}\\
\leq & \li{\zn-\I}{s}\li{A\circ\xn}{s}\lhn{\nabla\wn-\nabla\wm}{s}.
\end{split}
\end{align}
Then lemma \ref{composicionz} and \eqref{zetamenosi} implies
\begin{align*}
\lhn{d_1}{s}\leq C[v_0]\li{\zn-\I}{s}\htn{\wn-\wm}{s+1}\leq C[v_0]T^\frac{1}{4}\htn{\wn-\wm}{s+1}.
\end{align*}
For $d_2$ we have that
\begin{align*}
\lhn{d_2}{s}\leq \lhn{\nabla\wm}{s}\li{A\circ \xn}{s}\li{\zn-\zm}{s}.
\end{align*}
Thus, using lemma \ref{composiciondizeta} yields,
\begin{align*}
\lhn{d_2}{s}\leq C[v_0]T^\frac{1}{4}\lit{\xn-\xm}{s+1}.
\end{align*}
The estimate for $d_3$ follows the next steps
\begin{align*}
\lhn{d_3}{s}&\leq \lhn{\nabla\wm}{s}\li{\zm}{s}\li{A\circ\xn-A\circ\xm}{s}\\
&\leq C[v_0]\li{\xn-\xm}{s}\leq C[v_0]T^\frac{1}{4}\lit{\xn-\xm}{s}.
\end{align*}
And for $d_4$ we have that, using \ref{composicion1} and proceeding as in $\eqref{comoQ}$
\begin{align*}
\lhn{d_4}{s}&\leq \htn{\wn-\wm}{s+1}\li{A\circ \xn-A}{s}\\
&\leq C[v_0]\htn{\wn-\wm}{s+1}\li{\xn-\alpha}{s}\\
&\leq C[v_0]T^\frac{1}{4}\htn{\wn-\wm}{s+1}.
\end{align*}

It remains to estimate the $\hhn{\,\cdot\,}{\frac{s+1}{2}}{-1}$-norm.

We will split $\overline{g}_w^{(n)}-\overline{g}_w^{(n-1)}$ in the following terms
\begin{align*}
&\overline{g}_w^{(n)}-\overline{g}_w^{(n-1)}\\
&=Tr\p{(\nabla \wn-\nabla\wm)(\zn-\zeta_\phi)(A\circ \xn-A_\phi)}\\
&+Tr\p{(\nabla \wn-\nabla\wm)(\zn-\zeta_\phi)A_\phi}+Tr\p{(\nabla\wn-\nabla\wm)\zeta_\phi(A\circ\xn-A_\phi)}\\
&+Tr\p{\nabla \wm(\zn-\zm)(A\circ\xn-A_\phi)}+Tr\p{\nabla \wm(\zn-\zm)A_\phi}\\
&+Tr\p{\nabla \wm(\zm-\zeta_\phi)(A\circ \xn-A\circ \xm)}+Tr\p{\nabla\wm\zeta_\phi\p{A\circ\xn-A\circ \xm}}\\
&+Tr\p{(\nabla \wn-\nabla\wm)(\zeta_\phi-\I)A_\phi+(\nabla \wn-\nabla\wm)(A_\phi-A)}\\
&=D_1+D_2+D_3+D_4+D_5+D_6+D_7+D_8.
\end{align*}

For $D_1$ we have that

$$
D_1=-\int_{0}^t Tr(\pa_t((\nabla \wn-\nabla\wm)(\zn-\zeta_\phi) (A\circ \xn-A_\phi))) d\tau=D_{11}+D_{12}+D_{13}
$$
where
\begin{align}
D_{11}=&-\int_{0}^t Tr(\pa_t(\nabla \wn-\nabla\wm)(\zn-\zeta_\phi) (A\circ \xn-A_\phi))d\tau,\nonumber\\
D_{12}=&-\int_{0}^t Tr((\nabla \wn-\nabla\wm)\pa_t(\zn-\zeta_\phi) (A\circ \xn-A_\phi))d\tau,\label{d1j}\\
D_{13}=&-\int_{0}^t Tr((\nabla \wn-\nabla\wm)(\zn-\zeta_\phi)\pa_t (A\circ \xn-A_\phi))d\tau.\nonumber
\end{align}
By applying lemma \ref{2.4} with $\ep=0$ we obtain
\begin{align*}
\hhn{D_{11}}{\frac{s+1}{2}}{-1}\leq \hhn{\pa_t(\nabla \wn-\nabla\wm)(\zn-\zeta_\phi) (A\circ \xn-A_\phi)}{\frac{s-1}{2}}{-1}.
\end{align*}
Now we use lemma \ref{2.5} to yield
\begin{align*}
\hhn{D_{11}}{\frac{s+1}{2}}{-1}\leq \hhn{\pa_t(\nabla \wn-\nabla\wm)}{\frac{s-1}{2}}{-1}\hhn{\zn-\zeta_\phi}{\frac{s-1}{2}}{1+\delta}\hhn{A\circ\xn-A_\phi}{\frac{s-1}{2}}{1+\delta}
\end{align*}
In addition lemmas \ref{composicionAfi}, \ref{bestiario} and \ref{2.3} and proceeding as in \eqref{comoQ} imply
\begin{align*}
\hhn{D_{11}}{\frac{s+1}{2}}{-1}\leq C[v_0]\htn{\wn-\wm}{s+1}\hhn{\zn-\zeta_\phi}{\frac{s-1}{2}}{1+\delta}.
\end{align*}
And  we can apply   lemma \ref{2.4} to get
\begin{align*}
\hhn{\zn-\zeta_\phi}{\frac{s-1}{2}}{1+\delta}&\leq \hhn{\pa_t\int_{0}^t(\zn-\zeta_\phi)}{\frac{s-1}{2}}{1+\delta}\hhn{\int_{0}^t(\zn-\zeta_\phi)}{\frac{s-1}{2}+\ep +1-\ep}{1+\delta}\\
&\leq CT^\ep \hhn{\zn-\zeta_\phi}{\frac{s-1}{2}+\ep}{1+\delta},
\end{align*}
for $0<\frac{s-1}{2}+\ep<1$. Thus, by lemmas \ref{composicionzfi} and \ref{bestiario},
$$\hhn{D_{11}}{\frac{s+1}{2}}{-1} \leq C[v_0]T^\ep\htn{\wn-\wm}{s+1}.$$

Next we bound $D_{12}$.  Indeed
\begin{align*}
&\hhn{D_{12}}{\frac{s+1}{2}}{-1}=\hhn{(\nabla\wn-\nabla\wm)\pa_t(\zn-\zeta_\phi) (A\circ\xn-A_\phi)}{\frac{s-1}{2}}{-1}\\
&\leq C \hhn{(\nabla\wn-\nabla\wm)\pa_t(\zn-\zeta_\phi)}{\frac{s-1}{2}}{-1}\hhn{A\circ \xn-A_\phi}{\frac{s-1}{2}}{1+\delta}\\ &\leq C[v_0]\hhn{\nabla\wn-\nabla\wm}{\frac{s-1}{2}}{1}\hhn{\pa_t(\zn-\zeta_\phi)}{\frac{s-1}{2}}{0}.
\end{align*}
Since, $\hhn{\pa_t\zn-\pa_t\zeta_\phi}{\frac{s-1}{2}}{0}\leq C[v_0] T^\ep$, and
$
\hhn{\nabla\wn-\nabla\wm}{\frac{s-1}{2}}{1}\leq C[v_0]T^\ep\htn{\wn-\wm}{s+1}$
we obtain a suitable estimate for $D_{12}.$

The term $D_{13}$ does not cause any difficulty and it can be handled as before.

For $D_2$ we just split $(A_\phi)_{ij}=A_{ij}+t \pa_k A_{ij}A_{kl}v_{0}^l$. For the terms coming from $A$ we just notice that $A$ does not depend on $t$. For the term coming from $t\pa_k A_{ij}A_{kl}v_{0}^l$ we use lemma \ref{littlelemma} and the fact that  $\pa_k A_{ij}A_{kl}v_{0}^l$ does not depend on $t$. For $D_3$ a similar argument holds after splitting $\zeta_\phi =t-\nabla(Av_0)$.

The estimate for $D_4$ follows the following steps,
\begin{align}
\hhn{D_{4}}{\frac{s+1}{2}}{-1}\leq&\hhn{\nabla \wm (\zn-\zm)(A\circ \xn-A_\phi)}{\frac{s+1}{2}}{-1}\label{d2}\\
\leq &\hhn{\int_{0}^t\pa_t (\nabla \wm (\zn-\zm)(A\circ \xn-A_\phi))}{\frac{s+1}{2}}{-1}\nonumber,
\end{align}
and therefore
\begin{align*}
\hhn{D_{4}}{\frac{s+1}{2}}{-1}\leq & \hhn{\pa_t\nabla\wm(\zn-\zm)(A\circ\xn-A_\phi)}{\frac{s-1}{2}}{-1}\\
&+  \hhn{\nabla\wm\pa_t(\zn-\zm)(A\circ\xn-A_\phi)}{\frac{s-1}{2}}{-1}\\
&+\hhn{\nabla\wm(\zn-\zm)\pa_t(A\circ\xn-A_\phi)}{\frac{s-1}{2}}{-1}.
\end{align*}
Each addend is estimated in a different manner as follows
\begin{align*}
\hhn{D_{4}}{\frac{s+1}{2}}{-1}\leq & C[v_0]\hhn{\nabla \pa_t\wm}{\frac{s-1}{2}}{-1}\hhn{\zn-\zm}{\frac{s-1}{2}}{1+\delta}\hhn{A\circ\xn-A_\phi}{\frac{s-1}{2}}{1+\delta}  \\
+&C[v_0] \hhn{\nabla\wn}{\frac{s-1}{2}}{1}\hhn{\pa_t (\zn-\zm)}{\frac{s-1}{2}}{0}\hhn{A\circ\xn-A_\phi}{\frac{s-1}{2}}{1+\delta}\\
+ & C[v_0]\hhn{\nabla\wn}{\frac{s-1}{2}}{1}\hhn{\zn-\zm}{\frac{s-1}{2}}{0}\hhn{\pa_t (A\circ\xn-A_\phi)}{\frac{s-1}{2}}{1+\delta},
\end{align*}
and we have already bounded every term in here.

For $D_5$ we proceed as for $D_2$ and $D_3$.

We will give full detail for $D_6$.

$$
D_6=-\int_{0}^t Tr\pa_t(\nabla \wm (\zm-\zeta_\phi) (A\circ \xn-A\circ \xm))d\tau=d_{31}+d_{32}+d_{33}
$$
where
\begin{align}
D_{61}=&-\int_0^t Tr(\pa_t\nabla\wm (\zm-\zeta_\phi)(A\circ \xn-A\circ\xm))d\tau \nonumber\\
D_{62}=&-\int_0^t Tr(\nabla\wm \pa_t(\zm-\zeta_\phi) (A\circ \xn-A\circ\xm)) d\tau \label{d3j}\\
D_{63}=&-\int_0^t Tr(\nabla \wm (\zm-\zeta_\phi) \pa_t(A\circ \xn-A\circ\xm))d\tau. \nonumber
\end{align}

The estimates of these three terms follow similar steps as those one for $D_{1}$ and $D_4$. First we apply  lemma \ref{2.4} with $\ep=0$ since $0<\frac{s-1}{2}<1$ to find
\begin{align*}
&\hhn{D_{61}}{\frac{s+1}{2}}{-1}\leq \hhn{\pa_t\nabla\wm (\zm-\zeta_\phi)(A\circ \xn-A\circ\xm)}{\frac{s-1}{2}}{-1}\\
 &\leq C[v_0]\hhn{\nabla\pa_t\vm}{\frac{s-1}{2}}{-1}\hhn{\zm-\zeta_\phi}{\frac{s-1}{2}}{1+\delta})\hhn{A\circ \xn-A\circ\xm}{\frac{s-1}{2}}{1+\delta}\\
&\leq  C[v_0]\hhn{A\circ\xn-A\circ\xm}{\frac{s-1}{2}}{1+\delta}\leq C[v_0]\hhn{\xn-\xm}{\frac{s-1}{2}}{1+\delta}.
\end{align*}
Above  lemma \ref{2.5},  lemma \ref{2.6}, lemma \ref{composicionzfi} and lemma \ref{composicionAfi}. In addition
\begin{align*}
\hhn{D_{61}}{\frac{s+1}{2}}{-1}\leq & C[v_0]\hhn{\pa_t\int_{0}^t\xn-\xm}{\frac{s-1}{2}}{1+\delta}\leq \hhn{\int_{0}^t\xn- \xm}{\frac{s-1}{2}+\ep+1-\ep}{1+\delta}\\
 \leq & CT^\ep \hhn{\xn- \xm}{\frac{s-1}{2}+\ep}{1+\delta}\leq CT^\ep \fn{\xn-\xm},
\end{align*}
for $0<\frac{s-1}{2}+\ep<1$. The estimate for $D_{62}$ is quite similar. Indeed,  lemma \ref{2.4} with $\ep=0$ to obtain
\begin{align*}
\hhn{D_{62}}{\frac{s+1}{2}}{-1}\leq&C\hhn{\nabla\wm \pa_t(\zm-\zeta_\phi)(A\circ \xn-A\circ\xm)}{\frac{s-1}{2}}{-1} \\
\leq C[v_0]&\hhn{\nabla \wn}{\frac{s-1}{2}}{1}\hhn{\pa_t(\zm-\zeta_\phi)}{\frac{s-1}{2}}{0}\hhn{A\circ \xn-A\circ\xm}{\frac{s-1}{2}}{1+\delta}.
\end{align*}
Above we apply  lemma \ref{2.5} and the rest of the estimate follows the same steps. Next  we can make use of  \ref{2.4} with $\ep=0$ since $0<\frac{s-1}{2}<1$ to obtain
\begin{align*}
&\hhn{D_{63}}{\frac{s+1}{2}}{-1}\leq C[v_0]\hhn{\nabla\wn}{\frac{s-1}{2}}{1}\hhn{\zm-\zeta_\phi}{\frac{s-1}{2}}{1+\delta}\hhn{\pa_t (A\circ\xn-A\circ\xm)}{\frac{s-1}{2}}{0},
\end{align*}
and therefore, it is enough to estimate $\hhn{\pa_t (\xn-\xm)}{\frac{s-1}{2}}{0}$. As we did before we obtain
\begin{align*}
\hhn{\pa_t (\xn-\xm)}{\frac{s-1}{2}}{0}\leq & CT^\ep \hhn{\pa_t(\xn-\xm)}{\frac{s-1}{2}+\ep}{0} \\\leq& CT^\ep \fn{\xn-\xm}.
\end{align*}

To estimate $D_7$ we just split $\zeta_\phi=\I-t\nabla(Av_0)$ and use lemma \ref{littlelemma}.  Finally we estimate $D_8$ by using lemma \ref{littlelemma}.

\underline{P2.3. Estimate for $h^{(n)}-h^{(n-1)}$:}\\

We split using \eqref{hsplit} and decompose further $h^{(n)}_v-h^{(n-1)}_v=d_1+d_2+d_3+d_4$ where the differences are given by
\begin{equation}\label{diferenciahv}
	d_1=\nabla \vn(\zn-\zm)\gradj\xn n_0, \quad
	d_2=\nabla \vn\zm (\gradj\xn -\gradj\xm) n_0,
\end{equation}
\begin{equation*}
	d_3=\nabla (\vn-\vm)(\zm-\I)\gradj\xm n_0, \quad
	d_4=\nabla (\vn-\vm)(\gradj\xm -\I) n_0.
\end{equation*}
We estimate as follows. For $d_1$ we find:
\begin{align*}
	\lhb{d_1}{s-\frac12}&\leq  C \lhb{\nabla\vn}{s-\frac12} \lib{\zn-\zm}{s-\frac12}\lib{\grad\xn}{s-\frac12}\\
	&\leq C \lhn{\vn}{s+1} \li{\zn-\zm}{s}\li{\xn}{s+1}\\
	&\leq C[v_0]\li{\xn-\xm}{s+1}\leq C[v_0]T^\frac14 \fn{\xn-\xm}.
\end{align*}
The time norm needs the splitting $\vn=\wn+v_0+t\psi$ which gives
\begin{align*}
	\hlb{d_1}{\frac{s}2-\frac14}&\leq  C \Big(\hlb{\nabla \wn}{\frac{s}2-\frac14}+\|v_0\|_{H^{3/2}}+\hlb{t\nabla \psi}{\frac{s}2-\frac14}\Big)\hhb{(\zn-\zm)\gradj\xn}{\frac{s}2-\frac14}{\frac12+\delta}\\
	&\leq C \Big(\htn{\wn}{s+1}+C[v_0]\Big)\hhn{\zn-\zm}{\frac{s}2-\frac14}{1+\delta}\Big(
	\hhn{\gradj\xn-\I}{\frac{s}2-\frac14}{1+\delta}+1\Big)\\
	&\leq C[v_0]\hhn{\xn-\xm}{\frac{s}2-\frac14}{2+\delta}\leq C[v_0]T^\epsilon \fn{\xn-\xm}.
\end{align*}
Above we have used \eqref{tnYtraza} together with Lemmas \ref{2.4}, \ref{2.6} and \ref{composiciondizeta}. Similarly
\begin{align*}
	\lhb{d_2}{s-\frac12}&\leq  C \lhb{\nabla \vn}{s-\frac12} \lib{\zm}{s-\frac12}\lib{\grad\xn-\grad\xm}{s-\frac12}\\
	&\leq C[v_0]T^\frac14 \fn{\xn-\xm},
\end{align*}
and
\begin{align*}
	\hlb{d_2}{\frac{s}2-\frac14}&\leq  C \Big(\hlb{\nabla \wn}{\frac{s}2-\frac14}\!+\!\|v_0\|_{H^{3/2}}\!+\!\hlb{t\nabla \psi}{\frac{s}2-\frac14}\Big)\hhb{\zm(\gradj\xn-\gradj\xm)}{\frac{s}2-\frac14}{\frac12+\delta}\\
	&\leq C[v_0]\Big(\hhn{\zn-\I}{\frac{s}2-\frac14}{1+\delta}+1\Big)
	\hhn{\gradj\xn-\gradj\xm}{\frac{s}2-\frac14}{1+\delta}
	\\
	&\leq C[v_0]T^\epsilon \fn{\xn-\xm}.
\end{align*}
For $d_3$ we find:
\begin{align*}
	\lhb{d_3}{s-\frac12}&\leq  C \lhb{\nabla(\vn-\vm)}{s-\frac12} \lib{\zm-\I}{s-\frac12}\lib{\grad\xn}{s-\frac12}\\
	&\leq C \lhn{\vn-\vm}{s+1} \li{\zm-\I}{s}\li{\xn}{s+1}\\
	&\leq C[v_0]T^\frac14 \htn{\vn-\vm}{s+1},
\end{align*}
and
\begin{align*}
	\hlb{d_3}{\frac{s}2-\frac14}&\leq  C \hlb{\nabla (\vn-\vm)}{\frac{s}2-\frac14}\hhb{\zm-\I}{\frac{s}2-\frac14}{\frac12+\delta}\Big(\hhb{\gradj\xn-\I}{\frac{s}2-\frac14}{\frac12+\delta}+1\Big)\\
	&\leq C \htn{\vn-\vm}{s+1}\hhn{\zm-\I}{\frac{s}2-\frac14}{1+\delta}
	\Big(\hhn{\xn-\al}{\frac{s}2-\frac14}{2+\delta}+1\Big)\\
	&\leq C[v_0]T^\epsilon \htn{\vn-\vm}{s+1}.
\end{align*}
Above we have used \eqref{tnYtraza} together with Lemmas \ref{2.4}, \ref{bestiario} and \ref{composicionz}. Similarly
\begin{align*}
	\lhb{d_4}{s-\frac12}&\leq  C[v_0]T^\frac14 \htn{\vn-\vm}{s+1},\,\mbox{and}\,\hlb{d_4}{\frac{s}2-\frac14}\leq C[v_0]T^\epsilon \htn{\vn-\vm}{s+1}.
\end{align*}
We are done with $h_{v}^{(n)}-h_{v}^{(n-1)}$. In order to deal with $h_{v*}^{(n)}-h_{v*}^{(n-1)}=d_1^*+...+d_8^*$ we use the splitting where
\begin{equation*}
	d_1^*=(\nabla \vn(\zn-\zm) A\circ \xn)^* A^{-1}\circ \xn \gradj\xn n_0,
\end{equation*}
\begin{equation*}
	d_2^*=(\nabla \vn\zm (A\circ \xn-A\circ \xm))^* A^{-1}\circ \xn \gradj\xn n_0,
\end{equation*}
\begin{equation*}
	d_3^*=(\nabla \vn\zm A\circ \xm)^* (A^{-1}\circ \xn-A^{-1}\circ \xm))  \gradj\xn n_0,
\end{equation*}
\begin{equation}\label{diferenciahves}
	d_4^*=(\nabla \vn\zm A\circ\xm)^*A^{-1}\circ\xm(\gradj\xn-\gradj\xm) n_0,
\end{equation}
\begin{equation*}
	d_5^*=(\nabla (\vn-\vm)(\zm-\I) A\circ \xm)^* A^{-1}\circ \xm \gradj\xm n_0,
\end{equation*}
\begin{equation*}
	d_6^*=(\nabla (\vn-\vm)(A\circ \xm-A))^* A^{-1}\circ \xm \gradj\xm n_0,
\end{equation*}
\begin{equation*}
	d_7^*=(\nabla (\vn-\vm) A)^* (A^{-1}\circ \xm -A^{-1})\gradj\xm n_0,
\end{equation*}
\begin{equation*}
	d_8^*=(\nabla (\vn-\vm) A)^*  A^{-1}(\gradj\xn-\I) n_0.
\end{equation*}
Then
\begin{align*}
	\lhb{d_1^*}{s-\frac12}&\leq  C \lhb{\nabla\vn}{s-\frac12} \lib{\zn-\zm}{s-\frac12}\lib{\xn}{s-\frac12}^2\lib{\grad\xn}{s-\frac12}\\
	&\leq C \lhn{\vn}{s+1} \li{\zn-\zm}{s}\li{\xn}{s+1}^3\\
	&\leq C[v_0]\li{\xn-\xm}{s+1}\leq C[v_0]T^\frac14 \fn{\xn-\xm}.
\end{align*}
For the norm in time we split as follows
\begin{align*}
	\hlb{d_1^*}{\frac{s}2-\frac14}&\leq  C \Big(\hlb{\nabla \wn}{\frac{s}2-\frac14}+\|v_0\|_{H^{3/2}}+\hlb{t\nabla \psi}{\frac{s}2-\frac14}\Big)\\
	&\qquad\qquad \times \hhb{((\zn-\zm) A\circ \xn)^* A^{-1}\circ \xn \gradj\xn}{\frac{s}2-\frac14}{\frac12+\delta}\\
	&\leq C[v_0]\hhn{((\zn-\zm) A\circ \xn)^* A^{-1}\circ \xn \gradj\xn}{\frac{s}2-\frac14}{1+\delta}.
\end{align*}
With a decomposition similar to \eqref{nvchqpem}, it is possible to find
$$
\hlb{d_1^*}{\frac{s}2-\frac14}\leq  C[v_0]\hhn{\xn-\xm}{\frac{s}2-\frac14}{2+\delta}\leq C[v_0]T^\epsilon \fn{\xn-\xm},
$$
with the use of Lemmas \ref{2.4} and \ref{bestiario}. In an analogous manner
\begin{align*}
	\lhb{d_2^*}{s-\frac12}&\leq  C[v_0] \lhb{\nabla \vn}{s-\frac12} \lib{\zm}{s-\frac12}\lib{\xn-\xm}{s-\frac12}\li{\xn}{s+1}^2\\
	&\leq C[v_0]T^\frac14 \fn{\xn-\xm},
\end{align*}
and
\begin{align*}
	\hlb{d_2^*}{\frac{s}2-\frac14}&\leq  C[v_0] \hhb{((\zm) (A\circ \xn-A\circ \xm)^* A^{-1}\circ \xn \gradj\xn}{\frac{s}2-\frac14}{\frac12+\delta}\\
	&\leq  C[v_0]\hhn{\xn-\xm}{\frac{s}2-\frac14}{1+\delta}\leq C[v_0]T^\epsilon \fn{\xn-\xm}.
\end{align*}
We proceed for $d^*_3$ and $d^*_4$ as for $d^*_2$ to find
\begin{align*}
	\lhb{d^*_3}{s-\frac12}+\lhb{d^*_4}{s-\frac12}&\leq  C[v_0]T^\frac14 \fn{\xn-\xm},\\
	\hlb{d^*_3}{\frac{s}2-\frac14}+\hlb{d^*_4}{\frac{s}2-\frac14}&\leq C[v_0]T^\epsilon \fn{\xn-\xm}.
\end{align*}
A similar approach is used to find
\begin{align*}
	\lhb{d^*_5}{s-\frac12}&\leq  C[v_0] \lhb{\grad(\vn-\vm)}{s-\frac12}\lib{\zm-\I}{s-\frac12}\Big(\li{\xm}{s+1}^3+1\Big)\\
	&\leq C[v_0] \lhn{\vn-\vm}{s+1} \li{\zm-\I}{s}\\
	&\leq C[v_0]T^\frac14\htn{\wn-\wm}{s+1},
\end{align*}
and
\begin{align*}
	\hlb{d^*_5}{\frac{s}2-\frac14}&\leq  C \hlb{\nabla (\vn\!-\!\vm)}{\frac{s}2-\frac14}\hhb{(\zm\!-\!\I) A\circ \xm)^* A^{-1}\circ \xm \gradj\xm}{\frac{s}2-\frac14}{\frac12+\delta}\\
	&\leq C[v_0] \htn{\vn-\vm}{s+1}\hhn{\zm-\I}{\frac{s}2-\frac14}{1+\delta}\\
	&\leq C[v_0]T^\epsilon \htn{\wn-\wm}{s+1}.
\end{align*}
It is possible to distribute the norms in a similar manner as in $d_5^*$ to conclude that
\begin{align*}
	\lhb{d^*_6}{s-\frac12}+\lhb{d^*_7}{s-\frac12}+\lhb{d^*_8}{s-\frac12}&\leq  C[v_0]T^\frac14 \htn{\wn-\wm}{s+1},\\
	\hlb{d^*_6}{\frac{s}2-\frac14}+\hlb{d^*_7}{\frac{s}2-\frac14}+\lhb{d^*_8}{s-\frac12}&\leq C[v_0]T^\epsilon \htn{\wn-\wm}{s+1}.
\end{align*}

Next, we deal with $h^{(n)}_q-h^{(n-1)}_q=d_1^q+d_2^q+d_3^q+d_4^q$ where
\begin{equation*}
	d_1^q=\qm (A^{-1}\circ\xm-A^{-1}\circ\xn)\gradj\xm n_0,
\end{equation*}
\begin{equation}\label{diferenciahq}
	d_2^q=\qm A^{-1}\circ \xn (\gradj\xm-\gradj\xn) n_0,
\end{equation}
\begin{equation*}
	d_3^q=(\qm_w-\qn_w) (A^{-1}\circ \xn-A^{-1}) \gradj\xn n_0,
\end{equation*}
\begin{equation*}
	d_4^q=(\qm_w-\qn_w) A^{-1} (\gradj\xn-\I) n_0.
\end{equation*}
We start as follows
\begin{align*}
	\lhb{d_1^q}{s-\frac12}&\leq  C \lhb{\qm}{s-\frac12}\lib{\xn-\xm}{s-\frac12}\lib{\grad\xm}{s-\frac12}\\
	&\leq C[v_0]
	T^\frac14 \fn{\xn-\xm},
\end{align*}
and
\begin{align*}
	\hlb{d_1^q}{\frac{s}2-\frac14}&\leq  C \Big(\hlb{\qm_w}{\frac{s}2-\frac14}+|q_{\phi}|_{L^2}\Big)\hhb{(A^{-1}\circ\xm-A^{-1}\circ\xn)\gradj\xm}{\frac{s}2-\frac14}{\frac12+\delta}\\
	&\leq C[v_0]T^\epsilon \fn{\xn-\xm},
\end{align*}
following before estimates. Likewise, it is possible to find next
\begin{align*}
	\lhb{d_2^q}{s-\frac12}&\leq  C \lhb{\qm}{s-\frac12}\lib{\xn}{s-\frac12}\lib{\grad\xn-\grad\xm}{s-\frac12}\\
	&\leq C[v_0]T^\frac14 \fn{\xn-\xm},
\end{align*}
and
\begin{align*}
	\hlb{d_2^q}{\frac{s}2-\frac14}&\leq  C \Big(\hlb{\qm_w}{\frac{s}2-\frac14}+|q_{\phi}|_{L^2}\Big)\hhb{A^{-1}\circ \xn (\gradj\xm-\gradj\xn)}{\frac{s}2-\frac14}{\frac12+\delta}\\
	&\leq C[v_0]T^\epsilon \fn{\xn-\xm}.
\end{align*}
The next term provides
\begin{align*}
	\lhb{d_3^q}{s-\frac12}&\leq  C \lhb{\qn_w-\qm_w}{s-\frac12}\lib{\xn-\al}{s+\frac12}\lib{\grad\xm}{s-\frac12}\\
	&\leq C[v_0]T^\frac14 \|\qn_w-\qm_w\|_{H_{pr \, (0)}^{ht,s}},
\end{align*}
and
\begin{align*}
	\hlb{d_3^q}{\frac{s}2-\frac14}&\leq  C \hlb{\qn_w-\qm_w}{\frac{s}2-\frac14}\hhb{(A^{-1}\circ \xn-A^{-1}) \gradj\xn}{\frac{s}2-\frac14}{\frac12+\delta}\\
	&\leq C[v_0]T^\epsilon\|\qn_w-\qm_w\|_{H_{pr \, (0)}^{ht,s}}.
\end{align*}
As for $d_3^q$ we find
\begin{align*}
	\lhb{d_4^q}{s-\frac12}\leq C[v_0]T^\frac14\|\qn_w-\qm_w\|_{H_{pr \, (0)}^{ht,s}},\quad\hlb{d_4^q}{\frac{s}2-\frac14}\leq  C[v_0]T^\epsilon \|\qn_w-\qm_w\|_{H_{pr \, (0)}^{ht,s}}.
\end{align*}

\bibliographystyle{abbrv}
\bibliography{references}

\begin{tabular}{ll}

\textbf{Angel Castro} &  \\
{\small Departamento de Matem\'aticas} & \\
{\small Facultad de Ciencias} & \\
{\small Universidad Aut\'onoma de Madrid} & \\
{\small Campus Cantoblanco UAM, 28049 Madrid} & \\
\\
{\small Instituto de Ciencias Matem\'aticas, CSIC-UAM-UC3M-UCM} & \\
{\small C/ Nicol\'as Cabrera 13-15} & \\
{\small Campus Cantoblanco UAM, 28049 Madrid} & \\
{\small Email: angel\_castro@icmat.es} & \\
   & \\
\textbf{Diego C\'ordoba} &  \textbf{Charles Fefferman}\\
{\small Instituto de Ciencias Matem\'aticas} & {\small Department of Mathematics}\\
{\small Consejo Superior de Investigaciones Cient\'ificas} & {\small Princeton University}\\
{\small C/ Nicol\'{a}s Cabrera, 13-15} & {\small 1102 Fine Hall, Washington Rd, }\\
{\small Campus Cantoblanco UAM, 28049 Madrid} & {\small Princeton, NJ 08544, USA}\\
{\small Email: dcg@icmat.es} & {\small Email: cf@math.princeton.edu}\\
 & \\
\textbf{Francisco Gancedo} &  \textbf{Javier G\'omez-Serrano}\\
{\small Departamento de An\'alisis Matem\'atico} \& IMUS & {\small Department of Mathematics}\\
{\small Universidad de Sevilla} & {\small Princeton University}\\
{\small C/ Tarfia, s/n } & {\small 610 Fine Hall, Washington Rd,} \\
{\small Campus Reina Mercedes, 41012 Sevilla}  & {\small Princeton, NJ 08544, USA} \\
{\small Email: fgancedo@us.es} & {\small Email: jg27@math.princeton.edu}\\
\end{tabular}

\end{document}